\documentclass[twoside,a4paper,11pt,centertags]{amsart}

\usepackage[T1]{fontenc}
\usepackage[utf8]{inputenc}
\usepackage{subfiles}
\usepackage{multicol}

\makeatletter
\usepackage[section]{placeins}

\AtBeginDocument{%
	\expandafter\renewcommand\expandafter\subsection\expandafter{%
		\expandafter\@fb@secFB\subsection
	}%
}
\makeatother
\makeatletter
\@namedef{subjclassname@2020}{%
	\textup{2020} Mathematics Subject Classification}
\makeatother

\usepackage{amsfonts,amssymb,amsmath,amsthm}
\usepackage{mathrsfs}
\usepackage{nicefrac}
\usepackage{latexsym}
\usepackage{mathtools}
\usepackage[noadjust]{cite}
\usepackage{paralist}
\usepackage{longtable}
\usepackage{float}
\usepackage{pdfsync}
\usepackage{tikz,tikz-cd}
\usetikzlibrary{patterns}
\usepackage{hhline}
\usepackage{import}
\usepackage{dsfont}
\usepackage{bigints}

\usepackage{enumitem}

\usepackage[colorlinks=true,citecolor=magenta,linkcolor=cyan,
]{hyperref}
\AtBeginDocument{}

\numberwithin{equation}{section}

\theoremstyle{plain}
\newtheorem{thm}{Theorem}[section]
\newtheorem{lem}[thm]{Lemma}
\newtheorem{cor}[thm]{Corollary}
\newtheorem{prop}[thm]{Proposition}

\theoremstyle{definition}
\newtheorem{defn}[thm]{Definition}
\newtheorem{rem}[thm]{Remark}

\usepackage{color}

\renewcommand{\epsilon}{\varepsilon}

\newcommand{\IC}{\mathbb{C}}

\newcommand{\N}{\mathbb{N}}

\newcommand{\R}{\mathbb{R}}

\newcommand{\Z}{\mathbb{Z}}

\newcommand{\cC}{\mathcal{C}}
\newcommand{\cT}{\mathcal{T}}
\newcommand{\cD}{\mathcal{D}} 
\newcommand{\cF}{\mathcal{F}} 
\newcommand{\cG}{\mathcal{G}}

\newcommand{\cK}{\mathcal{K}} 
\newcommand{\cL}{\mathcal{L}}

\newcommand{\cP}{\mathcal{P}}

\newcommand{\cR}{\mathcal{R}}
\newcommand{\cS}{\mathcal{S}} 

\newcommand{\cU}{\mathcal{U}}

\renewcommand{\S}{\mathbb{S}}

\newcommand{\dense}{\cC}

\newcommand{\loc}{\operatorname{loc}}
\renewcommand{\L}{\operatorname{L}} 
\newcommand{\C}{\operatorname{C}} 
\renewcommand{\H}{\operatorname{H}} 
\newcommand{\W}{\operatorname{W}}
\newcommand{\B}{\operatorname{B}}
\newcommand{\A}{\operatorname{A}}

\newcommand{\X}{\operatorname{X}} 
\newcommand{\Y}{\operatorname{Y}}
\newcommand{\zZ}{\operatorname{Z}}
\newcommand{\U}{\operatorname{U}}
\newcommand{\I}{\operatorname{I}}

\newcommand{\frT}{\mathfrak{T}}
\newcommand{\frL}{\mathfrak{L}}
\newcommand{\sem}{\cR}

 \newcommand{\id}{\operatorname{Id}} 
 
\DeclareRobustCommand{\Hdot}{\dot{\H}\protect{\vphantom{H}}
} 
\DeclareRobustCommand{\Wdot}{\dot{\W}\protect{\vphantom{W}}} 
\DeclareRobustCommand{\Bdot}{\dot{\B}\protect{\vphantom{B}}} 
\DeclareRobustCommand{\Adot}{\dot{\A}\protect{\vphantom{A}}} 

\DeclareRobustCommand{\Xdot}{\dot{\X}\protect{\vphantom{X}}} 
\DeclareRobustCommand{\Ydot}{\dot{\Y}\protect{\vphantom{Y}}} 
\DeclareRobustCommand{\Zdot}{\dot{\zZ}\protect{\vphantom{Z}}} 
 
\DeclareRobustCommand{\Udot}{\dot{\U}\protect{\vphantom{U}}}
 
\DeclareRobustCommand{\cFdot}{\dot{\cF}\protect{\vphantom{\cF}}} 

\DeclareRobustCommand{\cLdot}{\dot{\cL}\protect{\vphantom{\cL}}} 
\DeclareRobustCommand{\cTdot}{\dot{\cT}\protect{\vphantom{\cT}}}
\DeclareRobustCommand{\cUdot}{\dot{\cU}\protect{\vphantom{\cU}}}
\DeclareRobustCommand{\cGdot}{\dot{\cG}\protect{\vphantom{\cG}}}

\newcommand{\eps}{\epsilon} 

\def\angle#1#2{\langle #1,#2 \rangle} 

\newcommand\supp{\operatorname{supp}}

\renewcommand{\iint}{\int_{}\kern-.34em \int} 
\renewcommand{\iiint}{\iint_{}\kern-.34em \int} 

\newcommand{\dx}{\, \mathrm{d} x}
\newcommand{\dd}{\, \mathrm{d} }
\newcommand{\dv}{\, \mathrm{d} v}

\newcommand{\ds}{\, \mathrm{d} s}
\newcommand{\dr}{\, \mathrm{d} r}
\newcommand{\dt}{\, \mathrm{d} t}

\newcommand{\abs}[1]{\left\lvert#1\right\rvert}

\newcommand{\norm}[1]{\left\|#1\right\|}
\newcommand{\homd}{\mathsf{K}}

\mathtoolsset{showmanualtags}

\title[Kinetic Sobolev Spaces]{Kinetic Sobolev Spaces}

\author{Pascal Auscher}
\address{Universit\'e Paris-Saclay, CNRS, Laboratoire de Math\'{e}matiques d'Orsay, 91405 Orsay, France\\ and 
French-Australian  Mathematical Sciences and Interactions, Australian National University-CNRS, Canberra, ACT, 2601, Australia}
\email{pascal.auscher@universite-paris-saclay.fr}

\author{Lukas Niebel}
\address{Institut f\"ur Analysis und Numerik,  Westf\"alische Wilhelms-Universit\"at M\"unster\\
Orl\'eans-Ring 10, 48149 M\"unster, Germany.}
\email{lukas.niebel@uni-muenster.de}

\date{March 17, 2026}
\thanks{The authors thank Cyril Imbert for discussions. The second author is funded by the Deutsche Forschungsgemeinschaft (DFG, German Research Foundation) under Germany's Excellence Strategy EXC 2044/2--390685587, Mathematics M\"unster: Dynamics--Geometry--Structure. A CC-BY 4.0 \url{https://creativecommons.org/licenses/by/4.0/} public copyright license has been applied by the authors to the present document and will be applied to all subsequent versions up to the Author Accepted Manuscript arising from this submission.}



\begin{document}
\allowdisplaybreaks
\begin{abstract}
We define and study homogeneous kinetic Sobolev spaces adapted to the Kolmogorov equation. We consider both local and non-local diffusion.  
The spaces are built from the Lebesgue spaces $\L^p$ for all integrability exponents $p \in (1,\infty)$ with regularity assumptions in the transport and diffusive directions according to the scaling of the Kolmogorov equation.
The regularity scale accommodates weak and strong solutions. 
We prove that the proposed spaces satisfy sharp embeddings quantifying the transfer-of-regularity \`a la Bouchut--H\"ormander, continuity-in-time in the spirit of Lions and the gain-of-integrability of Sobolev and Hardy--Littlewood--Sobolev type. 
A core tool are
mapping properties of the Kolmogorov operator, given by the fundamental solution, established between anisotropic homogeneous Sobolev spaces.
To achieve this, we prove $\L^p$ boundedness of related singular integral operators, for which
 we deduce novel kernel estimates by a Littlewood--Paley decomposition and geometric considerations. 
Moreover, we provide a new uniqueness criterion which allows us to show well-posedness of the Cauchy problem. 
\end{abstract}

\maketitle

\tableofcontents

\allowdisplaybreaks

\newpage

\section{Introduction}
\label{sec:intro}

In the same way the Sobolev scales associated with the Laplacian (and its fractional powers) and the corresponding heat operators are well adapted to the study of elliptic and parabolic problems, respectively, we may wonder what the analogous scales for kinetic problems are.
Since such problems involve both a transport field and a degenerate diffusion operator, it is natural to introduce scales of (homogeneous) spaces, which we call \emph{kinetic Sobolev spaces} associated with the local or non-local Kolmogorov equation 
with constant diffusion coefficient:
\begin{equation} \label{eq:kol}
		(\partial_t + v \cdot \nabla_x) f + (-\Delta_v)^{\beta}f = S  
\end{equation}
posed on the whole time line $t \in \R$ and the position-velocity full space $(x,v) \in \R^{2d}$, $d \ge 1$. 
Here, $\beta \in (0,1]$, $f = f(t,x,v) \colon \R^{1+2d} \to \R$ is a solution and $S = S(t,x,v) \colon \R^{1+2d} \to \R$ is a source term.
We speak of local diffusion in the case $\beta = 1$ and of non-local diffusion when $\beta \in (0,1)$. 
This partial differential equation is a central object in the theory of kinetic equations. 
It serves as a simplified model for linearisations of the Landau equation ($\beta=1$) and the Boltzmann equation 
($\beta\in(0,1)$).
Moreover, it generalises the heat equation (special case $f = f(t,v)$) and the Laplace equation (when $f = f(v)$).  

To start with, we recall various notions of solutions. 
Distributional solutions are distributions satisfying \eqref{eq:kol} in the distributional sense, that is, when all differential operators act on the test function only: this notion is what we use here to obtain the most general results. 
We emphasise that special care has to be taken when applying fractional diffusion operators to a distribution. 
A (variational/energy) weak solution is a distributional solution for which the diffusive term $(-\Delta_v)^{\beta}$ splits half on the solution and half on the test function. This may impose further restrictions.
A strong solution is a distributional solution (often a function) which satisfies \eqref{eq:kol}  with $S\in \L^p_{t,x,v}$ in an (a.e.) pointwise sense. It is clear that these three notions require adequate a priori regularity on $f$ itself.

\bigskip  

The equation \eqref{eq:kol} involves the free transport operator $\partial_t+v\cdot \nabla_x$ and the diffusion operator $(-\Delta_v)^{\beta}$ in the velocity variable. They do not commute. However, it is well-known from H\"ormander's work, at least in the local case, that the operator is hypo-elliptic and has excellent regularity properties. 
H\"ormander observed that one needs  spaces that quantify regularity and integrability of $(\partial_t + v \cdot \nabla_x) f$ 
and 
$(-\Delta_v)^{\beta}f$.
The general scheme is to define spaces of distributions  $(\X,\zZ)$ such that  if $\cL=\{ f \in \cD'(\R^{1+2d})\, ; \,  f \in \X \ \& \ (\partial_t + v \cdot \nabla_x)f \in \zZ\}$ then the differential operator $(\partial_t + v \cdot \nabla_x) + (-\Delta_v)^{\beta}$ is an isomorphism $\cL \to \zZ$. Moreover, its inverse should be the Kolmogorov operator given by the fundamental solution. If this operator maps $\zZ$ to a subspace $\Y$ of $\X$, then the gain from $\X$ to $\Y$ yields an embedding $\cL \hookrightarrow \Y$. As a consequence, it gives a priori regularity for solutions in $\X$ with source in $\zZ$. We stress that although  tied to the Kolmogorov equation, $\cL$ is \emph{not} a space of solutions. Having a good understanding  of $\cL$ can thus be used for other kinetic equations with non constant coefficients.

 As for Sobolev spaces, we require that $\cL$ describes a scale of spaces  distinguished by their  regularity and  integrability parameters. Having this in mind, three questions are central towards an explicit definition given fixed regularity and integrability parameters: 

\begin{enumerate}
	\item[\textbf{1.}]   What is the  largest distributional space in which distributional solutions are unique?  In other words, what is the largest choice for $\X$ ?
	\item[\textbf{2.}] On which  distributional spaces is the Kolmogorov operator bounded. In other words, find $\zZ$ and $\Y\subset \X$ such that  $\zZ\to \Y$ boundedness holds.
	\item[\textbf{3.}] Which space $\zZ$ is mapped by the Kolmogorov operator  into distributional solutions of \eqref{eq:kol}? In other words, when does $\zZ\to \cL$ boundedness hold? 
\end{enumerate}

\bigskip

Regarding the \textbf{first question}, a natural choice given the degenerate diffusion is to measure regularity only in the velocity variable (but this is not the only one) as the regularity in the transport direction is a consequence of the solution property. 
We develop an involved distributional argument that deeply uses the underlying kinetic structure to prove uniqueness of distributional solutions to \eqref{eq:kol} assuming only $f \in \L^p_{t,x}\Hdot^{\gamma,p}_{\vphantom{t}v}$, $p \in (1,\infty)$, $-\infty< \gamma<2\beta +\frac{d}{p}$.
The latter space is the homogeneous Sobolev space in the velocity variable $v$ with regularity $\gamma$ and $\L^p$-integrability in the other variables. 
The upper bound on $\gamma$ is there to ensure that the equation makes sense and will be explained later.

\bigskip

Concerning the \textbf{second question}, let us mention that there is a fundamental solution, with an explicit expression in physical variables when $\beta=1$ and a formula on the Fourier side in all cases $\beta \in (0,1]$. 
To this fundamental solution, we may associate the Kolmogorov operator, which is the convolution with respect to kinetic translation of the source term with the fundamental solution. 
A natural source space on which the Kolmogorov operator can act is $\L^p_{t,x}\Hdot^{\gamma-2\beta,p}_{\vphantom{t}v}$ with an expected range in $\L^p_{t,x}\Hdot^{\gamma,p}_{\vphantom{t}v}$ (and this is known for $\gamma=2\beta$),  but this choice is not the largest possible one. In fact, from the H\"ormander--Bouchut theory, one expects a transfer-of-regularity from the diffusive $v$- variable to the $x$-variable to gain $\L^p_{t,v}\Hdot_{\vphantom{t}x}^{\frac{\gamma}{2\beta+1},p}$. The exponent $\frac{\gamma}{2\beta+1}$ comes from scaling considerations 
Thus, dually, we want to put $\L^p_{t,v}\Hdot_{\vphantom{t}x}^{\frac{\gamma-2\beta}{2\beta+1},p}$ in the source space.  This leads us to introduce anisotropic Sobolev spaces in the $(x,v)$ variables. 

\noindent Furthermore, solutions should be in an anisotropic Besov space at any fixed time with, hopefully, continuity with respect to time, so one adds source terms being integrable in time and valued in this Besov space as a dual condition. With this large space of source terms, we prove an optimal boundedness result for the Kolmogorov operator for all possible exponents $\beta\in (0,1]$, $p\in (1,\infty)$ and $\gamma\in \R$. 

\bigskip

To understand the \textbf{third question}, a priori boundedness of the Kolmogorov operator does not guarantee that its extension (by density) maps into weak solutions (because we are using homogeneous spaces). This is where we need to formalise a definition of \emph{a kinetic Sobolev space}.  To illustrate our definition, we consider the two settings prominent in the literature: (variational/energy) weak solutions in $\L^2$ and strong solutions in $\L^p$. We formulate important results in the way they appear in the present article and compare them to the literature afterwards. 

\noindent We recall that in the local case ($\beta = 1$) one may rewrite a solution to \eqref{eq:kol} as
\begin{equation*}
	(\partial_t + v \cdot \nabla_x) f = \nabla_v \cdot S_0 \\
\end{equation*}
where $S_0 = \nabla_v f$. Thus, if $f \in \L^2_{t,x}\Hdot^{1,2}_v$ is a weak solution to the Kolmogorov equation, then the transport direction of $f$ is an element of a negative Sobolev space in the velocity variable. The kinetic Sobolev space should thus incorporate the regularity of the transport term in $\L^2_{t,x}\Hdot^{-1,2}_v$.

\noindent Back to the case $\beta \in (0,1]$ one defines the homogeneous kinetic Sobolev space as
\begin{equation*}
	\cFdot^{\beta,2}_\beta = \left\{ f \in \cD'(\R^{1+2d});  (\partial_t + v \cdot \nabla_x) f = D_v^\beta S_0, \quad S_0 \in \L^2_{t,x,v} \text{ and } D_v^\beta f \in \L^2_{t,x,v} \right\}.
\end{equation*}
Note that we do not require a priori that $f\in \L^2_{t,x,v}$ so that this space is invariant under kinetic scaling and its norm has a scaling relation $\|f(\lambda^{2\beta}t, \lambda^{2\beta+1}x, \lambda v)\|= \lambda^{\homd/2 -\beta}\|f(t,x,v)\|$, where 
\begin{equation*}
	\homd = 2\beta+ (2\beta+2)d
\end{equation*}
is the homogeneous dimension.
An important feature of kinetic equations is the transfer-of-regularity \`a la Bouchut--H\"ormander which manifests here in the embedding
\begin{equation*}
	\cFdot^{\beta,2}_\beta \ \hookrightarrow \ \L^2_{t,v}\Hdot_{\vphantom{t}x}^{\frac{\beta}{2\beta+1},2}.
\end{equation*}
Regarding the Cauchy problem, an embedding \`a la Lions is of utmost importance 
\begin{equation*}
	\cFdot^{\beta,2}_\beta \ \hookrightarrow \  \C^{}_0(\R;\L^2_{x,v}),
\end{equation*}
and in the kinetic De~Georgi--Nash--Moser theory the gain-of-integrability \`a la Sobolev  
\begin{equation*}
	\cFdot^{\beta,2}_\beta \ \hookrightarrow \ \L^{2\kappa}_{t,x,v}
\end{equation*}
with $\kappa = \frac{\homd}{\homd-2\beta}$,  is a central tool. We refer to \cite{AIN} for more information and proofs. 

Going to the $\L^p$-setting, which is the main purpose of the present article, the literature on kinetic equations studies strong solutions, for which we define the natural homogeneous kinetic Sobolev space as 
\begin{equation*}
	\cFdot^{2\beta,p}_\beta = \left\{ f \in \cD'(\R^{1+2d});  (\partial_t + v \cdot \nabla_x) f \in \L^p_{t,x,v} \text{ and } (-\Delta_v)^{\beta} f \in \L^p_{t,x,v} \right\}.
\end{equation*}
Note again that we do not require a priori that $f\in \L^p_{t,x,v}$ so that this space is invariant under kinetic scaling and its norm has a scaling relation $\|f(\lambda^{2\beta}t, \lambda^{2\beta+1}x, \lambda v)\|= \lambda^{\homd/p -2\beta}\|f(t,x,v)\|$.
In this situation, we prove the transfer-of-regularity  
$$
	\cFdot^{2\beta,p}_\beta  \hookrightarrow \L^p_{t,v}\Hdot_{\vphantom{t}x}^{\frac{2\beta}{2\beta+1},p},
$$
continuity-in-time 
$$
	\cFdot^{2\beta,p}_\beta \hookrightarrow \C^{}_0(\R;\Bdot_{\beta}^{2\beta(1-\frac{1}{p}),p}),
$$ 
and gain-of-integrability 
$$
	\cFdot^{2\beta,p}_\beta \hookrightarrow \L^{p\kappa}_{t,x,v}
$$ 
with $\kappa = \frac{\homd}{\homd-2\beta p}$, and provided $2\beta p <\homd$.
Here $\Bdot_{\beta}^{2\beta(1-\frac{1}{p}),p}$ is a homogeneous Besov space with anisotropy $(\frac 1{2\beta+1}, 1)$ in $(x,v)$ variables and  regularity  ${2\beta}(1-\frac{1}{p})$.

As mentioned above, one can allow for more source terms in \eqref{eq:kol}, and  in view of the transfer-of-regularity, it is natural to consider the anisotropic Sobolev space with the same anisotropy as above and regularity $\gamma$ as (the correct definition will be given later) 
\begin{equation*}
		\Xdot^{\gamma,p}_\beta = \left\{f \in \cS'(\R^{2d})\, ; \, \exists g \in \L^p_{x,v}, \ f = \cF^{-1}\left( \left( \abs{\varphi}^{\frac{1}{1+2\beta}}+\abs{\xi} \right)^{-\gamma} \hat{g}(\varphi,\xi) \right) \right\}.
\end{equation*}
In the case of weak solutions above, instead of $\L^2_{t,x}\Hdot^{-\beta,2}_v$ one considers data even in the larger space $\L^2_t\Xdot^{-\beta,2}_\beta+\L^1_t\L^2_{x,v}$ and this leads to the definition to a kinetic Sobolev space as
\begin{equation*}
	\cLdot^{\beta,2}_{\beta}= \left\{ f \in \cD'(\R^{1+2d})\, ; \,  f \in \L^2_{t,x}\Hdot^{\beta,2}_{v} \ \& \ (\partial_t + v \cdot \nabla_x)f \in \L^2_t\Xdot^{-\beta,2}_\beta+\L^1_t\L^2_{x,v} \right\}.
\end{equation*}
More generally, for fixed $\beta$, there is a corresponding kinetic Sobolev space $\cLdot^{\gamma,p}_{\beta}$ associated to a choice of solution space and source space given the regularity and integrability exponents $\gamma,p$. This is the object of our study here.

\bigskip

Let us also comment on our choice of working with \emph{homogeneous spaces}, for which we use the $\dot{}$ notation. When working on $\R^{1+2d}$ or even half-space $(0,\infty)\times \R^{2d}$, it is natural to  focus on homogeneous spaces. 
In fact, the Kolmogorov equation is not exponentially stable, so one cannot expect to find global solutions in $\L^p_{t,x,v}$ when the time interval is infinite (except in the case of regularity $\gamma=0$). The homogeneous setup allows us to work with global solutions nonetheless.
Moreover, in kinetic PDE, the velocity variable is often taken in $\R^d$ and thus considering the homogeneous Sobolev space in this variable is natural. 
Yet another advantage is that homogeneous spaces are adapted to scale-invariant estimates, which provide us with the sharp relations between exponents.

\bigskip 

First, we rigorously defining these spaces ($\Hdot^{\gamma,p}_v,\Xdot^{\gamma,p}_\beta,\Bdot_{\beta}^{\gamma,p},\cFdot^{\gamma,p}_\beta$, $\cLdot^{\gamma,p}_\beta$). Then, we obtain for  $p \in (1,\infty)$, $\beta \in (0,1]$ and $\gamma \in [0,2\beta]$ 
with $\gamma < \homd /p$,  concerning the kinetic Sobolev space $\cLdot^{\gamma,p}_\beta$ but also the smaller space $\cFdot^{\gamma,p}_\beta$,
\begin{itemize}
\item[{\tiny$ \bullet$}] sharp embeddings, which quantify
\subitem- the transfer-of-regularity \`a la Bouchut--H\"ormander,
\subitem- continuity-in-time in the spirit of Lions,
\subitem-  gain-of-integrability of Sobolev type, 
\item[{\tiny$ \bullet$}] the isomorphism property, implying it is a Banach space containing a  dense class of $\L^2$ functions,
\item[{\tiny$ \bullet$}] integral identities of energy type via absolute continuity.
\end{itemize}
This will allow us to prove well-posedness of Cauchy problems for the Kolmogorov equation \eqref{eq:kol} and also to establish  Hardy--Littlewood--Sobolev estimates of the Kolmogorov operators.

Also, we develop the corresponding theory for the inhomogeneous spaces. We emphasise that we provide for the first time the $\L^p$-estimates and a complete $\L^p$-theory for weak solutions to the Kolmogorov equation ($\gamma = \beta$) for $p \neq 2$. The manuscript encompasses a number of works in the literature, and the range on the coefficient $\gamma$ is sharp (for $\beta,p$ fixed).

\medskip

The focus on homogeneous Sobolev spaces has the advantage of naturally dealing with scale-invariance, but it introduces several technical difficulties which require careful treatment. 
This will be taken care of in Section \ref{sec:spaces}, where all the aforementioned  spaces will be defined. 
The study of the Kolmogorov operator is the core technical part of the paper. 
We introduce these operators in Section \ref{sec:forbackkol}. 
Their action on {appropriate source} spaces is the content of Section \ref{sec:towardsLp}, which we approach by stating $\L^p$-boundedness of a singular integral operator. To this end,
we establish novel (partially) integrated kernel estimates in Section \ref{sec:estimates}. 
These estimates will allow us to check H\"ormander's condition in the Coifman--Weiss theorem, that is, the analogue of Calder\'on--Zygmund weak $(1,1)$ extrapolation on spaces of homogeneous type, see Appendix \ref{sec:cw}, and deduce the $\L^p$-boundedness of the singular integral operator in Section \ref{sec:horm}. 
The initial value problem and corresponding temporal trace estimates are studied in Section \ref{sec:trace}. 
The latter three sections are calculation-heavy, and we advise a reader to skip the proofs on a first read. Next, we gather all the bounds obtained so far and, combined with embeddings for anisotropic Sobolev spaces, we prove Hardy--Littlewood--Sobolev estimates for the Kolmogorov operators with sharp exponents, see Section~\ref{sec:Lebesgueestimates}.
In Section \ref{sec:iso}, we prove that the Kolmogorov differential operators are isomorphisms from the kinetic Sobolev spaces $\cFdot$ or $\cLdot$ onto their corresponding source spaces. 
This goes by proving a novel uniqueness result, which is of interest itself. 
Section \ref{sec:cauchy} deals with the Cauchy problem for the Kolmogorov equation, and we provide a general existence and uniqueness result in Theorem \ref{thm:homCP-Lp}. 
Section \ref{sec:inhom} shows that our results continue to hold in inhomogeneous kinetic spaces.
Section~\ref{sec:NZ} discusses another possible choice for the kinetic Sobolev space which lifts the range restriction $[0,2\beta]$ on the regularity parameter $\gamma$.
In Section \ref{sec:local}, we explain how to prove one of the bounds on the Kolmogorov operator with local diffusion in the special case of weak solutions but without using the Fourier transform.

\bigskip
Let us now explain the related literature.
Fundamental are the works by H\"ormander \cite{hormander_hypoelliptic_1967}, Rothschild and Stein \cite{MR436223} and Folland \cite{folland_estimates_1974} on $\L^p$-estimates for hypoelliptic differential operators built from vector fields. The Kolmogorov equation is a particular example of this broad family of PDEs.

Let us focus on the literature concerned with the Kolmogorov equation. $\L^2$ estimates for the Kolmogorov equation \eqref{eq:kol} can be found in \cite{MR4444079} for $\beta=1$, in \cite{MR3906169,MR3826548} for strong solutions 
and $\beta\in(0,1]$, and in \cite{niebel_kinetic_2021} for a larger scale including weak solutions. 
Noteworthy is also the early work \cite{MR875086}.
Recently, together with C.~Imbert, we developed a complete $\L^2$ theory of weak solutions in \cite{AIN} including an analysis of the homogeneous kinetic Sobolev space $\cLdot^{\beta,2}_{\beta}$ mentioned above.
Our results may be viewed partly as an extension of \cite{AIN} to $\L^p$. We mention at this point that we get rid of the technical assumption $\beta < d/2$; compare \cite[Footnote 1 on p. 3]{AIN}.

$\L^p$ estimates for strong solutions in the case $\beta=1$ are classical \cite{MR436223,folland_estimates_1974}; another reference is \cite{MR1391154}. 
In \cite{MR2729292}, the estimates are obtained by a slightly different method; a proof of the homogeneous estimate can be found in \cite{hirao2026anisotropicmaximallpregularityestimates}.

For non-local diffusion $\beta\in(0,1)$, $\L^p$ estimates for strong solutions 
were proved in \cite{MR3906169,MR3826548} using different approaches. 
In both works, integrated kernel estimates (in $x$ and $v$) are derived via stochastic methods. 
The $\L^p$ bound in \cite{MR3826548} follows from interpolation of the $\L^2-\L^2$ and $\L^\infty-\mathrm{BMO}$ estimates via 
the Fefferman--Stein theorem; in \cite{MR3906169} the $\L^2-\L^2$ estimate is combined with the 
Coifman--Weiss theorem after verifying a H\"ormander condition on the kernel.

Concerning the Kolmogorov equation with variable rough coefficients, for $p=2$ and $\beta\in(0,1]$, 
 existence and uniqueness of weak solutions is established in \cite{AIN}; see also \cite{MR4312909,albritton2021variational} for $\beta = 1$. 
For strong solutions with $\beta=1$, $\L^p$-estimates for the Kolmogorov equation in non-divergence form have been proved under ellipticity, boundedness, and uniform continuity assumptions on the diffusion coefficients in \cite{MR3092273,niebel_kinetic_2022,niebel_analytic_2023},  the continuity assumptions being relaxed to $\mathrm{VMO}$ in \cite{MR4444079}.
An a priori $\L^p$ estimate for strong solutions with uniformly continuous jump kernel for $\beta\in(0,1)$ is proven in \cite{MR4299838}. Estimates in $\L^p$ spaces with spatial weights for strong solutions to the Kolmogorov equations are deduced in \cite{niebel_analytic_2023}.

We also mention \cite{niebel_kinetic_2022}, which treats the Cauchy problem for $\beta\in(0,1]$ 
and strong solutions with initial data in anisotropic Besov spaces, as well as mixed time/space integrability and temporal weights.
Moreover, a uniqueness result for strong solutions is proved; see also \cite{MR4444079} in the case $\beta = 1$. 

Estimates for the fundamental solution in the case $\beta=1$ with constant coefficients are classical, 
see \cite{MR1391154,lanconelli_class_1994}. 
Even for rough diffusion coefficients, a fundamental solution with two-sided pointwise bounds exists
\cite{auscher_fundamental_2024,dietert2025nashsgboundkolmogorov}. 
In the non-local case, integrated estimates appear in \cite{MR3826548,MR3906169}. 
First (non-optimal) pointwise estimates are derived in \cite{loher2024semilocal}; 
sharp pointwise bounds for the fundamental solution were recently obtained in \cite{hou_kernel_2024} (stochastic approach) 
and in \cite{grube2024pointwiseestimatesfundamentalsolution} for $d=1$ (analytic approach); 
for $d=1$ they can also be recovered from results on L\'evy measures \cite{MR3906169,marino2023weak}. 
A precise understanding of the fundamental solution is used in \cite{kassmann2024harnack} to show that the Harnack inequality fails for weak solutions to the {non-local} Kolmogorov equation ($\beta\in(0,1)$).

A space similar to the inhomogeneous version of $\cFdot^{\beta,2}_{2}$  but with Gaussian weight in the velocity variable appears in \cite{albritton2021variational}. The authors study the kinetic Fokker--Planck equation, i.e.\ \eqref{eq:kol} with $x$ in a torus or a bounded domain of $\R^d$ with an additional forcing term $v \cdot \nabla_v$ and bounded drift of gradient form. In this context, elements are $\L^2$ functions and transfer-of-regularity and  continuity-in-time were proven ($p=2$, $\beta=1$).
Results in standard Lebesgue spaces were obtained in \cite{niebel_kinetic_2021} ($p=2$, $\beta\in(0,1]$) 
in inhomogeneous spaces on finite time intervals. 
Sharp Lions-type statements in homogeneous spaces on possibly infinite time intervals 
appear in \cite{AIN} ($p=2$, $\beta\in(0,1]$). 
For strong solutions in inhomogeneous spaces, a related embedding (continuity in time with values in an anisotropic Besov space) 
is proved in \cite{niebel_kinetic_2022} ($p\in(1,\infty)$, $\beta\in(0,1]$).

H\"ormander's seminal work \cite{hormander_hypoelliptic_1967} quantifies the transfer of regularity in $\L^2$; 
see \cite{albritton2021variational} for a modern presentation. 
Averaging lemmas are a related line of research and originate in \cite{MR0808622,MR0923047}. 
Bouchut \cite{MR1949176} proves a scale-invariant transfer of regularity for strong solutions qualitatively in $\L^p$
with (non-)local diffusion ($p\in(1,\infty)$, $\beta\in(0,1]$) and for weak solutions ($p\in(1,\infty)$, $\beta=1$).
The transfer-of-regularity for weak solutions in $\L^2$ and $\beta\in(0,1]$ appears also in \cite{niebel_kinetic_2021}.
In our joint work with C.~Imbert \cite{AIN}, we establish transfer-of-regularity in a broad scale of $\L^2$-based 
homogeneous spaces for non-local diffusions ($\beta\in(0,1)$, and even $\beta\in (0,\infty)\setminus \N$) and the method applies to local differential operators of any order $(\beta\in \N)$.

A gain of integrability for weak solutions, although far from optimal, can already be deduced from 
H\"ormander's work ($p=2$, $\beta=1$). 
Even the presentation in \cite{albritton2021variational} does not yield the critical gain of integrability. 
Working with estimates for the fundamental solution, the first sharp kinetic Sobolev inequality is proved in
\cite{pascucci_mosers_2004} ($p=2$, $\beta=1$), and the method extends to all $p\in(1,\infty)$.  
The same method applies to non-local equations and is used in \cite{imbert_weak_2020} ($p=2$, $\beta \in (0,1)$); see also \cite{MR4688651,anceschi2025giorginashmosertheorykineticequations}. 
In the case of ultraparabolic equations with local diffusion, including the kinetic one, \cite{MR4700191} introduces a class of spaces called intrinsic Sobolev spaces roughly requiring iterated applications of the operators $|\partial_t +v \cdot \nabla_x |^{1/2}, \nabla_{v}$ in $ \L^p_{t,x,v}$ and obtains Sobolev embeddings for them. 
To compare their spaces to ours would probably require to show that any $f \in \cL^{1,p}_1$ satisfies $|\partial_t +v \cdot \nabla_x |^{1/2}f \in \L^p_{t,x,v}$ in a suitable sense; we leave this an interesting open problem.    
In \cite{AIN}, the kinetic Sobolev embedding is deduced from the transfer-of-regularity, combined with a 
Lions-type embedding ($p=2$, $\beta \in (0,1]$); 
see also \cite{golse_harnack_2019}. Sobolev-type embeddings for strong solutions are studied in \cite{MR4700191}. 
An alternative approach, not relying on the fundamental solution but on critical kinetic trajectories, is proposed in
 \cite{dietert2025criticaltrajectorieskineticgeometry} ($\beta=1$).

For Schauder estimates, we refer to 
\cite{di_francesco_schauder_2006,imbert_schauder_2021,imbert_schauder_2021-1,lunardi_schauder_1997,loher2023quantitativeschauderestimateshypoelliptic,menozzi_martingale_2018}. 

\medskip

\noindent \textbf{Notation.} 
We write $\lesssim$ ($ \gtrsim, \simeq$) whenever we estimate from above (below, or both ways) by a constant possibly
depending on the quantities $\beta,\gamma,d$ introduced below. 
If we want to highlight the dependence of the constant on the variable $\lambda$, 
we write $\lesssim_\lambda$ ($ \gtrsim_\lambda, \simeq_\lambda$). 
For $x \in \R^d$ and $r>0$, we denote the open ball by $B_r(x) \subset \R^d$. 
Convolution on $\R^d$ is denoted by $\ast$. 
We abbreviate $\R^* = \R \setminus \{0 \}$. 
Positive constants $C>0$ are assumed to be finite real numbers.  
{Moreover, $a \wedge b := \min \{ a,b \}$. }

\section{Kinetic spaces and Kolmogorov operators}
\label{sec:spaces}
\subsection{Main functional spaces}

We introduce the function spaces used throughout the paper.
Unless stated otherwise, the parameter ranges are $\gamma\in \R$, $\beta\in (0,\infty)$, $p \in (1,\infty)$.  
We use the following conventions. 
We indicate the variable $t\in \R$, $x\in \R^d$, $v\in \R^d$ by subscripts.
We let $D_{x}=(-\Delta_{x})^{1/2}$  and $D_{v}=(-\Delta_{v})^{1/2}$ where $\Delta_{x}$ and $\Delta_{v}$ 
are the Laplacians in the respective variable. 
{We will work with homogeneous norms and with (tempered) distributions. 
This is convenient for the transport term, which will be interpreted in the sense of distributions.}
We denote by $ \ \hat{} \ $ the Fourier transforms in $v$ or $(x,v)$ variables with dual variables $\xi$ or $(\varphi,\xi)$ 
and usually denote by $\cF^{-1}$ the inverse Fourier transforms.
We use the standard Fourier multiplication notation $m(D)f = \cF^{-1}(m \hat{f})$. 
 
\subsubsection{Homogeneous Sobolev spaces}\label{sec:homSobspaces} 
We begin with the Sobolev spaces on $\R^d$. 
We recall the following fact, which goes back to Peetre in the case of Besov spaces \cite{MR461123}. 
For $k\in \Z$, denote by $\cP_{k}$ the set of polynomials on $\R^d$ with degree $\deg P\le k$, 
with the convention that $\cP_{k}=\{0\}$ if $k<0$.

\begin{lem} \label{lem:convergenceLP} 
Let $\gamma\in \R$, $p \in (1,\infty)$. 
Let  $g_{j}\in \cS'(\R^d) \cap \C^\infty(\R^d)$, $j\in \Z$, with Fourier transform $\hat g_{j}$ 
supported in $C_{j}=\{\xi\in \R^d;  2^{j- 1} < |\xi|< 2^{j+1}\}$. 
If $\big\|\big(\sum_{j=-\infty}^\infty |2^{j\gamma} g_{j}|^2\big)^{1/2}\big\|_{\L^p_{v}}<\infty$ then 
$\sum_{j=1}^\infty g_{j}$ converges in $\cS'(\R^d)$ and $\sum_{j=-\infty}^0 g_{j}$ 
converges in $\cS'(\R^d)/\cP_{k}$ where $k=[\gamma-d/p]$, the greatest integer less than or equal to $\gamma-d/p$. 
\end{lem}
 
We introduce a normalised Littlewood--Paley decomposition   \begin{equation*}
	\sum_{j = -\infty}^\infty \hat{\psi}_j(\xi)=1, \qquad \xi\in \R^d\setminus \{0\},
\end{equation*}
with $\psi\in \cS(\R^d)$ where $\hat\psi$ is supported in $C_{0}$ and $\hat\psi>0$ if $1\le |\xi|\le 2$ and  $\hat{\psi}_j(\xi)= \hat\psi(2^{-j}\xi)$.  
With this lemma, one can give the following definition of the homogeneous Sobolev space via its Triebel--Lizorkin (semi-) norm.

\begin{defn}[Homogeneous Sobolev space]  
Let $\gamma\in \R$, $p \in (1,\infty)$. For $f\in \cS'(\R^d)$, we set
\[
	\|f\|_{\Hdot^{\gamma,p}_{v}} := \bigg\|\bigg(\sum_{j=-\infty}^\infty |2^{j\gamma} \psi_{j}\ast f|^2\bigg)^{1/2}\bigg\|_{\L^p_{v}}
\]
and we let 
\[
	\Hdot^{\gamma,p}_{v}:= \left\{ f\in \cS'(\R^d)\, ; \, \|f\|_{\Hdot^{\gamma,p}_{v}} <\infty\ \&\  f=\sum_{j=-\infty}^\infty  \psi_{j}\ast f \textrm{ \ in \ }  \cS'(\R^d)/\cP_{[\gamma-d/p]} \right\}.
\]
\end{defn}

Note that ${\Hdot^{\gamma,p}_{v}}$ is a normed space if $\gamma-d/p<0$ and a semi-normed space 
if $\gamma-d/p\ge 0$ which contains  $ \cP_{[\gamma-d/p]}$. 
It is complete for this semi-norm. 
This definition guarantees that we are dealing with tempered distributions, as $\Hdot^{\gamma,p}_{v} \subset \cS'(\R^d)$, 
but this inclusion is only continuous when $\gamma-d/p< 0$ (otherwise, if we argue modulo $\cP_{[\gamma-d/p]}$ it is continuous). 
It is known that changing the function $\psi$ to another one $\phi$ with the same support property 
and $\sum_{j = -\infty}^\infty |\hat{\phi}_j(\xi)|>0$, $  \xi\in \R^d\setminus \{0\},$ yields an equivalent (semi-)norm. And if the decomposition is normalized, i.e.\ $\sum_{j = -\infty}^\infty \hat{\phi}_j(\xi)=1$, $\xi\ne 0$, then the set defined using $\phi$ is the same. 

We observe that the  space of Schwartz functions $\cS$  with  Fourier transform supported in compact subsets of $\R^d\setminus \{0\}$
is dense in ${\Hdot^{\gamma,p}_{v}}$ for all $\gamma\in \R$, $p \in (1,\infty)$.   
The proof is the same for all parameters: first take an approximation by finite Littlewood--Paley sums by definition, and then take a convolution (in Fourier variable) with a mollifier.
Moreover, this space is preserved by the action of fractional powers of $-\Delta_{v}$, and their action is by Fourier multiplication
$D_{v}^\gamma S= \cF^{-1}(|\xi|^\gamma \hat S)$, where $\cF^{-1}$ denotes the inverse Fourier transform. 
In particular, for any such $S$, and all $\alpha,\gamma\in \R$, 
\[ 
	\|D_{v}^\alpha S\|_{{\Hdot^{\gamma,p}_{v}}}\sim \| S\|_{{\Hdot^{\gamma-\alpha,p}_{v}}}
\]
with constants that do not depend on $S$. This means that to prove estimates on such elements, we can use standard 
Fourier multiplication without using Littlewood--Paley series. 
We need, however, to define their action on general ${\Hdot^{\gamma,p}_{v}}$ elements. 

\begin{lem}	\label{lem:fractionalLaplacian} 
Let $\gamma\in \R$, $p \in (1,\infty)$. Let $f\in {\Hdot^{\gamma,p}_{v}}$.
\begin{enumerate}
	\item $D_{v}^\gamma f= (-\Delta_{v})^{\gamma/2}f$ is the unique $g\in \L^p_{v}$ defined by 
		$g =\sum_{j=-\infty}^\infty  (-\Delta_{v})^{\gamma/2} \psi_{j}(D)f$ in $ \cS'(\R^d)$ and we have 
		$\|f\|_{\Hdot^{\gamma,p}_{v}}\sim \| D_{v}^{\gamma} f\|_{\L^p_{v}}$.
	\item For $\alpha\in \R$, $\alpha>\gamma-d/p$, $D_{v}^\alpha f= (-\Delta_{v})^{\alpha/2}f$ is the unique 
		$g_{\gamma-\alpha}\in {\Hdot^{\gamma-\alpha,p}_{v}}$ defined by 
		$g_{\gamma-\alpha} =\sum_{j=-\infty}^\infty  (-\Delta_{v})^{\alpha/2} \psi_{j}(D)f$ in $\cS'(\R^d)$ 
		and we have  $${\|g_{\gamma-\alpha}\|_{\Hdot^{\gamma-\alpha,p}_v}\sim \|D_v^\gamma f\|_{L^p}\sim \|f\|_{\Hdot^{\gamma,p}}}.$$ 
		Moreover we have, with $g$ as in \emph{(i)},
$$
		(-\Delta_{v})^{\alpha/2}f= (-\Delta_{v})^{(\alpha-\gamma)/2} g = 
		(-\Delta_{v})^{(\alpha-\gamma)/2} (-\Delta_{v})^{\gamma/2}f.
		$$
		\end{enumerate}
\end{lem}

\begin{proof}
For the first item, we write $(-\Delta_{v})^{\gamma/2} \psi_{j}(D)= 2^{j\gamma} \phi_{j}(D)$ 
with $\hat\phi(\xi)=|\xi|^\gamma\hat\psi(\xi)$, and 
\begin{equation*}
	\bigg\|\bigg(\sum_{j=-\infty}^\infty |(-\Delta_{v})^{\gamma/2} \psi_{j}\ast f|^2\bigg)^{1/2}\bigg\|_{\L^p_{v}} 
	= \bigg\|\bigg(\sum_{j=-\infty}^\infty |2^{j\gamma} \phi_{j}\ast f|^2\bigg)^{1/2}\bigg\|_{\L^p_{v}} 
	\sim \|f\|_{\Hdot^{\gamma,p}_{v}}.
\end{equation*}
By the Lemma~\ref{lem:convergenceLP}, the series defining $g$ converges in $\cS'(\R^d)$ to an element in $\L^p_{v}$ 
and one easily checks that $g=\sum_{j=-\infty}^\infty   \psi_{j}\ast g$ in $\cS'(\R^d)$. 
The proof of the second item is similar.
\end{proof}

We note that we may take $\alpha<0$, but provided $\gamma-\alpha<d/p$.

\subsubsection{Homogeneous anisotropic Sobolev spaces}\label{sec:anhomSobspaces} 
We move to $\R^{2d}$ and fix $\beta>0$ (we will restrict to $\beta\le 1$ later).
For $k\in \Z$, denote by $\cP_{k}$ the set of polynomials on $\R^{2d}$ with $\deg P\le k$, with the convention that $\cP_{k}=\{0\}$ if $k<0$. We denote by $[a]$ the greatest integer less than or equal to $a$. 
 
We introduce the anisotropic {quasi-norm} on $\R^{2d}$: 
$$
 \abs{(\varphi,\xi)}_\beta = \abs{\varphi}^{\frac{1}{2\beta+1}} + \abs{\xi}.
$$
Note that $\abs{(\varphi,\xi)}_\beta>0$ if $(\varphi,\xi)\ne (0,0)$, 
$ \abs{(\lambda^{2\beta+1}\varphi,\lambda\xi)}_\beta= \lambda \abs{(\varphi,\xi)}_\beta$ for $\lambda>0$ and 
$ \abs{(\varphi+\varphi',\xi+\xi')}_\beta \le C_{\beta}(\abs{(\varphi,\xi)}_\beta + \abs{(\varphi',\xi')}_\beta)$.  
The anisotropy here is relative to the kinetic scaling with diffusion parameter $\beta$. 
 
We introduce homogeneous anisotropic Littlewood--Paley theory specified to the kinetic scaling.
{We draw inspiration from} \cite{MR417687} and refer to \cite{MR1950714} in the inhomogeneous case. 
First, we may replace $\abs{(\varphi,\xi)}_\beta$ by a smooth version \cite{MR417687}: 
there exists a function $d_{\beta}: \R^{2d} \to [0,\infty)$ in 
$\C^\infty(\R^{2d}\setminus\{(0,0)\})$, with 
\begin{itemize}
\item $d_{\beta}(\varphi,\xi) \sim \abs{(\varphi,\xi)}_\beta$,
\item $d_{\beta}(\varphi+\varphi',\xi+\xi') \le  d_{\beta}(\varphi,\xi)+d_{\beta}(\varphi',\xi')$ 
\item $d_{\beta}(\lambda^{2\beta+1}\varphi, \lambda\xi)=\lambda d_{\beta}(\varphi,\xi)$,
\item {$d_{\beta}(\varphi,\xi) = d_{\beta}(-\varphi,\xi)$,  $d_{\beta}(\varphi,\xi) = d_{\beta}(\varphi,-\xi)$}
\end{itemize}
for $(\varphi,\xi), (\varphi',\xi')\in \R^{2d}$ and $\lambda>0$.  Moreover,  
for all $s\in \R$, $\alpha_{1},\alpha_{2}\in \N^{d}$, there is a constant $C_{s,\alpha_{1}, \alpha_{2}}<\infty$ 
such that on  $\R^{2d}\setminus \{ (0,0)\}$,
$$
	|\partial^{\alpha_{1}}_{\varphi}\partial^{\alpha_{2}}_{\xi} d_{\beta}^s| \le C_{s,\alpha_{1}, \alpha_{2}}\ d_{\beta}^{s-(2\beta+1)|\alpha_{1}|- |\alpha_{2}|}
$$
using the standard notation for differentiation.

We consider any function $\theta \in \cS(\R^{2d})$ with Fourier transform supported in 
$ \{(\varphi,\xi)\in \R^{2d}\, ; \, \frac 1 2 < d_{\beta}(\varphi,\xi)< 4\}$ and $\hat \theta(\varphi,\xi)>0$ when 
$ 1  \le d_{\beta}(\varphi,\xi)\le 2$. 
Define $\theta_{j}$ with $\hat{\theta}_j(\varphi,\xi) = \hat{\theta}\left(2^{-(2\beta+1)j}\varphi,2^{-j}\xi\right)$, 
that is,  $\theta_{j}(x,v)=r^{(2\beta+2)d} \theta(r^{(2\beta+1)}x,r v)$ for $r=2^j$.  
We call $(\theta_{j})$ an anisotropic Littlewood--Paley family. 
In what follows, $\ast$ is standard convolution on $\R^{2d}$.

\begin{lem} \label{lem:anisoconvergenceLP} 
Let $\gamma\in \R$, $p \in (1,\infty)$. Let $(\theta_{j})$ be any anisotropic Littlewood--Paley family.
\begin{enumerate}
	\item For $\chi\in \cS(\R^{2d})$ with $\int_{\R^{2d}} \chi P \, \dd(x,v)=0$ for all  $P\in \cP_{k}$, 
	$k= [\gamma-(2\beta+2) d/p]$, then 
	\[
		\bigg\|\bigg(\sum_{j=-\infty}^\infty |2^{-j\gamma}\, \theta_{j}\ast\chi|^2\bigg)^{1/2}\bigg\|_{\L^{p'}_{x,v}}<\infty.
	\]
	\item Let  $g_{j}\in \cS'(\R^{2d}) \cap \C^\infty(\R^{2d})$, $j\in \Z$, with Fourier transform $\hat g_{j}$ supported in 
	$C_{\beta, j}=\{(\varphi,\xi)\in \R^{2d}\, ;\,  2^{j-1} < d_{\beta}(\varphi,\xi) < 2^{j+1}\}$. 
	If $\big\|\big(\sum_{j=-\infty}^\infty |2^{j\gamma} g_{j}|^2\big)^{1/2}\big\|_{\L^p_{x,v}}<\infty$ then 
	$\sum_{j=1}^\infty g_{j}$ converges in $\cS'(\R^{2d})$ and $\sum_{j=-\infty}^0 g_{j}$ converges in 
	$\cS'(\R^{2d})/\cP_{k}$ where $k=[\gamma-(2\beta+2) d/p]$.
\end{enumerate}
\end{lem}

We say that the anisotropic Littlewood--Paley family is normalised if it satisfies the identity   
\begin{equation*}
	\sum_{j = -\infty}^\infty \hat{\theta}_j(\varphi,\xi)=1, \qquad (\varphi,\xi)\in \R^{2d}\setminus \{(0,0)\},
\end{equation*}

\begin{defn}[Homogeneous anisotropic Sobolev space] \label{defn:HASS}  
Let $\beta>0$, $\gamma\in \R$, $p \in (1,\infty)$. Fix $(\theta_{j})$  a normalised anisotropic Littlewood--Paley family. 
For $f\in \cS'(\R^{2d})$, we set
\[
	\|f\|_{\Xdot^{\gamma,p}_{\beta}} := \bigg\|\bigg(\sum_{j=-\infty}^\infty |2^{j\gamma} \theta_{j}\ast f|^2\bigg)^{1/2}\bigg\|_{\L^p_{x,v}}
\]
and we let 
\[
	\Xdot^{\gamma,p}_{\beta}:= \left\{ f\in \cS'(\R^{2d})\, ; \, \|f\|_{\Xdot^{\gamma,p}_{\beta}} <\infty\ \&\  f=\sum_{j=-\infty}^\infty  \theta_{j}\ast f \textrm{ \ in \ } 
	\cS'(\R^{2d})/\cP_{[\gamma-(2\beta+2)d/p]} \right\}.
\]
\end{defn}

Note that ${\Xdot^{\gamma,p}_{\beta}}$ is a normed space if $\gamma-(2\beta+2)d/p<0$ and a semi-normed space which contains 
$ \cP_{[\gamma-(2\beta+2)d/p]}$  if $\gamma-(2\beta+2)d/p\ge 0$. 
It is complete for this semi-norm. Note that the  partial sums of the Littlewood--Paley series in the definition also converge to $f$ in this semi-norm.
This definition guarantees also that we have $\Xdot_\beta^{\gamma,p} \subset \cS'(\R^{2d})$ as a subspace, and the inclusion is continuous  if and only if 
$\gamma-(2\beta+2)d/p<0$.  
Changing the function $\theta$ to another one $\phi$ with the same support property and the non-degeneracy condition 
$\sum_{j = -\infty}^\infty |\hat{\phi}_j(\varphi,\xi)|>0,  (\varphi,\xi)\in \R^{2d}\setminus \{(0,0)\}$, 
yields an equivalent (semi-)norm. 
And if the decomposition is normalized ($\sum_{j = -\infty}^\infty \hat{\phi}_j(\varphi, \xi)=1$, $(\varphi,\xi)\ne (0,0)$), then the set defined using $\phi$ is the same. 

It will be convenient to have a universal dense subspace.

\begin{lem}
\label{lem:denseanisotropic}
 The  space $\dense$ of Schwartz functions in $\R^{2d}$  with  Fourier transform supported in compact subsets of 
$(\R_{\varphi}^{d}\setminus \{0\})\times (\R_{\xi}^{d}\setminus \{0\})$ is dense in ${\Xdot^{\gamma,p}_{\beta}}$ 
for all $\beta>0$,  $\gamma\in \R$ and $p \in (1,\infty)$.
\end{lem}

\begin{proof}
 First, we can do, by definition, a truncation of the anisotropic Littlewood--Paley series, which gives a compactly supported Fourier transform away from the origin. 
 Then, we may mollify in the variables $(\varphi,\xi)$ to have $\C^\infty$ regularity and still the support property.  Next, we may multiply by smooth Mikhlin symbols having support away from the axes $\varphi=0$ and $\xi=0$. We leave the verification of the details to the reader.  
\end{proof}

This dense space is preserved by the action of fractional powers $D_{x}^a$, $D_{v}^b$ of   $-\Delta_{x}$, $-\Delta_{v}$ 
and also by combinations 
$(\lambda D_{x}^a+\mu D_{v}^b)^c$, $\lambda,\mu\ge 0$, $a,b,c\in \R$  and their action is by 
Fourier multiplication  $(\lambda D_{x}^a+\mu D_{v}^b)^c S= \cF^{-1}((\lambda|\varphi|^a+ \mu|\xi|^b)^c \hat S)$. 
This means that to prove a priori estimates on such elements, we can use standard Fourier multiplication 
instead of the Littlewood--Paley series.

\begin{lem}\label{lem:multipliers} Let $\beta>0$, $\gamma>0$ and $p \in (1,\infty)$. 
	The symbols $m_{1}(\varphi,\xi)=(|\xi|/d_{\beta}(\varphi,\xi))^\gamma$,  
	$m_{2}(\varphi,\xi)=(|\varphi|^{\frac {1}{2\beta+1}}/d_{\beta}(\varphi,\xi))^\gamma$,  
	$m_{3}(\varphi,\xi)=d_{\beta}(\varphi,\xi)^\gamma/  (|\varphi|^{\frac {\gamma}{2\beta+1}}+|\xi|^\gamma)$
	and $m_{4}(\varphi,\xi)=1/m_{3}(\varphi,\xi)$  are  $\L^p_{x,v} = \L^p(\R^d\times \R^d)$ Fourier multipliers.
\end{lem}

\begin{proof}
This is an application of the Marcinkiewicz multiplier theorem \cite[Theorem 5.2.4]{MR2445437} as the four $\L^\infty$ symbols 
are $\C^\infty$ functions on $(\R_{\varphi}^{d}\setminus \{0\})\times (\R_{\xi}^{d}\setminus \{0\})$ and satisfy
\[
	|\partial_{\varphi}^{\alpha_{1}}   \partial_{\xi}^{\alpha_{2}} m_{i}(\varphi,\xi)| \le \frac{C}{|\varphi|^{|\alpha_{1}|}|\xi|^{|\alpha_{2}|}}
\]
for some constant $C = C(\alpha_{1},\alpha_{2})>0$.
\end{proof}

\begin{prop}[characterisation] \label{prop:characterisation}
Let $\beta>0$, $p \in (1,\infty)$ and $\gamma\in \R$.  
\begin{itemize}
	\item If $\gamma \ge 0$, then 
	$\Xdot^{\gamma,p}_{\beta} = \L^p_{x}\Hdot_{\vphantom{x} v}^{\gamma,p} \cap \L^p_{\vphantom{x} v}\Hdot_{\vphantom{x} x}^{\frac{\gamma}{2\beta+1},p}$
	with  
	\begin{equation} \label{eq:Xdotgammabetanorm>0}
 		 \|f\|_{\Xdot^{\gamma,p}_{\beta}}\sim_{d,\beta,\gamma}\| D_{v}^{\gamma} f\|_{\L^p_{\vphantom{\beta} x,v}}+ \| D_{ x}^{\frac{\gamma}{2\beta+1}} f\|_{\L^p_{\vphantom{\beta} x,v}}  .
 	\end{equation}
	\item  If $\gamma \le 0$, then
	$\Xdot^{\gamma,p}_{\beta}=  \L^p_{x}\Hdot^{\gamma,p}_{\vphantom{x} v} + \L^p_{\vphantom{x} v}\dot\H_{\vphantom{x} x}^{\frac{\gamma}{2\beta+1},p}$ with 
\begin{equation}
 \label{eq:Xdotgammabetanorm<0} \|f\|_{\Xdot^{\gamma,p}_{\beta}} \sim_{d,\beta,\gamma}\inf_{f=f_{1}+f_{2}}\| D_{v}^{\gamma} f_{1}\|_{\L^p_{\vphantom{\beta}x,v}}+ \| D_{x}^{\frac{\gamma}{2\beta+1}} f_{2}\|_{\L^p_{\vphantom{\beta}x,v}}.
  \end{equation}
\end{itemize}

\end{prop}

\begin{proof} When $\gamma=0$, there is nothing to prove. 
Let us first prove the equivalence when $f=S$ belongs to the dense class above, 
remarking that it is dense in each of the spaces involved. 
When $\gamma> 0$, then  we have 
\[
	d_{\beta}^\gamma \hat S= m_{3}\cdot (|\varphi|^{\frac{\gamma}{2\beta+1}}+|\xi|^{\gamma}) \hat S, \quad 
	|\varphi|^{\frac {\gamma}{2\beta+1}}\hat S= m_{2} \cdot d_{\beta}^\gamma \hat S, \quad 
	|\xi|^\gamma \hat S= m_{1} \cdot d_{\beta}^\gamma \hat S.
\]
From this and Lemma \ref{lem:multipliers}, the equivalence follows. 

When $\gamma<0$, then write $\hat T= d_{\beta}^\gamma \hat S$, which satisfies 
$\|T\|_{\L^p_{x,v}} \sim \|S\|_{\Xdot^{\gamma,p}_{\beta}}$ and 
decompose 
\[ 
	\hat S=   
	\frac{|\xi|^{-\gamma}\hat S}
	{|\varphi|^{\frac {-\gamma}{2\beta+1}}+|\xi|^{-\gamma}}
	+ \frac{|\varphi|^{\frac {-\gamma}{2\beta+1}} \hat S}{|\varphi|^{\frac {-\gamma}{2\beta+1}}+|\xi|^{-\gamma}}
	=:\hat S_{1}+\hat S_{2}. 
\]
Then  $\hat T_{1}:=|\xi|^\gamma \hat S_{1}=  m_{3}\hat T$ and 
$\hat T_{2}:=|\varphi|^{\frac {\gamma}{2\beta+1}} \hat S_{2}=  m_{3}\hat T$, 
where $m_{3}$ is as in Lemma \ref{lem:multipliers} with $\gamma$ changed to $-\gamma$. 
Thus, $\|T_{2}\|_{\L^p_{x,v}}\sim \| D_{x}^{\frac{\gamma}{2\beta+1}} S_{2}\|_{\L^p_{x,v}}$, $\|T_{1}\|_{\L^p_{x,v}}
\sim \| D_{v}^{\gamma} S_{1}\|_{\L^p_{x,v}}$ and the equivalence follows. 

Let us continue with proving the equality of sets. We go through Cauchy sequences of elements in the dense class. 
When $\gamma<0$, all the spaces involved are Banach spaces and convergence in them implies 
convergence in $\cS'_{x,v}=\cS'(\R^{2d})$. 
Equality is a direct consequence. 

When $\gamma>0$, we proceed as follows. Let $k=[\gamma-(2\beta+2)d/p]$, $k_{v}=[\gamma-d/p]$, $k_{x}=[\gamma-(2\beta+1)d/p]$. 
Observe that $k\le \min (k_{x},k_{v})$. Let us take a Cauchy sequence $(S_{j})$ for the two sides of the equivalence. 
It is not too hard to establish that there exist $f,\tilde f_{1}, \tilde f_{2} \in \cS'_{x,v}$ with 
$S_{j}\to f$ in  $\cS'_{x,v}/\cP_{k}[x,v]$, $S_{j} \to \tilde f_{1}$ in $\cS'_{x,v}/\cP_{k_{v}}[v]$ and  
$S_{j} \to \tilde f_{2}$ in $\cS'_{x,v}/\cP_{k_{x}}[x]$ (we indicated the variables of the polynomials for convenience). 
Fixing $f$, as $k\le k_{v}$,  we deduce that $f=\tilde f_{1}+P_{1}(v)$ with $\deg P_{1}\le k_{v}$. 
Similarly $f=\tilde f_{2}+P_{2}(x)$ with $\deg P_{2}\le k_{x}$. 
Setting $f_{1}=  \tilde f_{1}+P_{1}(v)$ and $f_{2}=\tilde f_{2}+P_{2}(x)$, 
we have shown $f=f_{1}=f_{2}$ in $\cS'(\R^{2d})$ and the set  equality follows. 
\end{proof} 

\subsubsection{Homogeneous anisotropic Besov spaces} 
\label{sec:HomAnBesov}
We follow the same strategy and start with the following lemma. 

\begin{lem} \label{lem:convergenceLPBeov} 
Let $\beta>0$, $\gamma\in \R$, $p \in (1,\infty)$. 
Let $(\theta_{j})$ be an anisotropic Littlewood--Paley family. 
\begin{enumerate}
	\item For $\chi\in \cS(\R^{2d})$ with $\int_{\R^{2d}} \chi P \, \dd(x,v) =0$ for all $P\in \cP_{k}$, 
	$k= [\gamma-(2\beta+2) d/p]$, we have 
	\[ 
		\sum_{j = - \infty}^\infty 2^{ -j \gamma p'} \norm{ \theta_j\ast \chi}_{\L^{p'}_{x,v}}^{p'} <\infty.
	\]
	\item Let  $g_{j}\in \cS'(\R^{2d}) \cap \C^\infty(\R^{2d})$, $j\in \Z$, with Fourier transform $\hat g_{j}$ supported in $C_{\beta, j}=\{(\varphi,\xi)\in \R^{2d}\, ;\,  2^{j-1} < d_{\beta}(\varphi,\xi) < 2^{j+1}\}$. 
	If $ \sum_{j = - \infty}^\infty 2^{ j \gamma p} \norm{ g_j}_{\L^p_{x,v}}^p <\infty$ then 
	$\sum_{j=1}^\infty g_{j}$ converges in $\cS'(\R^{2d})$ and $\sum_{j=-\infty}^0 g_{j}$ converges in 
	$\cS'(\R^{2d})/\cP_{k}$ where $k=[\gamma-(2\beta+2) d/p]$.
\end{enumerate}
\end{lem}

\begin{defn}[Homogeneous anisotropic Besov space] \label{defn:HABS}  
Let $\beta>0$, $\gamma\in \R$, $p \in (1,\infty)$. 
Let $(\theta_{j})$ be an anisotropic Littlewood--Paley family.
For $g\in \cS'(\R^{2d})$, we set
\[
	\norm{g}_{\Bdot_{\beta}^{\gamma,p}} = \bigg( \sum_{j = - \infty}^\infty 2^{ j \gamma p} \norm{ \theta_j\ast g}_{\L^p_{x,v}}^p \bigg)^\frac{1}{p} 
\]
and we let 
\[
	\Bdot^{\gamma,p}_{\beta}:= \left\{ g\in \cS'(\R^{2d})\, ; \, \|g\|_{\Bdot^{\gamma,p}_{\beta}} <\infty\ 
	\&\  g=\sum_{j=-\infty}^\infty  \theta_{j}\ast g \textrm{\ in \ } 
 	\cS'(\R^{2d})/\cP_{[\gamma-(2\beta+2)d/p]} \right\}.
\]
\end{defn}

The same paragraph and the density lemma after Definition~\ref{defn:HASS} apply to the space $\Bdot^{\gamma,p}_{\beta}$. 
For $p=2$, we have $\Bdot^{\gamma,2}_{\beta}=\Xdot^{\gamma,2}_{\beta}$.
For more information on anisotropic Besov spaces, we refer to \cite{MR2250142,MR3839617,MR4065179,MR2768550}. 

\subsubsection{Duality pairings}

We recall that there are duality pairings extending the $\L^2_{x,v}$ duality. 
For all $\beta\in (0,\infty)$, $p \in (1,\infty)$ and $\gamma\in \R$
\begin{align*}
| \angle f g | \lesssim \|f\|_{\Adot^{\gamma,p}_{\beta}}\|g\|_{\Adot^{-\gamma,p'}_{\beta}}
\end{align*}
for all $f\in \Adot^{\gamma,p}_{\beta}$ and $g\in \Adot^{-\gamma,p'}_{\beta}$,
where $\Adot=\Xdot$ or $\Bdot$ and $p'$ is the H\"older conjugate to $p$. When $f$ and $g$ are also $\L^2_{x,v}$ 
functions the bracket reduces to $\angle f g= \int_{\R^{2d}} f(x,v) \overline{g(x,v)} \dd(x,v) $.

\subsubsection{Kinetic spaces}
\label{sec:kineticspaces}

Fix $\beta>0$. Let $\gamma\in \R$ and $p\in (1,\infty)$. Having a definition of homogeneous anisotropic Sobolev and Besov spaces within tempered distributions, 
the precise definition of  kinetic spaces of regularity and integrability orders $\gamma, p$    on $\Omega=\R^{1+2d}$ is as follows. 
First, we introduce two spaces which will play a central role:  
\begin{equation}\label{eq:W}
	\Zdot^{\gamma,p}_{\beta}:= \L^p_{t}\Xdot^{\gamma-2\beta,p}_{\beta} + \L^1_{t}\Bdot^{\gamma-2\beta/p,p}_{\beta},
\end{equation}
\begin{equation}\label{eq:V}
	\Ydot^{\gamma,p}_{\beta}:= \L^p_{t}\Xdot^{\gamma,p}_{\beta} \cap \C^{}_{0}(\R^{}_{t}\, ;\, \Bdot^{\gamma-2\beta/p,p}_{\beta}).
\end{equation}
{They are equipped with the standard norms for sums and intersections of spaces.}
By construction of the anisotropic spaces, both $\Ydot^{\gamma,p}_{\beta} $ and $\Zdot^{\gamma,p}_{\beta}$ are subsets of $\cS'(\Omega)$ without restriction on $\gamma$. However, the inclusions in $\cS'(\Omega)$ are continuous if and only if 
$\gamma< \homd/p$, where 
\begin{equation} \label{eq:homodim}
	\homd:= 2\beta+ (2\beta+2)d
\end{equation}
is the homogeneous dimension for the kinetic dilations $(t,x,v) \to (\lambda^{2\beta} t, \lambda^{2\beta+1} x, \lambda v)$ 
for $\lambda>0$.
For $\Ydot^{\gamma,p}_{\beta}$, this is because it is an intersection, and the equivalence is already true for the second factor as 
$\gamma-2\beta/p< (2\beta+2)d/p$ exactly when $\gamma<\homd/p$, see Lemma~\ref{lem:convergenceLPBeov}; for $\Zdot^{\gamma,p}_{\beta} $, this is because the equivalence is true for the second factor as before and also because the condition $\gamma-2\beta  < (2\beta+2)d/p$ which is necessary and sufficient for the continuity of the inclusion of the first factor from Lemma~\ref{lem:anisoconvergenceLP}, is implied by $\gamma<\homd/p$.

 The case of weak solutions in $\L^2$, i.e. $p=2$ and $\gamma=\beta$, is very special as $\Bdot^{\gamma-2\beta/p,p}_{\beta}=\L^2_{x,v}$, so with the notation of \cite{AIN} and using Proposition~\ref{prop:characterisation}, $\Zdot^{\beta,2}_{\beta}=\Zdot^{\beta}$  and $\Ydot^{\beta,2}_{\beta}=\Ydot^{\beta}$.  The $\Zdot$ spaces are seen as source spaces and the $\Ydot$ spaces as target spaces for the Kolmogorov operators, below.
 
We introduce the kinetic Sobolev spaces by
\begin{equation}
\label{eq:Fdotgammabeta}
	\cFdot^{\gamma,p}_{\beta}= \left\{f\in \cD'(\Omega)\, ; \,  f \in \L^p_{t,x}\Hdot^{\gamma,p}_{v} \ \& \ (\partial_t + v \cdot \nabla_x)f \in \L^p_{t,x} \Hdot^{\gamma -2\beta,p}_{\vphantom{t,x} v} \right\}
\end{equation}
with (semi-)norm $ \|f\|_{\cFdot^{\gamma,p}_{\beta}}$ defined by 
\begin{equation}
\label{eq:Fdotgammabetanorm}
	\|f\|_{\cFdot^{\gamma,p}_{\beta}}^p= \| D_{v}^{\gamma} f\|_{\L^p_{t,x,v}}^p + \|(\partial_t + v \cdot \nabla_x)f\|_{\L^p_{t,x}\Hdot^{\gamma-2\beta,p}_{\vphantom{t,x} v}}^p.
 \end{equation}
Furthermore, we set
\begin{equation}
\label{eq:Kdotgammabeta}
	\cGdot^{\gamma,p}_{\beta}= \left\{ f \in \cD'(\Omega)\, ; \,  f \in \L^p_{t,x}\Hdot^{\gamma,p}_{v} \ \& \ (\partial_t + v \cdot \nabla_x)f \in \L^p_{t} \Xdot^{\gamma -2\beta,p}_\beta \right\}
\end{equation}
with (semi-)norm $ \|f\|_{\cGdot^{\gamma,p}_{\beta}}$ defined by
\begin{equation}
\label{eq:Gdotgammabetanorm}
 	\|f\|_{\cGdot^{\gamma,p}_{\beta}}^p= \| D_{v}^{\gamma} f\|_{\L^p_{t,x,v}}^p + \|(\partial_t + v \cdot \nabla_x)f\|_{\L^p_{t}\Xdot^{\gamma-2\beta,p}_\beta}^p,
\end{equation}
and
\begin{equation}
\label{eq:Ldotgammabeta}
	\cLdot^{\gamma,p}_{\beta}= \left\{ f \in \cD'(\Omega)\, ; \,  f \in \L^p_{t,x}\Hdot^{\gamma,p}_{v} \ \& \ (\partial_t + v \cdot \nabla_x)f \in \Zdot^{\gamma,p}_\beta \right\}
\end{equation}
with (semi-)norm $ \|f\|_{\cLdot^{\gamma,p}_{\beta}}$ defined by
\begin{equation}
\label{eq:Ldotgammabetanorm}
	\|f\|_{\cLdot^{\gamma,p}_{\beta}}^p= \| D_{v}^{\gamma} f\|_{\L^p_{t,x,v}}^p + \|(\partial_t + v \cdot \nabla_x)f\|_{\Zdot^{\gamma,p}_\beta}^p.
\end{equation}

At this stage, it is not clear whether they are normed spaces or complete. But they are all subsets of $\cS'(\Omega)$. 
Moreover, the following inclusions are continuous for the defining semi-norms. 
For all $\gamma\in \R$, we have $\cGdot^{\gamma,p}_{\beta}\subset \cLdot^{\gamma,p}_{\beta}$. 
Also by Proposition~\ref{prop:characterisation}, we have $\cFdot^{\gamma,p}_{\beta} \subset \cGdot^{\gamma,p}_{\beta}$ 
if $\gamma\le 2\beta$.

On $\R^{1+2d}$, we shall use the partial Fourier transform in the variables $(\varphi,\xi)$ dual to $(x,v)$ of a 
function $f= f(t,x,v)$  defined by
\[
	\widehat{f}(t,\varphi,\xi)=\int_{\R^{2d}} e^{-i(x\cdot \varphi+v\cdot\xi)}f(t,x,v) \dd(x,v) .
\]

\begin{lem}
\label{lem:density} 
Let $p \in (1,\infty)$, $\gamma\in \R$, $\beta>0$.
The space 
\[
	\cS_{K}= \left\{S\in \cS(\Omega)\, ;\, \widehat S \textrm{ has  compact support in  } \R^{}_{t}\times (\R_{\varphi}^{d}\setminus \{0\})\times (\R_{\xi}^{d}\setminus \{0\}) \right\}
\]
is dense in $\L^p_{t}\Xdot^{\gamma,p}_{\vphantom{t} \beta}$, in $\L^p_{t,x}\Hdot^{\gamma,p}_{v}$,  
in $\L^1_{t}\Bdot^{\gamma,p}_{\vphantom{t} \beta}$, in $\Zdot^{\gamma,p}_{\beta}$ and in $\Ydot^{\gamma,p}_{\beta}$.
\end{lem}

\begin{proof} 
The subspace $\cD(\R)\otimes \dense \subset \cS_{K}$, where $\dense$ is the subspace of Lemma~\ref{lem:denseanisotropic}, is dense by classical results for 
Bochner spaces $\L^p(\R;A)$ and for $\C_{0}^{}(\R; A)$ when $A$ is a Banach space and $\R$ equipped with the Lebesgue measure (here, we argue modulo polynomials to have complete normed spaces). 
Details are left to the reader.
 \end{proof}

\subsection{Forward and backward Kolmogorov operators of order $\beta$}
\label{sec:forbackkol}
First, we recall how one can construct solutions using a fundamental solution to \eqref{eq:kol}. 
Manipulations here are formal. Transforming the Kolmogorov equation in Fourier variables $(x,v) \mapsto (\varphi,\xi)$ yields  
\begin{equation*}
 	(\partial_t - \varphi \cdot \nabla_\xi) \hat{f} = - \abs{\xi}^{2\beta} \hat{f} + \hat{S}. 
\end{equation*}
Solving along characteristics yields the representation formula, 
\begin{equation*}
	\hat{f}(t,\varphi,\xi) = \int_{-\infty}^t \exp\left(-\int_s^{t} \abs{\xi +(t-\tau) \varphi}^{2\beta} \dd \tau \right) \hat{S}(s,\varphi,\xi+(t-s)\varphi) \ds, 
\end{equation*}
which corresponds to the kinetic convolution of Kolmogorov's fundamental solution with $S$.
We study this action in shifted variables under the Galilean change of variables $\Gamma$ (kinetic shift) defined 
by
\begin{equation}
\label{eq:kineticshift}
	[\Gamma f](t,x,v)=f(t,x+tv,v). 
\end{equation} 
As in  Fourier variables 
\begin{equation*}
	\widehat{\Gamma f}(t,\varphi,\xi) = \hat{f}(t,\varphi,\xi-t \varphi) 
\end{equation*}
we obtain
\begin{equation}
\label{e:conjugaison}
	\widehat{\Gamma f}= \frT_{\beta}^+(\widehat {\Gamma S}), 
\end{equation}
with
\begin{equation}
\label{eq:kolmogorovoperator}
	[\frT_{\beta}^+u](t,\varphi,\xi)= \int_{-\infty}^t \widehat{E}_{\beta}(t,s,\varphi,\xi)u(s,\varphi,\xi)\ds,
\end{equation}
where for $s,t\in \R$ and $(\varphi,\xi)\in \R^{2d}$, we set
\begin{equation}\label{eq:shiftedfundamentalsolution}
	\widehat{E}_\beta(t,s,\varphi,\xi) = \exp\bigg(-\int_s^{t} \abs{\xi -\tau \varphi}^{2\beta} \dd \tau \bigg).
\end{equation}
Of course, the integral in \eqref{eq:kolmogorovoperator} will be defined for some measurable functions $u$ of $(t,\varphi,\xi)$, 
but $S$ can be a distribution. 
As $\cS'(\Omega) \to \cS'(\Omega)$, $f \mapsto \widehat{\Gamma f}$ is an isomorphism, we set for $S \in \cS'(\Omega)$, 
\[ 
	f =\cK^{+}_{\beta} S 
\]  
whenever
\begin{equation}
\label{e:conjugaison1}
	 \frT_{\beta}^+(\widehat {\Gamma S}) \in \cS'(\Omega) \textrm{\ and\ } \widehat{\Gamma f}=\frT_{\beta}^+(\widehat {\Gamma S}).
\end{equation}
We define the operator $\cK^{-}_{\beta}$ as the formal adjoint to $\cK^{+}_{\beta}$ for the $\L^2_{t,x,v}$ sesquilinear duality. 
It is associated in the same fashion to the adjoint $\frT_{\beta}^-$ of $\frT_{\beta}^+$ given by the following formula,
\begin{equation}
\label{eq:kolmogorovoperatoradjoint}
	[\frT_{\beta}^-h](s,\varphi,\xi)= \int^{+\infty}_{s} \widehat{E}_\beta(t,s,\varphi, \xi)h(t,\varphi,\xi)\dd t.
\end{equation}
We call $\cK^{+}_\beta$ and $\cK^{-}_\beta$ the forward and backward Kolmogorov operators of order $\beta$, respectively.

We recall a result of \cite{AIN}, which allows us to perform calculations.

\begin{lem}\label{lem:equation}
If $S\in \cS_{K}$,  then $\cK^{\pm}_{\beta}S\in\C^{}_{0}(\R^{}_{t}\,;\, \L^2_{x,v})$, 
$ (\partial_{t}+v\cdot\nabla_{x})\cK^{\pm}_{\beta}S,  (-\Delta_{v})^\beta\cK^{\pm}_{\beta}S  \in \L^2_{t,x,v}$ and 
$\pm (\partial_{t}+v\cdot\nabla_{x})\cK^{\pm}_{\beta}S +(-\Delta_{v})^\beta\cK^{\pm}_{\beta}S=S$ in $\cS'(\Omega)$ 
$($even in the sense of strong solutions in $\L^2_{x,v})$. 
\end{lem}
In a first step, we may forget about the Kolmogorov equation and just study the boundedness properties of the Kolmogorov operators.

\subsection{Towards the $\L^p$ boundedness for Kolmogorov operators in the kinetic scale} \label{sec:towardsLp}
For $\gamma \in \R$, we are concerned with data $S \in \L^p_t\Xdot^{\gamma-2\beta,p}_\beta$, that is (again formally) 
\begin{equation*}
	\hat{S}(t,\varphi,\xi) = d_{\beta}(\varphi,\xi)^{2\beta-\gamma}	\hat{h}(t,\varphi,\xi),
\end{equation*} 
for some function $h \in \L^p_{t,x,v}$ and want to obtain  $f =\cK_{\beta}^+S \in \L^p_t\Xdot^{\gamma,p}_\beta$, that is, 
\begin{equation*}
	\cF^{-1}\left(d_{\beta}(\varphi,\xi)^\gamma \hat{f}(t,\varphi,\xi) \right) \in \L^p_{t,x,v}.
\end{equation*}
In other words, we are interested in the $\L^p_{t,x,v}$ boundedness of the operator $A$ formally defined as
\begin{align*}
	[\widehat{Ah}](t,\varphi,\xi)  &= \int_{-\infty}^t
	d_{\beta}(\varphi,\xi)^\gamma
	 \exp\left(-\int_s^{t} \abs{\xi +(t-\tau) \varphi}^{2\beta} \dd\tau \right) \\
	&\hspace{2cm} 
	\cdot d_{\beta}(\varphi,\xi+(t-s)\varphi)^{2\beta-\gamma}
	\hat{h}(s,\varphi,\xi+(t-s)\varphi) \ds. 
\end{align*}
Using the kinetic shift (with $g=\Gamma h$), we obtain with $m_{\beta}(t,\varphi,\xi)=d_{\beta}(\varphi,\xi-t\varphi)$ the relation
\[
	[\widehat{\Gamma A\Gamma^{-1} g}]  = m_{\beta}^\gamma \frT_{\beta}^+ (m_{\beta}^{2\beta-\gamma} \hat g),
\]
that is,
\begin{align}
\label{eq:Tgammahat}
	[\widehat{\Gamma A\Gamma^{-1} g}](t,\varphi,\xi)  
	&= \int_{-\infty}^t  d_{\beta}(\varphi,\xi-t\varphi)^\gamma d_{\beta}(\varphi,\xi-s\varphi)^{2\beta-\gamma}\\
	&\nonumber\hspace{2cm} 
	\cdot \exp\left(-\int_s^{t} \abs{\xi  -\tau \varphi}^{2\beta} \dd \tau \right) [\widehat{ \Gamma h}](s,\varphi,\xi) \ds. 
\end{align}

The interest of shifting variables is that for fixed $s<t$ we have a multiplier in the $(\varphi,\xi)$ variables,
hence to obtain an operator of convolution type in the spatial variables $(x,v)$. 
However, it is not of a convolution type in the time variable.
Setting $T_{\gamma}\coloneqq \Gamma A\Gamma^{-1}$, we have formally  
\begin{equation} \label{eq:defTgamma}
	[T_\gamma g](t,x,v) = \int_{\R^{1+2d}} \mathds{1}_{t-s>0}K_\gamma(t,s,x-y,v-w) g(s,y,w) \dd (s,y,w) 
\end{equation}  where
\begin{equation}
\label{eq:Kgamma}
	\widehat{K}_\gamma(t,s,\varphi,\xi)  = d_{\beta}(\varphi,\xi-t\varphi)^\gamma d_{\beta}(\varphi,\xi-s\varphi)^{2\beta-\gamma} \exp\left(-\int_s^{t} \abs{\xi  -\tau \varphi}^{2\beta} \dd\tau \right). 
\end{equation}
We have thus reduced our problem to proving $\L^p_{t,x,v}$ boundedness of $T_{\gamma}$. 
As said, we cannot apply multiplier theory in all variables simultaneously. 
But, we can apply the theory of singular integral operators on $\R^{1+2d}$ equipped with a (shifted) kinetic distance, 
and in particular the theorem of Coifman--Weiss in spaces of homogeneous type, Theorem~\ref{thm:CW}, to    show 
\begin{thm}\label{thm:boundednessTgamma}
	For $\beta \in (0,1]$ and $\gamma \in \R$, the operator $T_{\gamma}$ is bounded on $\L^p_{t,x,v}$, $p \in (1,\infty)$.
\end{thm}
Of course,  even the meaning of the representation of $T_{\gamma}$ is an issue. 
The first thing to address is $\L^2_{t,x,v}$ boundedness, which is already known from \cite{AIN}.

\begin{lem} \label{lem:TgammaL2}
	For all $\gamma\in \R$ and $\beta>0$, $T_{\gamma}$ is well-defined and bounded on $\L^2_{t,x,v}$. 
\end{lem}

\begin{proof} 
By Plancherel, it suffices to show that \eqref{eq:Tgammahat} defines a bounded operator on $\L^2_{t,\varphi,\xi}$, 
thus defining $T_{\gamma}$ by conjugation with the Fourier transform.  
This is a direct corollary of \cite[Prop~2.18]{AIN} and its proof. 
Indeed, observe that  
$ d_{\beta}(\varphi,\xi-t\varphi)^\gamma\sim  (|\varphi|^\frac{1}{2\beta+1}+ |\xi-t\varphi|)^{\gamma} \sim W(t,\varphi,\xi)^\gamma$
and 
$ d_{\beta}(\varphi,\xi-s\varphi)^{2\beta-\gamma} \sim (|\varphi|^\frac{1}{2\beta+1}+ |\xi-s\varphi|)^{2\beta-\gamma} \sim  W(s,\varphi,\xi)^{2\beta-\gamma}$
where $W$ is the weight used in \cite{AIN}. 
The proof of the first item of this proposition (even for all $\beta>0$) is done exclusively using size estimates 
for $W$ and the exponential factor, so the same proof applies with $d_{\beta}$ replacing $W$ and exactly gives 
the $\L^2_{t,x,v}$ boundedness of $T_{\gamma}$ by taking the inverse Fourier transform.  
\end{proof}

Next, we prove the required estimates for $K_{\gamma}$ to apply Theorem~\ref{thm:CW}.   

\begin{rem}
Let us also make a comment for experts on $\L^p$-estimates for autonomous parabolic PDEs. 
Here, the $\L^p$-boundedness at the level of strong solutions implies an $\L^p$ theory on the full Sobolev scale by ``differentiating'' the equation. 
This approach is not feasible in the kinetic setting due to derivatives in $v$ not commuting with the equation, see \cite[p.~59]{niebel_kinetic_2022}. 
\end{rem}

\section{Integrated kernel estimates}
\label{sec:estimates}

We prove pointwise estimates for the $x$-marginal and the $v$-marginal along characteristics of the kernel $K_\gamma(t,s,x,v)$ 
and its derivatives $\partial_s K_\gamma$ as well as $(\nabla_v -s \nabla_x)K_\gamma$. 
These estimates will be easy to manipulate and are sufficient for proving $\L^p$ estimates. 
Of course, it would be desirable to obtain pointwise estimates with good decay, 
but this seems out of reach. We comment on this in Remark \ref{rem:pointwiseest}.

In this section, we consider $\beta \in (0,1]$. The integrated estimates are valid for all $\beta\in \N$ though. We use L\'evy measures when $\beta<1$ and do not know what happens when $\beta\in (1,\infty)\setminus \N$.  We leave as an open problem to find analytic proofs of estimates that apply to all fractional $\beta$'s. 
 
\begin{prop}
\label{prop:intkernel}
	Let $\beta \in (0,1]$ and $\gamma \in \R$. 
	{For fixed $s<t$ the kernel is smooth, i.e.\ $K_\gamma(t,s,\cdot) \in \C^\infty(\R^{2d})$.}
	For any $\mathsf K \in \{ K_\gamma, (t-s)\partial_s K_\gamma, (t-s)^{\frac{1}{2\beta}}(\nabla_v- s \nabla_x) K_\gamma \}$ 
	{we have for the $x$-marginal}
	\begin{equation} \label{eq:Kxmarginal}
		\int_{\R^d} \abs{\mathsf K (t,s,x,v)}  \dx  \lesssim_{\beta,\gamma,d}
		\frac{1}{t-s} \frac{t-s}{(\abs{t-s}^{\frac{1}{2\beta}}+\abs{v})^{d+2\beta}},
	\end{equation}
	{as well as for the $v$-marginal along characteristics}
	\begin{equation}  \label{eq:Kvmarginal}
		\int_{\R^d} \abs{\mathsf K (t,s,x-sv,v)} \dv \lesssim_{\beta,\gamma,d} 
		\frac{1}{t-s} \frac{t-s}{(\abs{t-s}^{\frac{1}{2\beta}}+\abs{x}^{\frac{1}{2\beta+1}})^{(2\beta+1)d+2\beta}}
	\end{equation}
	for all $s<t$ and any $x,v \in \R^d$. 
	As a consequence, 
	\begin{equation}  \label{eq:Kxvmarginal}
		\int_{\R^{2d}} \abs{\mathsf K (t,s,x,v)}  \dd(x,v) \lesssim_{\beta,\gamma,d}\frac{1}{t-s}.
	\end{equation}
\end{prop} 

Another implication of our proof is the following upper bound.
\begin{prop} \label{prop:upper}
	We have the pointwise upper bound
\begin{equation*}
	\abs{K_\gamma(t,s,x,v)} \lesssim_{\beta,\gamma,d} \frac{1}{(t-s)^{d+\frac{d}{\beta}+1}}
\end{equation*}
for all $s<t$ and any $x,v \in \R^{d}$. 
\end{prop}

The following subsections are concerned with parts of the proof of Proposition \ref{prop:intkernel}, 
which is concluded in Section \ref{sec:proofintest}. 

\subsubsection{Rescaled variables}
We decompose  ${K}_\gamma = {p}_\gamma(D){E}_\beta $, where ${p}_\gamma(D)$ has symbol
\begin{equation*}
	{p}_\gamma(t,s,\varphi,\xi) = 
	d_{\beta}(\varphi,\xi-t\varphi)^\gamma
	d_{\beta}(\varphi,\xi-s\varphi)^{2\beta-\gamma}, 
\end{equation*} 
and is applied to the shifted  fundamental solution $E_{\beta}$, whose Fourier transform is 
\begin{equation*}
	\widehat{E}_\beta(t,s,\varphi,\xi) = \exp\bigg(-\int_s^{t} \abs{\xi -\tau \varphi}^{2\beta} \dd \tau \bigg).
\end{equation*}
Next, we introduce the new \textit{rescaled} variables $(\varphi',\xi')$ as
\begin{equation} \label{eq:rescaledvar}
	\xi' = (t-s)^{\frac{1}{2\beta}} (\xi-s\varphi) \;\mbox{ and }\; \varphi' = (t-s)^{1+\frac{1}{2\beta}} \varphi,
\end{equation}
then
\begin{equation*}
	(t-s) p_\gamma(t,s,\varphi,\xi) = \tilde{p}_\gamma(\varphi',\xi') = 
	d_{\beta}(\varphi',\xi'-\varphi')^\gamma
	d_{\beta}(\varphi',\xi')^{2\beta-\gamma}
\end{equation*}
and
\begin{equation} \label{eq:tildem}
	\widehat{E}_\beta(t,s,\varphi,\xi) = \hat{\tilde{E}}_\beta(\varphi',\xi') = \exp\bigg(- 	\int_0^1 \abs{\xi'-r\varphi'}^{2\beta}  \dr\bigg).
\end{equation} 
By abuse of notation, we {drop primes and} write $\varphi,\xi$ from now on (think $t = 1$, $s = 0$) 
and we set $\hat{\tilde{K}}_\gamma = {\tilde{p}}_\gamma \hat{\tilde{E}}_\beta$.

\medskip

Let us briefly explain the difficulties with $\tilde{K}_\gamma$ and the strategy to obtain integrated estimates. 
The symbol of $\tilde{p}_\gamma$ is not smooth for small frequencies ${(\varphi,\xi)} \approx (0,0)$. 
This limits the decay which we can expect in physical variables. 
Moreover, the rescaled and shifted fundamental solution $\tilde{E}_\beta$ has problems with differentiability 
at points in $\{ (\varphi,\xi); \xi = r\varphi, r \in [0,1] \}$. We make use of a decomposition of the exponential term into a Schwartz function $\Psi_g$ 
and measure part $\Psi_b$. 
The measure part introduces another limiting factor to the pointwise decay. We estimate $\tilde{p}_\gamma \Psi_g$ 
and combine it with estimates for the measure part to obtain the desired decay of the kernel integrated in either $x$ or $v$. 
	
After deriving integrated kernel estimates for $\tilde{K}_\gamma$, 
we deduce estimates for the kernel $K_\gamma$ by reversing the change of variables in \eqref{eq:rescaledvar}. 
	
Furthermore, we collect two observations. 
The symbol in the fundamental solution is nondegenerate  in the sense that 
\begin{equation} \label{eq:nondegen}
	\int_0^1 \abs{\xi-r\varphi}^{2\beta}  \dr \approx \abs{\varphi}^{2\beta}+\abs{\xi}^{2\beta}
\end{equation}
for all $(\varphi,\xi) \in \R^{2d}$; see \cite[Lemma 2.14]{AIN}. 
The symbol of $\tilde{p}_\gamma$ is of at most polynomial growth {(compare Lemma \ref{lem:xi-phi})} and hence 
$\tilde{K}_\gamma \in \C^\infty(\R^{2d})$ with an $\L^\infty$ bound due to the exponential decay of $\tilde{E}_\beta$. 

\subsubsection{Decomposition of the kernel into a good and a bad part}
If $\beta = 1$, then the exponent of $\tilde{E}_\beta$ is a coercive quadratic polynomial and 
hence $\tilde{E}_\beta \in \cS(\R^{2d})$ is a Schwartz function. 
In the case $\beta \in (0,1)$, the situation is more complicated.  
In this section, we will decompose the rescaled shifted fundamental solution as $\tilde{E}_\beta = \Psi_g \ast \Psi_b$, 
i.e.\ the convolution of a Schwartz function $\Psi_g$ and a probability measure $\Psi_b$.

Let $N \in \N$. We abbreviate 
$$
\omega^{(\beta)}:= \begin{cases}
	0 & \beta < \frac{1}{2} \\
	\omega \mathds{1}_{\abs{\omega} \le 1} & \beta = \frac{1}{2} \\
	\omega & \frac{1}{2} < \beta < 1
\end{cases},
$$
for any $\omega \in \R^{N}$. Moreover, we set 
\begin{equation*}
	\frL_\beta(a,b) = 1+ia \cdot b^{(\beta)} -\exp(i a \cdot b). 
\end{equation*}
for $a,b \in \R^N$. We start by recalling  
\begin{equation} \label{eq:charexpfraclaplace}
	\abs{a}^{2\beta} = c_{\beta,d}\int_{0}^\infty \int_{\S^{d-1}} \frL_\beta(a,\rho\omega) \rho^{-1-2\beta} \dd \omega \dd \rho 
\end{equation}
for all $a \in \R^d$ and some constant $c_{\beta,d}>0$, see \cite[Section 14]{sato_levy_1999}. 
{By symmetry the imaginary terms vanish, so effectively $\frL_\beta(a,b)$ can be replaced by $1-\cos(a\cdot b)$ in our applications.
We will later use this formula with $N = d$ and $N = 2d$.}

Inspired by \cite{MR3826548,hou_kernel_2024}, we split up the characteristic exponent with the new idea 
to keep some of the good coercive term in the bad term. We write
\begin{align*}
	\int_0^1 \abs{\xi-r\varphi}^{2\beta}  \dr &= c_{\beta,d}\int_0^1\int_{0}^\infty \int_{\S^{d-1}} \frL_\beta(\xi-r\varphi,\rho\omega) \rho^{-1-2\beta} \dd \omega \dd \rho \dr \\
	&= c_{\beta,d}\int_0^1\int_{0}^1 \int_{\S^{d-1}} \frL_\beta(\xi-r\varphi,\rho\omega) \rho^{-1-2\beta} \dd \omega \dd \rho \dr \\
	&\hphantom{=}+ c_{\beta,d}\int_0^1\int_{1}^\infty \int_{\S^{d-1}} \frL_\beta(\xi-r\varphi,\rho\omega) \rho^{-1-2\beta} \dd \omega \dd \rho \dr \\
	&= \frac{1}{2}c_{\beta,d}\int_0^1\int_{0}^1 \int_{\S^{d-1}} \frL_\beta(\xi-r\varphi,\rho\omega) \rho^{-1-2\beta} \dd \omega \dd \rho \dr \\
	&\hphantom{=}+\left( \frac{1}{2}c_{\beta,d}\int_0^1\int_{0}^1 \int_{\S^{d-1}} \frL_\beta(\xi-r\varphi,\rho\omega) \rho^{-1-2\beta} \dd \omega \dd \rho \dr \right. \\
	&\hphantom{=}+\left. c_{\beta,d}\int_0^1\int_{1}^\infty \int_{\S^{d-1}} \frL_\beta(\xi-r\varphi,\rho\omega) \rho^{-1-2\beta} \dd \omega \dd \rho \dr \right)\\
	&=: \psi_g(\varphi,\xi)+\psi_b(\varphi,\xi)
\end{align*}
for all $(\varphi,\xi) \in \R^{2d}$. 
In order to streamline notation we set $\psi_g(\varphi,\xi) = \int_0^1 \abs{\xi-r\varphi}^{2}  \dr$ 
in the case $\beta = 1$ with the convention that $\psi_b = 0$.

\begin{lem}
	The function $\Psi_g := \cF^{-1}(\exp(-\psi_g))$ is a Schwartz function and we have the estimate 
	\begin{equation} \label{eq:psigcoercive}
		\psi_g(\varphi,\xi) \ge c_0 (\abs{\varphi}^2+\abs{\xi}^2) \wedge (\abs{\varphi}^{2\beta}+\abs{\xi}^{2\beta})
	\end{equation}
	for some $c_0 > 0$ and any $(\varphi,\xi) \in \R^{2d}$. 
\end{lem}

\begin{proof}
	This has been proven in \cite[Proof of Lemma 2.5 part (iii)]{MR3826548}.
\end{proof}

\begin{lem}
	For $\beta \in (0,1)$ the function $\exp(-\psi_b)$ is the Fourier transform of a probability measure $\Psi_b$. 
	The probability measure is absolutely continuous with respect to the Lebesgue measure with a $\C^\infty$ density. 
	For $\beta = 1$, the function $1 \equiv \exp(-\psi_b)$ is the Fourier transform 
	of the Dirac measure with mass in $(x,v) = (0,0)$.
\end{lem}

\begin{proof}
	The case $\beta = 1$ is trivial.
	
	Let $\beta \in (0,1)$. We write
	\begin{align*}
		\psi_b(\varphi,\xi)&=\left( \frac{1}{2}c_{\beta,d}\int_0^1\int_{0}^\infty \int_{\S^{d-1}} \frL_\beta(\xi-r\varphi,\rho\omega) \rho^{-1-2\beta} \dd \omega \dd \rho \dr \right. \\
		&\hphantom{=}+\left. \frac{1}{2}c_{\beta,d}\int_0^1\int_{1}^\infty \int_{\S^{d-1}} \frL_\beta(\xi-r\varphi,\rho\omega) \rho^{-1-2\beta} \dd \omega \dd \rho \dr \right) =:I+I\!I.
	\end{align*}
	The first term is coercive by \eqref{eq:nondegen}. The second term is bounded. We know that 
	\begin{equation*}
		0 \le c_{\beta,d}\int_{1}^\infty \int_{\S^{d-1}} \frL_\beta(\xi-r\varphi,\rho\omega) \rho^{-1-2\beta} \dd \omega \dd \rho  \lesssim 1
	\end{equation*}
	for all $(\varphi,\xi) \in \R^{2d}$, $r \in [0,1]$. Indeed, due to the symmetry of $\dd \omega$ we may rewrite 
	\begin{align*}
		\int_{\S^{d-1}}\frL_\beta(\xi-r\varphi,\rho \omega) \dd \omega &= \mathrm{Re}\int_{\S^{d-1}}\left( 1+i(\xi-r\varphi)\cdot (\rho\omega)^{(\beta)}-\exp\left(i(\xi-r\varphi)\cdot (\rho\omega) \right) \right) \dd \omega \\
		&= \int_{\S^{d-1}} 1-\cos((\xi-r\varphi)\cdot \omega \rho) \dd \omega,	
	\end{align*}
	for any $\rho>0$, $(\varphi,\xi) \in \R^{2d}$. This can be seen by looking at the series representation of the exponential function. 
	The term $II$ is bounded, while $I$ is coercive. Consequently, $\psi_b$ satisfies
	\begin{equation} \label{eq:psibcoercive}
		\psi_b(\varphi,\xi) \gtrsim \abs{\varphi}^{2\beta}+\abs{\xi}^{2\beta}.
	\end{equation}
	Furthermore, as $\psi_b(0,0) = 0$ the function $\exp(-\psi_b)$ is the Fourier transform 
	of a (sign-changing) function with integral equal to one. 
	
	In order to prove the positivity of this function, we aim to apply the L\'evy-Khintchine theorem, \cite[Theorem 1.2.14]{MR2512800}. 
	For that, we rewrite the symbol as an effective $2d$-L\'evy measure; compare \cite{MR3906169}. 

We calculate  
\begin{align*}
	I\!I&=\frac{1}{2}c_{\beta,d}\int_0^1\int_{1}^\infty \int_{\S^{d-1}} \frL_\beta\left(\xi-r\varphi,\rho\omega\right) \rho^{-1-2\beta} \dd \omega \dd \rho \dr \\
	&= \frac{1}{2}c_{\beta,d}\int_0^1\int_{1}^\infty \int_{\S^{d-1}} \frL_\beta\left({\varphi \choose \xi },\rho {-r\omega \choose  \omega}\right) \rho^{-1-2\beta} \dd \omega \dd \rho \dr \\
	&= \frac{1}{2}c_{\beta,d}\int_0^1\int_{{\sqrt{1+r^2}}}^\infty \int_{\S^{d-1}} \frL_\beta\left({\varphi \choose \xi },\frac{\tilde{\rho}}{\sqrt{1+r^2}} {-r\omega \choose  \omega}\right) \tilde{\rho}^{-1-2\beta} (1+r^2)^{\beta} \dd \omega \dd \tilde{\rho} \dr \\
	&= \frac{1}{2}c_{\beta,d}\int_1^{{\sqrt{2}}} \int_{0}^{\sqrt{\rho^2-1}} \int_{\S^{d-1}} \frL_\beta\left({\varphi \choose \xi },\frac{{\rho}}{\sqrt{1+r^2}} {-r\omega \choose  \omega}\right) {\rho}^{-1-2\beta} (1+r^2)^{\beta} \dd \omega \dr  \dd {\rho} \\
	&\hphantom{=}+\frac{1}{2}c_{\beta,d}\int_{\sqrt{2}}^\infty \int_0^1 \int_{\S^{d-1}} \frL_\beta\left({\varphi \choose \xi }, \frac{{\rho}}{\sqrt{1+r^2}} {-r\omega \choose  \omega}\right) {\rho}^{-1-2\beta} (1+r^2)^{\beta} \dd \omega \dr  \dd {\rho} \\
	&=  \int_1^{{\sqrt{2}}} \int_{\S^{2d-1}} \frL_\beta\left({\varphi \choose \xi },{\rho}\theta \right) {\rho}^{-1-2\beta} \dd \bar{\mu}_1(\rho,\theta)  \dd {\rho} \\
	&\hphantom{=}+ \int_{\sqrt{2}}^\infty \int_{\S^{2d-1}} \frL_\beta\left({\varphi \choose \xi },{\rho}\theta \right) {\rho}^{-1-2\beta} \dd \bar{\mu}_2(\theta)  \dd {\rho}.
\end{align*}
Here, we substitute $\tilde{\rho} = \sqrt{1+r^2}\rho$ in the second equation and use Fubini's theorem in the third equation 
and drop the tildes in the fourth equality. 

We introduce $f \colon [0,1] \times \S^{d-1} \to \S^{2d-1}$ as 
$$
	f(r,\omega) = \frac{1}{\sqrt{1+r^2}}{-r\omega\choose \omega}
$$
and measures on $[0,1] \times \S^{d-1}$ as
\begin{equation*}
	\dd\mu_1(\rho;r,\omega) = \frac{1}{2}c_{\beta,d}(1+r^2)^{\beta} \mathds{1}_{\left[ 0,{\sqrt{\rho^2-1}} \right]}(r)\dd \omega \dr, \quad 
	\dd\mu_2(r,\omega) = \frac{1}{2}c_{\beta,d}(1+r^2)^{\beta} \dd \omega \dr
\end{equation*}
in order to define $\bar{\mu}_1(\rho;\cdot) = \mu_1(\rho;\cdot) \circ f^{-1}$, $\rho \in [1,\sqrt{2}]$, 
and $\bar{\mu}_2 = \mu_2 \circ f^{-1}$ as the pushforward measures on $\S^{2d-1}$ of $\mu_1(\rho,\cdot)$ and $\mu_2$ under $f$. 

Similarly, we derive a representation for 
\begin{align*}
	I&=\frac{1}{2}c_{\beta,d}\int_0^1\int_{0}^\infty \int_{\S^{d-1}} \frL_\beta\left(\xi-r\varphi,\rho\omega\right) \rho^{-1-2\beta} \dd \omega \dd \rho \dr \\
	&=  \int_0^\infty \int_{\S^{2d-1}} \frL_\beta\left({\varphi \choose \xi },{\rho}\theta \right) \rho^{-1-2\beta} \dd \bar{\mu}_3(\theta) \dd \rho
\end{align*}
with $\dd\mu_3(r,\omega) = \frac{1}{2}c_{\beta,d}(1+r^2)^{\beta} \dd \omega \dr$ and $\bar{\mu}_3 = \mu_3 \circ f^{-1}$.

We note that the measure
\begin{equation*}
	\dd\pi = \left(\mathds{1}_{[1,\sqrt{2}]} \dd \bar{\mu}_1(\rho,\theta) + \mathds{1}_{[\sqrt{2},\infty)}\dd \bar{\mu}_2(\theta) 
	+ \mathds{1}_{[0,\infty)}\dd \bar{\mu}_3(\theta) \right) \rho^{-1-2\beta} \dd \rho  
\end{equation*}
is a L\'evy measure, i.e. $\int_{\R^{2d} \setminus \{ 0 \}} \abs{z}^2 \land 1 \dd \pi (z) < \infty  $. 
This proves that $\Psi_b := \cF^{-1}(\exp(-\psi_b))$ is a probability measure. 
It is absolutely continuous with respect to the Lebesgue measure with a $\C^\infty$ density by the coercivity.
\end{proof}

\begin{rem}
	It is crucial for our proof that we know that $\Psi_b$ is a measure, and, in particular, that the density is nonnegative. 
	This is because we will only provide upper bounds on the marginals. 
\end{rem}

\begin{rem} \label{rem:analyticdecay}
	 Kernel estimates for degenerate L\'evy measures can be found in \cite{bogdan_heat_2020,watanabe_asymptotic_2007}, 
	 see also the references therein. 
	 The effective measure considered above is (locally) $d+1$-dimensional at $\S^{2d-1}$ in the sense of \cite{bogdan_heat_2020}.
	 This implies pointwise kernel estimates, proven in \cite{bogdan_heat_2020} with a very nice analytical argument, as
	$$
		|\tilde{E}_\beta(x, v)| \lesssim_{\beta,d} (1+|x|+|v|)^{-d-1-2 \beta}.
	$$
	This observation has been made by \cite[p. 28 at the top and p. 66]{marino2023weak} 
	in the context of density estimates for kinetic stochastic differential equations. 
	In general, this decay is integrable only for $d = 1$ and not strong enough for our applications if $d>1$. 
\end{rem}

\subsubsection{Estimates for the bad part}
In this section, we derive estimates for the marginals of the bad part. 
We use as a black-box kernel estimates for nondegenerate L\'evy measures as can be found in \cite{sztonyk_transition_2011}. 

We recall the following elementary lemma, which will be used again in the proof of Proposition \ref{prop:Ktilde}. 
It relates the computation of the marginal with an evaluation of one of the Fourier variables at the zero mode. 

\begin{lem} \label{lem:L1fourier}
	Given any function $f \in \L^1(\R^{2d})$ we have
\begin{align*}
	\int_{\R^d} f(x,v) \dx  = \cF_v^{-1} (\hat{f}(0,\cdot))(v), \; \mbox{a.e.}
\end{align*}
where $\hat{f} = \cF_{x,v}(f)$. This can be extended to the Fourier transform of a measure $\mu $ on $\R^{2d}$ as 
\begin{equation*}
	\int_{\R^{2d}} \eta(v) \dd \mu(x,v) = \int_{\R^d} \hat{\mu}(0,\xi) \cF^{-1}_v(\eta)(\xi) \dd \xi
\end{equation*}
for $\eta \in \C_c^\infty(\R^{d})$.
\end{lem}

\begin{lem} \label{lem:bad}
	Let $\beta \in (0,1)$. The $x$ and $v$ marginals of $\Psi_b$ are absolutely continuous with respect to the Lebesgue measure.
	For their respective densities, we have 
	\begin{equation*}
		 0 \le \cF^{-1}(\exp(-\psi_b(0,\cdot)))(v) \lesssim_{\beta,d} (1+\abs{v})^{-d-2\beta}
	\end{equation*}
	for all $v \in \R^d$ and
	\begin{equation*}
		 0 \le \cF^{-1}(\exp(-\psi_b(\cdot,0)))(x) \lesssim_{\beta,d} (1+\abs{x})^{-d-2\beta}
	\end{equation*}
	for all $x \in \R^d$.
\end{lem}

\begin{proof}
	Using the symmetry of $\dd \omega$ we may replace $\frL_\beta$ by $1-\cos$ in the following.
	
	We calculate the $x$-marginal first,
	\begin{align*}
		\psi_b(0,\xi) &=  \frac{1}{2}c_{\beta,d}\int_0^1\int_{0}^1 \int_{\S^{d-1}} \frL_\beta(\xi,\rho\omega) \rho^{-1-2\beta} \dd \omega \dd \rho \dr \\
	&\hphantom{=}+ c_{\beta,d}\int_0^1\int_{1}^\infty \int_{\S^{d-1}} \frL_\beta(\xi,\rho\omega) \rho^{-1-2\beta} \dd \omega \dd \rho \dd r \\
	&=\frac{1}{2}c_{\beta,d}\int_{0}^1 \int_{\S^{d-1}} \frL_\beta(\xi,\rho\omega) \rho^{-1-2\beta} \dd \omega \dd \rho  \\
	&\hphantom{=}+ c_{\beta,d}\int_{1}^\infty \int_{\S^{d-1}} \frL_\beta(\xi,\rho\omega) \rho^{-1-2\beta} \dd \omega \dd \rho \\
	&= \frac{1}{2}\abs{\xi}^{2\beta}+ \frac{1}{2}c_{\beta,d}\int_{1}^\infty \int_{\S^{d-1}} \frL_\beta(\xi,\rho\omega) \rho^{-1-2\beta} \dd \omega \dd \rho.
	\end{align*}

	The marginal is absolutely continuous with respect to the Lebesgue measure, and its density is given by 
	$\cF^{-1}(\exp(-\psi_b(0,\cdot)))$. 
	This is due to the coercivity in \eqref{eq:psibcoercive} with $\varphi = 0$, 
	which implies the integrability of the exponential. 
	Hence, the density $\cF^{-1}(\exp(-\psi_b(0,\cdot)))$ is bounded. 
	The characteristic exponent can be written in terms of the L\'evy measure
	\begin{equation*}
		\dd \pi = \frac{1}{2} c_{\beta,d} \left( \mathds{1}_{[0,1)}(\rho)+2\,\mathds{1}_{[1,\infty)}(\rho) \right) \rho^{-1-2\beta} \dd \omega \dd \rho. 
	\end{equation*}
	This is dominated by a multiple of the L\'evy measure of the $d$-dimensional $2\beta$ fractional Laplacian. 
	Applying \cite[Theorem 1]{sztonyk_transition_2011} at $t = 1$ with 
	$\gamma = d$, $q \equiv c_{\beta,d}$, $\phi \equiv 1$, $\alpha = 2\beta$, $\beta = \alpha$ 
	we conclude the desired density bound. 
	
Let us now consider the $v$-marginal. We derive 
	\begin{align*}
		\psi_b(\varphi,0) &= \frac{1}{2}c_{\beta,d}\int_0^1\int_{0}^\infty \int_{\S^{d-1}} \frL_\beta(-r\varphi,\rho\omega) \rho^{-1-2\beta} \dd \omega \dd \rho \dr  \\
		&\hphantom{=}+\frac{1}{2} c_{\beta,d}\int_0^1\int_{1}^\infty \int_{\S^{d-1}} \frL_\beta(-r\varphi,\rho\omega) \rho^{-1-2\beta} \dd \omega \dd \rho \dr  \\
		&= \frac{1}{2(2\beta+1)} \abs{\varphi}^{2\beta}+\frac{1}{2}c_{\beta,d}\int_0^1\int_{1}^\infty \int_{\S^{d-1}} \frL_\beta(r\varphi,\rho\omega) \rho^{-1-2\beta} \dd \omega \dd \rho \dr \\
		&= \frac{1}{2(2\beta+1)} \abs{\varphi}^{2\beta}+\frac{1}{2}c_{\beta,d}\int_0^1\int_{r}^\infty \int_{\S^{d-1}} \frL_\beta(\varphi,\rho\omega) \rho^{-1-2\beta} \dd \omega \, r^{2\beta }\dd \rho \dr \\
		&= \frac{1}{2(2\beta+1)} \abs{\varphi}^{2\beta}+\frac{1}{2}c_{\beta,d}\int_0^1\int_{0}^\rho \int_{\S^{d-1}} \frL_\beta(\varphi,\rho\omega) \rho^{-1-2\beta} \dd \omega\,  r^{2\beta }\dr \dd \rho  \\ 
		&\hphantom{=}+ \frac{1}{2}c_{\beta,d}\int_1^\infty\int_{0}^1 \int_{\S^{d-1}} \frL_\beta(\varphi,\rho\omega) \rho^{-1-2\beta} \dd \omega\,  r^{2\beta }\dr \dd \rho \\
		&= \frac{1}{2(2\beta+1)} \abs{\varphi}^{2\beta}+\frac{1}{2(2\beta+1)}c_{\beta,d}\int_0^1 \int_{\S^{d-1}} \frL_\beta(\varphi,\rho\omega)  \dd \omega\dd \rho  \\ 
		&\hphantom{=}+ \frac{1}{2(2\beta+1)}c_{\beta,d}\int_1^\infty \int_{\S^{d-1}} \frL_\beta(\varphi,\rho\omega) \rho^{-1-2\beta} \dd \omega \dd \rho.
	\end{align*}
	We use the identity \eqref{eq:charexpfraclaplace} in the first equation. 
	In the second equation, we substituted $\omega \mapsto -\omega$. 
	We rescaled $\rho $ in the third equation and integrated in $r$ in the last equation. 
		
	Again, due to the coercivity as in \eqref{eq:psibcoercive} with $\xi = 0$, 
	we obtain that the Fourier inverse of $\exp(-\psi_b(\cdot,0))$ is absolutely continuous with respect to the Lebesgue measure. 
	We identify the density with $\cF^{-1}(\exp(-\psi_b(\cdot,0)))(x)$, which is bounded.
	 
	In comparison to the above, we have one additional term. 
	Writing the exponent in terms of the L\'evy measure
	\begin{equation*}
		\dd \pi = \frac{1}{2(2\beta+1)}c_{\beta,d} \left( (1+\rho^{2\beta+1})\mathds{1}_{[0,1]}(\rho)+ 2\mathds{1}_{[1,\infty)}(\rho) \right) \rho^{-1-2\beta}\dd \omega \dd \rho
	\end{equation*}
	we are in the position to apply \cite[Theorem 1]{sztonyk_transition_2011} at $t = 1$ with 
	$\gamma = d$, $q \equiv \frac{1}{2\beta+1}c_{\beta,d}$, $\phi \equiv 1$, $\alpha = 2\beta$, $\beta = \alpha$. 
	This proves the lemma. 
\end{proof}

\subsubsection{A kinetic Littlewood--Paley lemma}

The following lemma is of independent interest, as it allows one to prove the anisotropic (w.r.t. the kinetic scaling) 
decay of a function built by a Littlewood--Paley sum. 
As usual, $\C^k(\R^{2d};\R)$ is the space of functions of class $\C^k$, bounded with all its derivatives up to order $k$ 
and endowed with the supremum of the $\L^\infty$ norms of all its partial derivatives.
 
\begin{lem} \label{lem:kinetic}
	Let $k \in \N$ with $k > \homd = (2\beta+2)d+2\beta$, $\Lambda$ be an index set, 
	and $(\eta_\lambda)_{\lambda \in \Lambda} \subset \C^k(\R^{2d};\R)$ be a family of functions with 
	\begin{enumerate}
		\item \begin{equation*}
			c_0 := \sup_{\lambda \in \Lambda} \norm{\eta_\lambda}_{\C^k} < \infty,
		\end{equation*}
		\item \begin{equation*}
			\supp \eta_\lambda \subset D := \left\{ \frac{1}{2} \le d_{\beta}(\varphi,\xi) \le 4 \right\}.
		\end{equation*}
	\end{enumerate}
	Then, for any sequence $(\lambda_j)_{j \le 0} \subset \Lambda$ we have	
	\begin{align*}
		&\abs{\cF^{-1}\left( (\varphi,\xi) \mapsto \sum_{j \le 0} 2^{2\beta j} \eta_{\lambda_j}\left( \frac{\varphi}{2^{j(2\beta+1)}}, \frac{\xi}{2^j} \right) \right)(x,v)} \\
		&\hspace{7cm}\lesssim_{\beta,c_0,d,k} \left(1+\abs{x}^{\frac{1}{2\beta+1}}+\abs{v}\right)^{-\homd}.
	\end{align*}	
\end{lem}

\begin{proof}
	We write $f$ for the function to be estimated and $(x,v) \in \R^{2d} \setminus \{ (0,0) \}$ in the following. 
	We start by deriving size estimates for each of the frequency packets
	\begin{equation*}
		f_j(x,v) = 2^{2\beta j} \cF^{-1}\left(\eta_{\lambda_j}\left( \frac{\cdot}{2^{j(2\beta+1)}}, \frac{\cdot}{2^j} \right) \right)(x,v), \quad j \le 0. 
	\end{equation*}
	In terms of $v$-decay, we obtain for all $n \in \N_0$, $n \le k$,
	\begin{align*}
		\abs{v}^{n}\abs{f_j(x,v)} &\lesssim \sum_{\substack{m \in \N_0^d\\\abs{m} = n}} \abs{v^{m}f_j(x,v)} \le 2^{2\beta j}\sum_{\substack{m \in \N_0^d \\ \abs{m} = n}} \norm{\partial_\xi^{m} \eta_{\lambda_j}\left( \frac{\cdot}{2^{j(2\beta+1)}}, \frac{\cdot}{2^j} \right)}_{\L^1_{x,v}} \\
		&\lesssim 2^{(\homd-n)j}c_0 \abs{D} \lesssim 2^{(\homd-n)j}
	\end{align*}
	where we have used (i) and (ii). 
	Similarly 
	\begin{align*}
		\abs{x}^{n}\abs{f_j(x,v)} 
		&\lesssim \sum_{\substack{m \in \N_0^d\\\abs{m} = n}} \abs{x^{m}f_j(x,v)} \le 2^{2\beta j}\sum_{\substack{m \in \N_0^d \\ \abs{m} = n}} \norm{\partial_\varphi^{m} \eta_{\lambda_j}\left( \frac{\cdot}{2^{j(2\beta+1)}}, \frac{\cdot}{2^j} \right)}_{\L^1_{x,v}} \\
		&\lesssim 2^{(\homd-(2\beta+1)n)j}c_0 \abs{D} \lesssim 2^{(\homd-(2\beta+1)n)j}
	\end{align*}
	
	Interpolation of the $x$-decay estimates combined with the $v$-decay yields
	\begin{equation*}
		\abs{f_j(x,v)} \lesssim (\abs{x}^{\frac{1}{2\beta+1}}+\abs{v})^{-n} 2^{j(\homd-{n})}
	\end{equation*}
	for all $0 \le n \le k$. 

	We conclude that 
	\begin{align*}
		\abs{f(x,v)} \lesssim \sum_{j = -\infty}^0 \abs{f_j(x,v)} 
		&\lesssim \sum_{2^j \le (\abs{x}^{\frac{1}{2\beta+1}}+\abs{v})^{-1}} \abs{f_j(x,v)}+\sum_{2^j > (\abs{x}^{\frac{1}{2\beta+1}}+\abs{v})^{-1}} \abs{f_j(x,v)} \\
		&\lesssim \sum_{2^j \le (\abs{x}^{\frac{1}{2\beta+1}}+\abs{v})^{-1}}  2^{j \homd }\\
		&\hphantom{=}+\sum_{2^j > (\abs{x}^{\frac{1}{2\beta+1}}+\abs{v})^{-1}} (\abs{x}^{\frac{1}{2\beta+1}}+\abs{v})^{-n} 2^{j(\homd-n)} \\
		&\lesssim(\abs{x}^{\frac{1}{2\beta+1}}+\abs{v})^{-\homd}
\end{align*}
for a choice of $n$ with $\homd<n \le k$.
\end{proof}

\subsubsection{Estimates for the good part}

The nondegeneracy of the symbol as in \eqref{eq:nondegen} is inherited by the good part, see \cite[Lemma A.7]{MR3826548}. 
Moreover, $\psi_g$ is smooth (see \cite[Theorem 3.7.13]{jacob_pdo_2005}) so that $\Psi_g$ is a Schwartz function and, 
in principle, we would get any decay we want for this function. 
If we apply the operator $\tilde{p}_\gamma(D)$ we have a singularity for low frequencies, which limits the decay 
according to the anisotropic scaling of $\tilde{p}_\gamma$. 
We recall the simple case for the fractional Laplacian. For any Schwartz function $f \colon \R^d \to \R$ we have 
\begin{equation*}
	\abs{\cF^{-1}(\abs{\xi}^{2\beta} f(\xi))(x)} \lesssim (1+\abs{x})^{-d-2\beta}.
\end{equation*}
The symbol $d_{\beta}(\varphi,\xi-\varphi)$ only has a singularity at $\xi = \varphi = 0$ 
and it can be controlled at the same order as $d_{\beta}(\varphi,\xi)$ close to this singularity. 
The following lemma quantifies this. 

\begin{lem} \label{lem:xi-phi} There are constants $0<A< B<\infty$ such that for all $\varphi,\xi\in \R^d$ and $\lambda\in [0,1]$, 
\[
	A\inf(d_{\beta}(\varphi,\xi)^{\frac{1}{2\beta+1}}, d_{\beta}(\varphi,\xi)) \le 
	d_{\beta}(\varphi,\xi-\lambda\varphi) \le B \sup(d_{\beta}(\varphi,\xi)^{{2\beta+1}}, d_{\beta}(\varphi,\xi)),
\]
and, in particular, if $d_{\beta}(\varphi,\xi)\le 1$, 
\[
	A d_{\beta}(\varphi,\xi) \le 
	d_{\beta}(\varphi,\xi-\lambda\varphi) \le B  d_{\beta}(\varphi,\xi).
\]
Finally, there is a constant $0<B'<\infty$ such that $d_{\beta}(\varphi,\xi)\ge 1$ implies 
$d_{\beta}(\varphi,\xi) \le B' (|\xi|+|\varphi|)$.
\end{lem}

\begin{proof}
Let $0<a<b<\infty$ be the best constants such that for all $(\varphi,\xi)\in \R^{2d}$, 
\[
	a (|\varphi|^{\frac{1}{2\beta+1}}+|\xi|) \le d_{\beta}(\varphi,\xi)\le b 
	(|\varphi|^{\frac{1}{2\beta+1}}+|\xi|).
\]
Then, by the triangle inequality and $\lambda\in [0,1]$,  
\begin{align*}
	d_{\beta}(\varphi,\xi-\lambda\varphi) \le b (|\varphi|^{\frac{1}{2\beta+1}}+|\xi|+\lambda|\varphi|) \le b \bigg(\frac{d_{\beta}(\varphi,\xi)}{a} +\bigg( \frac{d_{\beta}(\varphi,\xi)}{a}\bigg)^{2\beta+1}\bigg).
\end{align*}
Thus distinguishing ${d_{\beta}(\varphi,\xi)} \le 1$ or $\ge 1$, we find $B= \frac {b}{a}+ \frac {b}{a^{2\beta+1}}$. 
For the converse we observe that for $x,y\ge 0$, the inequality $y\le \sup(x^{2\beta+1},x)$ 
is equivalent to  $\inf(y^{\frac{1}{2\beta+1}},y) \le x$. 
The same argument from before yields
\[
	d_{\beta}(\varphi,\xi) \le 
	B \sup(d_{\beta}(\varphi,\xi-\lambda\varphi)^{{2\beta+1}}, d_{\beta}(\varphi,\xi-\lambda\varphi)),
\]
hence the desired inequality follows with $A=\inf (B^{-1}, B^{-\frac{1}{2\beta+1}})$. 

Next, we assume $d_{\beta}(\varphi,\xi)\ge 2b$.   If $|\varphi|\ge 1$, 
\[
	d_{\beta}(\varphi,\xi)\le b(|\xi|+|\varphi|^{\frac{1}{2\beta+1}}) \le  b(|\xi|+|\varphi|).
\]
If $|\varphi|\le 1 $, then $b(|\xi| +1)\ge d_{\beta}(\varphi,\xi)\ge 2b$, hence 
$|\xi|\ge 1$, and $d_{\beta}(\varphi,\xi)\le 2b|\xi| \le 2b (|\xi| +|\varphi|)$.
As  $d_{\beta}(\varphi,\xi)(|\xi| +|\varphi|)^{-1}$ is bounded on the compact set $1\le d_{\beta}(\varphi,\xi)\le 2b$, 
this concludes the argument. 
\end{proof}

We are now able to prove the following pointwise estimates. 
We introduce the following symbols 
\begin{equation*}
	m_1(\varphi,\xi) = \abs{\xi}^{2\beta}, \quad 
	m_2(\varphi,\xi) = \frac{\varphi \cdot [\nabla_\xi d_\beta](\varphi,\xi)}{d_\beta(\varphi,\xi)}.
\end{equation*}

\begin{lem} \label{lem:good}
	Let $\gamma \in \R$ and $\mathfrak{D} \in \left\{ \id, \nabla_v, m_1(D),m_2(D) \right\} $. 
	Then, for any Schwartz function $h \in \cS(\R^{2d})$ we have
	\begin{equation} \label{eq:pggdecay}
		\abs{[\mathfrak{D} \, \tilde{p}_\gamma(D) h] (x,v)}  
		\lesssim_{\beta,\gamma,d,h} (1+\abs{x}^{\frac{1}{2\beta+1}}+\abs{v})^{-\homd}.
	\end{equation}
	As a consequence,
	\begin{equation*}
		\int_{\R^d} \abs{[\mathfrak{D} \, \tilde{p}_\gamma(D) h ](x,v)}  \dd x   
		\lesssim_{\beta,\gamma,d,h} (1+\abs{v})^{-d-2\beta},
	\end{equation*}
	as well as
	\begin{equation*}
		\int_{\R^d} \abs{[\mathfrak{D} \, \tilde{p}_\gamma(D) h](x,v)} \dv  
		\lesssim_{\beta,\gamma,d,h} (1+\abs{x}^{\frac{1}{2\beta+1}})^{-(2\beta+1)d-2\beta}
	\end{equation*}
	for all $x,v \in \R^d$. In particular, this holds for $h = \Psi_g$. 
\end{lem}

\begin{proof}
	Consider any Schwartz function $h \in \cS(\R^{2d})$. 
	Using Lemma~\ref{lem:xi-phi}, the symbol $\tilde{p}_\gamma(\varphi,\xi)$ is $\C^\infty(\R^{2d}\setminus\{(0,0)\})$ and of at most polynomial growth. 
	Let $\epsilon>0$. 
	Consider any cutoff function $\chi \in \C_c^\infty(\R^{2d})$ which is equal to $1$ on $B_{\epsilon/2}(0)$ 
	and has support in $B_{\epsilon}(0)$. We write
	\begin{equation*}
		\tilde{p}_\gamma(D) h = [\tilde{p}_\gamma (1-\chi)](D) h+ [\tilde{p}_\gamma \chi] (D) h.
	\end{equation*}
	The first term is a $\C^\infty(\R^{2d})$ multiplier of polynomial growth applied to a Schwartz function and hence Schwartz. 
	For this function, we may get any polynomial decay we want, and, in particular, the decay of \eqref{eq:pggdecay}. 
	The same holds true when applying any of the choices for the differential operator $\mathfrak{D}$.
	
	Employing the triangle inequality we are left to estimate $\tilde{p}_\gamma(D)h$ 
	for some Schwartz function $h \in \cS(\R^{2d})$ whose Fourier transform is supported in a ball of radius $\epsilon$. 
	
	We start with $\mathfrak{D} = \id$. The goal is to apply Lemma \ref{lem:kinetic}. 
	First, we consider the Littlewood--Paley decomposition $(\theta_j)_{j \in \Z} \subset \C_c^\infty(\R^{2d})$ 
	used to define the spaces $\Xdot^{\gamma,p}_\beta$. 	
	For $\epsilon$ chosen small enough so that $B_\epsilon(0) \subset \{ (\varphi,\xi) \colon d_{\beta}(\varphi,\xi)\le 1 \}$ 
	(note the slightly different geometry of these two sets) we have
	\begin{align}
		\big[[\tilde{p}_\gamma\chi](D)h \big](x,v) 
		&= \cF^{-1}\bigg( \sum_{j \le 0} \hat{\theta}_j(\varphi,\xi)\tilde{p}_\gamma(\varphi,\xi) \chi(\varphi,\xi)\hat h(\varphi,\xi) \bigg)(x,v) \nonumber  \\ 
		&= \bigg[\cF^{-1}\bigg(  \sum_{j \le 0} 2^{2\beta j} \eta_{\lambda_j}\bigg( \frac{\varphi}{2^{j(2\beta+1)}},\frac{\xi}{2^j} \bigg)   \bigg)\ast \tilde h\bigg](x,v),  \label{eq:ptg_kinlem}
	\end{align}
	where we rescaled
	\begin{align*}
		\tilde{p}_\gamma(\varphi,\xi) &=
	d_{\beta}(\varphi,\xi-\varphi)^\gamma d_{\beta}(\varphi,\xi)^{2\beta-\gamma}	
		 \\
		&= 2^{2\beta j} d_{\beta}\bigg(\frac{\varphi}{2^{j(2\beta+1)}},\frac{\xi}{2^j}-2^{2\beta j}\frac{\varphi}{2^{j(2\beta+1)}}\bigg)^\gamma d_{\beta}\bigg(\frac{\varphi}{2^{j(2\beta+1)}},\frac{\xi}{2^j}\bigg)^{2\beta-\gamma}		
	\end{align*}
	and set $\hat{\tilde h}=\chi\hat h$ $\lambda_j = 2^{2\beta j}$,
	\begin{equation*}
		\eta_\lambda(\varphi,\xi) =
		\theta(\varphi,\xi) d_{\beta}(\varphi,\xi-\lambda\varphi)^\gamma
		d_{\beta}(\varphi,\xi)^{2\beta-\gamma}
	\end{equation*}
	for $\lambda \in [0,1]$.
	
	It remains to verify $ c_0 = \sup_{\lambda \in [0,1] }\norm{\eta_\lambda}_{\C^k}<\infty$ for $k$ large enough 
	in order to apply Lemma~\ref{lem:kinetic}. 
	Then, convolution with $h$ will preserve the estimates as $h$ is a Schwartz function. 
	Using again Lemma~\ref{lem:xi-phi} with the support property of $\theta$ and the properties of $d_{\beta}$, the map $[0,1]\times \R^{2d}\ni (\lambda,\varphi,\xi) \to \eta_\lambda(\varphi,\xi)$ is $\C^\infty$. 
	Hence,  the map $[0,1] \to \C^k(\R^{2d})$, $\lambda \mapsto \eta_\lambda$ is well-defined and continuous. 
	As $[0,1]$ is compact, we obtain the desired bound $c_0 < \infty$. 
	
	The case $\mathfrak{D} = \nabla_v$ translates to multiplication by $\xi_i$, $i = 1,\dots,d$, in Fourier variables. 
	It introduces a slightly different scaling, and the $2^{2\beta j}$ term in \eqref{eq:ptg_kinlem} needs to be replaced by $2^{(2\beta+1) j}$. 
	This would yield an even better decay of order $-\homd-1$, which we do not need and neglect.  
	From here, we can argue as before to apply Lemma \ref{lem:kinetic}.

	The differential operators $m_1(D)$ and $m_2(D)$ can be treated similarly. 
	The symbol $m_1$ and $m_2$ both give an additional $2^{2\beta j}$ factor, too, and in turn better decay than desired.
	\end{proof}

\begin{rem}
	In the case $\beta = 1$, we would even obtain the pointwise estimates of Lemma \ref{lem:good} for the full kernel. 
	In order to save writing and as it does not simplify the proof, we stick to the integrated estimates even in the case $\beta = 1$.
\end{rem}

\subsubsection{Integrated estimates for $\tilde{K}_\gamma$ and derivatives}

\begin{prop} \label{prop:Ktilde}
	For $\gamma \in \R$ and any choice 
	$\tilde{\mathsf K} \in \{\tilde{K}_\gamma, \nabla_v \tilde{K}_\gamma,  m_1(D) \tilde{K}_\gamma,m_2(D) \tilde{K}_\gamma \}$ 
	we have
	\begin{equation}\label{eq:intxmg}
		\int_{\R^d} \abs{\tilde{\mathsf K}(x,v)} \dv 
		\lesssim_{\beta,\gamma,d} (1+\abs{x}^{\frac{1}{2\beta+1}})^{-(2\beta+1)d-2\beta}
	\end{equation}
	for all $x \in \R^d$ and
	\begin{equation}\label{eq:intvmg2}		
		\int_{\R^d} \abs{\tilde{\mathsf K}(x,v)} \dx 
		\lesssim_{\beta,\gamma,d} (1+\abs{v})^{-d-2\beta}
	\end{equation}
	for any $v \in \R^d$.
\end{prop}

\begin{proof}
	For $\beta = 1$ we have $\Psi_b = \delta_{(0,0)}$ and the proposition follows directly from Lemma \ref{lem:good}. Assume now $\beta<1$.
	Let us consider $\tilde{\mathsf K} = \tilde{K}_\gamma = (\tilde{p}_\gamma(D) \Psi_g )\ast \Psi_b$ first. 
	To obtain integrated estimates (in one of the variables $x$ or $v$), we argue as follows
\begin{align*}
	&\int_{\R^d} \abs{(\tilde{p}_\gamma(D) \Psi_g )\ast \Psi_b}(x,v)  \dx \\ 
	&\le \int_{\R^d} \int_{\R^{2d}} \abs{(\tilde{p}_\gamma(D) \Psi_g )(x-y,v-w)} \dd \Psi_b(y,w)  \dx \\
	&=  \int_{\R^{2d}}\int_{\R^d} \abs{(\tilde{p}_\gamma(D) \Psi_g )(x-y,v-w)} \dx \dd \Psi_b(y,w) \\
	&= \int_{\R^{2d}}\int_{\R^d} \abs{(\tilde{p}_\gamma(D) \Psi_g )(x,v-w)} \dx \dd \Psi_b(y,w) \\
	&= \int_{\R^{d}}\int_{\R^d} \abs{(\tilde{p}_\gamma(D) \Psi_g )(x,v-w)} \dx \, \cF^{-1}(\exp(-\psi_b(0,\cdot)))(w) \dd w,
\end{align*}
where we have used Fubini's theorem and the translation invariance of the Lebesgue measure.

In Lemma \ref{lem:bad} we derived
\begin{equation*}
	0 \le \cF^{-1}(\exp(-\psi_b(0,\cdot)))(w) \lesssim (1+\abs{w})^{-d-2\beta}
\end{equation*}
and in Lemma \ref{lem:good} we proved
\begin{equation*}
	{\int_{\R^d} \abs{(\tilde{p}_\gamma(D) \Psi_g )(x,v-w)} \dx} \lesssim (1+\abs{v-w})^{-d-2\beta}.
\end{equation*}
Convolution preserves this decay. This type of calculation is well-known, see for example 
\cite{bogdan_heat_2020,sztonyk_transition_2011}. 
The estimate for the $v$-integral follows similarly. 

For the other terms, we note that we apply the differential operator $\nabla_v, m_1(D),m_2(D)$ only to the good part, 
i.e.\ to $\tilde{p}_\gamma(D) \Psi_g$ and the corresponding estimates are the content of Lemma \ref{lem:good}.
\end{proof}

\subsubsection{Conclusion}
\label{sec:proofintest}

\begin{proof}[Proof of Proposition \ref{prop:intkernel}]
	We are left to undo the change of variables from rescaled to original variables
 \begin{equation*}
 	\xi' = (t-s)^{\frac{1}{2\beta}} (\xi-s\varphi) \;\mbox{ and }\; \varphi' = (t-s)^{1+\frac{1}{2\beta}} \varphi.
 \end{equation*}
 
Recall that 
\begin{align*}
	\hat{K}_\gamma(t,s,\varphi,\xi)= (t-s)^{-1} \hat{\tilde{K}}_\gamma(\varphi',\xi'). 
\end{align*}
 
 We deduce
 \begin{align} \label{eq:Kg=Kgtilde}
 	K_\gamma(t,s,x,v)  \nonumber
 	&= \cF^{-1}\left( \hat{K}_\gamma(t,s,\varphi,\xi)\right)(x,v) \\
 	&= (t-s)^{-1} \cF^{-1}\left(\hat{\tilde{K}}_\gamma\left((t-s)^{1+\frac{1}{2\beta}} \varphi,(t-s)^{\frac{1}{2\beta}}(\xi-s\varphi) \right) \right)(x,v)\nonumber \\
 	&=(t-s)^{-1} \cF^{-1}\left(\hat{\tilde{K}}_\gamma\left((t-s)^{1+\frac{1}{2\beta}} \varphi,(t-s)^{\frac{1}{2\beta}}\xi \right) \right)(x+sv,v)\nonumber \\
 	&=(t-s)^{-d-\frac{d}{\beta}-1}\cF^{-1}\left(\hat{\tilde{K}}_\gamma\left( \varphi,\xi \right) \right)\left((t-s)^{-1-\frac{1}{2\beta}}(x+sv),(t-s)^{-\frac{1}{2\beta}}v\right ) \nonumber\\
 	&=(t-s)^{-d-\frac{d}{\beta}-1}\tilde{K}_\gamma\left((t-s)^{-1-\frac{1}{2\beta}}(x+sv),(t-s)^{-\frac{1}{2\beta}}v\right ). 
 \end{align}
 
Using the bound \eqref{eq:intvmg2} of Proposition \ref{prop:Ktilde} for $\tilde{K}_\gamma$ we obtain
 \begin{align} \label{eq:intKdx}
 	\int_{\R^d} \abs{K_\gamma(t,s,x,v)}  \dx 
 	&= (t-s)^{-d-\frac{d}{\beta}-1}	\int_{\R^d} \abs{\tilde{K}_\gamma\left((t-s)^{-1-\frac{1}{2\beta}}(x+sv),(t-s)^{-\frac{1}{2\beta}}v\right )}  \dx \nonumber  \\
 	&=(t-s)^{-d-\frac{d}{\beta}-1}	\int_{\R^d} \abs{\tilde{K}_\gamma\left((t-s)^{-1-\frac{1}{2\beta}}x,(t-s)^{-\frac{1}{2\beta}}v\right )}  \dx \nonumber  \\
 	&= (t-s)^{-\frac{d}{2\beta}-1} \int_{\R^d} \abs{\tilde{K}_\gamma\left(x,(t-s)^{-\frac{1}{2\beta}}v\right )}  \dx \nonumber  \\
 	&\lesssim (t-s)^{-\frac{d}{2\beta}-1}\left(1+\abs{(t-s)^{-\frac{1}{2\beta}}v}\right)^{-d-2\beta} \nonumber \\
 	&= \frac{1}{t-s} \frac{t-s}{\left(\abs{t-s}^{\frac{1}{2\beta}}+\abs{v}\right)^{d+2\beta}}.
 \end{align}

Employing the bound \eqref{eq:intxmg} of Proposition \ref{prop:Ktilde} for $\tilde{K}_\gamma$, 
we obtain a bound integrated in $v$ as
 \begin{align*}
 	\int_{\R^d} \abs{K_\gamma(t,s,x-sv,v)}  \dv 
 	&= (t-s)^{-d-\frac{d}{\beta}-1} \int_{\R^d}\abs{\tilde{K}_\gamma\left((t-s)^{-1-\frac{1}{2\beta}}x,(t-s)^{-\frac{1}{2\beta}}v\right )}  \dv \\
 	&= (t-s)^{-d-\frac{d}{2\beta}-1} \int_{\R^d}\abs{\tilde{K}_\gamma\left((t-s)^{-1-\frac{1}{2\beta}}x,v\right ) } \dv \\
 	&\lesssim  (t-s)^{-d-\frac{d}{2\beta}-1}\left(1+\abs{(t-s)^{-1-\frac{1}{2\beta}}x}^{\frac{1}{2\beta+1}}\right)^{-(2\beta+1)d-2\beta}\\
 	&\lesssim \frac{1}{t-s} \frac{t-s}{\left(\abs{t-s}^{\frac{1}{2\beta}}+\abs{x}^{\frac{1}{2\beta+1}}\right)^{(2\beta+1)d+2\beta}}.
 \end{align*}
This explains the case $\mathsf K = K_\gamma$. 
The operator $\nabla_v-s\nabla_x$ applied to \eqref{eq:Kg=Kgtilde} yields
\begin{align*}
    &(\nabla_v - s \nabla_x)K_\gamma(t,s,x,v) \\
    &= (t-s)^{-d-\frac{d}{\beta}-1} \left[ (t-s)^{-1-\frac{1}{2\beta}} s [\nabla_x \tilde{K}_\gamma]\left((t-s)^{-1-\frac{1}{2\beta}}(x+sv),(t-s)^{-\frac{1}{2\beta}}v\right ) \right. \\
    &+ (t-s)^{-\frac{1}{2\beta}} [\nabla_v \tilde{K}_\gamma]\left((t-s)^{-1-\frac{1}{2\beta}}(x+sv),(t-s)^{-\frac{1}{2\beta}}v\right ) \\
    &\left.-s (t-s)^{-1-\frac{1}{2\beta}}  [\nabla_x \tilde{K}_\gamma]\left((t-s)^{-1-\frac{1}{2\beta}}(x+sv),(t-s)^{-\frac{1}{2\beta}}v\right ) \right] \\
    &= (t-s)^{-d-\frac{d}{\beta}-1-\frac{1}{2\beta}} [\nabla_v \tilde{K}_\gamma]\left((t-s)^{-1-\frac{1}{2\beta}}(x+sv),(t-s)^{-\frac{1}{2\beta}}v\right ).
\end{align*}
The estimate for $\nabla_v \tilde{K}_\gamma$ from Proposition \ref{prop:Ktilde} and the same calculation as in \eqref{eq:intKdx} yields
\begin{equation*}
		\int_{\R^d} \abs{(t-s)^{\frac{1}{2\beta}}[(\nabla_v-s\nabla_x)K_\gamma](t,s,x,v)}  \dx \le \frac{1}{t-s} \frac{t-s}{\left(\abs{t-s}^{\frac{1}{2\beta}}+\abs{v}\right)^{d+2\beta}}. 
\end{equation*}
The estimate for the $v$ integral follows similarly. 

Concerning the $\partial_s$ term we differentiate \eqref{eq:Kgamma} directly
\begin{align*}
	\partial_s \hat{K}_\gamma(t,s,\varphi,\xi) &= \left( \abs{\xi-s\varphi}^{2\beta} -(2\beta-\gamma) \frac{\varphi \cdot [\nabla_\xi d_\beta](\varphi,\xi-s\varphi)}{d_\beta(\varphi,\xi-s\varphi)} \right)   \hat{K}_\gamma(t,s,\varphi,\xi) \\
	&=(t-s)^{-2} \left( \abs{\xi'}^{2\beta} -(2\beta-\gamma) \frac{\varphi' \cdot [\nabla_\xi d_\beta](\varphi',\xi')}{d_\beta(\varphi',\xi')} \right) \hat{\tilde{K}}_\gamma(\varphi',\xi') \\
	&=(t-s)^{-2} \left( m_1(\varphi',\xi') -(2\beta-\gamma)  m_2(\varphi',\xi') \right) \hat{\tilde{K}}_\gamma(\varphi',\xi')
\end{align*}
These two terms are estimated in Proposition \ref{prop:Ktilde}.
Undoing the change of variables as above yields the desired estimate.
This concludes the proof of Proposition \ref{prop:intkernel}. 
\end{proof}

\begin{proof}[Proof of Proposition~\ref{prop:upper}]
	This is an immediate consequence of \eqref{eq:Kg=Kgtilde} together with the fact that ${\tilde{K}}_\gamma$ 
	is bounded as its Fourier transform is integrable. 
\end{proof}

\begin{rem}
	If a bound on
  \begin{align*}
 	\int_{\R^d} \abs{K_\gamma(t,s,x,v)}  \dv &= (t-s)^{-d-\frac{d}{\beta}-1} \int_{\R^d}\abs{\tilde{K}_\gamma\left((t-s)^{-1-\frac{1}{2\beta}}(x+sv),(t-s)^{-\frac{1}{2\beta}}v\right)}  \dv \\
 	&= (t-s)^{-d-\frac{d}{2\beta}-1} \int_{\R^d}\abs{\tilde{K}_\gamma\left((t-s)^{-1-\frac{1}{2\beta}}x+ \frac{s}{t-s}v,v\right )}  \dv 
 \end{align*}
 is desired, we need to estimate 
  \begin{equation} \label{eq:intvmg1}
 	\int_{\R^d}\abs{\tilde{K}_\gamma\left(x+\frac{s}{t-s}v,v\right)}  \dv. 
 \end{equation}
	For this, we would need to consider $\psi_b(\varphi,-\alpha \varphi)$, 
	which does not seem to be controllable independently of $\alpha \in \R$. 
	We emphasise that the bounds on quantities evaluated at $(t,s,x-sv,v)$ are more natural due to the transport in $\nabla_v-s\nabla_x$. 
\end{rem} 

\begin{rem} \label{rem:comp_CZ_HMP}
	In \cite{MR3906169,MR3826548} the authors study the kernel $K_\gamma$, where $p_\gamma$ 
	is replaced by $\abs{\xi-t \varphi}^{2\beta}$ or $\abs{\varphi}^{\frac{2\beta}{2\beta+1}}$. 
	In this case, one can conveniently work in unshifted variables and kernel estimates can be deduced 
	by interpolation of estimates on (integer) derivatives of the fundamental solution of the Kolmogorov equation. 
	We cannot deduce estimates for $K_\gamma$ from their results, as our kernel is more involved. 
\end{rem}

\begin{rem} \label{rem:pointwiseest}
	At this point, we may also compare our estimates to the sharp pointwise estimates 
	for the fundamental solution of the Kolmogorov equation obtained in the preprint \cite{hou_kernel_2024}. 
	The estimates reflect the anisotropy of the diffusion (lack of differentiability where $\varphi \approx \xi$) 
	and need to be given in terms of the profile
	\begin{equation*}
		\frac{1}{(1+|(x,v)|)^{1+d+2\beta}}\left(\inf _{s \in[0,1]}\left|(x-sv,v)\right|+1\right)^{1-d-2\beta},
	\end{equation*}
	which is not compatible with the decay in powers of $(1+\abs{x}^{\frac{1}{2\beta+1}}+\abs{v})$ 
	induced by the Fourier multiplier ${p}_\gamma$. 
	
This underlines the flexibility of the integrated estimates. 
In fact, the marginals of the fundamental solution and of the decay induced by the Fourier multiplier 
$\tilde{p}_\gamma$ are compatible, in the sense that their decay has the same shape, see \cite[Lemma 1.5 and Remark 1.6]{hou_kernel_2024}.

If instead of ${p}_\gamma$ one considers the simpler Fourier multipliers as in Remark \ref{rem:comp_CZ_HMP}, 
one can deduce even pointwise estimates for the kernels considered in \cite{MR3906169,MR3826548} 
by interpolation of the estimate in \cite[Theorem 1.1 (ii)]{hou_kernel_2024}.
\end{rem}

\section{Proof of the $\L^p$ estimates}
\label{sec:horm}

To apply Theorem~\ref{thm:CW} to $T_{\gamma}$, we first equip $\R^{1+2d}$ with the following  structure of a space of homogenous type.
For $(t,x,v),(s,y,w) \in \R^{1+2d}$ we set 
\begin{align} \label{eq:kindist}
	&\rho((t,x,v),(s,y,w)) \\
	&= \abs{t-s}^{\frac{1}{2\beta}} + \frac{1}{2} \abs{x-y+t(v-w)}^{\frac{1}{2\beta+1}}+ \frac{1}{2}\abs{x-y+s(v-w)}^{\frac{1}{2\beta+1}}+ \abs{v-w}. \nonumber
\end{align}

This is the symmetrised kinetic distance in shifted variables. 
For two sets $A,B \subset \R^{1+2d}$ we write
\begin{equation*}
	\rho(A,B) = \inf 
	 \{\rho((t,x,v),(s,y,w)), \ (t,x,v) \in A,  (s,y,w) \in B\}.
\end{equation*}

\begin{lem} \label{lem:kindist}
	The mapping $\rho \colon \R^{1+2d} \times \R^{1+2d} \to \R$ is a quasi-distance, i.e.\ it is definite, 
	symmetric and satisfies the triangle inequality with some constant $C>0$. 
	Moreover, the triple $(\R^{1+2d},\rho, \dd(t,x,v))$, i.e.\ $\R^{1+2d}$ equipped with the quasi-distance 
	and the Lebesgue measure, is a  space of homogeneous type.
\end{lem}

\begin{proof}
	Shifting $\rho$ along the characteristics $(t,x-tv,v)$ (the inverse of the kinetic shift $\Gamma$), 
	one obtains the kinetic quasi-distance considered in \cite{MR3906169}. 
	The statement follows from \cite[Proposition C.2]{MR3906169}, where the doubling property transfers 
	because the Lebesgue measure is preserved by the kinetic shift.
\end{proof}

Next, we would like to check the three conditions (i)--(iii) 
on the kernel $\mathds{1}_{t-s>0}K_\gamma(t,s,x-y,v-w)$. 
We know that the kernel $K_\gamma(t,s,\cdot) \in \C^\infty(\R^{2d})$ for $s<t$ and that the only singularity 
is where $t -s\approx 0$. 
It is convenient to regularise this singularity to make the representation formula rigorous for an approximate operator.
Let $\epsilon \in (0,1)$ and $\zeta_\epsilon \in \C_c^\infty(\R)$ be such that 
$\supp \zeta_\epsilon \subset [\epsilon/2,2/{\epsilon}]$, $\zeta_\epsilon = 1$ for $r \in [\epsilon,1/\epsilon]$ 
with $M_\zeta = \sup_{r>0, \epsilon \in (0,1)} \abs{r \zeta'_\epsilon(r)}< \infty$ 
(e.g. $\zeta_\epsilon = \chi_1(\cdot/\eps) \chi_2(\eps \cdot)$ for suitable $\chi_1,\chi_2$).  
We define
\begin{equation} \label{eq:Kgammaeps}
	K_{\gamma,\epsilon}(t,s,x,v) = \zeta_\epsilon(t-s) K_\gamma(t,s,x,v)\mathds{1}_{s<t} 
\end{equation}
and set for $f \in \L^\infty(\R^{1+2d})$ with bounded support,
\begin{equation} \label{eq:defTgammaeps}
	[T_{\gamma,\epsilon} f](t,x,v) = \int_{\R^{1+2d}} K_{\gamma,\epsilon}(t,s,x-y,v-w) f(s,y,w) \dd(s,y,w).
\end{equation}
We prove estimates that are independent of $\epsilon>0$, and we let $\epsilon \to 0^+$ in the end.

In view of Lemma \ref{lem:TgammaL2} it is clear that $T_{\gamma,\eps}$ is bounded on $\L^2_{t,x,v}$. 
\begin{lem} \label{lem:TgammaepsL2}
	For all $\gamma\in \R$ and $\beta>0$, $T_{\gamma,\epsilon}$ extends boundedly to $\L^2_{t,x,v}$ 
	with operator norm bounded independently of $\epsilon$. 
\end{lem}

Thus, the conditions (i) and (iii) are met. It remains to check (ii), which is the main objective of this section.

\begin{lem} \label{lem:hormander}
	There exist constants $C= C(\beta,\gamma,d)<\infty$ and $N = N(\beta,\gamma,d) > 0$ 
	such that the kernels $K_{\gamma,\epsilon}$ satisfy, uniformly in $\varepsilon\in (0,1)$, the H\"ormander's condition: for all $r_0>0$ and any $(s,y,w),(s',y',w') \in \R^{1+2d}$ with $\rho((s,y,w),(s',y',w')) \le r_0$,
	\begin{enumerate}
		\item[$ $] \begin{equation} \label{eq:hormander1}
		 	\int\limits_{ 
			\rho((t,x,v),(s,y,w)) \ge N r_0}
			 \hspace{-0.8cm} \abs{K_{\gamma,\epsilon}(t,s,x-y,v-w)-K_{\gamma,\epsilon}(t,s',x-y',v-w')} \dd(t,x,v) \le C,
		\end{equation}
			\item[$ $] \begin{equation} \label{eq:hormander2}
		 	\int\limits_{ 
			\rho((t,x,v),(s,y,w)) \ge N r_0
			} \hspace{-0.8cm} \abs{K_{\gamma,\epsilon}(s,t,y-x,w-v)-K_{\gamma,\epsilon}(s',t,y'-x,w'-v)} \dd(t,x,v) \le C.
		\end{equation}
	\end{enumerate} 
	
\end{lem}

\begin{rem}
	The alert reader might have noticed that we only prove estimates on the $\partial_s$ and $\nabla_v-s\nabla_x$ 
	derivatives of $K_{\gamma,\epsilon}$. 
	Instead of directly estimating the difference
	\begin{equation*}
		(t,s,x-y,v-w) \to (t,s',x-y',v-w'),
	\end{equation*}
	where we would need control along the $\partial_s,\nabla_x, \nabla_v$ derivatives of $K_{\gamma,\epsilon}$, 
	we exhibit the underlying hypoelliptic structure. 
	We will use control only along the vector fields $\partial_s$ and $\nabla_v -s \nabla_x$. 
	Recall that we are working in shifted variables, so $\partial_s$ is the analogue of $\partial_s+v\cdot\nabla_x$ and 
	$\nabla_v-s\nabla_x$ corresponds to $\nabla_v$. 
	In particular, $[\nabla_v-s\nabla_x,\partial_s] = \nabla_x$ so that these vector fields satisfy the H\"ormander rank condition. 
	We refer to \cite{MR4875497,dietert2025criticaltrajectorieskineticgeometry} for a related analysis.
\end{rem}

\begin{proof}[Proof of Lemma \ref{lem:hormander}]
We let $N>0$, $\theta>0$ and $\alpha_1,\alpha_2,\alpha_3 \in [0,1]$ with $\alpha_1 + \alpha_2+\alpha_3 = 1$ with 
\begin{equation*}
	\alpha_3^{2\beta+1}-\alpha_1^{2\beta} \alpha_2>0
\end{equation*}
(e.g. $\alpha_1 = \frac{11}{20}$, $\alpha_2 = \frac{1}{20}$ and $\alpha_3 = \frac{4}{10}$ gives the lower bound $\frac{391}{8000}$). 
The parameter $\theta$ can be chosen arbitrarily, and $N$ will be chosen large at the end of the proof so that the quantities $B_{i,j}$ defined in each step are positive.

We will repeatedly use Proposition~\ref{prop:intkernel} for $K_{\gamma,\epsilon}$ and for the kernels obtained by applying the vector fields $\partial_s$ and $\nabla_v-s\nabla_x$. 
The only nontrivial $\epsilon$ contribution is the $\partial_s$--derivative, because the cutoff $\zeta_\epsilon(t-s)$ depends on $s$. 
We therefore record the identity
\[
	\partial_s K_{\gamma,\epsilon}(t,s,x,v)=\zeta_\epsilon(t-s)\,\partial_s K_\gamma(t,s,x,v)-\zeta_\epsilon'(t-s)\,K_\gamma(t,s,x,v),
\]
and hence
\[
	(t-s)\partial_s K_{\gamma,\epsilon}=\zeta_\epsilon(t-s)\,(t-s)\partial_s K_\gamma-(t-s)\zeta_\epsilon'(t-s)\,K_\gamma.
\]
We have the uniform bound $|(t-s)\zeta_\epsilon'(t-s)|\le M_\zeta$ for all $s<t$ and all $\epsilon>0$. 
Consequently, Proposition~\ref{prop:intkernel} for $K_\gamma$ and $(t-s)\partial_s K_\gamma$ implies the same marginal estimates for $(t-s)\partial_s K_{\gamma,\epsilon}$ with constants independent of $\epsilon$. 

Let $r_0>0$ and let $(s,y,w),(s',y',w') \in \R^{1+2d}$ with $\rho((s,y,w),(s',y',w')) \le r_0$. 
By symmetry we may assume $s' \leq s$ so that $s-s' = \abs{s'-s}$. 

\medskip
Let us first consider the case $s<t$. 
We aim to estimate
	\begin{equation*}
		\int\limits_{ 
		\rho((t,x,v),(s,y,w)) \ge N r_0} \abs{K_{\gamma,\epsilon}(t,s,x-y,v-w)-K_{\gamma,\epsilon}(t,s',x-y',v-w')}  \dd(t,x,v) 
	\end{equation*}
	by some constant from above. 
	
	Note that $\abs{s'-s}\le r_0^{2\beta}$. 
	For technical reasons, i.e.\ to control the spatial difference, we need to be able to control this difference from below. 
	Let $\hat{s} = s'-  \theta r_0^{2\beta}$, then $\theta r_0^{2\beta} \le  \abs{\hat{s}-s} \le (1+\theta) r_0^{2\beta}$.  
	We follow the trajectories $\gamma_i \colon [0,1] \to \R^{1+1+2d}$, $i = 1,\dots,4$ defined by
	\begin{align*}
		\gamma_1(r) &= (t,s,x-y-sc_1r,v-w+c_1r), \\
		\gamma_2(r) &= (t,s+r(\hat{s}-s),x-y-sc_1,v-w+c_1), \\
		\gamma_3(r) &= (t,\hat{s},x-y-sc_1-\hat{s}c_2r,v-w+c_1+c_2r) \\
		\gamma_4(r) &= (t,\hat{s}+r(s'-\hat{s}),x-y',v-w'),
	\end{align*}
	for $r \in [0,1]$ and
	\begin{equation*}
		c_1 = \frac{\hat{s}}{\hat{s}-s}(w-w') + \frac{1}{\hat{s}-s}(y-y'), \quad c_2 = -\frac{s}{\hat{s}-s}(w-w')-\frac{1}{\hat{s}-s}(y-y').
	\end{equation*}
	We remark that $\hat s < s'  \leq s <t$ and that the second component of these trajectories is smaller 
	or equal to $s$, and, in particular, smaller than $t$.
	Note that $\gamma_1,\gamma_3$ are along the vector field $\nabla_v-s\nabla_x$ while $\gamma_2,\gamma_4$ 
	moves along $\partial_s$.
	The coefficients $c_1,c_2$ are chosen such that the concatenation of the four paths is a continuous trajectory from  $ (t,s,x-y,v-w) $ to $(t,s',x-y',v-w')$.
	The choice of $\hat s$ implies  the following bounds 
	\begin{equation} \label{eq:boundci}
		\abs{c_i} \le \left( \frac{2^{2\beta+1}}{\theta} +1 \right) r_0, \quad i = 1,2. 
	\end{equation}
	
	\medskip
	We write 
	\begin{align}
		&K_{\gamma,\epsilon}(t,s,x-y,v-w)-K_{\gamma,\epsilon}(t,s',x-y',v-w') \nonumber\\
		&=K_{\gamma,\epsilon}(t,s,x-y,v-w)-K_{\gamma,\epsilon}(\gamma_1(1)) \label{eq:diffK1}\\
		&\hphantom{=}+K_{\gamma,\epsilon}(\gamma_2(0))-K_{\gamma,\epsilon}(\gamma_2(1))\label{eq:diffK2}\\
		&\hphantom{=}+K_{\gamma,\epsilon}(\gamma_3(0))-K_{\gamma,\epsilon}(\gamma_3(1))\label{eq:diffK3}\\
		&\hphantom{=}+K_{\gamma,\epsilon}(\gamma_4(0))-K_{\gamma,\epsilon}(t,s',x-y',v-w')\label{eq:diffK4}
	\end{align}
	noting that $\gamma_1(1) = \gamma_2(0)$, $\gamma_2(1) = \gamma_3(0)$ and $\gamma_3(1) = \gamma_4(0)$. 	
	
	The integrals are over the set 
	\begin{align*}
		&\rho((t,x,v),(s,y,w))\\
		&= \abs{t-s}^{\frac{1}{2\beta}} + \frac{1}{2} \abs{x-y+t(v-w)}^{\frac{1}{2\beta+1}}+ \frac{1}{2}\abs{x-y+s(v-w)}^{\frac{1}{2\beta+1}}+ \abs{v-w} \ge N r_0.
	\end{align*}
	We split this condition into the following disjoint cases 
	\begin{enumerate}
		\item[Case 1:] $\abs{t-s}^{\frac{1}{2\beta}} >  \alpha_1 N r_0$,
		\item[Case 2:] $\abs{t-s}^{\frac{1}{2\beta}} \le  \alpha_1 N r_0$ and $\abs{v-w} >\alpha_2 Nr_0$, 
		\item[Case 3:] $\abs{t-s}^{\frac{1}{2\beta}} \le  \alpha_1 N r_0$, $\abs{v-w} \le \alpha_2 Nr_0$ and 
		\begin{equation*}
			\frac{1}{2}\abs{x-y+s(v-w)}^{\frac{1}{2\beta+1}} > \frac{1}{2}\alpha_3 Nr_0.
		\end{equation*}
		\item[Case 4:] $\abs{t-s}^{\frac{1}{2\beta}} \le  \alpha_1 N r_0$, $\abs{v-w} \le \alpha_2 Nr_0$, $\frac{1}{2}\abs{x-y+s(v-w)}^{\frac{1}{2\beta+1}} \le \frac{1}{2}\alpha_3 Nr_0$ and 
		\begin{equation*}
			\frac{1}{2} \abs{x-y+t(v-w)}^{\frac{1}{2\beta+1}} > \frac{1}{2} \alpha_3 Nr_0.
		\end{equation*}
	\end{enumerate}
	Denote by $\cC_j$ the set of points $(t,x,v)$ satisfying the conditions in Case $j$, $j = 1,\dots,4$. 
	
	If neither of these cases is satisfied we have $\rho((t,x,v),(s,y,w)) \le N r_0 $ as 
	$\alpha_1 + \alpha_2 + \alpha_3 = 1$, hence
	\begin{equation*}
		\cC_1 \cup \cdots \cup \cC_4 = \left\{ (t,x,v) \in \R^{1+2d} : \rho((t,x,v),(s,y,w)) > N r_0 \right\},
	\end{equation*}
	and we are left to estimate four integrals on each of the domains $\cC_j$, $j = 1,\dots,4$. 
	In the first case, we use the estimate integrated in $x$ and $v$. 
	In the second case, we use the decay of the $x$-marginal, and in the last two cases, we use the decay of the $v$-marginal.
	
	Before we start, we recall the following elementary estimate
	\begin{align} 
		\int_{\abs{z} \ge a} \frac{1}{(b+\abs{z}^\lambda)^{c}} \dd z &\lesssim  \int_{a}^\infty r^{(d-1)-\lambda c}  \dr  = \frac{a^{d-\lambda c}}{\lambda c-d} 
		\label{eq:decayExtBall}
	\end{align}
	for all $a,b,c,\lambda>0$ with $\lambda c-d>0$.

	\medskip
	\textbf{Estimate along $\gamma_1$.} As $\gamma_1(0) = (t,s,x-y,v-w)$ we can write the difference in \eqref{eq:diffK1} as
	\begin{align*}
		I_1 &:= \int\limits_{\rho((t,x,v),(s,y,w)) \ge N r_0}\abs{K_{\gamma,\epsilon}(t,s,x-y,v-w)-K_{\gamma,\epsilon}(\gamma_1(1))}  \dd(t,x,v) \\
		&=\int\limits_{\rho((t,x,v),(s,y,w)) \ge N r_0}\abs{K_{\gamma,\epsilon}(\gamma_1(1))-K_{\gamma,\epsilon}(\gamma_1(0))}  \dd(t,x,v) =: I_1^1+I_1^2+I_1^3+I_1^4,
	\end{align*}
	where each of the terms $I_1^j$ corresponds to the integral on $\cC_j$, $j = 1,\dots, 4$. 
	
	\medskip
	\textbf{Case 1:}
	We drop the condition on $x$ and $v$, employ the full force of the mean value inequality and estimate by \eqref{eq:Kxvmarginal}
	\begin{align*}
		I_1^1&=\int\limits_{\abs{t-s}^{\frac{1}{2\beta}} \ge  \alpha_1 N r_0}  \int_{\R^{2d}} \abs{\int_0^1 \frac{\dd}{ \dr}K_{\gamma,\epsilon}(\gamma_1(r))  \dr} \dd(x,v)  \dt \\
		&\le \int\limits_{\abs{t-s}^{\frac{1}{2\beta}}\ge  \alpha_1 N r_0} \int_{\R^{2d}} \int_0^1 \abs{c_1} \abs{[(\nabla_v-s\nabla_x) K_{\gamma,\epsilon}](\gamma_1(r))}  \dr \dd(x,v)  \dt \\
		&= \abs{c_1}\int_0^1\int\limits_{\abs{t-s}^{\frac{1}{2\beta}}\ge  \alpha_1 N r_0}\int_{\R^{2d}}  \abs{[(\nabla_v-s\nabla_x) K_{\gamma,\epsilon}](\gamma_1(r))}  \dd(x,v)  \dt \dr \\
		&\lesssim \abs{c_1}\int_0^1\int_{s+(\alpha_1 N)^{2\beta}r_0^{2\beta}}^\infty \frac{1}{(t-s)^{1+\frac{1}{2\beta}}}  \dt \dr \\
		&= 2\beta \abs{c_1}\int_0^1 \frac{1}{\alpha_1 Nr_0}  \dr \le 2\beta \left(\frac{ 2^{2\beta+1}}{\theta}+1\right) \frac{1}{\alpha_1 N }. 
	\end{align*}
	We used that as $s<t$ the condition $\abs{t-s}^{\frac{1}{2\beta}}\ge \alpha_1 Nr_0$ 
	can be written as $s+(\alpha_1Nr_0)^{2\beta}\le t < \infty $.
	
	\medskip
	When $\abs{t-s}^{\frac{1}{2\beta}} \le \alpha_1 N r_0$, it suffices to use only the decay of each of the terms on their own 
	and we do not need to employ the mean value inequality.

	\textbf{Case 2:}
	We use the triangle inequality to estimate
	\begin{align*}
		I_1^2&=\int\limits_{\substack{\abs{t-s}^{\frac{1}{2\beta}} \le  \alpha_1 N r_0\\ \abs{v-w}\ge \alpha_2 N r_0}} \abs{K_{\gamma,\epsilon}(t,s,x-y,v-w)-K_{\gamma,\epsilon}(t,s,x-y-sc_1,v-w+c_1)}  \dd(t,x,v) \\
		&\le\int\limits_{\substack{\abs{t-s}^{\frac{1}{2\beta}} \le  \alpha_1 N r_0\\ \abs{v-w}\ge \alpha_2 N r_0}} \abs{K_{\gamma,\epsilon}(t,s,x-y,v-w)}+\abs{K_{\gamma,\epsilon}(t,s,x-y-sc_1,v-w+c_1)}  \dd(t,x,v) \\
		&=:J_0+J_1.
	\end{align*}

We substitute $\tilde{v} = v-w+c_1r$ and $\tilde{x} = x-y-sc_1r$, where we estimate the integral domain as
\begin{align*}
	\abs{\tilde{v}}&= \abs{v-w+c_1r} \\
	&\ge \abs{\alpha_2Nr_0-\abs{c_1}r} 
	\ge \left(\alpha_2 N - \frac{2^{2\beta+1}}{\theta}-1\right)r_0 =: B_{1,2} r_0
\end{align*}
as a consequence of $\abs{c_1} \le \left(\frac{2^{2\beta+1}}{\theta}+1\right) r_0$. 
Then estimate by \eqref{eq:Kxmarginal} and shift  $\tilde{t} = t-s$. 
Dropping tildes, we obtain
\begin{align*}
	J_1 &=  \int_{s}^{s+(\alpha_1 Nr_0)^{2\beta}}	\int_{\abs{v}\ge \alpha_2 Nr_0} \int_{\R^d} \abs{K_{\gamma,\epsilon}(t,s,x,v)}  \dx  \dv  \dt \\
	&\le  \int_{0}^{(\alpha_1 Nr_0)^{2\beta}} 	\int_{\abs{v}\ge B_{1,2} r_0} \frac{1}{(t^{\frac{1}{2\beta}}+\abs{v})^{d+2\beta}}  \dv  \dt\\
	&\lesssim_{\beta,d} \int_{0}^{(\alpha_1 Nr_0)^{2\beta}} (B_{1,2} r_0)^{-2\beta}	   \dt \\
	&\le   \left(\frac{\alpha_1N }{B_{1,2}} \right)^{2\beta}.
\end{align*}
The estimate for $J_0$ is similar. 

\textbf{Case 3:}
We use the triangle inequality and change variables $\tilde{v}= v-w+c_1 r$ and $\tilde{x}-s\tilde{v} = x-y-sc_1r$ for $r = 0,1$. 
The integral domains are
\begin{equation*}
	\abs{\tilde{v}-c_1r} \le \alpha_2 N r_0 \;\mbox{ and }\; \abs{\tilde{x}} \ge (\alpha_3Nr_0)^{2\beta+1}.
\end{equation*}
Dropping tildes and the $v$ condition, we are left to estimate
\begin{equation*}
	J = \int_{s}^{s+(\alpha_1 Nr_0)^{2\beta}}	\int_{\abs{x} > (\alpha_3 Nr_0)^{2\beta+1}} \int_{\R^d} \abs{K_{\gamma,\epsilon}(t,s,x-sv,v)}  \dv  \dx  \dt
\end{equation*}
for $r = 0,1$. 
Employing \eqref{eq:Kvmarginal} and shifting $\tilde{t} = t-s$ we deduce
\begin{align*}
	J &\le \int_{0}^{(\alpha_1 Nr_0)^{2\beta}}	\int_{\abs{x}  > (\alpha_3 Nr_0)^{2\beta+1}} \frac{1}{(t^{\frac{1}{2\beta}}+\abs{x}^{\frac{1}{2\beta+1}})^{(2\beta+1)d+2\beta}} \dx  \dt \\
	&\lesssim_{\beta,d} \left(\frac{\alpha_1 }{\alpha_3} \right)^{2\beta}.
\end{align*}
	
\textbf{Case 4:}	
We argue as in Case 3. 
The only difference is that the integral domain for the $(x,v)$-variables transforms to 
\begin{equation*}
	\abs{\tilde{v}-c_1r} \le {\alpha_2Nr_0} \;\mbox{ and }\; \abs{\tilde{x}+(t-s)(\tilde{v}-c_1r)} \ge (\alpha_3 N r_0)^{2\beta+1}
\end{equation*}
after the change of variables for $r = 0,1$. 
With that, we can estimate
\begin{equation*}
	\abs{x} \ge \left( \alpha_3^{2\beta+1}-\alpha_2\alpha_1^{2\beta} \right) N^{2\beta+1}r_0^{2\beta+1}=:B_{1,4}r_0^{2\beta+1}.
\end{equation*}
Details are left to the reader.

	\medskip
	\textbf{Estimate along $\gamma_2$.} To treat the $s$-difference in \eqref{eq:diffK2} we integrate along the second trajectory as follows
	\begin{align*}
		&\int\limits_{\rho((t,x,v),(s,y,w)) \ge N r_0}\abs{K_{\gamma,\epsilon}(\gamma_2(0))-K_{\gamma,\epsilon}(\gamma_2(1))}  \dd(t,x,v) \\
		&=\int\limits_{\rho((t,x,v),(s,y,w)) \ge N r_0} \abs{\int_0^1 \frac{\dd}{ \dr}K_{\gamma,\epsilon}(\gamma_2(r))  \dr}  \dd(t,x,v) \\
		&\le \int_0^1 \int\limits_{\rho((t,x,v),(s,y,w)) \ge N r_0}  \abs{\hat{s}-s}\abs{[\partial_s K_{\gamma,\epsilon}](\gamma_2(r))}   \dd(t,x,v) \dr.
	\end{align*}
	
	\textbf{Case 1:}
	We enlarge the domain of the $x$ and $v$ integral to be the full space and use the estimate in \eqref{eq:Kxvmarginal} to obtain 
	\begin{align*}
		&\int_0^1\int\limits_{\substack{\abs{t-s}^{\frac{1}{2\beta}} \ge  \alpha_1 N r_0}}  \abs{\hat{s}-s}\abs{[\partial_s K_{\gamma,\epsilon}](\gamma_2(r))}   \dd(t,x,v) \dr \\
		&\lesssim \abs{\hat{s}-s}\int_0^1\int\limits_{\abs{t-s}^{\frac{1}{2\beta}}\ge  \alpha_1 N r_0}\int_{\R^{2d}}  \abs{[\partial_s K_{\gamma,\epsilon}](\gamma_2(r))}  \dd(x,v)  \dt \dr \\
		&\le \abs{\hat{s}-s}\int_0^1\int\limits_{\abs{t-s}^{\frac{1}{2\beta}}\ge  \alpha_1 N r_0}  \frac{1}{(t-s-r(\hat{s}-s))^2}   \dt  \dr\\
		&= \abs{\hat{s}-s}\int_0^1  \int_{s+(\alpha_1 N)^{2\beta}r_0^{2\beta}}^\infty \frac{1}{(t-s-r(\hat{s}-s))^2}  \dt  \dr \\
		&= \abs{\hat{s}-s}\int_0^1 \frac{1}{(\alpha_1 N)^{2\beta}r_0^{2\beta}+\abs{\hat{s}-s}r}   \dr\\
		&= \int_0^{\abs{\hat{s}-s}} \frac{1}{(\alpha_1 N)^{2\beta} r_0^{2\beta}+\tau } \dd \tau  = \log\left( 1+ \frac{\abs{\hat{s}-s}}{(\alpha_1 N)^{2\beta} r_0^{2\beta}} \right)  \le \log\left( 1+ \frac{1+\theta}{(\alpha_1 N)^{2\beta}} \right).
	\end{align*}
	Here, the condition $\abs{t-s}^{\frac{1}{2\beta}}\ge \alpha_1 Nr_0$ translates to $s+(\alpha_1Nr_0)^{2\beta}\le t < \infty $.
	We have used that $\hat{s}<s$ and $\abs{\hat{s}-s} \le (1+\theta) r_0^{2\beta}$. 	
	
	\medskip
	\noindent\textbf{Case 2:} 
	We enlarge the domain of the $x$ integral to be the full space, then use translation invariance of the Lebesgue measure to substitute $\tilde{x} = x-y-sc_1$. 
	Next, we change variables as $\tilde{v} = v-w+c_1$, which transforms the condition $\abs{v-w}\ge \alpha_2 Nr_0$ to $\abs{\tilde{v}-c_1} \ge \alpha_2 Nr_0$. 
	Here, we can estimate as follows
	\begin{equation*}
		\abs{\tilde{v}}\ge  \abs{\abs{\tilde{v}-c_1}-\abs{c_1}} \ge \left(\alpha_2 N -\frac{2^{2\beta+1}}{\theta}-1\right) r_0 =: B_{2,2} r_0.
	\end{equation*}
	
	Together with \eqref{eq:Kxmarginal} we obtain 
	\begin{align*}
		&\int\limits_{\substack{\rho((t,x,v),(s,y,w)) \ge N r_0\\\abs{t-s}^{\frac{1}{2\beta}} \le  \alpha_1 N r_0\\ \abs{v-w}\ge \alpha_2 N r_0}} \int_0^1 \abs{\hat{s}-s}\abs{[\partial_s K_{\gamma,\epsilon}](\gamma_2(r))}  \dr  \dd(t,x,v) \\
		&\le \abs{\hat{s}-s}\int_0^1  \int\limits_{\abs{t-s}^{\frac{1}{2\beta}}\le  \alpha_1 N r_0} \int\limits_{\abs{v-w}\ge \alpha_2 N r_0}\int_{\R^{d}}  \abs{[\partial_s K_{\gamma,\epsilon}](\gamma_2(r))}   \dx  \dv  \dt  \dr \\
		&= \abs{\hat{s}-s}\int_0^1\!\!\!\!\!\! \int\limits_{\abs{t-s}^{\frac{1}{2\beta}}\le  \alpha_1 N r_0} \int\limits_{\abs{v-w}\ge \alpha_2 N r_0}\int_{\R^{d}}  \abs{[\partial_s K_{\gamma,\epsilon}](t,s+r(\hat{s}-s),x,v-w+c_1)}   \dx  \dv  \dt  \dr \\
		&= \abs{\hat{s}-s}\int_0^1\!\!\! \int\limits_{\abs{t-s}^{\frac{1}{2\beta}}\le  \alpha_1 N r_0} \int\limits_{\abs{\tilde{v}-c_1}\ge \alpha_2 N r_0}\int_{\R^{d}}  \abs{[\partial_s K_{\gamma,\epsilon}](t,s+r(\hat{s}-s),x,\tilde{v})}   \dx  \dv  \dt  \dr \\
		&\le \abs{\hat{s}-s}\int_0^1 \!\!\!\int\limits_{\abs{t-s}^{\frac{1}{2\beta}}\le  \alpha_1 N r_0} \int\limits_{\abs{\tilde{v}}\ge B_{2,2} r_0}\int_{\R^{d}}  \abs{[\partial_s K_{\gamma,\epsilon}](t,s+r(\hat{s}-s),x,\tilde{v})}   \dx  \dv  \dt  \dr \\
		&\lesssim \abs{\hat{s}-s}\int_0^1\!\!\!\!\!\! \int\limits_{\abs{t-s}^{\frac{1}{2\beta}}\le  \alpha_1 N r_0} \int\limits_{\abs{\tilde{v}}\ge B_{2,2} r_0} \frac{1}{\abs{t-s-r(\hat{s}-s)}} \frac{1}{\left(\abs{t-s-r(\hat{s}-s)}^{\frac{1}{2\beta}}+\abs{\tilde{v}}\right)^{d+2\beta}} \dd\tilde{v}  \dt  \dr \\
		&\lesssim \abs{\hat{s}-s}\int_0^1 \int\limits_{\abs{t-s}^{\frac{1}{2\beta}}\le  \alpha_1 N r_0} \frac{1}{\abs{t-s-r(\hat{s}-s)}} \frac{1}{(B_{2,2}r_0)^{2\beta}}  \dt  \dr \\
		&= \frac{\abs{\hat{s}-s}}{(B_{2,2}r_0)^{2\beta}}\int_0^1 \int\limits_{s}^{s+(\alpha_1 Nr_0)^{2\beta}}\frac{1}{t-s+r\abs{\hat{s}-s}}  \dt  \dr \\
		&= \frac{\abs{\hat{s}-s}}{(B_{2,2}r_0)^{2\beta}}\int_0^1 \log \left( 1+ \frac{(\alpha_1 Nr_0)^{2\beta}}{r \abs{\hat{s}-s}} \right)  \dr \\
		&= \frac{1}{B_{2,2}^{2\beta}}\int_0^{\frac{\abs{\hat{s}-s}}{ r_0^{2\beta}}} \log \left( 1+ \frac{(\alpha_1 N)^{2\beta}}{\tau } \right) \dd\tau \\
		&\le  \frac{1}{B_{2,2}^{2\beta}}\int_0^{1+\theta} \log \left( 1+ \frac{(\alpha_1 N)^{2\beta}}{\tau } \right) \dd\tau \lesssim 1,
	\end{align*}
	where we have used $ {\abs{\hat{s}-s}}\le (1+\theta) { r_0^{2\beta}}$ in the second to last estimate. 
	Moreover, we employed the estimate of \eqref{eq:decayExtBall}.
	 
	\medskip
	\textbf{Case 3:}
	If we want to use \eqref{eq:Kvmarginal} we need to write the kernel evaluated at $(t,\tilde{s},\tilde{x}-\tilde{s}\tilde{v},\tilde{v})$. 
	We abbreviate $\tilde{s}(r) = s+r(\hat{s}-s)$ and substitute
	\begin{equation*}
		\tilde{v} = v-w+c_1
	\end{equation*}
	and 
	\begin{equation*}
		\tilde{x} -\tilde{s}(r)\tilde{v}= x-y-sc_1.
	\end{equation*}
	The integral domain
	\begin{equation*}
		\abs{v-w}\le \alpha_2 N r_0 \; \mbox{ translates to } \; \abs{\tilde{v}-c_1}\le \alpha_2 N r_0 
	\end{equation*}
	and	
	\begin{equation*}
		\frac{1}{2}\abs{x-y+s(v-w)}^{\frac{1}{2\beta+1}} \ge \frac{1}{2}\alpha_3 Nr_0
	\end{equation*}
	transforms to
	\begin{equation*}
			\abs{\tilde{x}-r(\hat{s}-s)\tilde{v}}  \ge (\alpha_3 Nr_0)^{2\beta+1}.
	\end{equation*}
	
	We use that $\abs{\hat{s}-s} \le (1+\theta) r_0^{2\beta}$ to deduce
	\begin{align*}
		 (\alpha_3 Nr_0)^{2\beta+1} &\le \abs{\tilde{x}-r(\hat{s}-s)\tilde{v}}  \le \abs{\tilde{x}} + \abs{\hat{s}-s}\left(\abs{\tilde{v}-c_1}+\abs{c_1}\right)\\
		 &\le \abs{\tilde{x}} + (1+\theta)\alpha_2 N r_0^{2\beta+1} + (1+\theta)\left(\frac{2^{2\beta+1}}{\theta}+1\right) r_0^{2\beta+1},
	\end{align*}
	whence
	\begin{equation*}
		B_{2,3}r_0^{2\beta+1} := \left( \alpha_3^{2\beta+1} N^{2\beta+1}- (1+\theta)\alpha_2 N-(1+\theta)\left(\frac{2^{2\beta+1}}{\theta}+1\right) \right) r_0^{2\beta+1} \le \abs{\tilde{x}}.
	\end{equation*}

	We obtain
	\begin{align*}
		&\int\limits_{\substack{\abs{t-s}^{\frac{1}{2\beta}} \le  \alpha_1 N r_0\\ \abs{v-w}\le \alpha_2 N r_0\\	\abs{x-y+s(v-w)}^{\frac{1}{2\beta+1}} \ge \alpha_3 Nr_0}}  \int_0^1 \abs{\hat{s}-s}\abs{[\partial_s K_{\gamma,\epsilon}](t,s+r(\hat{s}-s),x-y-sc_1,v-w+c_1)}  \dr  \dd(t,x,v) \\
		&=\abs{\hat{s}-s}\int_0^1\int\limits_{\abs{t-s}^{\frac{1}{2\beta}} \le  \alpha_1 N r_0} \int\limits_{\abs{\tilde{v}-c_1}\le \alpha_2 N r_0 } \int\limits_{\abs{\tilde{x}-r(\hat{s}-s)\tilde{v}}^{\frac{1}{2\beta+1}} \ge \alpha_3 Nr_0} \\[1em]
		&\hspace{5cm}  \abs{[\partial_s K_{\gamma,\epsilon}](t,s+r(\hat{s}-s),\tilde{x}-\tilde{s}(r)\tilde{v},\tilde{v})} \dd \tilde{x} \dd \tilde{v}  \dt  \dr \\
		&\le \abs{\hat{s}-s}\int_0^1\int\limits_{\abs{t-s}^{\frac{1}{2\beta}} \le  \alpha_1 N r_0}  \int\limits_{\abs{\tilde{x}}\ge B_{2,3}r_0^{2\beta+1}} \int_{\R^{d}} \abs{[\partial_s K_{\gamma,\epsilon}](t,s+r(\hat{s}-s),\tilde{x}-\tilde{s}(r)\tilde{v},\tilde{v})} \dd \tilde{v} \dd\tilde{x}   \dt  \dr \\
		&\le \abs{\hat{s}-s}\int_0^1\int\limits_{\abs{t-s}^{\frac{1}{2\beta}} \le  \alpha_1 N r_0}  \int\limits_{\abs{x}\ge B_{2,3}r_0^{2\beta+1}} \frac{1}{\abs{t-\tilde{s}(r)}} \frac{1}{(\abs{t-\tilde{s}(r)}^{\frac{1}{2\beta}}+\abs{x}^{\frac{1}{2\beta+1}})^{(2\beta+1)d+2\beta}}  \dx   \dt  \dr \\
		&\lesssim \abs{\hat{s}-s}\int_0^1 \int\limits_{\abs{t-s}^{\frac{1}{2\beta}}\le  \alpha_1 N r_0} \frac{1}{\abs{t-s-r(\hat{s}-s)}} \frac{1}{B_{2,3}^{\frac{2\beta}{2\beta+1}}r_0^{2\beta}}  \dt  \dr \lesssim 1
	\end{align*}
	by the same estimation as in Case 2 at the end. 
	
	\medskip
	
	\textbf{Case 4:} As in Case 3 we want to use \eqref{eq:Kvmarginal}. We substitute
	\begin{equation*}
		\tilde{v} = v-w+c_1
	\end{equation*}
	and 
	\begin{equation*}
		\tilde{x} -\tilde{s}(r)\tilde{v}= x-y-sc_1
	\end{equation*}
	so that
	\begin{equation*}
		\abs{v-w}\le \alpha_2 N r_0 \; \mbox{ translates to } \; \abs{\tilde{v}-c_1}\le \alpha_2 N r_0 
	\end{equation*}
	and	
	\begin{equation*}
		\frac{1}{2} \abs{x-y+t(v-w)}^{\frac{1}{2\beta+1}} \ge \frac{1}{2}\alpha_3 Nr_0
	\end{equation*}
	transforms to
	\begin{equation*}
			\abs{\tilde{x}+(t-s-r(\hat{s}-s))\tilde{v}-(t-s)c_1}^{\frac{1}{2\beta+1}} \ge  \alpha_3 Nr_0.
	\end{equation*}
	As $\abs{\hat{s}-s} \le (1+\theta) r_0^{2\beta}$ we deduce
	\begin{align*}
		 (\alpha_3 Nr_0)^{2\beta+1} &\le \abs{\tilde{x}+(t-s-r(\hat{s}-s))\tilde{v}-(t-s)c_1}  \\
		 &\le \abs{\tilde{x}} + \abs{t-s}\abs{\tilde{v}-c_1}+\abs{\hat{s}-s}\left(\abs{\tilde{v}-c_1}+\abs{c_1}\right) \\
		 &\le \abs{\tilde{x}} \left((\alpha_1N)^{2\beta}\alpha_2N + \alpha_2 N(1+\theta) + (1+\theta)\left(\frac{2^{2\beta+1}}{\theta}+1\right)	\right)r_0^{2\beta+1},
	\end{align*}
	whence
	\begin{equation*}
		B_{2,4}r_0^{2\beta+1} := \left(  \left(\alpha_3^{2\beta+1} -\alpha_1^{2\beta}\alpha_2 \right)N^{2\beta+1} - \alpha_2 N(1+\theta) -(1+\theta)\left(\frac{2^{2\beta+1}}{\theta}+1\right) \right) r_0^{2\beta+1} \le \abs{\tilde{x}}.
	\end{equation*}
	From here, we deduce the bound along the lines of Case 3. 
	
	\medskip
	Combining all four cases, we have proven that 
	\begin{equation*}
		\int\limits_{\rho((t,x,v),(s,y,w)) \ge N r_0}\abs{K_{\gamma,\epsilon}(\gamma_2(0))-K_{\gamma,\epsilon}(\gamma_2(1))}  \dd(t,x,v) \lesssim 1. 
	\end{equation*}
	
	\medskip
	\textbf{Estimate along $\gamma_3$.}
	The goal is to estimate
		\begin{equation*}
		\int\limits_{\rho((t,x,v),(s,y,w)) \ge N r_0}\abs{K_{\gamma,\epsilon}(\gamma_3(0))-K_{\gamma,\epsilon}(\gamma_3(1))}  \dd(t,x,v) \lesssim 1. 
	\end{equation*}
	We follow the argumentation of the estimate along $\gamma_1$. 
	The main thing is to estimate the integral domains. 
	
	\textbf{Case 1.}
	Here, we use the integrated estimates in $(x,v)$. 
	Translation invariance allows us to deduce the estimate with the same argument and the same estimate as before, 
	but with $\abs{c_1}$ replaced by $\abs{c_2}$. 
	
	\medskip
	In the remaining cases, we use the triangle inequality and estimate each term on its own. 
	We only explain the integral domains the rest of the estimation follows along the lines of the estimate for $\gamma_1$.
	
	\medskip
	\textbf{Case 2.}
	We substitute
	\begin{equation*}
		 \tilde{v} = v-w+c_1+c_2r \;\mbox{ and }\;\tilde{x} = x-y-sc_1-\hat{s}c_2r,
	\end{equation*}
	which yields the integral domain
	\begin{equation*}
		\abs{\tilde{v}-c_1-c_2 r} \ge \alpha_2 N r_0.
	\end{equation*}
	We estimate 
	\begin{equation*}
		\abs{\tilde{v}} \ge \left( \alpha_2 N - \frac{2^{2\beta+2}}{\theta} \right) r_0.
	\end{equation*}
	
	\medskip
	\textbf{Case 3.}
	We change variables
	\begin{equation*}
		 \tilde{v} = v-w+c_1+c_2r \;\mbox{ and }\;\tilde{x}-\hat{s}\tilde{v} = x-y-sc_1-\hat{s}c_2r,
	\end{equation*}
	which transforms the integral domain into
	\begin{equation*}
		\abs{\tilde{v}-c_1-c_2r} \le \alpha_2 N r_0 \;\mbox{ and }\; 
		\abs{\tilde{x}-(\hat{s}-s)(\tilde{v}-c_2 r)} \ge (\alpha_3 N r_0)^{2\beta+1}.
	\end{equation*}
	With that, we estimate
	\begin{equation*}
		\abs{\tilde{x}} 
		\ge \left(\alpha_3^{2\beta+1} N^{2\beta+1}-(1+\theta)\alpha_2 N - \frac{1+\theta}{\theta}2^{2\beta+1} \right)r_0^{2\beta+1} 
		=: B_{3,3} r_0^{2\beta+1}.
	\end{equation*}
	
	\medskip
	\textbf{Case 4.}
	The same change of variables as in Case 3 leads to the integral domain
	\begin{equation*}
		\abs{\tilde{v}-c_1-c_2r} \le \alpha_2 N r_0 \;\mbox{ and }\;
		\abs{\tilde{x}-(t-\hat{s})(\tilde{v}-c_1-c_2 r)+(s-\hat{s})c_1} \ge (\alpha_3 N r_0)^{2\beta+1},
	\end{equation*}
	which allows us to estimate
	\begin{equation*}
		\abs{\tilde{x}} 
		\ge \left( \left(\alpha_3^{2\beta+1}-\alpha_1^{2\beta}\alpha_2 \right) N^{2\beta+1}-(1+\theta)\alpha_2 N-(1+\theta)\left(\frac{2^{2\beta+1}}{\theta}+1\right)  \right) r_0^{2\beta+1} 
		=: B_{3,4}r_0^{2\beta+1}
	\end{equation*}
	
	\medskip
	\textbf{Estimate along $\gamma_4$.} 
	Similarly, as for the estimate along $\gamma_2$, we deduce a bound for
	\begin{equation*}
		\int\limits_{\rho((t,x,v),(s,y,w)) \ge N r_0}\abs{K_{\gamma,\epsilon}(\gamma_4(0))-K_{\gamma,\epsilon}(\gamma_4(1))}  \dd(t,x,v) . 
	\end{equation*}
	To keep things short, we only state the needed bounds for $\alpha_1,\alpha_2,\alpha_3,N,\theta$. 
	In Case 1, we need no further assumptions. 
	In Case 2 we change variables $\tilde{v} = v-w'$ and need
	\begin{equation*}
		B_{4,2}:=\alpha_2 N - 1 >0.
	\end{equation*}
	The change of variables in Cases 3 and 4 is 
	\begin{equation*}
		\tilde{v} = v-w' \;\mbox{ and }\;\tilde{x}-\tilde{s}(r) \tilde{v} = x-y'
	\end{equation*}
	with $\tilde{s}(r) = \hat{s}+r(s'-\hat{s})=s'-(1-r)\theta r_0^{2\beta}$. 
	In order to obtain a lower bound in Case 3, we need
	\begin{equation*}
		B_{4,3} := \alpha_3^{2\beta+1}N^{2\beta+1}-(1+\theta)(\alpha_2 N+1)-2^{2\beta+1}>0,
	\end{equation*}
	and for Case 4, we need
	\begin{equation*}
		B_{4,4} := (\alpha_3^{2\beta+1}-\alpha_1^{2\beta}\alpha_2)N^{2\beta+1}-\alpha_1^{2\beta}N^{2\beta}-(1+\theta)(\alpha_2 N+1)-2^{2\beta+1}>0.
	\end{equation*}
		
	\textbf{Case $t<s$.} 
	We only need to treat the case $s'<t$, as for $t<s'$, both kernels are zero. 
	In this case, we need to estimate
	\begin{equation*}
		I = \int\limits_{\substack{\R^{1+2d}\\\rho((t,x,v),(s,y,w)) \ge N r_0}} \abs{K_{\gamma,\epsilon}(t,s',x-y',v-w')}  \dd(t,x,v) 
	\end{equation*}	
	from above.
	
	\textbf{Case 1.}
	For
	\begin{equation*}
		B_{5,1}:= \alpha_1 N -1>0
	\end{equation*}
	then the set of $t$, such that $\abs{t-s}^{\frac{1}{2\beta}} \ge \alpha_1 N r_0$ with $s'<t<s$ 
	and where $\abs{s'-s}^\frac{1}{2\beta} \le r_0$ holds, is empty. 
	So there is no estimate to prove in the first case. 
	
	\textbf{Case 2.} 
	We change variables $\tilde{v} = v-w'$, $\tilde{x} = x-x'$, $\tilde{t} = t-s$, estimate the domain as
	\begin{equation*}
		\abs{\tilde{v}} \ge B_{4,2}r_0
	\end{equation*}
	and conclude (similarly to the estimate for $J_2$ in Case 2 of the estimate for $\gamma_1$)
	\begin{align*}
		I& \lesssim \frac{1+(\alpha_1 N)^{2\beta}}{B_{4,2}^{2\beta}}.
	\end{align*} 
	
	\textbf{Case 3.}
	Here we set $\tilde{v} = v-w'$ and $\tilde{x}-s'\tilde{v} = x-y'$. 
	With that, we can estimate the new integral domain as 
	\begin{equation*}
		\abs{\tilde{x}} \ge \left(\alpha_3^{2\beta+1} N^{2\beta+1}-(\alpha_2 N+1) -2^{2\beta+1}\right)r_0^{2\beta+1} =: B_{4,3}r_0^{2\beta+1}
	\end{equation*}
	from which we conclude the estimate as in Case 2. 
	
	\textbf{Case 4.}
	With the same change of variables as in Case 3, we estimate the $\tilde{x}$ domain as
	\begin{equation*}
		\abs{\tilde{x}} 
		\ge \left(\left(\alpha_3^{2\beta+1}-\alpha_1^{2\beta}\alpha_2 \right) N^{2\beta+1}-\alpha_1^{2\beta}N^{2\beta}-(\alpha_2 N+1)-2^{2\beta+1}\right)r_0^{2\beta+1} 
		=: B_{4,4}r_0^{2\beta+1},
	\end{equation*}
	which allows us to deduce the desired estimate and concludes the proof of \eqref{eq:hormander1}. 
	
	\medskip
	Let us now explain the second estimate \eqref{eq:hormander2}. 
	We argue via duality to reduce the second estimate \eqref{eq:hormander2} to the first estimate \eqref{eq:hormander1}.
	First, we prove that for every $\gamma\in\mathbb R$, every $\epsilon>0$, every $s,t\in\mathbb R$ with $s>t$, 
	and every $(x,v)\in\mathbb R^{2d}$,
	\begin{equation}\label{eq:kernelidentity}
		K_{\gamma,\epsilon}(s,t,x,v)=K_{2\beta-\gamma,\epsilon}(-t,-s,-x,v).
	\end{equation}

	{
	In fact, let $\widehat{K}_{\gamma,\epsilon}(t,s,\varphi,\xi)$ denote the Fourier transform of $K_{\gamma,\epsilon}(t,s,\cdot,\cdot)$ in the $(x,v)$-variables, with dual variables $(\varphi,\xi)$.
	By \eqref{eq:Kgammaeps} and \eqref{eq:Kgamma} and the fact that $-t>-s$ when $s>t$, we compute
	\begin{align*}
		&\widehat{K}_{2\beta-\gamma,\epsilon}(-t,-s,-\varphi,\xi) \\
		&= d_{\beta}\big(-\varphi,\xi-(-t)(-\varphi)\big)^{2\beta-\gamma} d_{\beta}\big(-\varphi,\xi-(-s)(-\varphi)\big)^{\gamma} \\
		&\qquad \cdot \exp\left(-\int_{-s}^{-t}\big|\xi-\tau(-\varphi)\big|^{2\beta}\dd \tau\right) \zeta_\epsilon(s-t) \\
		&= d_{\beta}\big(\varphi,\xi-t\varphi\big)^{2\beta-\gamma} d_{\beta}\big(\varphi,\xi-s\varphi\big)^{\gamma} \exp\left(-\int_{t}^{s}\big|\xi-\sigma\varphi\big|^{2\beta} d\sigma\right) \zeta_\epsilon(s-t) \\
		&= d_{\beta}\big(\varphi,\xi-s\varphi\big)^{\gamma} d_{\beta}\big(\varphi,\xi-t\varphi\big)^{2\beta-\gamma} \exp\left(-\int_{t}^{s}\big|\xi-\sigma\varphi\big|^{2\beta} d\sigma\right) \zeta_\epsilon(s-t) \\
		&= \widehat{K}_{\gamma,\epsilon}(s,t,\varphi,\xi),
	\end{align*}
	where we used $d_{\beta}(-\varphi,\cdot)=d_{\beta}(\varphi,\cdot)$ 
	and the change of variables $\sigma=-\tau$ in the exponential.
	}
	
	{
	Taking inverse Fourier transform in $(\varphi,\xi)$, the identity 
	$\widehat{K}_{\gamma,\epsilon}(s,t,\varphi,\xi)=\widehat{K}_{2\beta-\gamma,\epsilon}(-t,-s,-\varphi,\xi)$ 
	yields
	\[
		K_{\gamma,\epsilon}(s,t,x,v) =\frac{1}{(2\pi)^{2d}}\int_{\mathbb R^{2d}} e^{i(x\cdot\varphi+v\cdot\xi)}  \widehat{K}_{2\beta-\gamma,\epsilon}(-t,-s,-\varphi,\xi) \dd (\varphi, \xi).
	\]
	Changing variables $\varphi\mapsto -\varphi$ gives \eqref{eq:kernelidentity}.
	}
	
	{
	Next, we deduce \eqref{eq:hormander2}. 
	Fix $r_0>0$ and $(s,y,w),(s',y',w')\in\mathbb R^{1+2d}$ such that we have $\rho((s,y,w),(s',y',w'))\le r_0$. 
	By \eqref{eq:kernelidentity}, for every $(t,x,v)\in\mathbb R^{1+2d}$,
	\[
		K_{\gamma,\epsilon}(s,t,y-x,w-v)=K_{2\beta-\gamma,\epsilon}(-t,-s,x-y,w-v),
	\]
	and similarly
	\[
		K_{\gamma,\epsilon}(s',t,y'-x,w'-v)=K_{2\beta-\gamma,\epsilon}(-t,-s',x-y',w'-v).
	\]
	Hence the left-hand side of \eqref{eq:hormander2} equals
	\begin{equation*}
	\int\limits_{ 
	\rho((t,x,v),(s,y,w))\ge N(\beta,2\beta-\gamma,d) r_0} \hspace{-1cm}\big|K_{2\beta-\gamma,\epsilon}(-t,-s,x-y,w-v)-K_{2\beta-\gamma,\epsilon}(-t,-s',x-y',w'-v)\big| \dd (t,x,v).
	\end{equation*}
	}
	
	{
	Now perform the change of variables
	\[
		\tau=-t,\qquad \eta=-v,
	\]
	so that $\dd (t,x,v)=\dd (\tau, x, \eta)$ and
	\[
		w-v=\eta-(-w), \qquad w'-v=\eta-(-w').
	\]
	A direct computation from the definition of $\rho$ shows
	\begin{equation}\label{eq:rho-invariance}
			\rho((t,x,v),(s,y,w))=\rho((\tau,x,\eta),(-s,y,-w)).
	\end{equation}
	Indeed,
	\begin{align*}
		&\rho((\tau,x,\eta),(-s,y,-w))\\
		&=|\tau+s|^{\frac1{2\beta}}+\frac12|x-y+\tau(\eta+w)|^{\frac1{2\beta+1}}+\frac12|x-y-s(\eta+w)|^{\frac1{2\beta+1}}+|\eta+w| \\
		&=|t-s|^{\frac1{2\beta}} +\frac12|x-y+t(v-w)|^{\frac1{2\beta+1}}+\frac12|x-y+s(v-w)|^{\frac1{2\beta+1}}+|v-w|.
	\end{align*}
	Using \eqref{eq:rho-invariance}, we obtain
	\begin{align*}
		&\int\limits_{ 
		\rho((t,x,v),(s,y,w))\ge N(\beta,2\beta-\gamma,d) r_0}\hspace{-1cm} \big|K_{\gamma,\epsilon}(s,t,y-x,w-v)-K_{\gamma,\epsilon}(s',t,y'-x,w'-v)\big|\dd (t,x,v) \\
		&=\hspace{0cm}\int\limits_{ 
		\rho((\tau,x,\eta),(-s,y,-w))\ge N(\beta,2\beta-\gamma,d)r_0}\hspace{-1.5cm}\big|K_{2\beta-\gamma,\epsilon}(\tau,-s,x-y,\eta-(-w))  \\
		&\hspace{7cm}    -K_{2\beta-\gamma,\epsilon}(\tau,-s',x-y',\eta-(-w'))\big| \dd  (\tau, x, \eta),
	\end{align*}
	where $N = N(\beta,2\beta-\gamma,d)$ is the constant of the first estimate \eqref{eq:hormander1} with $\gamma$ replaced by $2\beta-\gamma$.
	}
	
	{
	Finally, the same computation as above gives
	\[
		\rho((-s,y,-w),(-s',y',-w'))=\rho((s,y,w),(s',y',w'))\le r_0.
	\]
	Thus we may apply \eqref{eq:hormander2} with $\gamma$ replaced by $2\beta-\gamma$ and the points 
	$(-s,y,-w)$ and $(-s',y',-w')$, to conclude that the last integral is bounded by $C(\beta,2\beta-\gamma,d)$ 
	uniformly in $\epsilon>0$. 
	This proves \eqref{eq:hormander2} and concludes the proof of the lemma.
	}	
\end{proof}

{\begin{thm} \label{thm:TgammepsLp}
	Let $d\ge1$, $\beta\in(0,1]$ and $\gamma\in\R$. 
	Then for every $p \in (1,\infty)$ there exists $C_p=C_p(\beta,\gamma,d)<\infty$ such that:
	\begin{enumerate}
	\item[$($i$)$] $($\emph{Uniform $\L^p_{t,x,v}$-bounds for $T_{\gamma,\epsilon}$}$)$
	For all $\epsilon\in(0,1)$,
	\begin{equation}\label{eq:uniformLp}
		\|T_{\gamma,\epsilon}f\|_{\L^p_{t,x,v}}\le C_p\,\|f\|_{\L^p_{t,x,v}}\qquad \forall f\in \L^p_{t,x,v}.
	\end{equation}
	\item[$($ii$)$] $($\emph{Identification of the $\epsilon \to0$ limit on $\L^2$}$)$
	For all $f,g\in \L^2_{t,x,v}$,
	\begin{equation}\label{eq:bilinearconv}
		\langle g,T_{\gamma,\epsilon}f\rangle_{\L^2_{t,x,v}}\longrightarrow \langle g,T_\gamma f\rangle_{\L^2_{t,x,v}} \qquad \text{as }\epsilon \to0^+.
	\end{equation}
	\item[$($iii$)$] $($\emph{$\L^p$-bounded extension of $T_\gamma$ and convergence in the weak $\L^p$ sense}$)$
	The operator $T_\gamma$ admits a unique bounded extension
	\[
	T_\gamma\colon \L^p_{t,x,v}\to \L^p_{t,x,v},\qquad \|T_\gamma f\|_{\L^p_{t,x,v}}\le C_p\,\|f\|_{\L^p_{t,x,v}}.
	\]
	Moreover, for every $f\in \L^p_{t,x,v}\cap \L^2_{t,x,v}$ and every $g\in \L^{p'}_{t,x,v}\cap \L^2_{t,x,v}$ $($where $p'=\frac{p}{p-1}${}$)$,
	\begin{equation}\label{eq:weakLpconv}
		\langle g,T_{\gamma,\epsilon}f\rangle_{\L^2_{t,x,v}} \longrightarrow \langle g,T_\gamma f\rangle_{\L^2_{t,x,v}} \qquad\text{as }\epsilon \to0^+.
	\end{equation}
	In particular, $T_{\gamma,\epsilon}f\rightharpoonup T_\gamma f$ weakly in $\L^p_{t,x,v}$ for each $f\in \L^p_{t,x,v}\cap \L^2_{t,x,v}$.
\end{enumerate}
\end{thm}
}

\begin{proof}
{
We apply Theorem~\ref{thm:CW} on the space of homogeneous type $(X,\rho,\dd(t,x,v))$ with $X=\R^{1+2d}$, $\rho$ the quasi-distance of \eqref{eq:kindist}, and $\dd(t,x,v)$ the Lebesgue measure; see Lemma \ref{lem:kindist}.
}

{
\medskip
\noindent\textbf{Step 1: $\L^p_{t,x,v}$ bounds for $T_{\gamma,\epsilon}$.}
Fix $\epsilon \in(0,1)$. 
Since $\supp \zeta_\epsilon \subset[\epsilon /2,2/\epsilon ]$ and $K_\gamma(t,s,\cdot)\in \C^\infty(\R^{2d})$ for $s<t$, 
the kernel $K_{\gamma,\epsilon}$ is locally bounded on $X\times X\setminus\{(t,x,v,t,x,v): (t,x,v)\in X\}$ and, 
because $t-s$ is restricted to a compact set away from zero, it is locally integrable there. Thus item~(i) in Theorem~\ref{thm:CW} holds.
}

{
Lemma~\ref{lem:hormander} \eqref{eq:hormander1} gives H\"ormander's condition~(ii) for $K_{\gamma,\epsilon}$ 
with constants $C,N$ depending only on $(\beta,\gamma,d)$, uniformly in $\epsilon $.
}

{
Finally, item~(iii) in Theorem~\ref{thm:CW} holds because $T_{\gamma,\epsilon}$ is defined by kernel integration. 
For $f,g\in \L^\infty_{t,x,v}$ with compact support the double integral is absolutely convergent and Fubini's theorem yields
\[
	\langle g,T_{\gamma,\epsilon}f\rangle_{\L^2_{t,x,v}} =\iint_{X\times X} K_{\gamma,\epsilon}(t,t',x-x',v-v')\,f(t',x',v')g(t,x,v) \dd(t,x,v) \dd (t',x',v').
\]
}

{
Since $T_{\gamma,\epsilon}$ is bounded on $\L^2_{t,x,v}$ by Lemma \ref{lem:TgammaepsL2} with a bound independent of $\epsilon $,
 Theorem~\ref{thm:CW} applies and yields, for every $1<p<2$,
\begin{equation}\label{eq:TgammaepsLP}
	\|T_{\gamma,\epsilon}f\|_{\L^p_{t,x,v}}\le C_p\,\|f\|_{\L^p_{t,x,v}}\qquad (\epsilon \in(0,1)),
\end{equation}
with $C_p = C_p(\beta,d, \gamma)$ independent of $\epsilon $.
}

{
Lemma~\ref{lem:hormander} \eqref{eq:hormander2} gives the symmetrised H\"ormander condition and thus by Theorem \ref{thm:CW}
 together with Remark \ref{rem:CW} we deduce the boundedness for $p>2$. 
}

\medskip
{
\noindent\textbf{Step 2: Convergence of the bilinear forms on $\L^2_{t,x,v}$.}
Let $f,g\in \L^2(X)$. By Plancherel in the $(x,v)$ variables and the identity
$\widehat K_{\gamma,\epsilon}(t,s,\varphi,\xi)=\zeta_\epsilon (t-s)\,\widehat K_\gamma(t,s,\varphi,\xi)$,
we can write $\langle g,T_{\gamma,\epsilon}f\rangle_{\L^2_{t,x,v}}$ as an $(s,t,\varphi,\xi)$-integral whose integrand contains the 
factor $\zeta_\epsilon (t-s)$.
Since $\zeta_\epsilon (t-s)\to \mathds 1_{t-s>0}$ pointwise for $t\neq s$ and $|\zeta_\epsilon |\le 1$ we may use dominated
 convergence to pass to the limit. 
 In fact, the kernel of $T_\gamma$ provides a majorant; compare \cite{AIN}. 
}

\medskip
{
\noindent\textbf{Step 3: Extension of $T_\gamma$ to $\L^p_{t,x,v}$ and identification with the $\epsilon \to0$ limit.}
Fix $p \in (1,\infty)$ and $f\in \L^p_{t,x,v}\cap\L^2_{t,x,v}$. By \eqref{eq:uniformLp}, the family $\{T_{\gamma,\epsilon}f\}_{\epsilon \in(0,1)}$ 
is bounded in $\L^p_{t,x,v}$. Let $g\in \L^{p'}_{t,x,v}\cap\L^2_{t,x,v}$. Then, since $T_{\gamma,\epsilon}f\in\L^2_{t,x,v}$ for each $\epsilon $,
\[
	\langle g,T_{\gamma,\epsilon}f\rangle_{\L^2_{t,x,v}}\to \langle g,T_\gamma f\rangle_{\L^2_{t,x,v}}
\]
as $\epsilon \to 0$, by Step~2, which is exactly \eqref{eq:weakLpconv}. 
In particular, $T_{\gamma,\epsilon}f\rightharpoonup T_\gamma f$ weakly in $\L^p_{t,x,v}$ as $\epsilon \to 0^+$, 
because $\L^{p'}_{t,x,v}\cap\L^2_{t,x,v}$ is dense in $(\L^p_{t,x,v})^*=\L^{p'}_{t,x,v}$ and the uniform bound \eqref{eq:TgammaepsLP}.
}

{
By weak lower semicontinuity of the $\L^p$ norm,
\[
\|T_\gamma f\|_{\L^p_{t,x,v}}\le \liminf_{\epsilon \to0}\|T_{\gamma,\epsilon}f\|_{\L^p_{t,x,v}}
\le C_p\,\|f\|_{\L^p_{t,x,v}}.
\]
Thus $T_\gamma$ is bounded on the dense subspace $\L^p_{t,x,v}\cap\L^2_{t,x,v}$, and therefore extends uniquely to a bounded operator 
on all of $\L^p_{t,x,v}$ with the same bound. 
This completes the proof.
}
\end{proof}

\section{Regularity of the temporal trace}
\label{sec:trace}
The goal of this section is to prove the temporal regularity for the Kolmogorov operators. 
We start with the initial value problem and proceed by duality. 
We will use the (homogeneous) kinetic anisotropic Besov spaces defined earlier and rely on Fourier multiplier theory. 

We are going to use a continuous Littlewood--Paley version of the Besov norm, 
which is easier to work with. For the  equivalent definition of the $\Bdot_{\beta}^{\gamma,p}$-norm 
we consider any function $\psi \in \cS(\R^{2d})$ 
which satisfies 
\begin{equation*}
	\hat{\psi}(\varphi,\xi)>0 \quad \text{ for } \quad \frac{1}{2} < d_{\beta}(\varphi,\xi) < 2
\end{equation*}
and $\hat{\psi} = 0$ else. 
We set $\psi_s(x,v) = s^{-(1+\frac{1}\beta)d} \psi\left(s^{-(\frac{1}{2\beta}+1)}x,s^{-\frac{1}{2\beta}}v\right)$ for $s>0$, 
whose Fourier transform is  $\hat{\psi}(s^{\frac{1}{2\beta}+1}\varphi,s^{\frac{1}{2\beta}}\xi)$. 
With that, we have for $g\in \Bdot_{\beta}^{\gamma,p}$,
\begin{equation} \label{eq:besovcont}
	\norm{g}_{\Bdot_{\beta}^{\gamma,p}}^p  \simeq \int_0^\infty \left( \frac{\norm{\psi_{s} \ast g}_{\L^p_{x,v}}}{s^{\gamma/(2\beta)}} \right)^p \frac{ \ds}{s}.  
\end{equation}

\subsection{Estimates for the initial value problem of the Kolmogorov equation}

First, we are going to prove estimates for the solution $f$ of the Kolmogorov equation given by  
\begin{equation} \label{eq:solKolIV}
	\hat{f}(t,\varphi,\xi)= [\widehat{\sem^+_t g}](\varphi,\xi) := \exp\left( -\int_0^t \abs{\xi+(t-r)\varphi}^{2\beta}  \dr \right) \hat{g}(\varphi,\xi+t\varphi), \quad t\ge 0,
\end{equation}
that is, $ \widehat{\Gamma f}(t,\varphi,\xi)= \widehat{E_{\beta}}(t,0,\varphi,\xi)\hat g(\varphi,\xi)$
where $g$ is the initial value. 
The same discussion in this section applies to the solution of the backward Kolmogorov equation given by 
\begin{equation} \label{eq:solKolIVback}
	[\widehat{\sem^-_t g}](\varphi,\xi) = \exp\left( \int_0^t \abs{\xi+(t-r)\varphi}^{2\beta}  \dr \right) \hat{g}(\varphi,\xi+t\varphi), \quad t\le 0.
\end{equation}
Details in the backward case are left to the reader.
In the following, $\sem^{\pm}$ map the initial value to the associated solution as given by \eqref{eq:solKolIV} and \eqref{eq:solKolIVback} respectively.

\begin{lem}\label{lem:initialvalue} 
Let $g \in \cS(\R^{2d})$ whose Fourier transform has compact support in 
$(\R_{\varphi}^{d}\setminus \{0\})\times (\R_{\xi}^{d}\setminus \{0\})$. 
Then the function $f$ given by \eqref{eq:solKolIV} satisfies $f\in \C^{}_{0}([0,\infty)\, ;\, \L^2_{x,v})$ with 
$f(0,\cdot)=g$, $(\partial_{t}+v\cdot\nabla_{x})f, (-\Delta_{v})^{\beta}f\in \L^2_{t,x,v} $ 
and $f$ solves $(\partial_{t}+v\cdot\nabla_{x})f+ (-\Delta_{v})^{\beta}f=0$ in $(0,\infty)\times \R^{2d}$ 
in the sense of strong solutions in $\L^2_{x,v}$.  
 \end{lem}

\begin{proof} 
By Plancherel, it suffices to argue in $\L^2_{\varphi,\xi}$. By \cite[Lemma~2.19]{AIN}, 
\[
	\widehat{E_{\beta}}(t,0,\varphi,\xi) \le e^{-c_{\beta}t(|\xi-t\varphi|^{2\beta} + (t|\varphi|)^{2\beta})}.
\]
So from the isometry property of the change of variables 
$(\varphi,\xi)\mapsto (\varphi,\xi+t\varphi)$,  $\|\hat f(t,\cdot)\|_{\L^2_{\varphi,\xi}} \le \|\hat g\|_{\L^2_{\varphi,\xi}}$. 
The continuity on $\L^2_{\varphi,\xi}$ is clear and the limit at infinity follows from dominated convergence. 
Now, we have   
$$
	|\xi-t\varphi|^{2\beta}\widehat{E_{\beta}}(t,0,\varphi,\xi) 
	\lesssim 
	|\xi|^{2\beta}e^{-c'_{\beta}t|\xi|^{2\beta}}+ t^{2\beta}|\varphi|^{2\beta}e^{-c''_{\beta}t^{2\beta+1}|\varphi|^{2\beta}} 
$$
so that 
\begin{align*}
	\int_{0}^\infty\int_{\R^{2d}} ||\xi-t\varphi|^{2\beta}\widehat{E_{\beta}}(t,0,\varphi,\xi)\hat g(\varphi,\xi)|^2 \dd(\varphi,\xi) \dt 
	\lesssim
	\int_{\R^{2d}} \bigg(|\xi|^{2\beta}+ |\varphi|^{\frac{2\beta}{2\beta+1}}\bigg)|\hat g(\varphi,\xi)|^2  \dd(\varphi,\xi).
\end{align*}
This implies $(-\Delta_{v})^{\beta}f \in \L^2_{t,x,v}$ on $(0,\infty)\times \R^{2d}$. 
As $\partial_{t}\widehat{E_{\beta}}(t,0,\varphi,\xi)=-|\xi-t\varphi|^{2\beta}\widehat{E_{\beta}}(t,0,\varphi,\xi)$, 
we can use the estimates above and the conclusion follows. 
\end{proof}

We start with important multiplier estimates. 

\begin{lem} \label{lem:traceMult}
Let $\beta\in (0,1]$, $\gamma\in \R$ and $p \in (1,\infty)$. 
Let $s,t>0$. 
Let $\eta$ be a Schwartz function with Fourier transforms supported in $\frac{1}{4} < d_{\beta}(\varphi,\xi) < 4$. 
Then 
$(\varphi,\xi)\mapsto\widehat{\eta}_{s}(\varphi,\xi) d_{\beta}(\varphi, \xi-t\varphi)^\gamma \widehat{E_{\beta}}(t,0,\varphi,\xi)$ 
is an $\L^p_{x,v}$ Fourier multiplier with norm bounded by 
$t^{-\frac{\gamma}{2\beta}}\kappa_{\gamma}((t/s)^{\frac{1}{2\beta}})$ 
with $\kappa_{\gamma} (\Lambda)= c_{1}\Lambda^\gamma e^{-c_{0}\Lambda^{2\beta}}$ for some $c_0,c_1>0$ independent of $\Lambda$.
\end{lem}

\begin{proof} 
We want to estimate the $\L^p_{x,v}$ Fourier multiplier norm of the symbol
\[ 
	\widehat{\eta}_{s}(\varphi,\xi) d_{\beta}(\varphi, \xi-t\varphi)^\gamma \widehat{E_{\beta}}(t,0,\varphi,\xi).
\]
This norm is invariant under dilations in each of the variables $\varphi$ and $\xi$. 
Changing coordinates $\tilde{\varphi} = t^{1+\frac{1}{2\beta}}\varphi $ and $\tilde{\xi} = t^{\frac{1}{2\beta}} \xi $, 
it is the same as  the $\L^p_{x,v}$ multiplier norm of the symbol 
\begin{equation*}
	t^{-\frac{\gamma}{2\beta}}\widehat{\eta}\left( {\Big( \frac{s}{t} \Big)^{1+\frac{1}{2\beta}}} \tilde\varphi, {\Big( \frac{s}{t} \Big)^{\frac{1}{2\beta}}}{\tilde{\xi}}\right) d_{\beta}(\tilde\varphi, \tilde\xi-\tilde\varphi)^\gamma  \widehat{E_{\beta}}(1,0,\tilde\varphi,\tilde\xi).
\end{equation*}

When $\beta<1$, we split the exponential term as
\begin{equation*}
	\widehat{E_{\beta}}(1,0,\varphi,\xi)=\exp\left(-\int_0^1 \abs{\xi-r\varphi}^{2\beta} \dr \right) = \exp(-\psi_g(\varphi,\xi)-\psi_b(\varphi,\xi))
\end{equation*}
in a smooth part and a part induced by the Fourier transform of a measure, see Section \ref{sec:proofintest}. 
The measure part induces an $\L^p_{x,v}$-bounded operator. For the smooth part $\exp(-\psi_g)$ together with the rest 
of the multiplier, we want to employ Mikhlin's theorem to deduce the estimate. 
When $\beta=1$, we directly have smoothness of the integral and can write 
$\widehat{E_{\beta}}(1,0,\varphi,\xi)= \exp(-\psi_g(\varphi,\xi))$. 
	
Setting $\Lambda=(t/s)^{\frac{1}{2\beta}}$ and defining 	
\begin{equation*}
	M_\Lambda (\varphi,\xi) = \hat\eta\left( \frac{\varphi}{\Lambda^{2\beta+1}}, \frac{\xi}{\Lambda} \right) 
	d_{\beta}(\varphi,\xi-\varphi)^\gamma \exp\left(-\psi_g(\varphi,\xi)\right),  
\end{equation*}
the proof reduces to the following Lemma \ref{lem:helpMLambda}.
\end{proof} 
	
\begin{lem}  \label{lem:helpMLambda}
Let $\beta\in (0,1]$, $p \in (1,\infty)$ and $\gamma\in \R$.
$M_{\Lambda}$ is an $\L^p_{x,v}$ Fourier multiplier with bound 
$\kappa_{\gamma} (\Lambda)= c_{1}\Lambda^\gamma e^{-c_{0}\Lambda^{2\beta}}$ for some $c_0,c_1>0$
independent of $\Lambda$.
\end{lem}

\begin{proof}  
We begin with the first item. According to Lemma~\ref{lem:xi-phi} and since $M_{\Lambda}$ has support where 
$d_{\beta}(\varphi,\xi)\sim \Lambda$, $M_{\Lambda}$ is $\C^\infty$. 
By Marcinkiewicz's multiplier theorem, it suffices to show 
\begin{equation}\label{eq:MikGoal}
	\sup_{(\varphi,\xi) \in \R^{2d} \setminus \{(0,0) \}}
	\abs{\varphi}^{\abs{\alpha}}\abs{\xi}^{\abs{\sigma}}
	\abs{\partial_\varphi^\alpha \partial_\xi^\sigma M_\Lambda(\varphi,\xi)}
	\lesssim_{\alpha,\sigma} \Lambda^\gamma e^{-c_{0}\Lambda^{2\beta}}.
\end{equation}
First, 
\begin{equation*}
	\abs{\varphi}^{\abs{\alpha}}\abs{\xi}^{\abs{\sigma}}
  	\abs{\partial_\varphi^\alpha \partial_\xi^\sigma  \exp\left(-\psi_g(\varphi,\xi)\right)}   \lesssim_{\alpha,\sigma} 
  	\exp(-c_0(\Lambda^{2\beta} \wedge \Lambda^2)) \lesssim_{\alpha,\sigma} \inf (1, \exp(-c_0 \Lambda^{2\beta})).
\end{equation*}
Indeed, recall that when $\beta<1$
\begin{equation*}
	\psi_g(\varphi,\xi) = \frac{1}{2}c_{\beta,d}\int_0^1\int_{0}^1 \int_{\S^{d-1}} \frL_\beta(\xi-r\varphi,\rho\omega) \rho^{-1-2\beta} \dd \omega \dd \rho  \dr.
\end{equation*}
	
Fix $r \in [0,1]$. 
According to \cite[Lemma 7.5]{MR4709546} the function $h \colon \R^{2d} \to \R$,
\begin{equation*}
	h(\varphi,\xi) := \int_{0}^1 \int_{\S^{d-1}} \frL_\beta(\xi-r\varphi,\rho\omega) \rho^{-1-2\beta} \dd \omega \dd \rho
\end{equation*}
is $\C^\infty$ uniformly in $r$ with 
\begin{equation*}
	\abs{\partial_\varphi^\alpha \partial_\xi^\sigma h(\varphi,\xi)} \lesssim_{\alpha,\sigma} \left( (1+\abs{(\varphi,\xi)}) \mathds{1}_{\abs{\alpha} +\abs{\sigma}= 1}+\mathds{1}_{\abs{\alpha}+\abs{\sigma} \ge 2} \right).
\end{equation*}
 Hence, $\psi_g(\varphi,\xi)$ is $\C^\infty$ with 
\begin{equation*}
	\abs{\partial_\varphi^\alpha \partial_\xi^\sigma \psi_g(\varphi,\xi)} \lesssim_{\alpha,\sigma} \left( 1+ \abs{\xi}+\abs{\varphi} \right).
\end{equation*}
Due to the coercivity, see \eqref{eq:psigcoercive}, we know that every term 
\begin{equation*}
	\abs{\varphi}^{\abs{\alpha}}\abs{\xi}^{\abs{\sigma}}   \exp(-\psi_g(\varphi,\xi)) \lesssim_{\alpha,\sigma} \exp\left(-\frac{1}{2} \psi_g(\varphi,\xi)\right)
\end{equation*}
and we conclude by the Fa\`a-di-Bruno formula combining the two inequalities. 
For $\beta=1$, the same conclusion holds more easily. 

Next, by the Leibniz rule, write $\alpha=\alpha_{1}+\alpha_{2}+\alpha_{3}$ and $\sigma=\sigma_{1}+\sigma_{2}+\sigma_{3}$
 so that we have to estimate
\begin{equation}\label{eq:toestimate}
	\abs{\varphi}^{\abs{\alpha}}\abs{\xi}^{\abs{\sigma}}
	\abs{\partial_\varphi^{\alpha_{1}} \partial_\xi^{\sigma_{1}} \hat{\eta}\left( \frac{\varphi}{\Lambda^{2\beta+1}}, \frac{\xi}{\Lambda} \right)}
	\abs{\partial_\varphi^{\alpha_{2}} \partial_\xi^{\sigma_{2}} d_{\beta}(\varphi,\xi-\varphi)^\gamma}
	\abs{\partial_\varphi^{\alpha_{3}} \partial_\xi^{\sigma_{3}} \exp\left(-\psi_g(\varphi,\xi)\right)}.
\end{equation}
We begin with $\Lambda\le 1$. 
Recall that $\hat\eta$ is supported in the region where $1/4< d_{\beta}(\varphi,\xi)< 4$. 
In the  support of $\hat\eta\left( \frac{\varphi}{\Lambda^{2\beta+1}}, \frac{\xi}{\Lambda} \right)$, we know by 
Lemma~\ref{lem:xi-phi} that  $d_{\beta}(\varphi, \xi-\varphi)\sim d_{\beta}(\varphi,\xi) \sim \Lambda$. 
For the derivatives, we have 
\[ 
	\partial_\varphi^{\alpha_{2}} \partial_\xi^{\sigma_{2}} [d_{\beta}(\varphi,\xi-\varphi)^\gamma] = [(\partial_\varphi-\partial_\xi)^{\alpha_{2}} \partial_\xi^{\sigma_{2}} d_{\beta}^\gamma](\varphi,\xi-\varphi)
\]
so this last term is a combination of terms 
\[ 
	[\partial_\varphi^{\alpha_{2}^\flat} \partial_\xi^{\alpha_{2}^\sharp+\sigma_{2}} d_{\beta}^\gamma](\varphi,\xi-\varphi)
\]
where we write $\alpha_{2}=\alpha_{2}^\flat+\alpha_{2}^\sharp$ and the latter is bounded by 
\[ 	
	d_{\beta}(\varphi,\xi-\varphi)^{\gamma- (2\beta+1)|\alpha_{2}^\flat|-|\alpha_{2}^\sharp|-|\sigma_{2}|} \sim \Lambda^{\gamma- (2\beta+1)|\alpha_{2}^\flat|-|\alpha_{2}^\sharp|-|\sigma_{2}|}.
\]
Altogether,  and using $|\varphi|\lesssim \Lambda^{2\beta+1}$ and $|\xi|\lesssim \Lambda$, the term in 
\eqref{eq:toestimate} is bounded a constant times 
\begin{align*}
	& \Lambda^{(2\beta+1)|\alpha|+|\sigma|} \Lambda^{-(2\beta+1)|\alpha_{1}|-|\sigma_{1}|}  
	\Lambda^{\gamma- (2\beta+1)|\alpha_{2}^\flat|-|\alpha_{2}^\sharp|-|\sigma_{2}|} \\
	&\qquad=\Lambda^{\gamma}\Lambda^{(2\beta+1)(|\alpha_{2}|-|\alpha_{2}|^\flat+|\alpha_{3}|)}
	\Lambda^{|\sigma_{3}|-|\alpha_{2}^\sharp|} =\Lambda^{\gamma}\Lambda^{(2\beta+1)|\alpha_{3}|} \Lambda^{2\beta ||\alpha_{2}^\sharp|}
	\Lambda^{|\sigma_{3}|}. 
\end{align*}
As $\alpha_{3},\sigma_{3}, \alpha_{2}^\sharp$ can be $0$ and as $\Lambda \le 1$, the best possible bound is $\Lambda^\gamma$. 

When $\Lambda\ge 1$, we can do the same calculation, but here, 
Lemma~\ref{lem:xi-phi} yields $\Lambda^{\frac{1}{2\beta+1}}  \lesssim 
d_{\beta}(\varphi,\xi-\varphi) \lesssim \Lambda^{{2\beta+1}}.$ 
So the first two terms contribute to some power of $\Lambda$ and the value of the exponent 
does not matter since it will be absorbed by the exponential factor coming from $\psi_{g}$ when $\Lambda\ge 1$.
\end{proof}

\begin{lem}[Bounds for the initial value problem]\label{lem:boundsinitialvalueproblem}
Let $\beta\in (0,1]$, $p \in (1,\infty)$ and $\gamma\in \R$.
	The solution of the Kolmogorov equation given by \eqref{eq:solKolIV} with $g \in \cS(\R^{2d})$ whose Fourier transform has 
	compact support in $(\R_{\varphi}^{d}\setminus \{0\})\times (\R_{\xi}^{d}\setminus \{0\})$, satisfies the estimate
	\begin{equation} \label{eq:tracekol}
		\norm{g}_{\Bdot_{\beta}^{\gamma-{2\beta}/{p},p}} \lesssim \norm{f}_{\L^p_t((0,\infty);\Xdot^{\gamma,p}_{\beta})} \lesssim \norm{g}_{\Bdot_{\beta}^{\gamma-{2\beta}/{p},p}}.
	\end{equation}
	Moreover, if $g \in \Bdot^{\gamma-2\beta/p,p}_\beta$, then $f\in \C^{}_0\bigl([0,\infty);\Bdot^{\gamma-2\beta/p,p}_\beta\bigr)$ with 
	\begin{align*}
		\sup_{t\ge 0}\|f (t,\cdot)\|_{\Bdot^{\gamma-2\beta/p,p}_{\beta}}&\lesssim \norm{f}_{\L^p_t((0,\infty);\Xdot^{\gamma,p}_{\beta})}, \\
		\lim_{t\to0^+}\|f(t,\cdot)-g\|_{\Bdot^{\gamma-2\beta/p,p}_\beta}&=0, \\
		\lim_{t\to \infty}\|f (t,\cdot)\|_{\Bdot^{\gamma-2\beta/p,p}_{\beta}}&=0.
	\end{align*}
\end{lem}

\begin{proof}
{\textbf{Step 1: upper bound in \eqref{eq:tracekol}.} We start with the upper bound in the first estimate.} 
Pick $\eta\in \cS(\R^{2d})$  so that $\hat \eta=1$ on the support of $\hat\psi$ and 
with support in $\frac{1}{4} < d_{\beta}(\varphi,\xi) < 4$. 
{
We additionally assume the reproducing property
\[
\int_{0}^\infty \hat{\psi}(s^{\frac{1}{2\beta}+1}\varphi,s^{\frac{1}{2\beta}}\xi) \, \frac{\ds}{s} = 1
\]
for $(\varphi,\xi)\ne (0,0)$.
}
Thus, 
\begin{equation*}
	1 = \int_{0}^\infty \widehat{\psi_{s}}(\varphi,\xi) \, \frac{\ds}{s}
	= \int_{0}^\infty \widehat{\eta_{s}}(\varphi,\xi)\widehat{\psi_{s}}(\varphi,\xi) \, \frac{\ds}{s}.
\end{equation*}
	By the translation invariance of the Lebesgue measure, we have
	\begin{align*}
		&\norm{f}_{\L^p_t\Xdot^{\gamma,p}_{\beta}}^p \\
		&\simeq \int_0^\infty \norm{\cF^{-1}\left( d_{\beta}(\varphi,\xi)^\gamma \widehat{E_{\beta}}(t,0,\varphi,\xi+t\varphi)\hat{g}(\varphi,\xi+t\varphi) \right) }^p_{\L^p_{x,v}}  \dt \\
		&= \int_0^\infty   \norm{\cF^{-1}\left( d_{\beta}(\varphi,\xi-t\varphi)^\gamma\widehat{E_{\beta}}(t,0,\varphi,\xi)\hat{g}(\varphi,\xi) \right) }^p_{\L^p_{x,v}}    \dt 
		\\
		&= \int_0^\infty  \norm{\cF^{-1}\left(  \int_{0}^\infty \widehat{\eta_{s}}(\varphi,\xi)\widehat{\psi_{s}}(\varphi,\xi) \frac{\ds}{s} d_{\beta}(\varphi,\xi-t\varphi)^\gamma\widehat{E_{\beta}}(t,0,\varphi,\xi) \hat{g}(\varphi,\xi) \right) }^p_{\L^p_{x,v}}   \dt
		\\
		&=\int_0^\infty  \norm{\cF^{-1}\left( \int_{0}^\infty \widehat{\eta_{s}}(\varphi,\xi)d_{\beta}(\varphi,\xi-t\varphi)^\gamma\widehat{E_{\beta}}(t,0,\varphi,\xi) \widehat{\psi_{s}}(\varphi,\xi)\hat{g}(\varphi,\xi)  \frac{\ds}{s} \right) }^p_{\L^p_{x,v}}  \dt \\
		&\le \int_0^\infty \left( \int_{0}^\infty \norm{\cF^{-1}\left( \widehat{\eta_{s}}(\varphi,\xi)d_{\beta}(\varphi,\xi-t\varphi)^\gamma\widehat{E_{\beta}}(t,0,\varphi,\xi) \widehat{\psi_{s}}(\varphi,\xi)\hat{g}(\varphi,\xi) \right)}_{\L^p_{x,v}} \frac{\ds}{s}\right)^p    \dt\\
		&\le \int_0^\infty \left( \int_{0}^\infty  t^{-\frac{\gamma}{2\beta}} \kappa_{\gamma}((t/s)^{\frac{1}{2\beta}}) \norm{\psi_s\ast g}_{\L^p_{x,v}} \frac{\ds}{s} \right)^p    \dt
		\\ 
		&\le \int_0^\infty \left( \int_{0}^\infty  c_{1}(t/s)^{\frac{1}{p}} e^{-c_{0}{(t/s)}} s^{ -\frac1{2\beta}({\gamma-\frac{2\beta}{p}}) }\norm{\psi_s\ast g}_{\L^p_{x,v}} \frac{\ds}{s} \right)^p   \frac{ \dt}{t}
\\
		&
		\le C^p  \int_0^\infty \left(  s^{ -\frac{1}{2\beta}(\gamma-\frac{2\beta}{p}) }\norm{\psi_s\ast g}_{\L^p_{x,v}}  \right)^p   \frac{ \ds}{s} \simeq \|g\|_{\Bdot_{\beta}^{\gamma-{2\beta}/{p},p}}^p.
	\end{align*}
where $\kappa_{\gamma}$ is as in Lemma \ref{lem:traceMult}, and with 
$C= \int_{0}^\infty c_{1}r^{\frac{1}{p}} e^{-c_{0}r}\,  \frac{\dr}{r}$ after applying Schur's lemma 
in $\L^p((0,\infty);\frac{ \dt}{t})$.  
Remark that we used the properties of $\hat g$ to justify the first three lines. 
First, $(\varphi,\xi)\mapsto \widehat{E_{\beta}}(t,0,\varphi,\xi+t\varphi)\hat{g}(\varphi,\xi+t\varphi)$ belongs to 
$\L^2_{\varphi,\xi}$ and is supported away from $\varphi=0$. 
In particular, for each $t$, we have $0< m\le d_{\beta}(\varphi,\xi)< M(t)$, so we can multiply by 
$ d_{\beta}(\varphi,\xi)^\gamma$ and the computation of $\|f(t)\|_{\Xdot^{\gamma,p}_{\beta}}$ via the 
inverse Fourier transform is justified. 
The introduction of the Littlewood--Paley decomposition at each fixed $t$ is justified similarly.   

\bigskip

\textbf{Step 2: lower bound in \eqref{eq:tracekol}.}
As before, we use translation invariance to write
\begin{align}
	\|f\|_{\L^p((0,\infty);\Xdot_\beta^{\gamma,p})}^p
	&\simeq \int_0^\infty \Big\| \cF^{-1}\Big( d_\beta(\varphi,\xi)^\gamma\,\hat f(t,\varphi,\xi) \Big) \Big\|_{\L^p_{x,v}}^p \dd t \nonumber\\
	&= \int_0^\infty \Big\| \cF^{-1}\Big( d_\beta(\varphi,\xi-t\varphi)^\gamma\,\widehat{E_\beta}(t,0,\varphi,\xi)\,\hat g(\varphi,\xi) \Big) \Big\|_{\L^p_{x,v}}^p \dd t. \label{eq:def_Tt}
\end{align}
For $t>0$ we define the Fourier multiplier
\[
	m_t(\varphi,\xi) := d_\beta(\varphi,\xi-t\varphi)^\gamma\,\widehat{E_\beta}(t,0,\varphi,\xi),
\]
and the corresponding operator $m_t(D) g:=\cF^{-1}\big(m_t\,\hat g\big)$.
Then \eqref{eq:def_Tt} reads
\begin{equation}\label{eq:Lp_equals_Tt}
	\|f\|_{\L^p((0,\infty);\Xdot_\beta^{\gamma,p})}^p \simeq \int_0^\infty \|m_t(D) g\|_{\L^p_{x,v}}^p\dd t.
\end{equation}
The lower bound will be obtained by recovering a Littlewood--Paley piece of $g$ from
$m_t(D)g$ by means of a local inverse in the Wiener algebra.

\bigskip

\noindent
\emph{1. The Wiener algebra.}
We set
\[
	\A(\R^{2d}) := \mathcal F \L^1(\R^{2d})
	= \bigl\{\widehat h : h\in \L^1(\R^{2d})\bigr\},
	\qquad
	\A_1(\R^{2d}) := \IC \oplus \A(\R^{2d}).
\]
The space $\A(\R^{2d})$ is a Banach algebra under pointwise multiplication, because $\widehat h\,\widehat k=\widehat{h*k}
	$ for $ h,k\in \L^1(\R^{2d}),$
and $\A(\R^{2d})$ is an ideal in $\A_1(\R^{2d})$. Moreover, $\C_c^\infty(\R^{2d})\subset \A(\R^{2d})$, since for $a\in \C_c^\infty(\R^{2d})$ one has $\mathcal F^{-1}a\in \cS(\R^{2d})\subset \L^1(\R^{2d})$.

We set
\[
	u(\varphi,\xi):=\widehat E_\beta(1,0,\varphi,\xi)
	=\exp\!\left(-\int_0^1 |\xi-r\varphi|^{2\beta}\,dr\right). 
\]
We claim that $u\in \A(\R^{2d})$. Indeed, by the decomposition from Section \ref{sec:estimates},
\[
	u(\varphi,\xi)=e^{-\psi_g(\varphi,\xi)}e^{-\psi_b(\varphi,\xi)}
	=\widehat{\Psi_g}(\varphi,\xi)\,\widehat{\Psi_b}(\varphi,\xi)
	=\widehat{\Psi_g * \Psi_b}(\varphi,\xi),
\]
where $\Psi_g\in \cS(\R^{2d})\subset \L^1(\R^{2d})$, and $\Psi_b$ is a probability
measure (for $\beta=1$, $\Psi_b=\delta_0$). Hence $\Psi_g * \Psi_b\in \L^1(\R^{2d})$, and therefore $u\in \A(\R^{2d})$.
Next, $u(\varphi,\xi)>0$ for every $(\varphi,\xi)\in\R^{2d}$.

\medskip

\noindent
\emph{2. Construction of a local inverse.}
Choose $\vartheta\in\mathcal S(\R^{2d})$ such that
\[
	\hat\vartheta>0
	\quad\text{on}\quad
	\Bigl\{\frac12<d_\beta(\varphi,\xi)<2\Bigr\},
	\qquad
	\supp\hat\vartheta\subset
	\Bigl\{\frac14<d_\beta(\varphi,\xi)<4\Bigr\}.
\]
We recall that any such cutoff may be used in the continuous
Littlewood--Paley characterisation of $\Bdot_\beta^{s,p}$.

Since $\supp\hat\vartheta$ is a compact subset of the annulus
$\{\frac14<d_\beta(\varphi,\xi)<4\}$, Lemma~\ref{lem:xi-phi} implies that
\[
	d_\beta(\varphi,\xi-\varphi)\sim d_\beta(\varphi,\xi)\sim 1
	\qquad\text{on }\supp\hat\vartheta.
\]
In particular, there exists an open neighborhood $\Omega$ of $\supp\hat\vartheta$ and a constant
$c_*>0$ such that
\[
	d_\beta(\varphi,\xi-\varphi)\ge c_*
	\qquad\text{for all }(\varphi,\xi)\in\Omega.
\]
Using the smoothness properties of $d_\beta$ away from the origin, the function
\[
	(\varphi,\xi)\longmapsto d_\beta(\varphi,\xi-\varphi)^{-\gamma}
\]
is $\C^\infty$ on $\Omega$.

Set
\[
	\chi(\varphi,\xi):=\hat\vartheta(\varphi,\xi)\,d_\beta(\varphi,\xi-\varphi)^{-\gamma},
\]
then $\chi\in \C_c^\infty(\R^{2d})\subset \A(\R^{2d})$.

Next choose $\rho_0\in \C_c^\infty(\R^{2d})$ such that
\[
	0\le \rho_0\le 1,
	\qquad
	\rho_0=1 \text{ on } \supp\chi .
\]
Define
\[
	U_0(\varphi,\xi):=1-\rho_0(\varphi,\xi)+\rho_0(\varphi,\xi)\,u(\varphi,\xi).
\]
Since $\rho_0\in \A(\R^{2d})$ and $u\in \A(\R^{2d})$, we have $U_0\in \A_1(\R^{2d})$.

We now verify that $U_0$ has no zeros. Outside $\supp\rho_0$, we have $U_0=1$.
On $\supp\rho_0$, the function $u$ is continuous and strictly positive, and
$\supp\rho_0$ is compact; therefore
\[
	m_0:=\inf_{(\varphi,\xi) \in \supp\rho_0} u(\varphi,\xi) >0.
\]
Hence for every $(\varphi,\xi)\in\R^{2d}$,
\[
	U_0(\varphi,\xi)\ge \min\{1,m_0\}>0.
\]
Therefore the function $z\mapsto z^{-1}$ is holomorphic on a neighborhood of the range of $U_0$,
and the Wiener--L\'evy theorem (\cite{MR2039503}) in the unital algebra $A_1(\R^{2d})$ yields
\[
	U_0^{-1}\in A_1(\R^{2d}).
\]

Now define
\[
	b(\varphi,\xi):=\chi(\varphi,\xi)\,U_0(\varphi,\xi)^{-1}.
\]
Since $\chi\in \A(\R^{2d})$ and $\A(\R^{2d})$ is an ideal in $\A_1(\R^{2d})$, we have $b\in \A(\R^{2d})$.

Finally, because $\rho_0=1$ on $\supp\chi$, we have $U_0=u$ on $\supp\chi$. Hence
\begin{equation}
\label{eq:localinverse}
b(\varphi,\xi)\,u(\varphi,\xi)=\chi(\varphi,\xi)
	=\hat\vartheta(\varphi,\xi)\,d_\beta(\varphi,\xi-\varphi)^{-\gamma}
\end{equation}
for every $(\varphi,\xi)\in\R^{2d}$. This is the desired local inverse relation.

\medskip
\noindent
\emph{3. Rescaling and recovery of a Littlewood--Paley piece of $g$.}
For $t>0$, define the anisotropic rescaling of $b$ by
\[
	b_t(\varphi,\xi):=
	b\Bigl(t^{1+\frac1{2\beta}}\varphi,\; t^{\frac1{2\beta}}\xi\Bigr).
\]
Since $b\in \A(\R^{2d})$, there exists $k\in \L^1(\R^{2d})$ such that $b=\widehat k$.
Define its anisotropic dilate by
\[
	k_t(x,v):=
	t^{-(1+\frac1\beta)d}\,
	k\Bigl(t^{-(1+\frac1{2\beta})}x,\;t^{-\frac1{2\beta}}v\Bigr).
\]
Then $\widehat{k_t}=b_t$ and $\|k_t\|_{\L^1_{x,v}}=\|k\|_{\L^1_{x,v}}$. Consequently,
for every $h\in \L^p(\R^{2d})$,
\[
	\|b_t(D)h\|_{\L^p_{x,v}}
	=\|k_t*h\|_{\L^p_{x,v}}
	\le \|k_t\|_{\L^1_{x,v}}\|h\|_{\L^p_{x,v}}
	=\|k\|_{\L^1_{x,v}}\|h\|_{\L^p_{x,v}}.
\]
Thus the operators $b_t(D)$ are bounded on $\L^p_{x,v}$, uniformly in $t>0$.

We now compute the relation between $b_t$, $m_t$, and $\vartheta_t$.
Set
\[
	\tilde\varphi:=t^{1+\frac1{2\beta}}\varphi,
	\qquad
	\tilde\xi:=t^{\frac1{2\beta}}\xi .
\]
By the anisotropic homogeneity of $d_\beta$ we have $d_\beta(\varphi,\xi-t\varphi) = t^{-\frac1{2\beta}}\,d_\beta(\tilde\varphi,\tilde\xi-\tilde\varphi)$.
Moreover, using the change of variables $r=ts$, we deduce
\begin{align*}
	\widehat E_\beta(t,0,\varphi,\xi)
	&=
	\exp\!\left(-\int_0^t |\xi-r\varphi|^{2\beta}\,dr\right) =
	\exp\!\left(
	-\int_0^t
	t^{-1}\bigl|\tilde\xi-(r/t)\tilde\varphi\bigr|^{2\beta}\,dr
	\right) \\
	&=
	\exp\!\left(-\int_0^1 |\tilde\xi-s\tilde\varphi|^{2\beta}\,ds\right)
	=
	u(\tilde\varphi,\tilde\xi).
\end{align*}
Therefore
\[
	m_t(\varphi,\xi)
	=
	t^{-\frac{\gamma}{2\beta}}\,
	d_\beta(\tilde\varphi,\tilde\xi-\tilde\varphi)^\gamma\,
	u(\tilde\varphi,\tilde\xi).
\]
Multiplying by $b_t(\varphi,\xi)=b(\tilde\varphi,\tilde\xi)$ and using the local inverse relation \eqref{eq:localinverse}, we obtain
\begin{align*}
	b_t(\varphi,\xi)\,m_t(\varphi,\xi)
	&=
	t^{-\frac{\gamma}{2\beta}}\,
	b(\tilde\varphi,\tilde\xi)\,
	d_\beta(\tilde\varphi,\tilde\xi-\tilde\varphi)^\gamma\,
	u(\tilde\varphi,\tilde\xi) \\
	&=
	t^{-\frac{\gamma}{2\beta}}\,
	\hat\vartheta(\tilde\varphi,\tilde\xi)
	=
	t^{-\frac{\gamma}{2\beta}}\,
	\widehat{\vartheta_t}(\varphi,\xi).
\end{align*}
Hence, on the Fourier side,
\[
	[\widehat{\vartheta_t*g}](\varphi,\xi)
	=
	\widehat{\vartheta_t}(\varphi,\xi)\,\hat g(\varphi,\xi)
	=
	t^{\frac{\gamma}{2\beta}}\,
	b_t(\varphi,\xi)\,m_t(\varphi,\xi)\,\hat g(\varphi,\xi).
\]
Taking inverse Fourier transform and $\L^p$-norms and using the uniform $\L^p$-boundedness of $b_t(D)$, we infer
\[
	t^{-\frac{\gamma}{2\beta}}
	\|\vartheta_t*g\|_{\L^p_{x,v}}
	\lesssim
	\|m_t(D)g\|_{\L^p_{x,v}}
	\qquad\text{for every }t>0.
\]

\medskip

\noindent
\emph{4. Integration in $t$.}
Using the continuous Littlewood--Paley characterisation with the admissible cutoff $\vartheta$, we obtain
\begin{align*}
	\|g\|_{\Bdot_\beta^{\gamma-\frac{2\beta}{p},p}}^p
	&\simeq
	\int_0^\infty
	\left(
	\frac{\|\vartheta_t*g\|_{\L^p_{x,v}}}
	{t^{(\gamma-\frac{2\beta}{p})/(2\beta)}}
	\right)^p
	\frac{\dd t}{t} \\
	&=
	\int_0^\infty
	t^{-\frac{\gamma p}{2\beta}}
	\|\vartheta_t*g\|_{\L^p_{x,v}}^p\dd t \\
	&\lesssim
	\int_0^\infty
	\|m_t(D)g\|_{\L^p_{x,v}}^p\dd t \\
	&\simeq
	\|f\|_{\L^p((0,\infty);\Xdot_\beta^{\gamma,p})}^p .
\end{align*}
This proves the lower bound in \eqref{eq:tracekol}.

\medskip
{
\noindent\textbf{Step 3: continuity and limits.}
}
	
	{
	Since $\supp\hat g\subset (\R^d_\varphi\setminus\{0\})\times(\R^d_\xi\setminus\{0\})$ is compact, 
	there exists $\delta_\varphi>0$ such that $|\varphi|\ge\delta_\varphi$ on $\supp\hat g$. 
	For every fixed $t\ge0$,
	\[
		\supp \hat f(t,\cdot)
		\subset \{(\varphi,\xi):(\varphi,\xi+t\varphi)\in\supp\hat g\},
	\]
	so in particular $|\varphi|\ge\delta_\varphi$ on $\supp\hat f(t,\cdot)$. 
	Hence $d_\beta(\varphi,\xi)\gtrsim_\beta |\varphi|^{\frac1{2\beta+1}}\gtrsim_{\beta,g} 1$ 
	on $\supp\hat f(t,\cdot)$, uniformly in $t$.
	The same is true for differences $f(t+h)-f(t)$: their Fourier supports are contained in unions of such sets.
	In particular, all homogeneous Besov norms appearing below are unambiguously defined and finite for these data, 
	and the trace estimate proved in Steps~1--2 applies to them with the same constants.
	}
	
	{
	For $t\ge 0$, we recall the Kolmogorov semigroup {$\sem^+_t$} defined in \eqref{eq:solKolIV} 
	\begin{equation}\label{eq:def_semigroup}
		[\widehat{\sem^+_t g}](\varphi,\xi)
		:= \exp\!\Big( -\int_0^t \abs{\xi+(t-r)\varphi}^{2\beta}\, \dr \Big)\, \hat g(\varphi,\xi+t\varphi).
	\end{equation}
	Then $f(t)=\sem^+_t g$ for $t\ge 0$, and we have the semigroup property
	\begin{equation}\label{eq:semigroup_property}
		\sem^+_{t+s}=\sem^+_t\sem^+_s\qquad\text{for all }t,s\ge 0.
	\end{equation}
	Moreover, we fix $t\ge 0$ and define the time-shifted function $F_t(s):=f(t+s)$ for $s\ge 0$.
	}
	
	{
	By \eqref{eq:semigroup_property}, $F_t(s)=\sem^+_s(f(t))$, so $F_t$ is the  solution to the forward Kolmogorov equation when $s>0$ with initial datum $f(t)$.
	Applying \eqref{eq:tracekol} to $F_t$ yields the $\L^\infty_t$ bound
	\begin{equation} \label{eq:Linfty-trace-bound}
		\|f(t)\|_{\Bdot_\beta^{{\gamma-{2\beta}/{p}},p}} \lesssim \|F_t\|_{\L^p((0,\infty);\Xdot_\beta^{\gamma,p})}= \|f\|_{\L^p((t,\infty);\Xdot_\beta^{\gamma,p})}.
	\end{equation}
	}
	
	{
	Fix $t_0\ge 0$ and $h>0$, and define
	\[
		W_{t_0,h}(s):=f(t_0+s+h)-f(t_0+s),\qquad s\ge 0.
	\]
	Using \eqref{eq:semigroup_property},
	\[
		W_{t_0,h}(s)=\sem^+_{t_0+s+h}g-\sem^+_{t_0+s}g
		=\sem^+_s\big(\sem^+_{t_0+h}g-\sem^+_{t_0}g\big)
		=\sem^+_s\big(f(t_0+h)-f(t_0)\big).
	\]
	Thus $W_{t_0,h}$ is the forward Kolmogorov solution with initial datum $W_{t_0,h}(0)=f(t_0+h)-f(t_0)$.
	Applying the trace bound \eqref{eq:tracekol} to $W_{t_0,h}$ gives us
	\begin{equation}\label{eq:modulus_of_continuity}
		\|f(t_0+h)-f(t_0)\|_{\Bdot_\beta^{{\gamma-{2\beta}/{p}},p}} \lesssim		\|W_{t_0,h}\|_{\L^p((0,\infty);\Xdot_\beta^{\gamma,p})}		=		\|f(\cdot+h)-f(\cdot)\|_{\L^p((t_0,\infty);\Xdot_\beta^{\gamma,p})}.
	\end{equation}
	Since $f\in \L^p((0,\infty);\Xdot_\beta^{\gamma,p})$ the translation invariance of vector valued 
	$\L^p$ spaces together with \eqref{eq:modulus_of_continuity} yields
	\[
		\lim_{h\to0^+}\|f(t_0+h)-f(t_0)\|_{\Bdot_\beta^{{\gamma-{2\beta}/{p}},p}}=0
		\qquad\text{for every }t_0\ge 0,
	\]
	i.e.\ right-continuity for all $t_0\ge0$.
	}

	{
	To prove the left-continuity at $t_0>0$, we let $0<h<t_0$ and apply \eqref{eq:modulus_of_continuity} with $t_0$ replaced by $t_0-h$:
	\[
		\|f(t_0)-f(t_0-h)\|_{\Bdot_\beta^{{\gamma-{2\beta}/{p}},p}}
		\lesssim
		\|f(\cdot+h)-f(\cdot)\|_{\L^p((t_0-h,\infty);\Xdot_\beta^{\gamma,p})}.
	\]
	Splitting $(t_0-h,\infty)=(t_0,\infty)\cup(t_0-h,t_0)$ and using $\|a-b\|^p\le 2^{p-1}(\|a\|^p+\|b\|^p)$,
	\[
		\|f(\cdot+h)-f(\cdot)\|_{\L^p((t_0-h,t_0);\Xdot_\beta^{\gamma,p})}^p
		\le 2^{p-1}\int_{t_0-h}^{t_0+h}\|f(\tau)\|_{\Xdot_\beta^{\gamma,p}}^p \dd\tau \xrightarrow[h\to0^+]{} 0
	\]
	by absolute continuity of the Lebesgue integral, while the tail term over $(t_0,\infty)$ tends to $0$ by translation invariance.
	We have proven continuity on $(0,\infty)$ with values in $\Bdot_\beta^{{\gamma-{2\beta}/{p}},p}$.
	}
	
	{
	Furthermore, taking the limit $t \to \infty$ in \eqref{eq:Linfty-trace-bound} yields
	\[
		\|f(t)\|_{\Bdot_\beta^{{\gamma-{2\beta}/{p}},p}}
		\lesssim \|f\|_{\L^p((t,\infty);\Xdot_\beta^{\gamma,p})}\xrightarrow[t\to\infty]{}0,
	\]
	so $f\in \C_0([0,\infty);\Bdot_\beta^{{\gamma-{2\beta}/{p}},p})$. 
	Finally, taking $t_0=0$ and using that $f(0)=g$ by construction, the same right-continuity argument gives
	$\|f(t)-g\|_{\Bdot_\beta^{{\gamma-{2\beta}/{p}},p}}\to0$ as $t\to0^+$.
	}	
\end{proof}

\subsection{Consequences for the Kolmogorov operators}
We first recall a consequence of the energy equality. 

\begin{lem}\label{lem:energy} 
Let $g\in \cS(\R^{2d})$ whose Fourier transform has compact  support in 
$(\R_{\varphi}^{d}\setminus \{0\})\times (\R_{\xi}^{d}\setminus \{0\})$,  
$S\in \cS_{K}$ and $\tau\in \R$. Then
\[
	\int_{\R^{2d}} g(x,v) (\cK_{\beta}^-S)(\tau,x,v)\, \dd(x,v)= \int_{\tau}^\infty\int_{\R^{2d}} \sem^+_{t-\tau}g(x,v)S(t,x,v)\,  \dd(x,v)\dt.
\]
Similarly,
\[
	\int_{\R^{2d}} g(x,v) (\cK_{\beta}^+S)(\tau,x,v)\, \dd(x,v)= \int^{\tau}_{-\infty}\int_{\R^{2d}} \sem^-_{t-\tau}g(x,v)S(t,x,v)\,  \dd(x,v)\dt.
\]
 \end{lem}

\begin{proof} 
We only prove this first one. Let $f(t, x,v)=\sem^+_{t}g(x,v)$ be given by \eqref{eq:solKolIV}.
Shifting time, we have for $t> \tau$, 
$(\partial_{t}+v\cdot\nabla_{x})f(t-\tau, x,v)+ (-\Delta_{v})^\beta f(t-\tau,x,v)= 0$, 
the equality being in the sense of strong solutions in $\L^2_{x,v}$ by Lemma~\ref{lem:initialvalue}, with $f(0,x,v)=g(x,v)$.  
By Lemma~\ref{lem:equation}, also in the sense of strong solutions in $\L^2_{x,v}$,  
$-(\partial_{t}+v\cdot\nabla_{x})(\cK_{\beta}^-S)+ (-\Delta_{v})^\beta (\cK_{\beta}^- S)=S$. 
Hence the formula follows as in \cite[Theorem~3.9]{AIN} using the absolute continuity of 
$t\mapsto \angle{f(t-\tau)}{(\cK_{\beta}^-S)(t)}$ on $[\tau,\infty)$ and zero limit at $\infty$. 
 \end{proof}

\begin{cor}[A priori bounds involving Besov spaces] \label{cor:besovbounds} 
Let $\beta\in (0,1]$, $\gamma\in \R$ and $p \in (1,\infty)$. Let $S\in \cS_{K}$.   
\begin{itemize}
\item $ \cK_{\beta}^\pm S$ belongs to 
$\C^{}_{0}(\R^{}_{t}\,;\, \Bdot_{\beta}^{\gamma-{2\beta}/{p},p})$ with 
\[
\|  \cK_{\beta}^\pm S\|_{\L^\infty_{t}\Bdot_{\beta}^{\gamma-{2\beta}/{p},p}} \lesssim_{\gamma,p} \|S\|_{\L^p_{t}\Xdot^{\gamma-2\beta,p}_{\beta}}. 
\]
\item
$ \cK_{\beta}^\pm S $ belongs to 
$\C^{}_{0}(\R^{}_{t}\,;\, \Bdot_{\beta}^{\gamma-2\beta/p,p})$ with 
\[
\|  \cK_{\beta}^\pm S \|_{\L^\infty_{t}\Bdot_{\beta}^{\gamma-2\beta/p,p}} \lesssim_{\gamma,p} \|S\|_{\L^1_{t}\Bdot^{\gamma-2\beta/p,p}_{\beta}}.
\] 
\item
$ \cK_{\beta}^\pm S $ belongs to 
$\L^p_{t}\Xdot^{\gamma,p}_{\beta}$ with 
\[
\|  \cK_{\beta}^\pm S \|_{\L^p_{t}\Xdot^{\gamma,p}_{\beta}} \lesssim_{\gamma,p} \|S\|_{\L^1_{t}\Bdot^{\gamma-2\beta/p,p}_{\beta}}. \]
\end{itemize}
\end{cor}

\begin{proof}
We prove the first two items for $\cK_{\beta}^-$ as it is the same for $\cK_{\beta}^+$ 
starting from the backward initial value problem given by \eqref{eq:solKolIVback}.
  
Using the formula for  $g$, $S$ and $f$  as in Lemma~\ref{lem:energy},  the duality pairing for anisotropic  Sobolev spaces and Lemma~\ref{lem:boundsinitialvalueproblem},
\begin{align*}
	\bigg|\int_{\R^{2d}} g(x,v) (\cK_{\beta}^-S)(\tau,x,v)\, \dd(x,v) \bigg| 
	&\le 	\|f\|_{\L^{p'}_{t}\Xdot^{-(\gamma-2\beta),p'}_{\beta}} \|S\|_{\L^p_{t}\Xdot^{\gamma-2\beta,p}_{\beta}} \\
	&\lesssim \|g\|_{\Bdot_{\beta}^{-\gamma+{2\beta}/{p},p'}}\|S\|_{\L^p_{t}\Xdot^{\gamma-2\beta,p}_{\beta}}.
\end{align*}
As $g$ describes a dense class of $\Bdot_{\beta}^{-\gamma+{2\beta}/{p},p'}$, we have obtained
\[	
	\|  (\cK_{\beta}^- S)(\tau)\|_{\Bdot_{\beta}^{\gamma-{2\beta}/{p},p}} \lesssim   \|S\|_{\L^p_{t}\Xdot^{\gamma-2\beta,p}_{\beta}}
\]
uniformly in $\tau$.

For the continuity and limits, writing $S_{\tau}(t,x,v)=S(t+\tau,x,v)$, we have similarly for $\tau<\tau'$,
\[
	\| (\cK_{\beta}^- S)(\tau)-(\cK_{\beta}^- S)(\tau')\|_{\Bdot_{\beta}^{\gamma-{2\beta}/{p},p}} \lesssim   \|S_{\tau}-S_{\tau'}\|_{\L^p_{t}\Xdot^{\gamma-2\beta,p}_{\beta}}
\]
and as $S$ is $\C^\infty$ with compact support in time valued in $\Xdot^{\gamma-2\beta,p}_{\beta}$, the conclusion follows.  

For the second item, using again Lemma~\ref{lem:boundsinitialvalueproblem}, replace the first inequality with 
\begin{align*}
	\bigg|\int_{\R^{2d}} g(x,v) (\cK_{\beta}^-S)(\tau,x,v)\, \dd(x,v) \bigg| &\le 
	\|f\|_{\vphantom{\L^1_{t}}\L^{\infty}_{t}\Bdot^{-(\gamma-2\beta/p),p'}_{\beta}} \|S\|_{\L^1_{t}\Bdot^{\gamma-2\beta/p,p}_{\beta}} \\
	&\lesssim \|g\|_{\Bdot_{\beta}^{-(\gamma-2\beta/p),p'}}\|S\|_{\L^1_{t}\Bdot^{\gamma-2\beta/p,p}_{\beta}}
\end{align*}
and argue similarly. 
We skip further details.

For the third item, we use that $\cK^+_{\beta}$ is the dual operator to $\cK_{\beta}^-$ in the extended $\L^2_{t,x,v}$ duality.
 So the estimate follows from the first item, on changing $(\gamma,p)$ to $(-(\gamma-2\beta),p')$.
\end{proof}


\section{Hardy--Littlewood--Sobolev estimates for Kolmogorov operators}
\label{sec:Lebesgueestimates}

Collecting the results of the previous sections, we have obtained the following estimates for the Kolmogorov operators and semigroups. 

\begin{thm}[Bounds for the Kolmogorov operators and semigroups]\label{thm:bounds} 
Let $\beta\in (0,1]$, $\gamma\in \R$, $p\in(1,\infty)$. 
\begin{itemize}
\item $\cK^{\pm}_{\beta}$ extend to bounded operators from  
$\Zdot^{\gamma,p}_{\beta}$ to $\Ydot^{\gamma,p}_{\beta}$.
\item If $0\le \gamma\le 2\beta$, then  $\cK^{\pm}_{\beta}$ extend to bounded operators from either $\L^p_{t,x} \Hdot^{\gamma -2\beta,p}_{\vphantom{t,x} v}$ or $\L^p_{t,v}\Hdot_{\vphantom{t,v} x}^{\frac{\gamma-2\beta}{2\beta+1},p}$ to both $\L^p_{t,x} \Hdot^{\gamma ,p}_{\vphantom{t,x} v}$ and  $\L^p_{t,v}\Hdot_{\vphantom{t,v} x}^{\frac{\gamma}{2\beta+1},p}$.
\item The semigroups $\sem^\pm$ extend to bounded operators from $ \Bdot^{\gamma-2\beta/p,p}_{\beta}$ to $\Ydot^{\gamma,p}_{\beta, \pm}$ where the $\pm $ index indicates whether we work with positive times or negative times. 
\item  If $0\le \gamma\le 2\beta$, $\sem^\pm$ extend to bounded operators from $ \Bdot^{\gamma-2\beta/p,p}_{\beta}$ to both $\L^p_{t,x} \Hdot^{\gamma ,p}_{\vphantom{t,x} v}$ and  $\L^p_{t,v}\Hdot_{\vphantom{t,v} x}^{\frac{\gamma}{2\beta+1},p}$.
\end{itemize}
\end{thm}

\begin{proof} 
For the first item, the a priori bounds
\[
	\|\cK^{\pm}_{\beta}S\|_{\L^p_{t} \Xdot^{\gamma,p}_{\beta}}
	\lesssim \|S\|_{\L^p_{t} \Xdot^{\gamma-2\beta,p}_{\beta}},
\]
for  $S\in \cS_{K}$ exactly follow from  the $\L^p_{t,x,v}$ boundedness of the operator $T_{\gamma}$ established in 
Theorem~\ref{thm:boundednessTgamma} and the three other ones involving the Besov bounds are in Corollary~\ref{cor:besovbounds}. 
Then, the extension is obtained by the  density Lemma~\ref{lem:density}.
Note that on the left-hand side we have to use that Cauchy sequences  for the $ \Ydot^{\gamma,p}_{\beta}$ semi-norm of elements in 
$\L^2_{t,x,v}\cap  \Ydot^{\gamma,p}_{\beta}$  converge (for that semi-norm) to elements in $\Ydot^{\gamma,p}_{\beta}$ by completeness. 

The second item follows directly from the first and the characterisation in Proposition~\ref{prop:characterisation} as $\gamma\ge 0$ and $\gamma-2\beta\le 0$.  

The third item for $\sem^+$ is a consequence of Lemma~\ref{lem:boundsinitialvalueproblem}. The last one  follows directly from the third one and the characterisation in Proposition~\ref{prop:characterisation} as $\gamma\ge 0$ and $\gamma-2\beta\le 0$. The proof is the same for $\sem^-$.
\end{proof}

We obtain the following kinetic Hardy--Littlewood--Sobolev estimates as a consequence. Given any positive number with $a < \frac{\homd}{\beta}$ we set $a^* = \frac{a \homd}{\homd - a\beta }$ the kinetic Sobolev exponent of $a$. If $a< \frac{\homd}{2\beta}$, then $a^*< \frac{\homd}{\beta}$  and we set $a^{**} = (a^*)^*$.

\begin{thm}  \label{thm:HLS} 
 Let  $\beta\in (0,1]$.
 \begin{enumerate}
\item If $1<a< \frac{\homd}{2\beta}$, then
$$\cK_{\beta}^\pm: \L^{a}_{t,x,v}\to  \L^{a^{**}}_{t,x,v}.
$$
\item If $1<b<\infty$, then 
$$
\sem^\pm: \L^{b}_{x,v} \to \L^{\frac{b \homd}{\homd - 2\beta}}_{t,x,v}.
$$
\end{enumerate}
\end{thm}

Again, the restriction to $\beta\le 1$ (it works also for all integers) is only due to availability of kernel estimates in Section~\ref{sec:estimates}. Recall that for $\beta=1$ we have $\frac{\homd}{2\beta}=2d+1$.

To illustrate this result, we obtain the following concrete corollary, studied more often in the literature via direct estimates on the fundamental solution and kinetic convolution. 

\begin{cor}
\label{cor:Lebesguesestimates}
	Let  $\beta\in (0,1]$. Let $1<a< \frac{\homd}{2\beta}$,  $\psi \in \L^{a^\flat}_{x,v}$ with $a^\flat= \frac{a^{**}(\homd-2\beta)}{\homd}$, and $S \in \L^a_{t,x,v}$. Set $f(t)=\sem^+_{t}\psi+[\cK_{\beta}^+S](t)$, $t>0$. 
	  Then $f \in \L^{a^{**}}_{t,x,v}$   with 
\[
	\|f\|_{\L^{\vphantom{a^\flat}a^{**}}_{t,x,v}} \lesssim \| S \|_{\L^{\vphantom{a^\flat}a}_{t,x,v}}+ \|\psi\|_{\L^{a^\flat}_{x,v}}. 
\]
\end{cor}

\begin{rem} 
If we want $a^\flat=2$, then $a^{**}=\frac{2\homd}{\homd -2\beta}=2^*$ and $a=\frac{2\homd}{\homd + 2\beta}=:2_{*}$, also the conjugate exponent to $2^*$, so the formula specializes to
\[
	\|f\|_{\L^{2^*}_{t,x,v}} \lesssim \| S \|_{\L^{2_{*}}_{t,x,v}}+ \|\psi\|_{\L^{2}_{x,v}}. 
\] 
Note that for the term in $S$, the inequality is a consequence of the reverse embedding $\L^{2_{*}}_{t,x,v} \hookrightarrow \Zdot^{\beta,2}_{\beta} $ proved in Lemma~\ref{lem:embedLinZ}, corresponding to the  Sobolev embedding $\Ydot^{\beta,2}_{\beta}\hookrightarrow \L^{2^*}_{t,x,v}$ proved in Lemma~\ref{lem:Sobolev2} below. If we decompose $S=S_{1}+S_{2}+S_{3}$ from the definition, one has the stronger inequalites
\[
	\|f\|_{\L^{2^*}_{t,x,v}} \lesssim \|f\|_{\Ydot^{\beta,2}_{\beta}}  \lesssim \| S_{1} \|_{\L^2_{t,x} \Hdot^{-\beta , 2}_{\vphantom{t,x} v}}+  \|S_{2}\|_{\L^2_{t,v}\Hdot_{\vphantom{t,v} x}^{-\frac{\beta }{2\beta+1},2}}+ \|S_{3} \|_{\L^1_{t}\L^2_{x,v}}+ \|\psi\|_{\L^{2}_{x,v}}.
	\]
\end{rem}

To prove the Hardy--Littlewood--Sobolev inequalities we need some further embeddings, which are of interest on their own. For the following mixed anisotropic spaces, the regularity condition $\gamma<\homd/p$ for the parameters $\gamma,p$ appears naturally and is likely optimal by  the underlying scaling.

\begin{lem} \label{lem:Sobolev2}
Let $p \in (1,\infty)$ and $\beta >0$. Fix $\gamma\in(0,\frac\homd p)$ and define $\kappa:=\frac{\homd}{\homd-\gamma p}$.
We have the continuous embedding
\[
	\L^{p}_t\Xdot^{\gamma,p}_\beta \cap \L^\infty_t\Bdot^{\gamma-2\beta/p,p}_\beta
	\ \hookrightarrow\ \L^{p \kappa }_{t,x,v},
\]
with the estimate
\[
	\|f\|_{\L^{p\kappa }_{t,x,v}}
	\;\lesssim\;
	\|f\|_{\L^{p}_t\Xdot^{\gamma,p}_\beta}^{1/\kappa}\,
	\|f\|_{\L^\infty_t\Bdot^{\gamma-2\beta/p,p}_\beta}^{1-1/\kappa}.
\]
\end{lem}

\begin{proof}
We use the anisotropic Littlewood--Paley decomposition. 
In the range $\gamma<\homd/p$, the second space in the intersection is continuously embedded into $\cS'(\Omega)$ and the anisotropic 
Littlewood--Paley decomposition converges in $\cS'(\Omega)$ on it (the variable $t$ is not involved in the decomposition).
Hence, it suffices to obtain the inequality for norms.
Let $f\in  \L^{p}_t\Xdot^{\gamma,p}_\beta \cap \L^\infty_t\Bdot^{\gamma-2\beta/p,p}_\beta$ and set
$$
	F_{1}= \bigg(\sum_{j=-\infty}^\infty |2^{j\gamma} \theta_{j}\ast f|^2\bigg)^{1/2}, \quad 
	F_{2}= \sup_{j\in\Z}|2^{j(\gamma-\homd/p)} \theta_{j}\ast f|, \quad  
	F_{3}=\bigg(\sum_{j=-\infty}^\infty | \theta_{j}\ast f|^2\bigg)^{1/2}.
$$
We want to control $\|F_{3}\|_{\L^r_{t,x,v}}$, where $r=p\kappa$ (or, equivalently, $\frac 1 r= \frac 1 p - \frac {\gamma}{ \homd}$).  First, by definition, $F_{1}\in \L^p_{t,x,v}$ with 
$$
	\|F_{1}\|_{\vphantom{\Xdot}\L^p_{t,x,v}} \sim \|f\|_{\L^p_{t}\Xdot^{\gamma,p}_{\beta}}.
$$ 
Second, by  Bernstein inequalities in the anisotropic scaling, 
$$
	| \theta_{j}\ast f(t,\cdot) | \lesssim 2^{j(2\beta+2)d/p}\,   \|\theta_{j}\ast f(t,\cdot)\|_{\L^p_{x,v}}.
$$ 
Hence by the relation $\gamma-\homd/p+(2\beta+2)d/p=\gamma-2\beta/p$ and the inclusion of $\ell^p(\Z)\subset \ell^\infty(\Z)$, 
$F_{2} \in \L^\infty_{t,x,v}$ with 
$$
	\|F_{2}\|_{\L^\infty_{t,x,v}} \lesssim  \|f\|_{\L^\infty_{ t}\Bdot^{\gamma-2\beta/p,p}_{\beta}}.
$$
Now, for any $j_{0}\in \Z$,
\begin{align*}
	F_{3}
	&\lesssim 2^{-j_{0}\gamma} \bigg(\sum_{j=j_{0}}^\infty |2^{j\gamma} \theta_{j}\ast f|^2\bigg)^{1/2} + \bigg(\sum_{j=-\infty}^{j_{0}-1} | \theta_{j}\ast f|^2\bigg)^{1/2} \\
	&\lesssim 2^{-j_{0}\gamma} F_{1}+ 2^{j_{0}(\homd/p-\gamma)}F_{2}.
\end{align*}
So optimizing in $j_{0}$ and using the relations between $p$ and $r$, we have
\[
	F_{3}\lesssim F_{1}^{p/r}F_{2}^{1-p/r} \le F_{1}^{p/r}\|F_{2}\|_{\L^\infty_{t,x,v}}^{1-p/r}
\]
and by integration in $(t,x,v)$ we obtain 
$$
	\|F_{3}\|^{\vphantom{p/r}}_{\L^r_{t,x,v}} \lesssim \|F_{1}\|_{\L^p_{t,x,v}}^{p/r}\|F_{2}\|_{\L^\infty_{t,x,v}}^{1-p/r}.
$$ 
This concludes the proof. 
\end{proof}

Using duality we may deduce from that a reverse embedding for $\Zdot^{\gamma,p}_\beta$.

\begin{lem}
\label{lem:embedLinZ}
Let $p \in (1,\infty)$, $\beta>0$, and $\gamma\in\R$. Assume $0<2\beta-\gamma<\frac{\homd}{p'}$. Define $q\in(1,\infty)$ by
\begin{equation}\label{eq:B}
\frac1q=\frac1p+\frac{2\beta-\gamma}{\homd}
\qquad\Big(\text{equivalently, } q=\frac{p\,\homd}{\homd+p(2\beta-\gamma)}\Big).
\end{equation}
Then there is the continuous embedding $\L^q_{t,x,v}\ \hookrightarrow\ \Zdot^{\gamma,p}_\beta$ with  the norm estimate
\begin{equation}\label{eq:SchwartzEstimate}
\|S\|_{\Zdot^{\gamma,p}_\beta}\ \lesssim\ \|S\|_{\L^q_{t,x,v}}.
\end{equation}
\end{lem}

\begin{proof} Let us begin by noting that the conjugate exponent to $q$ in  \eqref{eq:B}  is given by
 
\[
	 q'=\frac{p'\homd}{\homd-(2\beta-\gamma) p'}
	\quad\text{as}\quad
	\frac1{q'}=\frac1{p'}-\frac{2\beta-\gamma}{\homd}.
\]
Note that, arguing in $\cS'(\R^{2d})$ modulo polynomials in which all spaces embed as Banach spaces,  $(\Xdot^{s,p}_\beta)' = \Xdot^{-s,p'}_\beta$ and $(\Bdot^{s,p}_\beta)' = \Bdot^{-s,p'}_\beta$ and  
\begin{equation}\label{eq:dualZ}
(\Zdot^{\gamma,p}_\beta)'
=
\L^{p'}_t\Xdot^{2\beta-\gamma,p'}_\beta \cap \L^\infty_t\Bdot^{-\gamma+2\beta/p,p'}_\beta,
\end{equation}
with the norm given by the maximum of the two component norms. 
\medskip

Applying Lemma \ref{lem:Sobolev2} with $\tilde{p} = p'$ and $\tilde{\gamma} = 2\beta-\gamma$ (note $-\gamma+2\beta/p = 2\beta-\gamma -2\beta/p'$ and $0<2\beta-\gamma < \homd /p'$), we deduce $(\Zdot^{\gamma,p}_\beta)' \hookrightarrow \L_{t,x,v}^{q'}$ with 
 the inequality
\begin{equation} \label{eq:embeddingH}
	\|f\|_{\L^{q'}_{t,x,v}}\lesssim \|f\|_{(\Zdot^{\gamma,p}_\beta)'}.
\end{equation} 

Let $S\in \cS_{K}\subset \Zdot^{\gamma,p}_\beta$. By the Hahn-Banach theorem,  
\begin{equation}\label{eq:normFormula}
\|S\|_{\Zdot^{\gamma,p}_\beta}
=
\sup\Big\{\,|\langle f,S\rangle| \;:\; f\in (\Zdot^{\gamma,p}_\beta)',\ \|f\|_{(\Zdot^{\gamma,p}_\beta)'}\le 1\Big\}.
\end{equation}
 By \eqref{eq:embeddingH}, every such $f$ satisfies $\|f\|_{\L^{q'}_{t,x,v}}\lesssim 1$. Therefore H\"older's inequality gives us
\[
|\langle f,S\rangle|
=\Big|\int_{\R^{1+2d}} f(t,x,v)\,S(t,x,v) \dd (t,x,v) \Big|
\le \|S\|_{\L^{\vphantom{q'}q}_{t,x,v}}\,\|f\|_{\L^{q'}_{t,x,v}}
\lesssim \|S\|_{\L^{\vphantom{q'}q}_{t,x,v}}.
\]
Taking the supremum over all $f$ as in \eqref{eq:normFormula},
we obtain \eqref{eq:SchwartzEstimate} for all  $S\in \cS_{K}$ and we conclude by density since $\cS_{K}$ is dense in $\L^q_{t,x,v}$  by Lemma~\ref{lem:density}.   \end{proof}

We further provide the Sobolev embedding for anisotropic Besov spaces: the threshold is $(2\beta+2)d/p=(\homd -2\beta)/p$.

\begin{lem}[Sobolev embedding for homogeneous anisotropic Besov spaces]\label{lem:besov_sobolev_embedding}
Let $\beta \in (0,\infty)$, $p \in (1,\infty)$, and assume $0<\gamma<\frac{\homd-2\beta}{p}$.
Define $q\in(p,\infty)$ by
\[
	\frac1q=\frac1p-\frac{\gamma}{\homd-2\beta}.
\]
Then there is the continuous embedding $\Bdot^{\gamma,p}_\beta \hookrightarrow \L^q_{x,v}$,
and the estimate
\[
	\|g\|_{\L^q_{x,v}}\ \lesssim\ \|g\|_{\Bdot^{\gamma,p}_\beta}
	\qquad\text{for all }g\in \Bdot^{\gamma,p}_\beta.
\]
\end{lem}

\begin{proof}
Fix $g\in \cC$, dense in $\Bdot^{\gamma,p}_\beta$. Define the functions on $\R^{2d}$,
\[
	F_{1}:=\Big(\sum_{j\in\Z}\big|2^{j\gamma}\,\theta_j\ast g \big|^p\Big)^{1/p},\quad
	F_{2}:=\sup_{j\in\Z}\big|2^{j(\gamma-(\homd-2\beta)/p)}\,\theta_j\ast g\big|,
\quad 	F_{3}:=\Big(\sum_{j\in\Z}\big|\theta_j\ast g\big|^2\Big)^{1/2}.
\]
By the Littlewood--Paley square-function estimate (valid for $1<q<\infty$), $\|g\|_{\L^q_{x,v}}\ \simeq \ \|F_3\|_{\L^q_{x,v}}$, so it suffices to prove $\|F_3\|_{\L^q_{x,v}} \lesssim \|g\|_{\Bdot^{\gamma,p}_\beta}$ and we conclude by density. 

First, 
by Fubini's theorem, 
\[
	\|F_1\|_{\L^p_{x,v}}^p
	=\int_{\R^{2d}}\sum_{j\in\Z}\big|2^{j\gamma}\theta_j\ast g(x,v)\big|^p\dd (x,v)
	=\sum_{j\in\Z}2^{j\gamma p}\|\theta_j\ast g\|_{\L^p_{x,v}}^p
	=\|g\|_{\Bdot^{\gamma,p}_\beta}^p.
\]
Hence $\|F_1\|_{\L^p_{x,v}}=\|g\|_{\Bdot^{\gamma,p}_\beta}$.

Next, by the anisotropic Bernstein inequality,
\[
\|\theta_j\ast g\|_{\L^\infty_{x,v}}\ \lesssim\ 2^{j(\homd-2\beta)/p}\,\|\theta_j\ast g\|_{\L^p_{x,v}}.
\]
Therefore
\begin{align*}
	\|F_2\|_{\L^\infty_{x,v}}
	&=\|\sup_{j}2^{j(\gamma-(\homd-2\beta)/p)}|\theta_j\ast g |\|_{\L^\infty_{x,v}}
	\le \sup_{j}2^{j(\gamma-(\homd-2\beta)/p)}\|\theta_j\ast g\|_{\L^\infty_{x,v}} \\
	&\lesssim \sup_{j}2^{j\gamma}\|\theta_j\ast g\|_{\L^p_{x,v}}
	\lesssim \Big(\sum_{j\in\Z}2^{j\gamma p}\|\theta_j\ast g\|_{\L^p_{x,v}}^p\Big)^{1/p}
	=\|g\|_{\Bdot^{\gamma,p}_\beta},
\end{align*}
where we used $\ell^p(\Z)\subset \ell^\infty(\Z)$ in the last inequality.

Now, for any $j_0\in\Z$,
\begin{align*}
 F_3
&\le \Big(\sum_{j\ge j_0}|\theta_j\ast g|^2\Big)^{1/2}
+\Big(\sum_{j<j_0}|\theta_j\ast g|^2\Big)^{1/2}
\\
&\lesssim 2^{-j_0\gamma}F_1+ \Big(\sum_{j< j_0}2^{2j((\homd-2\beta)/p-\gamma)}\Big)^{1/2}F_2
\\
& \lesssim 2^{-j_0\gamma}F_1 + 2^{j_{0}((\homd-2\beta)/p-\gamma)}F_2
\end{align*}
since $0<\gamma<(\homd-2\beta)/p$. 
Optimizing $j_0\in \Z$ with 
$
2^{-j_0\gamma}F_1 \sim 2^{j_0((\homd-2\beta)/p-\gamma)}F_2$
yields
\[
F_3\ \lesssim\ F_1^{\,1-\gamma p/(\homd-2\beta)}\,F_2^{\,\gamma p/(\homd-2\beta)} \le F_1^{\,1-\gamma p/(\homd-2\beta)}\,\|F_2\|_{\L^\infty_{x,v}}^{\,\gamma p/(\homd-2\beta)}.
\]
Using the definition of $q$ in the statement, this is the same as 
\[
F_3\ \lesssim\ F_1^{\,p/q}\,\|F_2\|_{\L^\infty_{x,v}}^{\,1-p/q}.
\]
Thus, 
\[
\|F_3\|^{\vphantom{p/q}{}}_{\L^q_{x,v}}
\lesssim \|F_1\|_{\L^p_{x,v}}^{p/q}\,\|F_2\|_{\L^\infty_{x,v}}^{1-p/q}
\lesssim \|g\|_{\Bdot^{\gamma,p}_\beta}^{p/q}\,\|g\|_{\Bdot^{\gamma,p}_\beta}^{1-p/q}
\lesssim \|g\|^{\vphantom{p/q}{}}_{\Bdot^{\gamma,p}_\beta},
\]
which proves the embedding.
\end{proof}

\begin{cor}[Reverse Sobolev embedding for homogeneous anisotropic Besov spaces]\label{cor:Lq_into_Besov}
Let $\beta \in (0,\infty)$, $p \in (1,\infty)$ and assume $0<\gamma<\frac{\homd-2\beta}{p}.$ 
Set $q \in(p,\infty)$ as 
\[
	\frac1q=\frac1p-\frac{\gamma}{\homd-2\beta}.
\]
 Then, there is the continuous embedding $\L^{q'}_{x,v} \hookrightarrow \Bdot^{-\gamma,p'}_\beta$, and moreover
\begin{equation}
\label{eq:reverseembedBesov}
\|h\|_{\Bdot^{-\gamma,p'}_\beta}\ \lesssim\ \|h\|_{\vphantom{\Bdot}\L^{q'}_{\vphantom{\beta}x,v}}
	\qquad\text{for all }h\in \L^{q'}_{x,v}.
\end{equation}
\end{cor}

\begin{proof}
Fix $h\in \L^{q'}_{x,v}$. For any $g\in  \cC$, see Lemma~\ref{lem:denseanisotropic},
H\"older's inequality and Lemma~\ref{lem:besov_sobolev_embedding} yield
\[
\big|\langle h,g\rangle\big|
\le \|h\|_{\vphantom{\Bdot}\L^{q'}_{\vphantom{\beta}x,v}}\,\|g\|_{\vphantom{\Bdot}\L^{q}_{\vphantom{\beta}x,v}}
\lesssim \|h\|_{\vphantom{\Bdot}\L^{q'}_{\vphantom{\beta}x,v}}\,\|g\|_{\Bdot^{\gamma,p}_\beta}.
\]
As $\cC$ is dense in $\Bdot^{\gamma,p}_\beta$ (see Section~\ref{sec:HomAnBesov}),  $h$ defines a bounded linear functional on $\Bdot^{\gamma,p}_\beta$ and
\[
\|h\|_{(\Bdot^{\gamma,p}_\beta)'} \lesssim \|h\|_{\vphantom{\Bdot}\L^{q'}_{\vphantom{\beta}x,v}}.
\]

By duality for homogeneous Besov spaces we have $(\Bdot^{\gamma,p}_\beta)'= \Bdot^{-\gamma,p'}_\beta$.
Therefore $h\in \Bdot^{-\gamma,p'}_\beta$ and \eqref{eq:reverseembedBesov} is proved. 
\end{proof}

We can now establish the announced Hardy--Littlewood--Sobolev inequalities for the Kolmogorov operators and semigroups.

\begin{proof}[Proof of Theorem~\ref{thm:HLS}]
 For (i), the conditions  
 $$
 p\in (1,\infty) \quad \text{and}\quad
 \max \left\{ 0,2\beta-\homd+ \frac{\homd}{p} \right\}< \gamma < \min\left\{{2\beta},\frac{\homd}{p} \right\}
 $$ 
 allow us to use successively Lemma~\ref{lem:embedLinZ}, the first item of Theorem \ref{thm:bounds} and Lemma~\ref{lem:Sobolev2}. We obtain
 \[
  \L^{\frac{p \homd}{\homd + p (2\beta-\gamma)}}_{t,x,v}\hookrightarrow  \Zdot^{\gamma,p}_\beta \xrightarrow{\cK_{\beta}^\pm} \Ydot^{\gamma,p}_{\beta} \hookrightarrow   \L^{\frac{p \homd}{\homd - p \gamma}}_{t,x,v}.
 \]
 Now we observe that 
 $$\frac{p \homd}{\homd + p (2\beta-\gamma)}=  \frac{ \homd}{\frac{\homd}{p}- \gamma + 2\beta}
 $$ takes any value in the interval $(1, \frac{\homd}{2\beta})$ for all $(\gamma,p)$ as above. If $1<a< \frac{\homd}{2\beta}$, then setting $a=\frac{p \homd}{\homd + p (2\beta-\gamma)}$ for some $\gamma,p$ as above, a direct computation yields
 $$
 	\frac{p \homd}{\homd - p \gamma}= \frac{a \homd}{\homd - 2\beta a}
 $$
 and we are done. 
 
 For (ii), we begin with $\sem^+$ and work for $t>0$.  The conditions
 $$
 	1<p<\infty \quad \text{and}\quad \max \left\{ 0, 2\beta-\homd +\frac{\homd}{p} \right\} < \gamma < \frac{2\beta}{p}
 $$
 allow us to use successively  Corollary \ref{cor:Lq_into_Besov},  the third item of Proposition~\ref{prop:mapping} and Lemma~\ref{lem:Sobolev2}. We obtain
 \[
 \L^{\frac{p (\homd-2\beta)}{\homd - p \gamma}}_{x,v} \hookrightarrow  \Bdot^{\gamma-2\beta/p,p}_\beta \xrightarrow{\sem^+} \Ydot^{\gamma,p}_{\beta} \hookrightarrow   \L^{\frac{p \homd}{\homd - p \gamma}}_{t,x,v}.
 \]
 Now, we observe that 
 $$\frac{p (\homd-2\beta)}{\homd - p \gamma}= \frac{ \homd-2\beta}{\frac \homd p -  \gamma}$$ takes any value in the interval $(1,\infty)$ for $(\gamma,p)$ with the conditions above. If $1<b<\infty$, then setting $b= \frac{p (\homd-2\beta)}{\homd - p \gamma}$ with some $(\gamma,p)$ as above, we have $$\frac{p \homd}{\homd - p \gamma}= \frac{b \homd}{\homd - 2\beta}$$
 and the conclusion follows. 
 
The proof for the  backward Kolmogorov semigroup $\sem^-$  is the same. 
\end{proof}

\section{Isomorphism property of the Kolmogorov operators}
\label{sec:iso}
So far, we have explored the boundedness properties of the Kolmogorov operators and have obtained a complete picture between the scale of $\Zdot$ and $\Ydot$ spaces.  This was summarised up in Theorem~\ref{thm:bounds}. 
We now go back to the essence of these operators seen as fundamental solutions for the Kolmogorov equations. 
For which exponents $\gamma,p$ can we assert that they map into solutions? 
It is precisely for this purpose that we introduced the kinetic Sobolev spaces.

We shall show the following statement.
Recall that $\homd$ denotes the homogeneous dimension for the kinetic scaling: $$\homd=2\beta+(2\beta+2)d.$$ 
The main result of this section is the following. 
\begin{thm}[Isomorphism] \label{lem:isom-embed}
Let $\beta\in (0,1]$,    $p \in (1,\infty)$ and  $\gamma\in [0,2\beta]$ with $\gamma<\homd/p$.  
Then, the operators $\pm (\partial_{t}+v\cdot\nabla_{x}) +(-\Delta_{v})^\beta$ are isomorphisms from  
\begin{itemize}
\item  $\cLdot^{\gamma,p}_{\beta}$ onto $\Zdot^{\gamma,p}_{\beta}$,
\item $\cGdot^{\gamma,p}_{\beta}$ onto $\L^p_{t} \Xdot^{\gamma -2\beta,p}_{\beta}$, 
\item $\cFdot^{\gamma,p}_{\beta} $ onto $\L^p_{t,x} \Hdot^{\gamma -2\beta,p}_{\vphantom{t,x} v}$. 
\end{itemize}
The  inverses are respectively given by the Kolmogorov operators $\cK^\pm_\beta$ defined by the fundamental solutions.  
In particular we have the maximal regularity bounds
\begin{align*}
	\|\pm(\partial_{t}+v\cdot\nabla_{x})f+ (-\Delta_{v})^\beta f \|_{\Adot^{\gamma,p}_{\beta}} 
	&\sim \|f\|_{\L^p_{t,x} \Hdot^{\gamma,p}_{\vphantom{t,x} v}} + \|(\partial_{t}+v\cdot\nabla_{x})f\|_{\Adot^{\gamma,p}_{\beta}}\\
	\|S\|_{\Adot^{\gamma,p}_{\beta}}   \sim \|\cK^\pm_\beta S\|_{\L^p_{t,x} \Hdot^{\gamma,p}_{\vphantom{t,x} v}}
	& + \|(\partial_{t}+v\cdot\nabla_{x})\cK^\pm_\beta S\|_{\Adot^{\gamma,p}_{\beta}},
\end{align*}
 where $\Adot^{\gamma,p}_{\beta}$ denotes any of the three target spaces of the isomorphisms. 
 \end{thm}

\begin{proof}
The ingredients for the proof of this theorem are supplied in later parts of this section.
In the range for $\gamma$ prescribed in the statement,  Corollary~\ref{cor:ontoness} ensures boundedness and ontoness 
with the Kolmogorov operators as bounded right inverses and Lemma~\ref{lem:uniquenessH} yields  injectivity. 
\end{proof}

\begin{rem}
We always have $\beta<\homd/2 = \beta+(\beta+1)d$, hence the statement applies in full when $\gamma=\beta$ and $p=2$ with our definition of $\Hdot^{\beta,2}_{v}$.  
This removes the assumption $f\in \Udot^\beta$ when $\beta\ge d/2$ made in \cite[Theorem~1.16]{AIN} as only $f\in \L^2_{t,x} \Hdot^{\beta,2}_{\vphantom{t,x} v}$ is required.  
\end{rem}
 
\begin{rem}  
Let $\Wdot^{1,p}_{v}$, $p \in (1,\infty)$, be the space of distributions $f$ with $\nabla_{v} f\in \L^p_{v}$ and 
$\Wdot^{-1,p'}_{v}$ be its dual space. 
In the case $\gamma=\beta=1$ and $p=2$, \cite[Proposition~8.4]{AIN} establishes several isomorphisms. 
Let us look at the ones  saying that $\cK^\pm_{1}$ are isomorphisms from $\L^2_{t,x} \Wdot^{-1,2}_{\vphantom{t,x} v}$ onto 
\begin{align*}
	\cFdot= \{ f \in \cD'(\Omega)\, ; \,  f\in  \L^2_{t,x}\Wdot^{1,2}_{\vphantom{t,x} v} \ \& \ (\partial_{ t} + v \cdot \nabla_x)f \in \L^2_{t,x} \Wdot^{-1,2}_{\vphantom{t,x} v}\}
\end{align*}
having semi-norm defined by
\begin{align*}
	\|f\|_{\cFdot}= \| \nabla_{v} f\|_{\L^2_{t,x,v}} + \|(\partial_t + v \cdot \nabla_x)f\|_{\L^2_{t,x} \Wdot^{-1,2}_{\vphantom{t,x} v}}.
\end{align*}
Since $\Hdot^{-1,2}_{v}=\Wdot^{-1,2}_{v}$, it follows that $\cFdot^{1,2}_{1}=\cFdot$. 
The other two isomorphisms lead to similar conclusion for $\cGdot^{1,2}_{1}$ and $\cLdot^{1,2}_{1}$.
We can do the same thing with $p\ne 2$ to obtain 
\begin{align*}  
	\cFdot^{1,p}_{1}= \{ f \in \cD'(\Omega)\, ; \,  f\in  \L^p_{t,x}\Wdot^{1,p}_{\vphantom{t,x} v} \ 
	\& \ (\partial_{ t} + v \cdot \nabla_x)f \in \L^p_{t,x} \Wdot^{-1,p}_{\vphantom{t,x} v}\}
\end{align*}
and similarly for $\cGdot^{1,p}_{1}$ and $\cLdot^{1,p}_{1}$.
\end{rem}

\subsection{Consequences of Theorem~\ref{lem:isom-embed}}

\begin{cor}[Completeness, dense class and interpolation] 
Let $\beta\in (0,1]$, $p \in (1,\infty)$ and $\gamma\in [0,2\beta]$ with $\gamma<\homd/p$.
The spaces $\cFdot^{\gamma,p}_{\beta}$,  $\cGdot^{\gamma,p}_{\beta}$ and $\cLdot^{\gamma,p}_{\beta}$ are complete Banach spaces 
of tempered distributions containing the images of $\cS_{K}$ under $\cK_{\beta}^\pm$ as dense subspaces.  
The  $\cF$ and  $\cG$ spaces  interpolate by the complex method along each segment drawn in the convex region 
$(\gamma, 1/p) \in [0,2\beta] \times (0,1)$ defined by $\gamma<\homd/p$.
\end{cor}

\begin{proof}
This is clear because, in the respective prescribed ranges for the exponents $(\gamma,1/p)$, 
the  spaces $\L^p_{t,x} \Hdot^{\gamma -2\beta,p}_{\vphantom{t,x} v}$,  $\L^p_{t} \Xdot^{\gamma -2\beta,p}_{\beta}$ 
and $\Zdot^{\gamma,p}_\beta$  are Banach spaces, containing $\cS_{K}$ as a dense class,  
and the first two families complex interpolate.
Theorem \ref{lem:isom-embed} yields the claim.
\end{proof}

\begin{rem}
The issue of dense classes in the homogeneous kinetic Sobolev spaces above is nontrivial. 
It is not clear from their definition that spaces such as their intersection with $\cD(\Omega)$ are dense in any of the above spaces. Tricks such as regularisation by kinetic convolution (see the proof of Theorem~\ref{thm:inhomkinspaceLp} below) would likely require further information such as the embeddings in Corollary~\ref{cor:embedding} below.  Lemma~\ref{lem:equation} provides us with the properties of elements in the class appearing in the above statement, enough to justify basic calculations that one may need to perform. We mention that applying \cite[Lemma 2.14]{AIN} provides us with a stronger density result when $d\ge 2$: in particular, the intersection of $\cS(\Omega)$ with those spaces is dense. 
\end{rem}

\begin{rem}
The interpolation property of sums of mixed spaces with different integrability exponents is not known. This is why we have no conclusion   for the $\Zdot$-family. 
 \end{rem}

\begin{cor}[Embedding and transfer of regularity]\label{cor:embedding}  
Let $\beta\in (0,1]$, $p \in (1,\infty)$ and  $\gamma\in [0,2\beta]$ with $\gamma<\homd/p$.  
All elements of $\cLdot^{\gamma,p}_{\beta}$ belong to 
$\L^p_{t,v}\dot\H_{\vphantom{t,v} x}^{\frac{\gamma}{2\beta+1},p}\cap \C^{}_{0}(\R^{}_{t}\, ;\, \Bdot^{\gamma-2\beta/p,p}_{\beta})$ with
\begin{equation}\label{eq:embedding}
	\| D_{x}^{\frac{\gamma}{2\beta+1}}f\|_{\L^p_{t,x,v}}+ \sup_{t\in \R}\|f(t,\cdot)\|_{\Bdot^{\gamma-2\beta/p,p}_{\beta}}
	\lesssim_{\beta,\gamma,d,p}\| D_{v}^{\gamma}f\|_{\L^p_{t,x,v}}+ \|(\partial_t + v \cdot \nabla_x)f\|_{\Zdot^{\gamma,p}_\beta}.
\end{equation}
\end{cor}

\begin{proof}
 It follows from the isomorphisms in Theorem~\ref{lem:isom-embed} and also Theorem~\ref{thm:bounds} which was instrumental to prove the isomorphism property.
\end{proof}

\begin{cor}[Polarized energy equalities]\label{cor:polarizedenergyequalities} 
Let $\beta\in (0,1]$, $p\in(1,\infty)$, $\gamma,\tilde \gamma\in [0,2\beta]$ with $\gamma<\homd/p$ and 
$ \tilde \gamma<\homd/p'$ with $\gamma+\tilde \gamma=2\beta$.

Let $f, \tilde f\in \cD'(\Omega)$ be such that $f\in \cLdot^{\gamma,p}_{\beta}$ and $\tilde f\in \cLdot^{\tilde \gamma,p'}_{\beta}$. 
Then the map  $t\mapsto\angle{f(t)}{\tilde f(t)}$, where the bracket is for the duality between 
$\Bdot^{\gamma-2\beta/p,p}_{\beta}, \Bdot^{\tilde \gamma- 2\beta/p',p'}_{\beta}$,  
is absolutely continuous on $\R$ and its almost everywhere derivative can be computed as follows. 
  
Write 
$$
	(\partial_{t}+v\cdot\nabla_{x})f =S_{1}+S_{2}
$$
with $S_1 \in \L^p_{t}\Xdot^{\gamma-2\beta,p}_{\vphantom{t} \beta}$ and 
$S_2 \in \L^1_{\vphantom{t,x} t}\Bdot_{\vphantom{t,x} \beta}^{\gamma-2\beta/p,p}$, and 
$$
	(\partial_{t}+v\cdot\nabla_{x})\tilde f =\widetilde S_{1}+ \widetilde S_{2}
$$
with $\widetilde S_1 \in \L^{p'}_{t}\Xdot^{\tilde \gamma-2\beta,p'}_{\vphantom{t} \beta}$ 
and $\widetilde S_2 \in \L^1_{\vphantom{t,x} t}\Bdot_{\vphantom{t,x} \beta}^{\tilde\gamma-2\beta/p',p'}$.
Then almost everywhere on  $\R_{t}$, 
\begin{equation}\label{eq:derivative}
	\frac{\mathrm{d} }{\mathrm{d}t }\angle{f}{\tilde f} 
	=    \angle {S_{1}}{\tilde f} + \angle { f}{\widetilde S_{1}}     + \angle   {S_{2}}{\tilde f} + \angle   {f} {\widetilde S_{2}}
\end{equation}
where each bracket denotes a different sesquilinear duality extending the $\L^2_{x,v}$ inner product by tracing the spaces
involved, and is an integrable function of $t$.

Moreover, we may also decompose 
$ S_{1} =S_{1,1}+S_{1,2} $
with $S_{1,1} \in \L^p_{t,x}\Hdot^{\gamma-2\beta,p}_{\vphantom{t,x} v}$, 
$S_{1,2} \in \L^p_{t,v}\Hdot_{\vphantom{t,x} x}^{\frac{\gamma-2\beta}{2\beta+1},p}$ and 
$ \widetilde S_{1} =\widetilde S_{1,1}+ \widetilde S_{1,2} $
with $\widetilde S_{1,1} \in \L^{p'}_{t,x}\Hdot^{\tilde \gamma-2\beta,p'}_{\vphantom{t,x} v}$, 
$\widetilde S_{1,2} \in \L^{p'}_{t,v}\Hdot_{\vphantom{t,x} x}^{\frac{\tilde \gamma-2\beta}{2\beta+1},p'}$  
and the formula for the  derivative  becomes 
\begin{multline*}
	\frac{\mathrm{d} }{\mathrm{d}t }\angle{f}{\tilde f} 
	=    \int_{\R^d} (\angle {S_{1,1}}{\tilde f} + \angle { f}{\widetilde S_{1,1}}) \dx + \int_{\R^d} (\angle {S_{1,2}}{\tilde f}+  \angle { f}{\widetilde S_{1,2}}) \dv  + \angle   {S_{2}}{\tilde f} + \angle   {f} {\widetilde S_{2}}
\end{multline*}
with the same interpretation for the brackets being integrable in the missing variables. 
\end{cor}

\begin{proof} 
We can write $f=\cK_{\beta}^+S$ and $\tilde f= \cK_{\beta}^-\tilde S$ with $S,\tilde S\in \tilde  \cS_{K}$. 
In this case, the equality
\[
	\frac{\mathrm{d} }{\mathrm{d}t }\int_{\R^{2d}}{f(t,x,v})\overline {\tilde f(t,x,v)} \dd(x,v) = \int_{\R^{2d}} (S(t,x,v)\overline{\tilde f(t,x,v)} + f(t,x,v)\overline{\tilde S(t,x,v)}) \dd(x,v)
\]
is already in \cite{AIN}. 
Then, we use density  of $\cS_{K}$ in $\Zdot^{\gamma,p}_{\beta}$ and $\Zdot^{\tilde \gamma,p'}_{\beta}$, 
and boundedness of  $\cK_{\beta}^\pm $ to conclude with the proper interpretation of the dualities. 
 \end{proof}

\begin{cor}[Kinetic Sobolev embedding]\label{cor:Sobolevembeddings} 
Let $\beta\in (0,1]$, $p\in(1,\infty)$ and  $\gamma\in [0,2\beta]$ with $0<\gamma<\homd/p$. 
We have $\cLdot^{\gamma,p}_{\beta} \subset \L^{p\kappa}_{t,x,v}$ with 
\[
	\|f\|_{\L^{p\kappa}_{t,x,v}} \lesssim \| D_{v}^{\gamma}f\|_{\L^p_{t,x,v}}+ \|(\partial_t + v \cdot \nabla_x)f\|_{\Zdot^{\gamma,p}_\beta}
\]
where $\kappa= \frac{\homd}{\homd-\gamma p}.$
\end{cor}

\begin{proof} 
By the embedding in Corollary~\ref{cor:embedding}, $\cLdot^{\gamma,p}_{\beta} \subset \Ydot^{\gamma, p}_{\beta}$ 
with continuous inclusion. Hence, we may apply Lemma~\ref{lem:Sobolev2} as  $\Ydot^{\gamma, p}_{\beta} \subset \L^{p}_t\Xdot^{\gamma,p}_\beta \cap \L^\infty_t\Bdot^{\gamma-2\beta/p,p}_\beta$ in this range of exponents.
\end{proof}

\subsection{Relating the Kolmogorov operator and solutions of the Kolmogorov equation}
 
The following proposition makes the link with the Kolmogorov equation, via the  kinetic spaces. As we want to deal exclusively with distributions, this puts some restrictions on the parameter $\gamma$.

\begin{prop}[Kolmogorov operators map into solutions]\label{prop:mapping}  
Let $\beta\in (0,1]$,  $p\in (1,\infty)$ and $ \gamma\in [0,\homd/p)$. 
\begin{itemize}
\item $\cK^{\pm}_{\beta}$ extend to bounded operators from  $ \Zdot^{\gamma,p}_{\beta}$ to $ \cLdot^{\gamma,p}_{\beta}$.
\item $\cK^{\pm}_{\beta}$ extend to bounded operators from  $\L^p_{t} \Xdot^{\gamma -2\beta,p}_{\beta}$ to $ \cGdot^{\gamma,p}_{\beta}$.  
\item  If, in addition $\gamma\in [0,2\beta]$, $\cK^{\pm}_{\beta}$ extend to bounded operators from  $\L^p_{t,x} \Hdot^{\gamma -2\beta,p}_{\vphantom{t,x} v}$ to $\cFdot^{\gamma,p}_{\beta} $.
\end{itemize}
In each case, $\cK_{\beta}^\pm$ map sources $S$ into (tempered) distributional solutions $f$ to 
$\pm (\partial_{t}+v\cdot\nabla_{x})f +(-\Delta_{v})^\beta f=S$, respectively.
\end{prop}

\begin{proof} 
Let us consider the case of the forward operator $\cK^{+}_{\beta}$ as the case of the backward one is similar. 

We first consider the boundedness from $ \Zdot^{\gamma,p}_{\beta}$ to $\cLdot^{\gamma,p}_{\beta}$. 
Let $ S \in \Zdot^{\gamma,p}_{\beta}$ and set $f=\cK^{+}_{\beta}S \in  \Ydot^{\gamma,p}_{\beta}$ by Theorem~\ref{thm:bounds}. 
Let  $S_{j}\in \cS_{K}$ converge to $ S$ in $\Zdot^{\gamma,p}_{\beta}$. 
We know that $f_{j}=\cK_{\beta}^+S_{j}$ satisfies the equation  and we pass to the limit in $\cS'(\Omega)$ next. As observed in Section \ref{sec:kineticspaces},
the inclusions of $\Zdot^{\gamma,p}_{\beta}$ and $\Ydot^{\gamma,p}_{\beta}$ in $\cS'(\Omega)$ are continuous if and only if $\gamma<\homd/p$. 
Hence, we have $S_{j}\to S$, $f_{j}\to f$ and 
$(\partial_{t}+v\cdot\nabla_{x})f_{j} \to  (\partial_{t}+v\cdot\nabla_{x})f$ in $\cS'(\Omega)$. 
We next claim that $(-\Delta_{v})^\beta f_{j}\to (-\Delta_{v})^\beta f$ in $\cS'(\Omega)$. 
The operator $(-\Delta_{v})^\beta$ is defined and bounded from $\L^p_{t} \Xdot^{\gamma,p}_{\beta}$ to 
$\L^p_{t} \Xdot^{\gamma -2\beta,p}_{\beta}$ when $\gamma<2\beta+(2\beta+2)d/p$, which is the condition for the latter space 
to be continuously contained in $\cS'(\Omega)$. 
Indeed, if $g\in \cS_{K}$, write
\[
	d_{\beta}(\varphi,\xi)^{\gamma-2\beta} |\xi|^{2\beta} \hat g
	= \left(\frac{|\xi|}{d_{\beta}(\varphi,\xi)}\right)^{2\beta} d_{\beta}(\varphi,\xi)^\gamma \hat g.
\]
By Lemma~\ref{lem:multipliers}, $(|\xi|/d_{\beta}(\varphi,\xi))^{2\beta}$ is an $\L^p_{x,v}$ Fourier multiplier. 
Hence, $\|(-\Delta_{v})^\beta g\|_{\L^p_{t} \Xdot^{\gamma -2\beta,p}_{\beta}} \lesssim \| g\|_{\L^p_{t} \Xdot^{\gamma,p}_{\beta}}$.
Then we conclude by density. 
Eventually, as $\gamma<\homd/p$, we have  $\gamma<2\beta+(2\beta+2)d/p$ and the claim follows.  
The equation for $f$ has been verified and we conclude for $f\in \cLdot^{\gamma,p}_{\beta}$. 
Indeed, we have $ f\in \Ydot^{\gamma,p}_{\beta} \subset \L^p_{t,x} \Hdot^{\gamma,p}_{v}$ as $\gamma\ge 0$, 
and $(\partial_{t}+v\cdot\nabla_{x})f= S- (-\Delta_{v})^\beta f\in \Zdot^{\gamma,p}_{\beta}$. 
 
The argument for the boundedness for the second item is the same. 

For the third item, we observe that if $\gamma \le 2\beta$, 
$S\in \L^p_{t,x} \Hdot^{\gamma -2\beta,p}_{v} \subset  \L^p_{t} \Xdot^{\gamma -2\beta,p}_{\beta}\subset \Zdot^{\gamma,p}_{\beta}$. 
Hence we can repeat the first part and obtain that $f=\cK_{\beta}^+S$ satisfies the equation and 
$f\in \L^p_{t,x} \Hdot^{\gamma,p}_{v}$ as $\gamma\ge 0$.  
Since $(\partial_t +v \cdot \nabla_x)f$ already belongs to the target space by definition, it remains to control $(-\Delta_v)^{\beta} f$.
Here,  $(-\Delta_{v})^\beta f \in \L^p_{t,x} \Hdot^{\gamma-2\beta,p}_{v}$  as $\gamma \le 2\beta$. 
Thus $(\partial_{t}+v\cdot\nabla_{x})f= S- (-\Delta_{v})^\beta f\in \L^p_{t,x} \Hdot^{\gamma-2\beta,p}_{v}$. 
\end{proof}
 
Next, we go the other direction, applying the differential operator mapping from the kinetic Sobolev spaces.

\begin{lem}[Boundedness] 
Let $\beta\in (0,1]$,  $p\in (1,\infty)$ and $\gamma\in (-\infty, 2\beta]$. 
The operators  $\pm (\partial_{t}+v\cdot\nabla_{x}) +(-\Delta_{v})^\beta$ are bounded 
\begin{itemize}
\item from $\cLdot^{\gamma,p}_{\beta}$ to $ \Zdot^{\gamma,p}_{\beta}$.
\item  from $\cGdot^{\gamma,p}_{\beta}$ to $\L^p_{t} \Xdot^{\gamma -2\beta,p}_{\beta}$.
\item  from $\cFdot^{\gamma,p}_{\beta} $ to  $\L^p_{t,x} \Hdot^{\gamma -2\beta,p}_{\vphantom{t,x} v}$   
\end{itemize}
\end{lem}

\begin{proof}  
It suffices to remark that for $\gamma\le 2\beta$, $(-\Delta_{v})^{\beta}$ maps  
$\L^p_{t,x} \Hdot^{\gamma,p}_{\vphantom{t,x} v}$ into 
$\L^p_{t,x} \Hdot^{\gamma-2\beta,p}_{\vphantom{t,x} v} \subset \L^p_{t} \Xdot^{\gamma -2\beta,p}_{\beta} \subset \Zdot^{\gamma,p}_{\beta}$.
\end{proof}

\begin{cor}[Ontoness]\label{cor:ontoness} 
Let $\beta\in (0,1]$,  $p\in (1,\infty)$ and $\gamma\in [0, 2\beta]$ with $\gamma<\homd/p$. 
The operators  $\pm (\partial_{t}+v\cdot\nabla_{x}) +(-\Delta_{v})^\beta$ are bounded  and onto.
\begin{itemize}
\item from $\cLdot^{\gamma,p}_{\beta}$ to $ \Zdot^{\gamma,p}_{\beta}$.
\item  from $\cGdot^{\gamma,p}_{\beta}$ to $\L^p_{t} \Xdot^{\gamma -2\beta,p}_{\beta}$.
\item  from $\cFdot^{\gamma,p}_{\beta} $ to $\L^p_{t,x} \Hdot^{\gamma -2\beta,p}_{\vphantom{t,x} v}$   
\end{itemize}
with the forward and backward Kolmogorov operators as right inverses, respectively.
\end{cor}

\begin{proof} 
Combine the previous results. 
The Kolmogorov operators are right inverses for the respective equations from Proposition~\ref{prop:mapping}, 
yielding ontoness in the three cases.  
\end{proof}

\subsection{Injectivity}

The only remaining issue is injectivity. 
We prove a uniqueness statement that covers more equations ($\beta>1$ possible) and works as well in the non favorable Sobolev range ($\gamma\ge d/p$). 
It improves the one obtained for $p=2$ in \cite{AIN}  because we weaken the assumption $f\in \Udot^\gamma$ 
made there when $\gamma\ge d/2$.  
The main point is to have a  uniqueness class involving regularity information with respect to $v$ only, 
hence directly relevant to the equation.

Thinking about it, we cannot argue by some duality argument using the ontoness proved in the previous section because the transport
 operator is only continuous in the sense of (tempered) distributions and not on the spaces involved. 
 For example, the transport operator does not map $ \L^p_{t,x}\Hdot^{\gamma,p}_{v}$ into $\L^p_{t,x}\Hdot^{\gamma-2\beta,p}_{v}$. 
 It is only when starting from a solution $f$ in the first space that we conclude that 
 $(\partial_{t}+v\cdot\nabla_{x})f $ belongs to the second space.  
 This forces us to develop an involved distributional argument that deeply uses the structure of kinetic equations. 
 The range of $\gamma$ with respect to $\beta, d$ and $p$ is the one guaranteeing that we work within (tempered) distributions. We recall $\Omega=\R^{1+2d}$.

\begin{lem}[Uniqueness in $\L^p_{t,x}\Hdot^{\gamma,p}_{v}$]
\label{lem:uniquenessH}
Let $\beta>0$, $p \in (1,\infty)$ and $-\infty<\gamma < 2\beta+d/p$. 
If   $f\in  \L^p_{t,x}\Hdot^{\gamma,p}_{v}$   satisfies  $\pm (\partial_{t}+v\cdot\nabla_{x})f +(-\Delta_{v})^{\beta}f=0$ 
in the sense of distributions on $\Omega$, then  $f=0$. 
\end{lem}

\begin{proof} 
We consider the forward equation, and the proof is essentially the same for the backward equation.
We use here $(\cdot,\cdot)$ for the bilinear duality between distributions and test functions.

\medskip
\paragraph 
{\textbf{Step 1: general facts.}} 

First, the two terms of the equation belong to $\cS'(\Omega)$. Indeed, 
$f\in   \L^p_{t,x}\Hdot^{\gamma,p}_{\vphantom{t,x}v} \subset \cS'(\Omega)$ so the transport term is in $\cS'(\Omega)$. 
Next, set $g=(-\Delta_{v})^{\gamma/2}f\in \L^p_{t,x,v}$. 
By Lemma~\ref{lem:fractionalLaplacian}, the condition $\gamma-2\beta<d/p$ ensures that $(-\Delta_{v})^{\beta}f$ is defined in 
$\cS'(\Omega)$ as $(-\Delta_{v})^{\beta}f= (-\Delta_{v})^{\beta-\gamma/2}g \in \cS'(\Omega)$.  
Thus the equation makes sense in  $\cS'(\Omega)$ 
and taking the kinetic shift and partial Fourier transform yields
$$
	\partial_{t}\widehat{\Gamma f} =- \widehat{[\Gamma (-\Delta_{v})^{\beta}f }]= -\widehat{[\Gamma(-\Delta_{v})^{\beta-\gamma/2}g]}
$$ 
in $\cS'(\Omega)$, hence in $\cD'(\Omega)$.

Secondly,  if $O$ is  the open set of $\Omega$ defined by $\xi-t\varphi\ne 0$, then the distributional restriction 
$\widehat{\Gamma f}_{\mid O}$ of $\widehat{\Gamma f}$ to $O$ satisfies 
$|\xi-t\varphi|^{\gamma} (\widehat{\Gamma f}_{\mid O})=\widehat{\Gamma g} _{\mid O}$. 
Hence we have the equation  
\begin{equation}\label{eq:ODE}
	\partial_{t}(\widehat{\Gamma f}_{\mid O})+ |\xi-t\varphi|^{2\beta} (\widehat{\Gamma f} _{\mid O})=0
\end{equation}
in $\cD'(O)$. 
 
\medskip
\paragraph{\textbf{Step 2: argument when $d\ge 2$.}}

The first  point is to solve the equation as an ODE. 
Consider the open set $\widetilde O=\R\times U$ where $U=\{(\varphi,\xi)\, ;\, \varphi\ne 0, d(\xi,\R\varphi)\ne 0\}$ and 
$d(\xi, \R\varphi)$ is the distance of $\xi$ to the set $\R\varphi$. 
When $d\ge 2$, the complement of $\widetilde O$ has null Lebesgue measure in $\Omega$ and the space 
$\widetilde \cS_{K}$ of those Schwartz functions whose Fourier transforms have compact supports in $\widetilde O$ 
is dense in $\L^{p'}_{t,x,v}$, see \cite[Lemma~2.14]{AIN}.  

By restricting further, we can solve the ODE in $\cD'(\widetilde O)$. 
Since $(t,\varphi,\xi)\mapsto |\xi-t\varphi|^{2\beta}$ is $\C^\infty(\widetilde O)$, one obtains that 
$\widehat{\Gamma f}_{\mid \widetilde O }$ 
has a representative in $\C^1(\R^{}_{t}; \cD'(U))$ and identifying with the representative,   
\begin{equation}\label{eq:ODEsolved}
	\widehat{\Gamma f}_{\mid \widetilde O }(t,\varphi,\xi)= m(t,\varphi,\xi) h(\varphi,\xi)
\end{equation}
for some $h\in \cD'(U)$ with $m$ the $\C^\infty(\widetilde O)$ function given by
$m(t,\varphi,\xi)=\exp\big(-\int_0^{t} \abs{\xi  -\tau \varphi}^{2\beta}\dd\tau \big)$. 
We indicate the $\varphi,\xi$ variables by abuse since we deal with distributions in order to ease the presentation. 

We now want to show that $h=0$. As $\widehat{\Gamma f}\in \cS'(\Omega)$, there exists an integer $N$ and a constant 
$0<C<\infty$ such that for all $\chi\in \cS(\Omega)$, 
\begin{equation} \label{eq:S'}
	|(\widehat{\Gamma f},\chi)| \le C \sup_{|\alpha|\le N, (t,\varphi,\xi)\in \Omega} |1+(t,\varphi,\xi)|^N|\partial^{\alpha}\chi(t,\varphi,\xi)|.
\end{equation}

Observe that  $h(\varphi,\xi)=  m(t,\varphi,\xi) ^{-1} (\widehat{\Gamma f}_{\mid \widetilde O })(t,\varphi,\xi)$ 
in $\cD'(\widetilde O)$ is independent of $t\in \R$. Fix $A$ an open set with compact closure in $U$. 
Remark that in particular $|\varphi|$ and $|\xi|$  are bounded from below by  positive constants when $(\varphi,\xi)\in A$.  
Let $I=(t_{0}-1,t_{0})$ for $t_{0}\in \R$.  
There is  $M\ge 0$ and positive constants $C_{1},C_{2},C_{3}$ such that  if $t_{0}\le -M$ then on $I\times A$ we have 
\begin{align*}
	\int_0^{t} \abs{\xi  -\tau \varphi}^{2\beta}\dd\tau = -\int_{t}^0 \abs{\xi  -\tau \varphi}^{2\beta}\dd\tau \le -C_{1}|t_{0}|^{2\beta+1}, \\
	C_{2}|t_{0}|\le |\xi-t\varphi|\le C_{3}|t_{0}|.
\end{align*}
The first inequality follows from \eqref{eq:nondegen} and a change of variable together with straightforward majorations. 
Let $\Phi\in \cD(A)$  and $\eta\in \cD(\R)$ with support in $(-1,0)$ and $\int_{-1}^0\eta(t)\dt=1$. 
Let $\eta_{t_{0}}(t)=\eta(t-t_{0})$. 
Then the support of $\eta_{t_{0}}\otimes \Phi$ is contained in $I\times A$ and  
\[
	(h,\Phi)=\int  \eta_{t_{0}}(t) (h,\Phi)\dt= (\widehat{\Gamma f}_{\mid \widetilde O }, m ^{-1}\eta_{t_{0}}\otimes \Phi).
\]
Applying \eqref{eq:S'} with the Fa\`a-di-Bruno identity and the Leibniz rule, we obtain the estimate
\[
	|(h,\Phi)| \lesssim_{d,\beta,N, C, A}  e^{-c|t_{0}|^{2\beta+1}}
\]
for some positive constant $c$  independent of $t_{0}\le -M$ as well as the implicit constant. 
Letting $t_{0}\to -\infty$ we obtain $(h,\Phi)=0$.  
As $A$ and $\Phi$ were arbitrary, we deduce $h=0$. 

The next step is to show that $g=0$. 
It follows from $h=0$ that  
$\widehat{\Gamma f}_{\mid \widetilde O }=0$. 
Since $|\xi-t\varphi|^{\gamma} (\widehat{\Gamma f}_{\mid \widetilde O})=\widehat{\Gamma g} _{\mid \widetilde O}$, we 
deduce $\widehat{\Gamma g}_{\mid \widetilde O }=0$. 
This implies that $(\Gamma g,\chi)=0$ for any $\chi\in \widetilde S_{K}$. 
Now $\Gamma g\in \L^p_{t,x,v}$ and  $\widetilde S_{K}$ is dense in $\L^{p'}_{t,x,v}$. 
Thus $\Gamma g=0$, hence $g=0$.

The last  point is to conclude that $f=0$. 
Since $\|f\|_{\L^p_{t,x}\Hdot^{\gamma,p}_{\vphantom{t,x}v}} \sim \|g\|_{\L^p_{t,x,v}}=0$, 
we have $f\in \L^p_{t,x}(\cP_{k}[v])$ where $k=[\gamma-d/p]$. 
We are done when $k<0$ since $\cP_{k}[V]=\{0\}$.  
Otherwise, we have in particular $f\in \L^p_{t,x}\L^p_{\loc,v}$. 
Let us come back to the equation.   
Since $(-\Delta_{v})^{\beta}f=  (-\Delta_{v})^{\beta-\gamma/2}g=0$ in $\cS'(\Omega)$, we obtain  
$(\partial_{t}+v\cdot\nabla_{x})f=0$. 
Thus, we also have $f(t,x,v)=T(x-tv,v)$ for some distribution $T\in \cD'(\R^{2d})$. 
If $K_{v}$ is any compact subset of $\R^d_{v}$, this implies
\[
	\int_{\R^{}_{t}\times \R^d_{x}\times K^{}_{v}} |T(x-tv,v)|^p\, \dd(t,x,v) <\infty
\]
and invariance by translation in $x$ shows this is not possible unless $T=0$ and thus $f=0$.

\medskip
\paragraph{\textbf{Step 3: argument when $d=1$.}}\
Introduce  the open set  $O_{1}=\{\xi-t\varphi\ne  0 \, \& \, \varphi\ne 0\} \subset \Omega=\R^3$.  
By restricting further, we  solve 
$\partial_{t}(\widehat{\Gamma f}_{\mid O_{1}})+ |\xi-t\varphi|^{2\beta} ( \widehat{\Gamma f}_{\mid O_{1}} )=0$ in $\cD'(O_{1})$. 
If we fix $\varphi\ne 0$ and $\xi \in \R$, then $\{t\in \R\, ;\, \xi-t\varphi\ne0\}$ has two connected components, 
$(-\infty, \xi/\varphi)$ and $(\xi/\varphi,\infty)$ and this will necessitate to address behaviour at $t=\xi/\varphi$.

Thus, if $I\times U$ is an open set in $O_{1}^+=\{\varphi\ne 0, t-\xi/\varphi>0\}$ with $I$ an interval, 
then $\widehat{\Gamma f}_{\mid O_{1}}$ restricted to $I\times U$ belongs to $ \C^1(I, \cD'(U))$ and  we find, 
abusively indicating the variables $\varphi,\xi$ for distributions to ease the presentation,    
$$
	(\widehat{\Gamma f}_{\mid O_{1}})(t,\varphi,\xi)
	= \exp\big(-\int_{t_{0}}^{t} \abs{\xi  -\tau \varphi}^{2\beta}\dd\tau \big) 
	(\widehat{\Gamma f}_{\mid O_{1}})(t_{0},\varphi,\xi)
$$
in $\cD'(U)$ for any $t_{0},t\in I$. 
We introduce $O_{2}=\R\times \R^*\times \R$.
Define $m:O_{2}\to \R$, 
\begin{equation}\label{eq:m}
	m(t,\varphi,\xi):=\exp\big(-\int_{\xi/\varphi}^{t} \abs{\xi  -\tau \varphi}^{2\beta}\dd\tau \big),
\end{equation}
which is clearly $\C^\infty$ on $O_{1}$. 
Thus  $m^{-1}\cdot\widehat{\Gamma f}_{\mid O_{1}}$ is a distribution on $O_{1}$ and its restriction to $I\times U$ does not depend on $t$. 
By patching and doing exactly the same thing on $O_{1}^-=\{\varphi\ne 0,t-\xi/\varphi<0\}$, we have obtained, 
\begin{equation*}
	(\widehat{\Gamma f}_{\mid O_{1}})(t,\varphi,\xi)= 
	\begin{cases}
	m(t,\varphi,\xi) \,h^{+}(\varphi,\xi), \quad \mathrm{on}\  O_{1}^+, \\
	m(t,\varphi,\xi) \, h^{-}(\varphi,\xi),  \quad  \mathrm{on}\  O_{1}^-,
	\end{cases}
\end{equation*}
for some  distributions $h^{\pm}$ on $\R^*_{\varphi}\times \R^{}_{\xi}$. 
The main step now  is to show $h^{-} = h^+=0$.

For this we introduce a class of test functions as follows. 
We let
\[ 
	\cD_{\beta}(O_{2})=\{\chi:O_{2}\to \IC\, ; \, \exists \theta\in \cD(O_{2}) \text{ with } \chi(t,\varphi,\xi)= |\xi-t\varphi|^{2\beta}\theta(t,\varphi,\xi)\}.
\]
We recall $\widehat\Gamma  \theta(t,\varphi, \xi):= \theta(t,\varphi,\xi+ t\varphi)$ is the conjugate of $\Gamma$ 
via the Fourier transform. 
Hence, $\chi\in \cD_{\beta}(O_{2})$ if and only if $\widehat\Gamma\chi=|\xi|^{2\beta}\widehat\Gamma \theta$ and 
$\widehat\Gamma \theta\in \cD(O_{2})$. 
Recall that we assume $\gamma<2\beta+1/p$ throughout.
We first claim that 
\begin{enumerate}
\item If $\chi\in \cD_{\beta}(O_{2})$ then $|\xi-t\varphi|^{-\gamma}\chi$ is the (partial) Fourier transform of a function 
in $\L^{p'}_{t,x,v}$ whose norm is controlled by a semi-norm for the function 
$\theta=|\xi-t\varphi|^{-2\beta}\chi \in \cD(O_{2})$.
\item Let $f\in \L^p_{t,x}\Hdot^{\gamma,p}_{v}$ and $\chi\in \cD_{\beta}(O_{2})$. 
If $\chi_{n}$, $n\in \N$, is defined via a normalized Littlewood--Paley family with $(\psi_{j})$ of Section~\ref{sec:homSobspaces} as 
$\widehat{ \Gamma \chi_{n}}= \sum_{|j|\le n} \widehat {\psi_{j}}\widehat {\Gamma \chi}$, then 
$\lim_{n\to \infty} (\widehat{\Gamma f}, \chi_{n})$  exists  and setting  $(\widehat{\Gamma f}, \chi)$ 
its limit, we have  
$$
	(\widehat{\Gamma f}, \chi)= (\widehat{\Gamma g}, |\xi-t\varphi|^{-\gamma}\chi)
$$ 
is independent of the choice of $\chi_{n}$, and 
$|(\widehat{\Gamma f}, \chi)| \le C_{\chi}\|f\|_{\L^p_{t,x}\Hdot^{\gamma,p}_{\vphantom{t,x} v}}$ with 
$C_{\chi}\sim  \|h\|_{\L^{p'}_{t,x,v}}$, where $h = \cF(|\xi|^{-\gamma}\widehat\Gamma \chi )$.  
\item If $f\in \L^p_{t,x}\Hdot^{\gamma,p}_{v}$ and $\chi\in \cD(O_{2})$, then 
$(\widehat{[\Gamma (-\Delta_{v})^{\beta}f]}, \chi)= (\widehat{\Gamma f}, |\xi-t\varphi|^{2\beta}\chi)$.
\end{enumerate}

To see (i), we first remark that $m(\xi)=|\xi|^{2\beta-\gamma}\theta(\xi)$ for some $\theta\in \cD(\R)$ is the 
(one dimensional) Fourier transform of an $\L^{p'}$~function.  
If $2\beta-\gamma\ge 0$, 
there is no difficulty of checking  the Mikhlin multiplier condition $|m|+|\xi m'|$ bounded. 
Hence, $m$ is an $\L^q$ Fourier multiplier for any $q\in (1,\infty)$. 
As $m$ has bounded support, one can write $m=m \Theta$ for some smooth bump function $\Theta$, hence the Fourier transform of an $\L^q$ function.  
This shows that $m$ is the Fourier transform of  an $\L^q$ function.
If $0>2\beta-\gamma>-1/p$, we let $1/q=2\beta-\gamma+1/p$, that is $2\beta-\gamma=1/p'-1/q'$ 
and observe that  $|\xi|^{2\beta-\gamma}$ is the multiplier of the Riesz potential mapping $\L^{q'}$ to $\L^{p'}$.  
Thus the conclusion follows on seeing $\theta$ as the Fourier transform of an $\L^{q'}$~function.  
Now, we come back to a function $\chi(t,\varphi,\xi)= |\xi-t\varphi|^{2\beta}\theta(t,\varphi,\xi)$ with $\theta \in \cD(O_2)$. 
The above argument plus the dependence in $t,\varphi$ shows that
$|\xi|^{2\beta-\gamma}\widehat \Gamma \theta \in \cD_{t,\varphi}\cF_{v}\L^{p'}_{v} \subset \cD_{t}\cF_{x,v}\L^{p'}_{x,v}$. 
This proves (i). 

To see (ii), we show that  $(\widehat f, \widehat \Gamma \chi_{n})$ has a limit equal to 
$(\widehat{g}, |\xi|^{-\gamma}\widehat \Gamma\chi)$. 
Call $h\in \L^{p'}_{t,x,v}$ the function whose (partial) Fourier transform is  $ |\xi|^{-\gamma}\widehat \Gamma \chi $ by (i). 
Recomputing the sum we have $\sum_{|j|\le n} \widehat \psi(2^{-j}\xi)=\widehat\Phi(\xi/2^n)-\widehat\Phi(\xi 2^{n+1})$ for some 
function $\widehat\Phi\in \cD(\R)$ which is equal to $1$ identically on a neighborhood of $0$. 
For $n$ large enough, $\widehat\Phi(\xi/2^n)$ is equal to $1$ on  the support of $\widehat \Gamma \chi $. 
Computing we have for such $n$
\begin{align*}
	(\widehat f, \widehat \Gamma \chi_{n}) &= (\widehat g, (\widehat\Phi(\xi/2^n)-\widehat\Phi(\xi 2^{n+1})) |\xi|^{-\gamma} \widehat  \Gamma \chi) \\
	&= (\widehat g,   |\xi|^{-\gamma} \widehat \Gamma \chi) - (\widehat g, \widehat\Phi(\xi 2^{n+1}) |\xi|^{-\gamma} \widehat  \Gamma \chi) \\
	&=  (\widehat g,   |\xi|^{-\gamma} \widehat \Gamma \chi) -  (\widehat g,    \widehat{h_{n }})
\end{align*} 
where $h_{n}= \Phi_{-n-1}\star_{v} h$ with $\Phi_{j}(v)= 2^{j}\Phi(2^{j}v)$. 
The last term can be controlled by $2\pi \|g\|_{\vphantom{L^{p'}}\L^p_{t,x,v}}\| h_{n}\|_{\L^{p'}_{t,x,v}}$, and 
$\|h_{n}\|_{\L^{p'}_{t,x,v}}\to 0$ as $n\to \infty$. 
In particular, the limit does not depend on the choice of the Littlewood--Paley decomposition. 
We get the bound for $C_{\chi}$ by invoking (i)  together with the isometry property of $\Gamma$ on $\L^{p'}_{t,x,v}$.

To see the equality in (iii), by definition of $f\in \L^p_{t,x}\Hdot^{\gamma,p}_{v}$ and as 
$\widehat \Gamma \chi\in \cD(O_{2})$, we can compute 
\begin{align*}
	(\widehat{[ (-\Delta_{v})^{\beta}f]}, \widehat \Gamma\chi) &= \lim_{n \to \infty} \sum_{|j|\le n} (\widehat \psi(2^{-j}\xi) \widehat{[ (-\Delta_{v})^{\beta}f]}, \widehat \Gamma\chi)\\
	& =  \lim_{n \to \infty} \sum_{|j|\le n} ( \widehat{f}, \widehat \psi(2^{-j}\xi) |\xi|^{2\beta} \widehat \Gamma\chi)\\
	&= (\widehat f, |\xi|^{2\beta}\widehat \Gamma\chi)
\end{align*}
as  $|\xi-t\varphi|^{2\beta}\chi\in \cD_{\beta}(O_{2})$ allows us to use (ii).

To move on, we introduce some notation. For $\chi\in \cD_{\beta}(O_{2})$, we set
\[
	[m\chi]^{+}(\varphi,\xi)= \int_{\xi/\varphi}^\infty m(t,\varphi,\xi) \chi(t,\varphi,\xi)\dt , \quad [m\chi]^{-}(\varphi,\xi)= \int^{\xi/\varphi}_{-\infty} m(t,\varphi,\xi) \chi(t,\varphi,\xi)\dt. 
\]
We claim that both $[m\chi]^{\pm}$ belong to $\cD(\R^*\times \R)$. 
To see this we re-express the integrals. 
A calculation shows that 
\[	
	m(t,\varphi,\xi) = \exp\bigg(-\textrm{sign} (t-\xi/\varphi) \frac{|\xi-t\varphi|^{2\beta+1}}{(2\beta+1)|\varphi|}\bigg).
\]
Hence, writing $\chi(t,\varphi,\xi)=|\xi-t\varphi|^{2\beta}\theta(t,\varphi,\xi)$ and  changing variables $u=\pm(t-\xi/\varphi)|\varphi|^{2\beta/(2\beta+1)}$ we find
\[
	[m\chi]^{\pm}(\varphi,\xi)= \int_{0}^\infty u^{2\beta} \exp\bigg(\mp\frac{u^{2\beta+1}}{(2\beta+1)}\bigg)\theta(\xi/\varphi \pm u/|\varphi|^{2\beta/(2\beta+1)}, \varphi,\xi) \dd u.
\]
By the properties of $\theta$, including the bounded support and that $|\varphi|\ge c>0$ on the support, 
the integral is computed on a finite interval and this
allows us to differentiate repeatedly under the integral sign and $[m\chi]^{\pm}\in \C^\infty(\R^*\times \R)$. Also the statement on the support is clear from that of $\theta$. 
The main claim is the representation formula, for all $\chi\in \cD_{\beta}(O_{2})$, 
\begin{equation}\label{eq:representation}
	(\widehat{\Gamma f}, \chi)= (h^+, [m\chi]^{+})+ (h^-, [m\chi]^{-}),
\end{equation}
where
the brackets on the right-hand side are for the distributions--test functions in $\R^*\times \R$ duality. 

To prove this, we first remark that it is enough to obtain the formula for $\chi$ replaced by its approximations $\chi_{n}$ as this passes to the limit. 
Indeed, for the left-hand side, this is our definition and for the right-hand side, we observe that
$[m\chi_{n}]^\pm \to [m\chi]^\pm$ in  $\cD(\R^*\times \R)$ 
using dominated convergence for the integrals and their partial derivatives.  
Thus we may assume that $\chi$ vanishes in a neighborhood of $t-\xi/\varphi=0$, so that it becomes a function in $\cD(O_{2})$. 
Pick two functions $\eta^\pm$ with $\eta^+\in \cD(0,\infty)$ and $\int_{0}^\infty \eta^+(u)\dd u= 1$,  
$\eta^-\in \cD(-\infty,0)$ and $\int_{-\infty}^0 \eta^-(u)\dd u= 1$, and define
\[
	\tilde \chi(t,\varphi,\xi)= \chi(t,\varphi,\xi) - [m\chi]^{+}(\varphi,\xi)\eta^+(t-{\xi/\varphi}) - [m\chi]^{-}(\varphi,\xi)\eta^-(t-{\xi/\varphi}).
\]
It is clear that $\tilde \chi \in \cD(O_{2})$ and vanishes near $t-\xi/\varphi=0$.  
We need to establish $(\widehat{\Gamma f}, \tilde \chi)=0$ because the last two terms give the right-hand side in \eqref{eq:representation}. 
To see this,  we consider the solution, for fixed $(\varphi,\xi)$, of the ODE on $\R$, 
\[
	\partial_{t}\chi^\sharp(t,\varphi,\xi) - |\xi-t\varphi|^{2\beta} \chi^\sharp(t,\varphi,\xi)= \tilde \chi(t,\varphi,\xi).
\]
This function can be expressed as 
\[ 
	\chi^\sharp(t,\varphi,\xi)=m(t,\varphi,\xi)^{-1} \int_{-\infty}^t m(u,\varphi,\xi)\tilde \chi(u,\varphi,\xi)\, \dd u.
\]
By construction, we  have $[m\tilde\chi]^\pm=0$ and we can write 
\begin{equation*}
	\chi^\sharp(t,\varphi,\xi)= 
	\begin{cases}
	\ \ m(t,\varphi,\xi)^{-1} \int_{\xi/\varphi}^t m(u,\varphi,\xi)\tilde \chi(u,\varphi,\xi)\, \dd u, \quad \mathrm{on}\  O_{1}^+,\\
	-m(t,\varphi,\xi)^{-1} \int^{\xi/\varphi}_{t} m(u,\varphi,\xi)\tilde \chi(u,\varphi,\xi)\, \dd u,  \quad  \mathrm{on}\  O_{1}^-.
	\end{cases}
\end{equation*}
Because $\tilde \chi(u,\varphi,\xi)$ is supported away from $u-\xi/\varphi=0$, it follows from this formula that 
$\chi^\sharp(t,\varphi,\xi)$ is also supported away from $t-\xi/\varphi=0$ 
and one can check that it belongs to $\C^\infty(O_{2})$, and has compact support using again $[m\tilde\chi]^\pm=0$. 

Since $f$ is a solution, we have (with the standard calculus because of supports)
\[
	0=  (\partial_{t}\widehat{\Gamma f}+\widehat{\Gamma (-\Delta_{v})^{\beta}f} , \chi^\sharp)= (\widehat{\Gamma f}, -\partial_{t}\chi^\sharp+|\xi-t\varphi|^{2\beta}\chi^\sharp )= -(\widehat{\Gamma f}, \tilde \chi).
\]
Now we can prove that $h^+=h^-$. Indeed, consider $\chi^\flat(t,\varphi,\xi)= \theta(\varphi,\xi)\eta(t-\xi/\varphi)$ where 
$\theta$ is an arbitrary function in $\cD(\R^*\times \R)$ and $\eta\in \cD(\R)$ which is 1 in a neighborhood of zero. 
Set $\tilde \chi^\flat(t,\varphi,\xi)= \partial_{t}  \chi^\flat(t,\varphi,\xi)- |\xi-t\varphi|^{2\beta} \chi^\flat(t,\varphi,\xi)$ 
and observe that $\tilde \chi^\flat\in \cD_{\beta}(O_{2})$. 
By using (ii) and (iii)  this time,  
\[
	0=  (\partial_{t}\widehat{\Gamma f}+\widehat{\Gamma (-\Delta_{v})^{\beta}f} , \chi^\flat)= (\widehat{\Gamma f}, -\partial_{t}\chi^\flat+|\xi-t\varphi|^{2\beta}\chi^\flat )= -(\widehat{\Gamma f}, \tilde \chi^\flat).
\]
Using the representation \eqref{eq:representation} we find
\[
	-(\widehat{\Gamma f}, \tilde \chi^\flat)= (h^+-h^-, \theta).
\]
This shows that $h^+=h^-$. 
Writing $h$ for this distribution and resinserting in the representation for arbitrary $\chi\in \cD_{\beta}(O_{2})$ we have obtained
\[
	(\widehat{\Gamma f}, \chi)= (h, [m\chi]^{+}+ [m\chi]^{-})= \bigg(h, \int_{\R} m(t,\varphi,\xi) \chi(t,\varphi,\xi)\dt\bigg).
\]
As in Step 2 by looking at the behavior of $m$ when $t\to -\infty$ and using the bound for $C_{\chi}$ in (ii) and (i) 
replacing \eqref{eq:S'}, we can conclude that $h=0$.

We now infer that $g=0$. 
Indeed, undoing the kinetic shift, we have proved in particular that $\widehat{f}_{\,\mid \{\xi\varphi\ne 0\}}=0$ 
and as $|\xi|^\gamma(\widehat{f}_{\,\mid \{\xi\varphi\ne 0\}})= \widehat{g}_{\,\mid \{\xi\varphi\ne 0\}} $ 
we conclude that $\widehat{g}_{\,\mid \{\xi\varphi\ne 0\}}=0$. 
This implies that $(g,\chi)=0$ for any $\chi\in \cS_{K}$. 
As $\cS_{K}$ is dense in $\L^{p'}_{t,x,v}$ we conclude that $g=0$. 

From there, we can argue exactly as in Step 2  and conclude that $f=0$.   
\end{proof}

\section{The kinetic Cauchy problem on the half-line}
\label{sec:cauchy}

This section is devoted to the kinetic Cauchy problem on $\Omega_{+}=
(0,\infty)\times \R^d_x\times \R^d_v$
for the Kolmogorov equation. The case of a strip $(0,T)\times \R^d_x\times \R^d_v$ will be considered in the next section. 
We work with the homogeneous anisotropic Sobolev/Besov spaces and the homogeneous kinetic spaces defined as before with time restricted to the interval $(0,\infty)$. We use the notation $\L^p((0,\infty)\times \R^d_{x}\, ; \, \Hdot^{\gamma,p}_v)$ etc, for mixed-spaces, and also $\Zdot^{\gamma,p}_{\beta,+}$, $\cLdot^{\gamma,p}_{\beta,+}$, etc. 

The key result is the following.

\begin{lem}
\label{lem:embedOmega+} Let $\beta\in(0,1]$, $p \in (1,\infty)$ and $0\le \gamma\le 2\beta$ with $\gamma<\homd/p$. If $f\in \cLdot^{\gamma,p}_{\beta,+}$ then 
\[
	f\in 
	 \L^p\big((0,\infty)\times \R^d_{v}\, ; \, \dot\H^{\frac{\gamma}{2\beta+1},p}_x\big)
	\cap \C^{}_0([0,\infty);\Bdot^{\gamma-{2\beta}/{p},p}_\beta),
\]
and one has the estimate
\begin{equation}
\label{eq:embeddingomega+}
		\|D_x^{\frac{\gamma}{2\beta+1}} f\|_{\L^p(\Omega_+)}
	+\sup_{t\ge 0}\|f(t)\|_{\Bdot^{\gamma-{2\beta}/{p},p}_\beta}\\
	\lesssim_{\beta,\gamma,d,p} \|D_v^\gamma f\|_{\L^p(\Omega_+)}
+ \|(\partial_t+v\cdot \nabla_x)f\|_{\Zdot^{\gamma,p}_{\beta, +}}.
\end{equation}
 \end{lem}

\begin{proof} 
We let $S=-(\partial_t+v\cdot\nabla_x)f+(-\Delta_v)^\beta f \in \Zdot^{\gamma,p}_{\beta,+}$, $\tilde S$ the extension of $S$ by 0 to $\Omega$ and $\tilde f= \cK^-_{\beta}(\tilde S)$, where $\cK_{\beta}^-$ is the backward Kolmogorov operator. Then, we have $-(\partial_t+v\cdot\nabla_x)(f-\tilde f)+(-\Delta_v)^\beta (f-\tilde f)=0$ on $\Omega_{+}$. The proof is in the same spirit as the one of Lemma~\ref{lem:uniquenessH} with $t>0$ instead of $t\in \R$ with a slight change in one dimension: in the regions where $\varphi,\xi$ have different signs, the function $\xi-t\varphi$ has only one sign. So on this region, we can follow the strategy adopted in higher dimension. When $\varphi,\xi$ have same sign, nothing  changes. It is important that we use the backward operator in this argument to pick up decay when  $t\to \infty$. Thus $f-\tilde f=0$ on $\Omega_{+}$. Since $\tilde S\in \Zdot^{\gamma,p}_{\beta}$, it follows from Theorem~\ref{thm:bounds} that $\tilde f\in  \Ydot^{\gamma,p}_{\beta}$, so $f$  has the desired properties and estimates by taking  restriction of $\tilde f$ to $\Omega_{+}$.
 \end{proof}
 
\begin{rem}
 The proof shows that the image under $\cK_{\beta}^-$ of $\cS_{K}$, restricted to $\Omega_{+}$, that is $\{(\cK_{\beta}^-S)\!\mid _{\Omega_{+}}\, ;\, S\in \cS_{K}\}$,  is a dense subspace of $\cLdot^{\gamma,p}_{\beta,+}$ in the same range of exponents.  
\end{rem}

\begin{rem}
For the smaller subspaces  $\cFdot^{\gamma,p}_{\beta,+}$ and $\cGdot^{\gamma,p}_{\beta,+}$ we have the embedding too. These spaces also have the complex interpolation property respectively in the same  range of exponents as in the lemma.  
\end{rem}
 
\begin{rem}
 On $\Omega_{-}$, corresponding to $\I=(-\infty,0)$, the same result holds, either by a symmetry argument from the previous lemma or a direct argument using this time the forward Kolmogorov operator. 
\end{rem}

\subsection{The forward kinetic Cauchy problem on $(0,\infty)$}

We consider the forward kinetic Cauchy problem
\begin{align} \label{eq:CP}
	\begin{cases}
		(\partial_t+v\cdot\nabla_x)f+(-\Delta_v)^\beta f = S, & \text{in }\Omega_+
		,\\
		f(0)=\psi, & \text{in }\R^{2d}.
	\end{cases}
\end{align}

\begin{defn}
\label{defn:sol-cauchyLp}
Let $\beta\in(0,1]$, $p \in (1,\infty)$ and $0\le \gamma\le 2\beta$ with $\gamma<\homd/p$.
Let $S \in \cD' (\Omega_+)$ and $\psi\in \cD'(\R^{2d})$.
A distribution $f\in \L^p_{t,x}\Hdot^{\gamma,p}_{\vphantom{t}v}$ on $\Omega_{+}$
is said to be a (distributional solution) to \eqref{eq:CP} if 
the first equation in \eqref{eq:CP} holds in $\cD'(\Omega_{+})$ and $f(t)\to \psi$ in $\cD'(\R^{2d})$ as $t\to 0^+$.
\end{defn}

We remark that $f(t)$ exists for almost every $t>0$ in $\cD(\R^{2d})$ by definition of $\L^p_{t,x}\Hdot^{\gamma,p}_{\vphantom{t}v}$. So the limit $t\to 0^+$ can be taken along those times of existence.

\begin{thm}[Existence, uniqueness and representation for the kinetic Cauchy problem]
\label{thm:homCP-Lp}
Let $\beta\in(0,1]$, $p \in (1,\infty)$ and $0\le \gamma\le 2\beta$ with $\gamma<\homd/p$.
Let $S\in \Zdot^{\gamma,p}_{\beta,+}$ and $\psi\in \Bdot^{\gamma-{2\beta}/{p},p}_\beta$.
Then there exists a unique distributional solution $f\in \L^p((0,\infty)\times \R^d_{x}\, ; \, \Hdot^{\gamma,p}_v)$ to \eqref{eq:CP} in the sense of
Definition~\ref{defn:sol-cauchyLp}. Moreover, $f\in \Ydot^{\gamma,p}_{\beta,+}$  
and one has the estimate
\begin{multline}
\label{eq:CP-estimate-Lp}\|f\|_{\Ydot^{\gamma,p}_{\beta,+}}\simeq
	\|D_v^\gamma f\|_{\L^p(\Omega_+)}
	+\|D_x^{\frac{\gamma}{2\beta+1}} f\|_{\L^p(\Omega_+)}
	+\sup_{t\ge 0}\|f(t)\|_{\Bdot^{\gamma-{2\beta}/{p},p}_\beta}\\
	\lesssim_{\beta,\gamma,d,p}
	\|S\|_{\Zdot^{\gamma,p}_{\beta,+}}
	+\|\psi\|_{\Bdot^{\gamma-{2\beta}/{p},p}_\beta}.
\end{multline}

Finally, if $\tilde{S}$ denotes the zero extension of $S$ to $t<0$, then $f$ admits the representation
\begin{equation}
\label{eq:CP-representation}
	f(t)=\sem^+_t\psi + \big(\cK_\beta^+ \tilde{S}\big)(t),\qquad t\ge 0,
\end{equation}
where $\sem^+_t$ is the forward Kolmogorov semigroup associated with \eqref{eq:solKolIV}
and $\cK_\beta^+$ is the forward Kolmogorov operator $($defined on $\Omega$ via the fundamental solution$)$. The equality is meant in $\Bdot^{\gamma-{2\beta}/{p},p}_\beta$.
\end{thm}

\begin{proof}
Let $\tilde{S}$ be the zero extension of $S$ to $t<0$ and define
\[
	f(t):=\sem^+_t\psi+(\cK_\beta^+\tilde{S})(t),\qquad t\ge0.
\]
By   Proposition~\ref{prop:mapping}, $\cK_\beta^+\tilde{S}\in\cLdot^{\gamma,p}_\beta$ and 
solves $((\partial_t+v\cdot\nabla_x)+(-\Delta_v)^\beta)(\cK_\beta^+\tilde{S})=\tilde{S}$ in $\cD'(\Omega)$. 
By Lemma~\ref{lem:boundsinitialvalueproblem} and the same density argument as in Proposition~\ref{prop:mapping}, $t\mapsto \sem^+_t\psi$ belongs to
$\cLdot^{\gamma,p}_{\beta,+}$, solves the homogeneous equation $((\partial_t+v\cdot\nabla_x)+(-\Delta_v)^\beta)(\sem^+\psi)=0$ in $\cD'(\Omega_+)$ and satisfies $\sem^+_t\psi\to\psi$ 
in $\Bdot^{\gamma-{2\beta}/{p},p}_\beta$ as $t\to0^+$.
Moreover, they both vanish for $t<0$ by definition, and continuity implies $\cK_\beta^+(\tilde{S})(0) = 0$.
Next,  \eqref{eq:CP-estimate-Lp} for $f$ is already in Theorem~\ref{thm:bounds}. 

It remains to show the uniqueness.   Let $f\in \L^p((0,\infty)\times \R^d_{x}\, ; \, \Hdot^{\gamma,p}_v)
$ be a distributional solution of \eqref{eq:CP} with  $S=0$ and $\psi=0$.  We know  $f\in \cLdot^{\gamma,p}_{\beta,+}$ by definition, hence $f\in \Ydot^{\gamma,p}_{\beta,+}$ by Lemma~\ref{lem:embedOmega+} with $ f(0,\cdot)=\psi=0$. Thus,   the extension $\tilde f$ by $0$ of $f$ to $t<0$ belongs to $\Ydot^{\gamma,p}_\beta$ with $\tilde f(0,\cdot)=0$. Let $\theta_{\varepsilon}(t)= \theta(t/\varepsilon)$, where $\varepsilon>0$ and $\theta:\R\to [0,1]$ is smooth, even, with $\theta=0$ on $[0,1]$ and $1$ on $[2,\infty)$. 
Set $\tilde f_{\varepsilon}= \theta_{\varepsilon}\tilde f$. Clearly 
$(\partial_t + v \cdot \nabla_x) \tilde f_{\varepsilon}+ (-\Delta_v)^{{\beta}} \tilde f_{\varepsilon} = \theta_{\varepsilon}' \tilde f$ in $\cD'(\Omega)$ as $\theta_\epsilon$ vanishes near $t = 0$. As $\tilde f_{\varepsilon}\in \L^p_{t,x}(\Hdot^{\gamma,p}_v)$ and $\theta_{\varepsilon}' \tilde f \in \L^1_{t}{\Bdot^{\gamma-{2\beta}/{p},p}_\beta}$, by the isomorphism on $\Omega$, we have $\tilde f_{\varepsilon}= \cK_{\beta}^+ (\theta_{\varepsilon}' \tilde f)$. By Theorem~\ref{thm:bounds}, we have in particular $$ \|\tilde f_{\varepsilon}\|_{\L^p_{t,x}(\Hdot^{\gamma,p}_{\vphantom{t}v})} \le C \|\theta_{\varepsilon}' \tilde f\|_{\L^1_{t}{\Bdot^{\gamma-{2\beta}/{p},p}_\beta}} \lesssim \frac 1 \varepsilon \int_{-2\varepsilon}^{2\varepsilon} \| \tilde f(s)\|_{{\Bdot^{\gamma-{2\beta}/{p},p}_\beta}}\, \dd s.$$ If $
\varepsilon\to 0$,  the left-hand side converges to 
$\|\tilde f\|_{\L^p_{t,x}\Hdot^{\gamma,p}_{\vphantom{t}v}}$ by dominated convergence, and the right-hand side converges to 0 as $\tilde f\in \C^{}_0(\R;\Bdot^{\gamma-{2\beta}/{p},p}_\beta)$ with $\tilde f(0,\cdot)=0$. Hence $\tilde f=0$.  
\end{proof}

\begin{rem}
Note that when $\gamma>0$  we  have $f\in \L^{\frac{p\homd}{\homd -p\gamma}}_{t,x,v}(\Omega_{+})$.  This follows from  the  embeddings of Section~\ref{sec:Lebesgueestimates}. By  Corollary~\ref{cor:Lebesguesestimates}, one can solve the  Cauchy problem with data  $\psi\in \L^{a^\flat}_{x,v}$ and $S\in \L^a_{t,x,v}(\Omega_{+})$ for  $a=\frac{p\homd}{\homd + p(2\beta-\gamma)}$ provided $1<a<\frac \homd{2\beta}$.
Furthemore, it is known that the Kolmogorov fundamental solution is, for each $t>0$, the density of a probability measure. Thus $\sem^+$ is bounded  $\L^r_{x,v}\to \L^\infty_{t}((0,\infty); \L^r_{x,v})$ for all $r\in [1,\infty]$ and $\cK_{\beta}^\pm$ is bounded  $ \L^s_{t}((0,T); \L^q_{x,v}) \to \L^s_{t}((0,T); \L^q_{x,v})$ for $s,q\in [1,\infty]$ and $0<T<\infty$. Combining all this with interpolation gives a wealth of estimates on bounded intervals for the solution. We leave details to the reader.
\end{rem}

\begin{cor}[Causality principle]
\label{cor:causality-Lp}
Let $\beta\in(0,1]$, $p \in (1,\infty)$ and $0\le \gamma\le 2\beta$ with $\gamma<\homd/p$.
Let $S\in \Zdot^{\gamma,p}_{\beta,+}$ and let $\tilde{S}$ be its zero extension to $t<0$.
Then $\cK_\beta^+\tilde{S}$ vanishes for $t\le 0$, and its restriction to $\Omega_+$ is the unique distributional solution of \eqref{eq:CP}
with $\psi=0$.
\end{cor}

\begin{rem} 
The embeddings for $\cLdot^{\gamma,p}_{\beta,+}$ imply that we can write the integral identities as in Corollary~\ref{cor:polarizedenergyequalities}. In particular,  
if $h\in \cD (\overline{\Omega_+})$, then  $t\mapsto \angle{f(t)}{h(t)}_{B}$  is absolutely continuous on $[0,\infty)$ and we have  for each $t\ge 0$, 
\begin{align*}
 \angle{f(t)}{h(t)}_{\B}- \angle{\psi}{h(0)}_{\B}&= \int_{0}^t \angle{f}{(\partial_t + v \cdot \nabla_x)h}_{\B}\dd s - \int_{0}^t\int_{\R^{d}} \langle f , (-\Delta_v)^{{\beta}} h \rangle_{\H}\dd x\dd s 
 \\
 & \qquad + \int_{0}^t 
\angle {S_{1}}{h}_{\X}     + \angle   {S_{2}}{h}_{\B} \dd s 
\end{align*}
if $S=S_{1}+S_{2}$, $S_1 \in \L^p(0,\infty\, ;\, \Xdot^{\gamma-2\beta,p}_{\vphantom{t} \beta})$ and 
$S_2 \in \L^1(0,\infty\, ; \, \Bdot_{\vphantom{t,x} \beta}^{\gamma-2\beta/p,p})$. The brackets are distributions-test functions but subscripts indicate  the duality for which  the brackets can be understood: $\B$ for anisotropic Besov, $\H$ for Sobolev and $\X$ for anisotropic Sobolev.
\end{rem}

\begin{rem}[Weak solutions]
\label{rem:weaksol}
	In the case $\gamma = \beta$ and $p=2$,  and for $t$ large in the formula of the previous remark,   since $h$ has compact support, with $D_{v}=(-\Delta_v)^{{1/2}}$,
	\[
	\int_{0}^t\int_{\R^{d}} \langle f , (-\Delta_v)^{{\beta}} h \rangle_{\H}\dd x\dd s
	= \int_{\Omega_{+}} D_v^\beta f \,\overline{D_v^\beta h}\, \dd (s,x,v)
	\]
	as an immediate consequence of Lemma \ref{lem:fractionalLaplacian}, 
	and if $\beta<d/2$, 
	\[\int_{0}^t \angle{f}{(\partial_t + v \cdot \nabla_x)h}_{\B}\dd s = 
	\int_{\Omega_{+}} f \ (\partial_t + v \cdot \nabla_x)\overline h\,  \dd(s,x,v)
	\]
	because by Sobolev embedding $f\in \L^{2}((0,\infty)\times \R^d_{x}\, ;\, \L^q_{\vphantom{t}v})$, with $q=\frac {2d}{d-2\beta}$. 	So we obtain the definition of weak solutions in \cite{AIN} when $\beta<d/2$. 
	When $\beta\ge d/2$ (if $\beta\le 1$, $d=1,2$ are the only concerned dimensions but the theory for $p=2$ works for all $\beta>0$), the last integral may not exist when $f\in \L^2((0,\infty)\times \R^d_{x}\, ;\,\Hdot^{\beta,2}_{v})$ and  \cite{AIN} assumes a further a priori condition on $f$  to keep the integral  formulation in the definition valid. Here, we took the distributional definition. But  as distributional solutions are shown to have $\C^{}_{t}\L^2_{x,v}$ regularity, one can use the integral formulation right away. We emphasise that  this regularity is not trivial: it is, in particular,  a consequence of the  stronger uniqueness result proved here (for $p=2$, the proof can be simplified a little).
\end{rem}

\begin{rem}[Strong solutions]
	In the case  $\gamma = 2\beta$, the solutions of Theorem \ref{thm:homCP-Lp} are strong solutions in the sense that the Kolmogorov equation holds in $\L^p_{t,x,v}(\Omega_+)$, and in particular,  pointwise almost everywhere.  
\end{rem}

\begin{rem}[Backward kinetic Cauchy problem]
By the same strategy, one can solve the backward kinetic Cauchy problem
\[
	-(\partial_t+v\cdot\nabla_x)f+(-\Delta_v)^\beta f=S \quad\text{on }(-\infty,0)\times\R^{2d}, \qquad f(0)=\psi,
\]
in the same range of parameters, with the backward Kolmogorov operator $\cK^-_\beta$ 
and the backward semigroup $\sem^-$ (cf.\ \eqref{eq:solKolIVback}) replacing $\cK^+_\beta$ and $\sem^+$.
\end{rem}

\section{Statements on inhomogeneous spaces}
\label{sec:inhom}

The homogeneous theory developed above has a fully parallel inhomogeneous counterpart.
The only change is the treatment of low frequencies: one replaces homogeneous Sobolev/Besov norms by their inhomogeneous versions 
(equivalently, one adds an $\L^p$-term to the homogeneous norms).
In particular, the auxiliary restriction $\gamma<\homd/p$ that was only used to avoid quotient issues in 
homogeneous Besov/Sobolev spaces disappears in the inhomogeneous framework. Moreover, inhomogeneous versions appear naturally when looking at the Cauchy problem on finite time intervals using the usual exponential trick. So we present a short account on this.
Throughout this section $\beta\in(0,1]$ and $p \in (1,\infty)$.

\subsection{Inhomogeneous kinetic spaces and embeddings} $ $ \\

\textbf{Inhomogeneous kinetic Besov and Sobolev spaces on $\R^{2d}$.}
Let $d_\beta$ be the smooth anisotropic quasi-norm on the Fourier side introduced previously.
We denote by $\X^{\gamma,p}_\beta$ the inhomogeneous anisotropic Bessel potential space associated with $d_\beta$:
\begin{align*}
	\X^{\gamma,p}_\beta &:=\Big\{ g\in\cS'(\R^{2d}) \,;\, \cF^{-1}\!\big((1+d_\beta(\varphi,\xi)^2)^{\gamma/2}\hat g(\varphi,\xi)\big)\in \L^p(\R^{2d}) \Big\}, \\ 
	\|g\|_{\X^{\gamma,p}_\beta}&:=\Big\|\cF^{-1}\!\big((1+d_\beta^2)^{\gamma/2}\hat g\big)\Big\|_{\L^p_{x,v}}.
\end{align*}
Similarly, $\B^{\gamma,p}_\beta$ denotes the inhomogeneous anisotropic Besov space.
For instance, in the continuous Littlewood--Paley formulation, fix $\psi\in \cS(\R^{2d})$ as in \eqref{eq:besovcont} 
and choose $\Phi\in\cS(\R^{2d})$ with $\widehat\Phi$ supported near $(0,0)$ such that
\[
	\widehat\Phi(\varphi,\xi) + \int_{0}^{1}\widehat\psi\big(s^{\frac{1}{2\beta}+1}\varphi,\,s^{\frac{1}{2\beta}}\xi\big)\,\frac{\ds}{s}=1 \qquad \text{for all }(\varphi,\xi)\in\R^{2d}.
\]
Then, for all $\gamma\in\R$,
\begin{equation}\label{eq:inhomBesov-cont}
	\|g\|_{\B^{\gamma,p}_\beta}^p \simeq \|\Phi*g\|_{\L^p_{x,v}}^p + \int_{0}^{1}\left(\frac{\|\psi_s*g\|_{\L^p_{x,v}}}{s^{\gamma/(2\beta)}}\right)^p\frac{\ds}{s}.
\end{equation}
As usual, this is equivalent to the dyadic definition.
\medskip

\textbf{Inhomogeneous kinetic spaces on $\Omega_{\I}$.} 
Let $\I\subset\R$ be an open interval and recall that $\Omega_{\I}:=\I\times\R^d_x\times\R^d_v$.
We define the inhomogeneous analogues of the spaces $\Ydot^{\gamma,p}_\beta$ and $\Zdot^{\gamma,p}_\beta$ on $\Omega_{\I}$ by
\begin{align*}
	\Y^{\gamma,p}_\beta(\I) &:= \L^p\!\big(\I\,;\,\X^{\gamma,p}_\beta\big) \cap \C^{}_{0}\!\big(\bar{\I}\, ;\, \B^{\gamma-{2\beta}/{p},p}_\beta\big), \\
	\|f\|_{\Y^{\gamma,p}_\beta(\I)} &:= \|f\|_{\L^p(\I\, ;\, \X^{\gamma,p}_\beta)}+\|f\|_{\L^\infty(\I\, ;\, \B^{\gamma-{2\beta}/{p},p}_\beta)}, \\[0.2cm]
	\zZ^{\gamma,p}_\beta(\I) &:= \L^p\!\big(\I\,;\,\X^{\gamma-2\beta,p}_\beta\big) + \L^1\!\big(\I\,;\,\B^{\gamma-{2\beta}/{p},p}_\beta\big),\\
	\|S\|_{\zZ^{\gamma,p}_\beta(\I)} &:=	\inf_{S=S_1+S_2}\Big( \|S_1\|_{\L^p(\I\, ;\, \X^{\gamma-2\beta,p}_\beta)} + \|S_2\|_{\L^1(\I\, ;\,\B^{\gamma-{2\beta}/{p},p}_\beta)} \Big).
\end{align*}
Here, the suffix $0$ for $\C^{}_{0}$ means that the limits at infinite endpoints of $\I$ vanishes if $\I$ is unbounded.

We also define the inhomogeneous kinetic spaces $\cF^{\gamma,p}_\beta(\I)$, $\cG^{\gamma,p}_\beta(\I)$ and 
$\cL^{\gamma,p}_\beta(\I)$ by replacing everywhere in definitions
$\Hdot^{s,p}$, $\Xdot^{s,p}_\beta$ and $\Bdot^{s,p}_\beta$ with
$\H^{s,p}$, $\X^{s,p}_\beta$ and $\B^{s,p}_\beta$ in the corresponding homogeneous definitions and working with $t\in \I$.
In particular,  we set
\begin{align*}
	\cL^{\gamma,p}_\beta(\I) &:= \Big\{ f\in \L^p\!\big(\I\times\R^d_x\, ;\,\H^{\gamma,p}_v\big) \ ;\ (\partial_t+v\cdot\nabla_x)f\in \zZ^{\gamma,p}_\beta(\I) \Big\}, \\
	\|f\|_{\cL^{\gamma,p}_\beta(\I)}&:= \|f\|_{\L^p(\I\times\R^d_x\, ;\,\H^{\gamma,p}_v)} + \|(\partial_t+v\cdot\nabla_x)f\|_{\zZ^{\gamma,p}_\beta(\I)}.
\end{align*}

\medskip

One may work either with the forward/backward Kolmogorov operators
$\cK_\beta^\pm$ if $\I$ is bounded or, equivalently, with the resolvents
\[
	\cK_{1,\beta}^{\pm}:=\big(\pm(\partial_t+v\cdot\nabla_x)+(-\Delta_v)^\beta+1\big)^{-1},
\]
which are convenient to handle low frequencies. If $\I$ is unbounded, we consider the resolvents.
{In the inhomogeneous case we do not need to impose $\gamma < \homd/p$ as we are dealing with functions.}
Concerning an inhomogeneous kinetic embedding and transfer of regularity, all proofs for $\I=\R$ are unchanged and easier.  In the case $\I=(0,\infty)$, one can adapt the proof of Lemma~\ref{lem:embedOmega+} using the resolvent to go back to the case of $\R$. This holds by symmetry for $\I=(-\infty,0)$. For $\I$ bounded, one can localise at each endpoint and apply the previous case with half-infinite intervals after a shift in time. Altogether, this yields the following statement.

\begin{thm}[Kinetic embeddings and transfer of regularity: inhomogeneous case]
\label{thm:inhomkinspaceLp}
Assume $\gamma\in[0,2\beta]$ and $\I\subset\R$ is an open interval. We have the following embeddings and properties.
\begin{enumerate}
\item \emph{(Transfer-of-regularity)} 
\[
	\cL^{\gamma,p}_\beta(\I)\hookrightarrow \L^p\big(\I\times\R^d_v\, ;\,\H^{\frac{\gamma}{2\beta+1},p}_{\vphantom{t}x}\big), \qquad \|f\|_{\L^p(\I\times\R^d_v\, ;\,\H^{\frac{\gamma}{2\beta+1},p}_{\vphantom{t}x})} \lesssim \|f\|_{\cL^{\gamma,p}_\beta(\I)}.
\]

\item \emph{(Continuity-in-time)} 
\[
	\cL^{\gamma,p}_\beta(\I)\hookrightarrow \C^{}_{0}(\overline \I\,;\,\B^{\gamma-{2\beta}/{p},p}_\beta), \qquad \sup_{t\in \overline\I}\|f(t)\|_{\B^{\gamma-{2\beta}/{p},p}_\beta}	\lesssim \|f\|_{\cL^{\gamma,p}_\beta(\I)}.
\]

\item \emph{(Gain-of-integrability)} For $p\le q\le \frac{p\homd}{\homd -p\gamma}$ if $\gamma<\frac \homd p$ and  $p\le q<\infty$ if  $\gamma\ge \frac \homd p$, 
\[
	\cL^{\gamma,p}_\beta(\I)\hookrightarrow 
	\L^{q}_{t}(\I;\L^{q}_{x,v}), \qquad \|f\|_{\L^{q}_{t,x,v}}\lesssim \|f\|_{\cL^{\gamma,p}_\beta(\I)}.
\]

\item \emph{(Dense class)} The set
\begin{equation*}
	\mathfrak{C}:=\Big\{\, f=\tilde f\big|_{\Omega_{\I}}\;  ; \; \tilde f\in \C_c^\infty(\R_t\times\R_x^d\times\R_v^d)\,\Big\}
\end{equation*}
 is dense in $\cF^{\gamma,p}_\beta(\I)$, $\cG^{\gamma,p}_\beta(\I)$, $\cL^{\gamma,p}_\beta(\I)$.

\item \emph{(Complex interpolation)} The spaces  $\cF^{\gamma,p}_\beta(\I)$ and $ \cG^{\gamma,p}_\beta(\I)$ respectively interpolate by the complex method along each segment drawn in the box $(\gamma, 1/p) \in [0,2\beta] \times (0,1)$. 
\end{enumerate}

The implicit constants depend on $d,\beta,\gamma,p$, and on $|\I|$ if $\I$ is bounded. 
\end{thm}

\begin{proof} 
The proof of (i) and (ii) has already been commented. The interpolation statement (v) comes from the isomorphism and interpolation for the target spaces corresponding to $\cF$ and $\cG$ spaces.

As for (iii), we have already seen the result at  $q=\frac{p\homd}{\homd -p\gamma}$ if $\gamma<\frac \homd p$ in the homogeneous case.  This is the same in the inhomogeneous case and we interpolate with $f\in \L^p_{t,x,v}$ by assumption. If   $\gamma\ge \frac \homd p$, then we have $f\in 
 \Y^{\gamma,p}_{\beta}(\I)$ by (i) and (ii). The spaces being inhomogeneous, we also have  $f\in 
 \Y^{\delta,p}_{\beta}(\I)$ for all $\delta\in [0, \gamma]$. If $q\in [p,\infty)$, it suffices to pick $\delta$ with $\delta<\homd/p$ and $q=\frac{p\homd}{\homd -p\delta}$ to conclude.

Regarding (iv) on $\R^{1+2d}$, the key point is that, as we deal with inhomogeneous spaces, we  may regularize by kinetic convolution $\ast_{\rm kin}$ defined as
\begin{equation*}
	[f \ast_{\rm kin} g](t,x,v) = \int_{\R^{1+2d}}f(t-s,x-y-(t-s)w,v-w) g(s,y,w) \dd(s,y,w),
\end{equation*}
i.e.,\ by averaging against the Galilean translations which commute with $\partial_t+v\cdot\nabla_x$. With this tool, one can mollify (according to kinetic scaling) and then multiply by smooth cut-offs in order to produce approximations in the kinetic norms. We refer to \cite{MR4527757} for the special case corresponding to weak solutions $\gamma=\beta$ and $p=2$.
If $\I\ne \R$, one may use continuous time truncations to get back to the case $\I=\R$. Details are left to the reader. 
\end{proof}

\subsection{The kinetic Cauchy problem on finite intervals}

Let $0<T<\infty$ and set $\I=(0,T)$.
All inhomogeneous kinetic spaces are understood on $\Omega_{\I}$ as above.
We emphasise that the kinetic Cauchy problem \eqref{eq:CP} in inhomogeneous spaces is in general only well-posed on finite time intervals as the Kolmogorov semigroup is not exponentially stable. 

We now record a clean maximal-regularity formulation of the kinetic Cauchy problem
for the (fractional) Kolmogorov equation. Our notion of distributional solution is as before with  $f\in \L^p(\I\times\R^d_x\, ;\,\H^{\gamma,p}_v)$.

\begin{thm}[Cauchy problem in inhomogeneous kinetic spaces]
\label{thm:CP-inhom-Lp}
Let $\beta\in(0,1]$, $p \in (1,\infty)$, $\gamma\in[0,2\beta]$ and  $\I=(0,T)$ with $0<T<\infty$. 
Let
\[
	S\in \zZ^{\gamma,p}_\beta(\I) \qquad\text{and}\qquad \psi\in \B^{\gamma-{2\beta}/{p},p}_\beta.
\]
Then there exists a unique distributional solution $f\in \L^p(\I\times\R^d_x\, ;\,\H^{\gamma,p}_v)$
such that
\begin{equation}\label{eq:CP-inhom-Lp}
	\begin{cases}
		(\partial_t+v\cdot\nabla_x)f+(-\Delta_v)^\beta f = S, & \text{in }\cD'(\Omega_{\I}), \\
		f(0)=\psi & \text{in }\B^{\gamma-{2\beta}/{p},p}_\beta.
	\end{cases}
\end{equation}
Moreover,  $f\in \Y^{\gamma,p}_{\beta}(\I)$ and satisfies the  estimate
\begin{align}\label{eq:CP-inhom-Lp-est}
	&\|f\|_{\L^p(\I\times\R^d_x\, ;\,\H^{\gamma,p}_v)} +\|f\|_{\L^p(\I\times\R^d_v\, ;\,\H^{\frac{\gamma}{2\beta+1},p}_{\vphantom{t}x})}
+ \sup_{t\in \overline \I}\|f(t)\|_{\B^{\gamma-{2\beta}/{p},p}_\beta} \\
	&\hspace{6cm}\lesssim_{d,\beta,\gamma,p, T}\ \|S\|_{\zZ^{\gamma,p}_\beta(\I)}+\|\psi\|_{\B^{\gamma-{2\beta}/{p},p}_\beta}. \nonumber
\end{align}
The solution is given by the Duhamel formula
\[
	f(t)=\sem^+_t\psi+\int_{0}^{t} \sem^+_{t-s}S(s)\,\ds=\sem^+_t\psi+[\cK_{\beta}^+\widetilde S](t) ,\qquad t\in [0,T),
\]
where the equality is meant in $\B^{\gamma-{2\beta}/{p},p}_\beta$ and the integral is understood in the weak sense, i.e.\ $\int_{0}^t \angle{S(s)}{(\sem^+_{t-s})^*\,g}\ds$ against arbitrary test functions $g${} $($seen as a dense elements in the dual of the  Besov space$)$ and $\widetilde S$ is the extension of $S$ by $0$ outside of $\Omega_{\I}$. 
\end{thm}

\begin{rem}
From Theorem~\ref{thm:inhomkinspaceLp}, we  have $f\in \L^{q}_{t,x,v}$ for  $p\le q\le \frac{p\homd}{\homd -p\gamma}$ if $\homd -p\gamma>0$ and $p\le q<\infty$ if $\homd -p\gamma\le 0$.    By  Corollary~\ref{cor:Lebesguesestimates}, one can take in particular  $\psi\in \L^{a^\flat}_{x,v}$ and $S\in \L^a_{t,x,v}$ for  $a=\frac{p\homd}{\homd + p(2\beta-\gamma)}$ provided $1<a<\frac \homd{2\beta}$. 
\end{rem}

\begin{rem}[Backward Cauchy problem]
The same statements hold for the backward kinetic Cauchy problem
\[
	-(\partial_t+v\cdot\nabla_x)f+(-\Delta_v)^\beta f = S \quad\text{on }(0,T)\times\R^{2d},
	\qquad f(T)=\psi,
\]
with $\cK_\beta^-$, $\sem^-$ in place of $\cK_\beta^+$, $\sem^+
$ and with the obvious time-reversal modifications.
\end{rem}

\section{Beyond the restriction $\gamma \in [0,2\beta]$}
\label{sec:NZ}

The assumption $f \in \L^p_{t,x}\Hdot^{\gamma,p}_{v}$ only, which appears in the definition of the homogeneous kinetic Sobolev spaces, is motivated by the theory of distributional solutions to the Kolmogorov equations, since this is the only information we expect to use there. The definition, however, requires a restriction on the range of $\gamma$. Yet, for quasilinear equations one typically wants to bootstrap regularity all the way to $\C^\infty$. For this reason, we introduce two additional scales of kinetic Sobolev spaces to
remove the restriction $\gamma \in [0,2\beta]$. In the homogeneous setting, we must still assume that  $\gamma < \homd/p$. In the inhomogeneous setting, this restriction can also be removed and one can work with $\gamma \in \R$. Bootstrap arguments can thus be set up.

We therefore present the details for this two scales only for the homogeneous case, since the inhomogeneous case is analogous and simpler.

 Let 
\begin{equation}
 \label{eq:Tdotgammabeta}
 \cTdot^{\gamma,p}_{\beta}= \left\{ f \in \cD'(\Omega)\, : \,  f \in \L^p_{t}\Xdot^{\gamma,p}_{\beta} \ \& \ (\partial_t + v \cdot \nabla_x)f \in \L^p_{t} \Xdot^{\gamma -2\beta,p}_\beta \right\}
\end{equation}
with (semi-)norm defined by
\begin{equation}
 \label{eq:Tdotgammabetanorm}
 \|f\|_{\cTdot^{\gamma,p}_{\beta}}^p= \|  f\|_{\L^p_{t}\Xdot^{\gamma,p}_{\beta}}^p + \|(\partial_t + v \cdot \nabla_x)f\|_{\L^p_{t}\Xdot^{\gamma-2\beta,p}_{\vphantom{t} \beta}}^p,
 \end{equation}
 and \begin{equation}
 \label{eq:Udotgammabeta}
 \cUdot^{\gamma,p}_{\beta}= \left\{ f \in \cD'(\Omega)\, : \,  f \in \L^p_{t}\Xdot^{\gamma,p}_{\beta} \ \& \ (\partial_t + v \cdot \nabla_x)f \in \Zdot^{\gamma,p}_\beta \right\}
\end{equation}
with (semi-)norm defined by
\begin{equation}
 \label{eq:Udotgammabetanorm}
 \|f\|_{\cUdot^{\gamma,p}_{\beta}}^p= \|  f\|_{\L^p_{t}\Xdot^{\gamma,p}_{\beta}}^p + \|(\partial_t + v \cdot \nabla_x)f\|_{\Zdot^{\gamma,p}_{\vphantom{t} \beta}}^p.
 \end{equation}

For all $\gamma\in \R$, we have  $\cTdot^{\gamma,p}_{\beta} \subset \cUdot^{\gamma,p}_{\beta}$. 
Also by Proposition~\ref{prop:characterisation}, we have $\cTdot^{\gamma,p}_{\beta} \subset \cGdot^{\gamma,p}_{\beta}$ 
and $\cUdot^{\gamma,p}_{\beta} \subset \cLdot^{\gamma,p}_{\beta}$ if $\gamma\ge 0$ and the opposite inclusions if $\gamma\le 0$. 

The following lemma, analogous to Lemma~\ref{lem:uniquenessH} in this context, gives the range of parameters for which the operators $\pm(\partial_{t}+v\cdot\nabla_{x})+ (-\Delta_{v})^\beta  $ are well-defined in the distribution sense and injective on $\L^p_{t}\Xdot^{\gamma,p}_{\beta}.$

\begin{lem}[Uniqueness in $ \L^p_{t}\Xdot^{\gamma,p}_{\beta}$]
\label{lem:uniquenessX}
Let $\beta>0$,  $1<p<\infty$ and   $-\infty<\gamma < 2\beta+(2\beta+2)d/p$. 
If   $f\in  \L^p_{t}\Xdot^{\gamma,p}_{\beta}$   satisfies  $\pm (\partial_{t}+v\cdot\nabla_{x})f +(-\Delta_{v})^{\beta}f=0$ 
in the sense of distributions on $\Omega$, then  $f=0$. 
\end{lem}

\begin{proof} 
The proof follows the same pattern as the  one of Lemma~\ref{lem:uniquenessH} with many differences which we detail. 
Again, we only look at the case of the forward equation. We divide the argument in five steps. 

\medskip
\paragraph 
{\textbf{Step 1: meaning of the equation.}} 

First, $f\in   \L^p_{t}\Xdot^{\gamma,p}_{\beta} \subset \cS'(\Omega)$ and  the equation may be interpreted in $\cS'(\Omega)$.
Indeed, we know from the proof of  Proposition~\ref{prop:mapping} that the condition $\gamma-2\beta<(2\beta+2)d/p$ ensures that 
$(-\Delta_{v})^{\beta}f$ is defined in $\cS'(\Omega)$ as $(-\Delta_{v})^{\beta}f = D_{\beta}^{2\beta-\gamma}m_{2\beta}\, g$,
 and belongs to $\L^p_{t}\Xdot^{\gamma-2\beta,p}_{\beta}$,  where $g=D_{\beta}^{\gamma}f$,  $D_{\beta}$ being the 
 Fourier multiplier with symbol $d_{\beta}(\varphi,\xi)$, and $M_{2\beta}$ is the  Fourier multiplier with symbol 
 $m_{2\beta}(\varphi,\xi)= (|\xi|/d_{\beta}(\varphi,\xi))^{2\beta}$. 
 More precisely, we took the extension of the Fourier multiplier $D_{\beta}^{2\beta-\gamma}m_{2\beta}$ 
 from the class $\cS_{K}$ by density in $\L^p_{t,x,v}$ into $\L^p_{t}\Xdot^{\gamma-2\beta,p}_{\beta}$.  

By taking the kinetic shift and partial Fourier transform, we have
$$
	\partial_{t}\widehat{\Gamma f} =- \widehat{[\Gamma (-\Delta_{v})^{\beta}f }]= -\widehat{[\Gamma D_{\beta}^{2\beta-\gamma}M_{2\beta}\, g]}
$$ 
in $\cS'(\Omega)$, hence in $\cD'(\Omega)$.

\medskip
\paragraph 
{\textbf{Step 2: an ODE for a restriction of $\widehat{\Gamma f}$.}} 
If $O$ is  the open set of $\Omega$ defined by  $\xi-t\varphi\ne 0$, then the distributional restriction 
$\widehat{\Gamma f}_{\mid O}$ of $\widehat{\Gamma f}$ to $O$  satisfies the equation  
\begin{equation}\label{eq:ODE1}
	\partial_{t}(\widehat{\Gamma f}_{\mid O})+ |\xi-t\varphi|^{2\beta} (\widehat{\Gamma f} _{\mid O})=0
\end{equation}
in $\cD'(O)$. 
Indeed, by taking partial Fourier transform and kinetic shift from $g=D_{\beta}^{\gamma}f$ we have that 
$d_{\beta}(\varphi,\xi-t\varphi)^\gamma\,  (\widehat {\Gamma f}_{\, \mid \{d_{\beta}(\varphi,\xi-t\varphi)\ne 0\}})
= \widehat{\Gamma g}_{\, \mid \{d_{\beta}(\varphi,\xi-t\varphi)\ne 0\}}$ 
as $d_{\beta}$ is $\C^\infty$ away from the origin.  Next, as $\xi-t\varphi\ne 0$ implies 
$d_{\beta}(\varphi,\xi-t\varphi)\ne 0$ and  $ |\xi-t\varphi|^{2\beta}$ is $\C^\infty$ on $O$,  
we have  the following identities
\begin{align*}
	|\xi-t\varphi|^{2\beta} (\widehat{\Gamma f} _{\mid O})
	= d_{\beta}(\varphi,\xi-t\varphi)^{2\beta-\gamma}m_{2\beta}(\varphi,\xi-t\varphi)\, (\widehat{\Gamma g}_{\mid O})
	= -(\partial_{t}\widehat{\Gamma f})_{\mid O} =-\partial_{t}(\widehat{\Gamma f}_{\mid O}),
\end{align*}
in $\cD'(O)$, where the first two equalities come from Step 1. 

\medskip
\paragraph 
{\textbf{Step 3: solving the ODE and getting $g=0$ when $d\ge 2$.}} 

We solve the ODE for a further restriction of $\widehat{\Gamma f}_{\mid \widetilde O}$ exactly as in Step 2  of the proof of  Lemma~\ref{lem:uniquenessH}. This is a purely distributional argument.  To prove that $g=0$, we only have to replace at the very end $|\xi-t\varphi|$ by $d_{\beta}(\varphi,\xi-t\varphi)$ in the relation between $\widehat{\Gamma f}_{\mid \widetilde O}$ and  $\widehat{\Gamma g}_{\mid\widetilde O}$.   

\medskip
\paragraph 
{\textbf{Step 4: solving the ODE and getting $g=0$ when $d=1$.}} 

We proceed exactly the same way as in step 3  of the proof of  Lemma~\ref{lem:uniquenessH}.  
The only change is in the list (i)--(iii) which should be as follows. 
Assume  $\gamma< 2\beta+ (2\beta+2)d/p$ and let $g\in \L^p_{t,x,v}$ be defined as above.
\begin{enumerate}
\item If $\chi\in \cD_{\beta}(O_{2})$ then $d_{\beta}(\varphi,\xi-t\varphi)^{-\gamma}\chi$ is the (partial) Fourier transform of 
a function in $\L^{p'}_{t,x,v}$ whose norm is controlled by a semi-norm 
for the function $\theta=|\xi-t\varphi|^{-2\beta}\chi \in \cD(O_{2})$.
\item Let $f\in \L^p_{t}\Xdot^{\gamma,p}_{\beta}$ and $\chi\in \cD_{\beta}(O_{2})$. 
\item If  $\chi_{n}$,  $n\in \N$, is defined via a normalized anisotropic Littlewood--Paley family $(\theta_{j})$ of 
Section~\ref{sec:anhomSobspaces} as $\widehat \Gamma \chi_{n}= \sum_{|j|\le n} \widehat {\theta_{j}}\widehat \Gamma \chi $,
then $\lim_{n\to \infty} (\widehat{\Gamma f}, \chi_{n})$ exists, and  setting $(\widehat{\Gamma f}, \chi)$ its limit, we have   
$$
	(\widehat{\Gamma f}, \chi)= (\widehat{\Gamma g}, d_{\beta}^{-\gamma}\chi)
$$ 
is independent of the choice of $\chi_{n}$ and  
$|(\widehat{\Gamma f}, \chi)|  \le C_{\chi}\|f\|_{\L^p_{t}\Xdot^{\gamma,p}_{\vphantom{t} \beta}}$ with 
$C_{\chi}\sim  \|h\|_{\L^{p'}_{t,x,v}}$ where $h=\cF^{-1}(d_{\beta}^{-\gamma}\widehat \Gamma\chi)$. 
\item If $f\in \L^p_{t}\Xdot^{\gamma,p}_{\beta}$ and $\chi\in \cD(O_{2})$, then 
$(\widehat{[\Gamma (-\Delta_{v})^{\beta}f]}, \chi)= (\widehat{\Gamma f}, |\xi-t\varphi|^{2\beta}\chi)$.
\end{enumerate}

We only indicate some details for (i), as the proofs of (ii), (iii) and (iv) follow the same patterns. 
We write 
\[
	d_{\beta}^{-\gamma}\widehat \Gamma \chi  = m_{2\beta} d_{\beta}^{2\beta-\gamma}\widehat \Gamma \theta
\]
with $\theta\in \cD(O_{2})$.  
We have already seen that $m_{2\beta}$ is an $\L^q_{x,v}$ Fourier multiplier for any $q\in (1,\infty)$ and similar arguments using
either Mikhlin conditions in $\varphi$ and $\xi$ separately, or anisotropic Sobolev inequalities allow us to conclude.

\medskip
\paragraph 
{\textbf{Step 5:  concluding $f=0$.}}

Since $\|f\|_{\L^p_{t}\Xdot^{\gamma,p}_{\vphantom{t}\beta}} \sim\|g\|_{\L^p_{t,x,v}}=0$, 
we have   $f\in \L^p_{t}(\cP_{k}[x,v])$ where $k=[\gamma-(2\beta+2)d/p]$. 
We are done when $k<0$ since $\cP_{k}[x,v]=\{0\}$. 
Otherwise, we can write $f(t,x,v)=\sum_{|\alpha|\le k} a_{\alpha}(t,x)v^\alpha$ where 
$a_{\alpha}\in \L^p_{\vphantom{\loc,x}t}\L^p_{\loc,x}$ because they are polynomial in $x$ with coefficients in $\L^p_{t}$. 
Let us come back to the equation. By Step 1 and $g=0$, we have    $(-\Delta_{v})^{\beta}f=0$ in $\cS'(\Omega)$. 
Hence,  $(\partial_{t}+v\cdot\nabla_{x})f=0$ and it follows that  $f(t,x,v)=T(x-tv,v)$ for some distribution $T\in \cD'(\R^{2d})$. 
But the equality $\sum_{|\alpha|\le k} a_{\alpha}(t,x)v^\alpha = T(x-tv,v)$ in $\cD'(\Omega)$ is only possible if $T=0$. 
Indeed, the coefficients $a_{\alpha}(t,x)$ can be computed on taking the partial derivative $\partial_{v}^\alpha$  
in $\cD'(\Omega)$  and evaluating at $v=0$. 
Doing this operation for $T(x-tv,v)$ yields a polynomial in $t$ with coefficients being distributions on $\R_{x}$. 
As this polynomial  belongs to $\L^p_{\vphantom{\loc,x} t}\L^p_{\loc,x}$, it must be 0. 
Therefore all $a_{\alpha}$ are 0 and thus $f=0$.
\end{proof}

Next, we look at the Kolmogorov operators. The first item in Theorem~\ref{thm:bounds}  applies and gives boundedness with parameters $\beta\in (0,1], p\in (1,\infty), \gamma\in \R$. We want to relate boundedness properties of $\cK^\pm_{\beta}$ and the corresponding differential operators. The same strategy as before gives us a comparable result to Theorem~\ref{lem:isom-embed} where the condition $\gamma\in [0,2\beta]$ drops, as well as in the list of corollaries below.  

\begin{thm}[Isomorphism] \label{lem:isom-embedNZ}
Let $\beta\in (0,1]$,    $p \in (1,\infty)$ and  $\gamma\in \R$ with $\gamma<\homd/p$.  
Then, the operators $\pm (\partial_{t}+v\cdot\nabla_{x}) +(-\Delta_{v})^\beta$ are isomorphisms from  
\begin{itemize}
\item  $\cUdot^{\gamma,p}_{\beta}$ onto $\Zdot^{\gamma,p}_{\beta}$,
\item $\cTdot^{\gamma,p}_{\beta}$ onto $\L^p_{t} \Xdot^{\gamma -2\beta,p}_{\beta}$. 
\end{itemize}
The inverses are respectively given by the Kolmogorov operators $\cK^\pm_\beta$ defined via the fundamental solutions.  
In particular we have the maximal regularity bounds
\begin{align*}
	\|\pm(\partial_{t}+v\cdot\nabla_{x})f+ (-\Delta_{v})^\beta f \|_{\Adot^{\gamma,p}_{\beta}} 
	&\sim \|f\|_{\L^p_{t} \Xdot^{\gamma,p}_{\vphantom{t} \beta}}+ \|(\partial_{t}+v\cdot\nabla_{x})f\|_{\Adot^{\gamma,p}_{\beta}}\\
	\|S\|_{\Adot^{\gamma,p}_{\beta}}   \sim \|\cK^\pm_\beta S\|_{\L^p_{t} \Xdot^{\gamma,p}_{\vphantom{t} \beta}}
	& + \|(\partial_{t}+v\cdot\nabla_{x})\cK^\pm_\beta S\|_{\Adot^{\gamma,p}_{\beta}},
\end{align*}
 where $\Adot^{\gamma,p}_{\beta}$ denotes any of the two target spaces of the isomorphisms. 
 \end{thm}

\begin{cor}[Completeness, dense class and interpolation] 
Let $\beta\in (0,1]$, $p \in (1,\infty)$ and $\gamma\in \R$ with $\gamma<\homd/p$.
The spaces $\cTdot^{\gamma,p}_{\beta}$ and  $\cUdot^{\gamma,p}_{\beta}$  are complete Banach spaces 
of tempered distributions containing the images of $\cS_{K}$ under $\cK_{\beta}^\pm$ as dense subspaces.  
The  $\cT$  spaces  interpolate by the complex method along each segment drawn in the convex region 
$(\gamma, 1/p) \in \R \times (0,1)$ defined by $\gamma<\homd/p$.
\end{cor}

\begin{cor} 
Let $\beta\in (0,1]$,    $p \in (1,\infty)$ and  $\gamma\in \R$ with $\gamma<\homd/p$.  
Then $\cUdot^{\gamma,p}_{\beta}\subset \C^{}_{0}(\R^{}_{t}\, ;\, \Bdot^{\gamma-2\beta/p,p}_{\beta})$ with   
\begin{equation}\label{eq:embedding2}
	\sup_{t\in \R}\|f(t,\cdot)\|_{\Bdot^{\gamma-2\beta/p,p}_{\beta}} 
	\lesssim_{\beta,\gamma,d,p}\| f\|_{\L^p_{t}\Xdot^{\gamma,p}_\beta}+ \|(\partial_t + v \cdot \nabla_x)f\|_{\Zdot^{\gamma,p}_\beta}.
\end{equation}
\end{cor}

\begin{cor}[Equalities between kinetic spaces] 
Let $\beta\in (0,1]$,  $p\in(1,\infty)$ and $\gamma\in [0,2\beta]$ with $\gamma<\homd/p$. 
We have $\cUdot^{\gamma,p}_{\beta}=\cLdot^{\gamma,p}_{\beta}$ and $\cTdot^{\gamma,p}_{\beta}=\cGdot^{\gamma,p}_{\beta}$.
\end{cor}
 
\begin{proof}
The inclusions $\subset$ were already observed by definition when $\gamma\ge 0$. 
Also elements $f$ in any of those spaces belong to $\Ydot^{\gamma,p}_{\beta}$ by the embedding above. 
So the other inclusions follow.
\end{proof}

\begin{cor}
The formula \eqref{eq:derivative} holds for $f\in \cUdot^{\gamma,p}_{\beta}$ and 
$\tilde f\in \cUdot^{\tilde \gamma,p'}_{\beta}$, under the condition $2\beta-\homd/p'<\gamma<\homd/p$ and with $\tilde{\gamma} = 2\beta-\gamma$.
\end{cor}

There are inhomogeneous versions of these spaces by adding the condition $f\in \L^p_{t,x,v}$ in the definitions. The isomorphism property is for the differential operators $ \pm (\partial_{t}+v\cdot\nabla_{x}) +(-\Delta_{v})^\beta +1$ in the full range $\gamma\in \R$. Such results can be used to study Cauchy problems on strips $[0,T]\times \R^{2d}$ and, in particular, to bootstrap regularity all the way to $\C^\infty$ in applications to non-linear problems.

\section{The case of local diffusion}
\label{sec:local}

In this section we provide an alternative approach to the $\L^p$-estimate in the special case of local diffusion $\beta = 1$ and weak solutions $\gamma = 1$. 
In the case of weak solutions, we want to find the minimal requirements on $S$ such that $\nabla_v f \in \L^p_{t,x,v}$. 
One choice is to consider a source term of the form $S = \nabla_v \cdot S_0$ for $S_0 \in \L^p(\R^{1+2d};\R^d)$. 
We write $S \in \L^p_{t,x}\Wdot^{-1,p}_{\vphantom{t}v}$ equipped with the norm $\norm{S}_{\L^p_{t,x}\Wdot^{-1,p}_{\vphantom{t}v}} = \inf \norm{S_0}_{\L^p_{t,x,v}}$, the infimum being taken over all such $S_{0}$.
As before, when $p=2$, $\L^2_{t,x,v}$ bounds for $\nabla_v f$ can be proved by employing the Fourier transform and the theorem of 
Plancherel; see \cite{AIN}. 
To consider $p\ne 2$, we prove a representation formula in physical variables $(t,x,v)$ and do not use Fourier transform as in Section \ref{sec:towardsLp}. 
One then proves that the resulting singular integral operator is bounded from $\L^p_{t,x,v}$ to $\L^p_{t,x,v}$ 
and we deduce the following global $\L^p_{t,x,v}$ estimate for weak solutions to the Kolmogorov equation when $\beta = 1$. 

\begin{thm} \label{thm:weakLplocal}
	Let $\beta = 1$ and $p\in (1,\infty)$. The solution $f$ given by the fundamental solution to the Kolmogorov equation with data $S=\nabla_{v}\cdot S_{0}$ for $S_{0} \in \C_c^\infty(\R^{1+2d} ; \R^d)$ satisfies
	\begin{equation*}
		\norm{(\partial_t+v\cdot \nabla_x)f}_{\L^p_{t,x}\Wdot^{-1,p}_{\vphantom{t}v}} + \norm{\nabla_v f}_{\L^p_{t,x,v}} \lesssim \norm{S}_{\L^p_{t,x}\Wdot^{-1,p}_{\vphantom{t}v}}.
	\end{equation*} 
\end{thm}

Let $G_1$ be Kolmogorov's fundamental solution
\[
	G_1(t,x,v)=\frac{3^{\frac d2}}{(2\pi)^d\,t^{2d}} \exp\!\left(-\frac{|v|^2}{4t}-\frac{3}{t^3}\Big|x-\frac t2v\Big|^2\right)\mathds 1_{(0,\infty)}(t).
\]
Moreover, we fix a homogeneous quasi-norm (smooth away from zero) for the Kolmogorov dilations
\[
	\delta_r(t,x,v):=(r^2t,r^3x,rv), \qquad \rho(t,x,v):= \left(|t|^3+|x|^{2}+|v|^6\right)^{\frac{1}{6}}.
\]

\begin{lem}\label{lem:repweak}
	Let $S_0\in \C_c^\infty(\mathbb R^{1+2d};\mathbb R^d)$ and set $S:=\nabla_v\!\cdot S_0$.
	Define $f$ by the Duhamel formula
	\begin{equation}\label{eq:f_def_convolution}
		f(t,x,v):=\int_{\mathbb R^{1+2d}}G_1\bigl(t-s,\,x-y-(t-s)w,\,v-w\bigr)\,S(s,y,w) \dd(s,y,w)
	\end{equation}
	Fix a nondecreasing cutoff $\chi\in \C^\infty([0,\infty)\, ; \, [0,1])$ such that $\chi(r)=0$ for $0\le r\le 1$ and $\chi(r)=1$ for $r\ge 2$, and define $\chi_\epsilon(t,x,v):=\chi(\rho(t,x,v)/\epsilon)$. For $\tau>0$ and $(\zeta,\eta)\in\mathbb R^{2d}$ define the $d\times d$ matrix-valued kernel
	\begin{equation}\label{eq:K_def}
		\mathcal K(\tau,\zeta,\eta):=\bigl(\mathcal K_{ij}(\tau,\zeta,\eta)\bigr)_{1\le i,j\le d},
		\qquad
		\mathcal K_{ij}(\tau,\zeta,\eta):=(\tau\partial_{\zeta_j}+\partial_{\eta_j})\,\partial_{\eta_i}G_1(\tau,\zeta,\eta),
	\end{equation}
	and set $\mathcal K(\tau,\zeta,\eta)=0$ for $\tau\le 0$.

	Then for every $(t,x,v)\in\mathbb R^{1+2d}$ one has
	\begin{align}\label{eq:repweak_lemma}
		[\nabla_v f](t,x,v) \nonumber
		&= \mathrm{p.v.}\!\int_{\mathbb R^{1+2d}} \mathcal K\bigl(t-s,\,x-y-(t-s)w,\,v-w\bigr)\,S_0(s,y,w)\dd(s,y,w) \\
		&\hphantom{=}+\mathfrak c\,S_0(t,x,v),
	\end{align}
	where the principal value is defined using the cutoff $\chi_\epsilon$ by
	\begin{align}\label{eq:pv_def}
		\lim_{\epsilon\to 0} \int_{\mathbb R^{1+2d}} [\chi_\epsilon
		\mathcal K]\bigl(t-s,\,x-y-(t-s)w,\,v-w\bigr)  \cdot S_0(s,y,w)\dd(s,y,w),	\end{align}
	and $\mathfrak c\in\mathbb R^{d\times d}$ is a constant matrix $($depending only on $d$ and on the choice of $\chi${}$)$
	given by
	\begin{equation}\label{eq:c_def}
		\mathfrak c_{ij}:=
		\int_{\mathbb R^{1+2d}} \partial_{\eta_i}G_1(\tau,\zeta,\eta)\, (\tau\partial_{\zeta_j}+\partial_{\eta_j})\chi_1(\tau,\zeta,\eta)\dd(\tau,\zeta,\eta).
	\end{equation}
\end{lem}

\begin{rem}
	We refer to \cite[Lemma 3.1]{MR1662349} for a related (localised) representation formula for higher-order Kolmogorov equations as studied in \cite{pascucci_mosers_2004}.  
\end{rem}

\begin{proof}
	Fix $(t,x,v)\in\mathbb R^{1+2d}$ and $(s,y,w)\in\mathbb R^{1+2d}$.
	Introduce the variables
	\begin{equation}\label{eq:diff_variables}
		\tau:=t-s,\qquad \zeta:=x-y-\tau w,\qquad \eta:=v-w.
	\end{equation}
	Thus the kernel appearing in \eqref{eq:f_def_convolution} is $G_1(\tau,\zeta,\eta)$.

	A direct chain-rule computation from \eqref{eq:diff_variables} yields, for any smooth $\Phi=\Phi(\tau,\zeta,\eta)$,
	\begin{equation}\label{eq:w_derivative_identity}
		\partial_{w_j}\bigl[\Phi(\tau,\zeta,\eta)\bigr] = -\bigl(\tau\partial_{\zeta_j}+\partial_{\eta_j}\bigr)\Phi(\tau,\zeta,\eta).
	\end{equation}
	Indeed, $\partial_{w_j}\tau=0$, $\partial_{w_j}\zeta=-\tau e_j$, and $\partial_{w_j}\eta=-e_j$.

	From the explicit formula for $G_1$ one checks the scaling
	\begin{equation}\label{eq:G_scaling}
		G_1(\delta_r(\tau,\zeta,\eta)) = r^{-4d}G_1(\tau,\zeta,\eta)\qquad \tau\in\mathbb R,\ \zeta,\eta\in\mathbb R^d,\ r>0.
	\end{equation}
	Differentiating \eqref{eq:G_scaling} gives, for $r>0$,
	\begin{align}\label{eq:dG_scaling}
		\partial_{\eta_i}G_1(\delta_r(\tau,\zeta,\eta)) &= r^{-(4d+1)}\,(\partial_{\eta_i}G_1)(\tau,\zeta,\eta), \nonumber \\
		\mathcal K_{ij}(\delta_r(\tau,\zeta,\eta)) &= r^{-(4d+2)}\,\mathcal K_{ij}(\tau,\zeta,\eta).
	\end{align}
	We recall the homogeneous dimension associated with the dilations $\delta_r$ is $\homd = 4d+2$.
	Under the change of variables $\zeta=(\tau,\zeta,\eta)=\delta_r(\tilde\tau,\tilde\zeta,\tilde\eta)$ one has 
	$\dd (\tau,\zeta,\eta)= r^{\homd}\dd (\tilde\tau,\tilde\zeta,\tilde\eta)$.

	We now record the key local integrability fact. For fixed $\tau>0$, $\partial_{\eta_i}G_1$ is a polynomial times a Gaussian 
	in $(\zeta,\eta)$, hence integrable. 
	As $\tau\to0$, one has $\|\partial_{\eta_i}G_1(\tau, \cdot)\|_{\L^1_{\zeta,\eta}}\lesssim \tau^{-1/2}$, 
	so the $\tau$-singularity is integrable. 
	Therefore $\partial_\eta G_1\in \L^1_{\mathrm{loc}}(\R^{1+2d})$.
	
	Hence, using \eqref{eq:dG_scaling} with $r=\epsilon$ and the dilation scaling of $\rho$, 
	$\rho(\delta_\epsilon(\tilde\tau,\tilde\zeta,\tilde\eta))=\epsilon\,\rho(\tilde\tau,\tilde\zeta,\tilde\eta)$
	we obtain, for any $\epsilon\in(0,1]$,
	\begin{align*}
		\int_{\rho(\tau,\zeta,\eta)\le 2\epsilon} |\partial_{\eta_i}G_1(\tau,\zeta,\eta)|\dd(\tau,\zeta,\eta)
		&= \int_{\rho(\tilde\tau,\tilde\zeta,\tilde\eta)\le 2} |\partial_{\eta_i}G_1(\delta_\epsilon(\tilde\tau,\tilde\zeta,\tilde\eta))|\,\epsilon^{\homd}\dd(\tilde\tau,\tilde\zeta,\tilde\eta) \\
		&= \epsilon^{\homd}\epsilon^{-(4d+1)} \int_{\rho(\tilde\tau,\tilde\zeta,\tilde\eta)\le 2} |\partial_{\eta_i}G_1(\tilde\tau,\tilde\zeta,\tilde\eta)|\dd(\tilde\tau,\tilde\zeta,\tilde\eta) \\
		&= \epsilon^{(4d+2)-(4d+1)}\, \int_{\rho(\tilde\tau,\tilde\zeta,\tilde\eta)\le 2} |\partial_{\eta_i}G_1(\tilde\tau,\tilde\zeta,\tilde\eta)|\dd(\tilde\tau,\tilde\zeta,\tilde\eta) \\
		&\lesssim \epsilon.
	\end{align*}

	Since $S=\nabla_v\cdot S_0\in \C_c^\infty$ and $G_1$ is smooth for $\tau\neq 0$ with Gaussian decay, 
	standard dominated-convergence arguments (using the local integrability of $\partial_{\eta}G_1$) show that $f\in \C_v^1$ and,
	for each $i=1,\dots,d$,
	\begin{equation}\label{eq:dvf_first}
		[\partial_{v_i}f](t,x,v) = \int_{\mathbb R^{1+2d}} \partial_{\eta_i}G_1(\tau,\zeta,\eta)\,S(s,y,w)\dd(s,y,w).
	\end{equation}
	Fix $\epsilon\in(0,1)$ and define
	\begin{equation}\label{eq:Ieps_def}
		I_{i,\epsilon}(t,x,v) := \int_{\mathbb R^{1+2d}} \partial_{\eta_i}G_1(\tau,\zeta,\eta)\, \chi_\epsilon(\tau,\zeta,\eta)\, S(s,y,w)\dd(s,y,w).
	\end{equation}
	Because $\chi_\epsilon\to 1$ pointwise and $|\partial_{\eta_i}G_1|$ is locally integrable, while $S$ is compactly supported, we have
	\begin{equation}\label{eq:Ieps_to_dvf}
		\lim_{\epsilon\to 0} I_{i,\epsilon}(t,x,v)=[\partial_{v_i}f](t,x,v).
	\end{equation}
	Indeed,
	\[
		\bigl|[\partial_{v_i}f](t,x,v) - I_{i,\epsilon}(t,x,v)\bigr| \le \|S\|_{L^\infty}\int_{\rho(\tau,\zeta,\eta)\le 2\epsilon}|\partial_{\eta_i}G_1(\tau,\zeta,\eta)|\dd(\tau,\zeta,\eta) \lesssim \epsilon\|S\|_{L^\infty}.
	\]

	Recall that $S=\sum_{j = 1}^d \partial_{w_j}S_{0,j}$. Using \eqref{eq:Ieps_def} and integration by parts in $w$ 
	(no boundary terms since $S_0$ is compactly supported), we get
	\begin{align}
		I_{i,\epsilon}(t,x,v)
		&= \sum_{j = 1}^d\int \partial_{\eta_i}G_1(\tau,\zeta,\eta)\,\chi_\epsilon(\tau,\zeta,\eta)\,\partial_{w_j}S_{0,j}(s,y,w)\dd(s,y,w) \nonumber\\
		&= -\sum_{j = 1}^d\int \partial_{w_j}\Big(\partial_{\eta_i}G_1(\tau,\zeta,\eta)\,\chi_\epsilon(\tau,\zeta,\eta)\Big)\,S_{0,j}(s,y,w)\dd(s,y,w). \label{eq:Ieps_ibp}
	\end{align}
	Apply \eqref{eq:w_derivative_identity} to the product in \eqref{eq:Ieps_ibp}:
	\begin{align*}
		-\partial_{w_j}\Big(\partial_{\eta_i}G_1\,\chi_\epsilon\Big) &= (\tau\partial_{\zeta_j}+\partial_{\eta_j})\Big(\partial_{\eta_i}G_1\,\chi_\epsilon\Big) \\
		&=  {(\tau\partial_{\zeta_j}+\partial_{\eta_j})\partial_{\eta_i}G_1}\ \chi_\epsilon
		+\partial_{\eta_i}G_1\,(\tau\partial_{\zeta_j}+\partial_{\eta_j})\chi_\epsilon \\
		&=  \mathcal K_{ij}\ \chi_\epsilon
		+\partial_{\eta_i}G_1\,(\tau\partial_{\zeta_j}+\partial_{\eta_j})\chi_\epsilon.
	\end{align*}
	Hence \eqref{eq:Ieps_ibp} becomes
	\begin{equation}\label{eq:Ieps_split}
		I_{i,\epsilon}(t,x,v) = A_{i,\epsilon}(t,x,v)+B_{i,\epsilon}(t,x,v),
	\end{equation}
	where
	\begin{align}
		&A_{i,\epsilon}(t,x,v) \nonumber\\
		&:= \sum_{j = 1}^d\int_{\mathbb R^{1+2d}} \mathcal K_{ij}(\tau,\zeta,\eta)\,\chi_\epsilon(\tau,\zeta,\eta)\,S_{0,j}(s,y,w)\dd(s,y,w), \label{eq:A_def}\\
		&B_{i,\epsilon}(t,x,v) \nonumber\\
		&:= \sum_{j = 1}^d\int_{\mathbb R^{1+2d}} \partial_{\eta_i}G_1(\tau,\zeta,\eta)\,(\tau\partial_{\zeta_j}+\partial_{\eta_j})\chi_\epsilon(\tau,\zeta,\eta)\,S_{0,j}(s,y,w)\dd(s,y,w).
	\label{eq:B_def}
	\end{align}

	First observe that $(\tau\partial_{\zeta_j}+\partial_{\eta_j})\chi_\epsilon$ is supported in the annulus 
	$\{\epsilon\le \rho(\tau,\zeta,\eta)\le 2\epsilon\}$.

	It is convenient to change variables from $(s,y,w)$ to $(\tau,\zeta,\eta)$ for fixed $(t,x,v)$: 
	the map \eqref{eq:diff_variables} is a smooth bijection with Jacobian determinant $\pm 1$, hence
	\begin{equation}\label{eq:haar_change}
		\dd(s,y,w) = \dd(\tau,\zeta,\eta),
		\qquad
		(s,y,w) = (t-\tau,\ x-\zeta-\tau(v-\eta),\ v-\eta).
	\end{equation}
	Using \eqref{eq:haar_change}, we can rewrite \eqref{eq:B_def} as
	\begin{align}\label{eq:B_def_zeta}
		B_{i,\epsilon}(t,x,v) &= \sum_{j = 1}^d\int_{\mathbb R^{1+2d}} \partial_{\eta_i}G_1(\tau,\zeta,\eta)\,(\tau\partial_{\zeta_j}+\partial_{\eta_j})\chi_\epsilon(\tau,\zeta,\eta)\, \nonumber \\
		&\hspace{2cm} \cdot S_{0,j}\bigl(t-\tau,\,x-\zeta-\tau(v-\eta),\,v-\eta\bigr)\dd(\tau,\zeta,\eta).
	\end{align}

	We decompose the last factor into its value at $(t,x,v)$ plus a remainder:
	\[
		S_{0,j}\bigl(t-\tau,\,x-\zeta-\tau(v-\eta),\,v-\eta\bigr) = S_{0,j}(t,x,v)+\Big(S_{0,j}\bigl(t-\tau,\,x-\zeta-\tau(v-\eta),\,v-\eta\bigr)-S_{0,j}(t,x,v)\Big).
	\]
	Accordingly, write $B_{i,\epsilon}=B^{(0)}_{i,\epsilon}+R_{i,\epsilon}$ with
	\begin{align}
		&B^{(0)}_{i,\epsilon}(t,x,v) \nonumber\\ &:= \sum_{j = 1}^dS_{0,j}(t,x,v)\int_{\mathbb R^{1+2d}} \partial_{\eta_i}G_1(\tau,\zeta,\eta)\,(\tau\partial_{\zeta_j}+\partial_{\eta_j})\chi_\epsilon(\tau,\zeta,\eta)\, \dd(\tau,\zeta,\eta), \label{eq:B0_def}\\
		&\nonumber R_{i,\epsilon}(t,x,v) := \sum_{j = 1}^d\int_{\mathbb R^{1+2d}} \partial_{\eta_i}G_1(\tau,\zeta,\eta)\,(\tau\partial_{\zeta_j}+\partial_{\eta_j})\chi_\epsilon(\tau,\zeta,\eta)\\
		&\hspace{6cm}\cdot \Big(S_{0,j}(\cdots)-S_{0,j}(t,x,v)\Big)\dd(\tau,\zeta,\eta). \label{eq:R_def}
	\end{align}

	\emph{Claim 1: $R_{i,\epsilon}(t,x,v)\to 0$ as $\epsilon\to 0$.}
	On the support of $(\tau\partial_{\zeta_j}+\partial_{\eta_j})\chi_\epsilon$ we have $\rho(\tau,\zeta,\eta)\sim\epsilon$, hence the mean value theorem implies
	\[
		\big|S_{0,j}(\cdots)-S_{0,j}(t,x,v)\big|\lesssim_{S_0}\epsilon.
	\]
	Moreover, by scaling considerations we have the pointwise bound
	\[
		\big|(\tau\partial_{\zeta_j}+\partial_{\eta_j})\chi_\epsilon(\tau,\zeta,\eta)\big| \lesssim \epsilon^{-1}\mathds 1_{\{\epsilon\le \rho(\tau,\zeta,\eta)\le 2\epsilon\}}.
	\]
	Combining these estimates with the scaling of $\partial_{\eta_i}G_1$ in \eqref{eq:dG_scaling} yields
	\begin{align*}
		|R_{i,\epsilon}(t,x,v)|
		&\lesssim_{S_0}\epsilon\int_{\{\epsilon\le \rho\le 2\epsilon\}} |\partial_{\eta_i}G_1(\tau,\zeta,\eta)|\,\epsilon^{-1}\dd(\tau,\zeta,\eta) \\
		&\lesssim_{S_0} \int_{\{\epsilon\le \rho\le 2\epsilon\}} |\partial_{\eta_i}G_1(\tau,\zeta,\eta)|\dd(\tau,\zeta,\eta) \lesssim \epsilon.
	\end{align*}
	Therefore $R_{i,\epsilon}(t,x,v)\to 0$.

	\emph{Claim 2: the coefficient in \eqref{eq:B0_def} is independent of $\epsilon$.}
	Define
	\[
		c_{ij}(\epsilon) :=
		\int_{\mathbb R^{1+2d}} \partial_{\eta_i}G_1(\tau,\zeta,\eta)\,(\tau\partial_{\zeta_j}+\partial_{\eta_j})\chi_\epsilon(\tau,\zeta,\eta) \dd(\tau,\zeta,\eta).
	\]
	Let $(\tau,\zeta,\eta)=\delta_\epsilon(\tilde\tau,\tilde\zeta,\tilde\eta)$.
	Then $\dd(\tau,\zeta,\eta) = \epsilon^{\homd}\dd(\tilde\tau,\tilde\zeta,\tilde\eta)$.
	Using \eqref{eq:dG_scaling} and the fact that $(\tau\partial_{\zeta_j}+\partial_{\eta_j})$ has homogeneous degree $-1$ with respect to $\delta_r$, we obtain
	\begin{align*}
		\partial_{\eta_i}G_1(\delta_\epsilon(\tilde\tau,\tilde\zeta,\tilde\eta))&=\epsilon^{-(4d+1)}\partial_{\eta_i}G_1(\tilde\tau,\tilde\zeta,\tilde\eta), \\
		(\tau\partial_{\zeta_j}+\partial_{\eta_j})\chi_\epsilon(\delta_\epsilon(\tilde\tau,\tilde\zeta,\tilde\eta))&=\epsilon^{-1}(\tilde\tau\partial_{\tilde\zeta_j}+\partial_{\tilde\eta_j})\chi_1(\tilde\tau,\tilde\zeta,\tilde\eta).
	\end{align*}
	Hence
	\begin{align*}
		c_{ij}(\epsilon)
		&= \int_{\mathbb R^{1+2d}} \epsilon^{-(4d+1)}\partial_{\tilde\eta_i}G_1(\tilde\tau,\tilde\zeta,\tilde\eta)\, \epsilon^{-1}(\tilde\tau\partial_{\tilde\zeta_j}+\partial_{\tilde\eta_j})\chi_1(\tilde\tau,\tilde\zeta,\tilde\eta)\,\epsilon^{\homd}\dd (\tilde\tau,\tilde\zeta,\tilde\eta)\\
		&= \epsilon^{\homd-(4d+2)}\int_{\mathbb R^{1+2d}}\partial_{\tilde\eta_i}G_1(\tilde\tau,\tilde\zeta,\tilde\eta)\,(\tilde\tau\partial_{\tilde\zeta_j}+\partial_{\tilde\eta_j})\chi_1(\tilde\tau,\tilde\zeta,\tilde\eta)\,d(\tilde\tau,\tilde\zeta,\tilde\eta).
	\end{align*}
	Therefore $c_{ij}(\epsilon)=c_{ij}(1)$ for all $\epsilon>0$.
	Define $\mathfrak c_{ij}=c_{ij}(1)$ as in \eqref{eq:c_def}.
	Consequently,
	\[
		B^{(0)}_{i,\epsilon}(t,x,v)= \sum_{j = 1}^d \mathfrak c_{ij}\,S_{0,j}(t,x,v)\qquad\text{for all }\epsilon>0.
	\]
	Together with Claim~1 we have shown
	\begin{equation}\label{eq:B_limit}
		\lim_{\epsilon\to 0}B_{i,\epsilon}(t,x,v)=\mathfrak c_{ij}\,S_{0,j}(t,x,v).
	\end{equation}

	Combine \eqref{eq:Ieps_to_dvf}, \eqref{eq:Ieps_split}, and \eqref{eq:B_limit}:
	\[
		\partial_{v_i}f(t,x,v) = \lim_{\epsilon\to 0}I_{i,\epsilon}(t,x,v) = \lim_{\epsilon\to 0}A_{i,\epsilon}(t,x,v) + \mathfrak c_{ij}\,S_{0,j}(t,x,v).
	\]
	Therefore the limit $\lim_{\epsilon\to 0}A_{i,\epsilon}(t,x,v)$ exists.
\end{proof}

It remains to prove that this singular integral operator $T_{\varepsilon}$ is bounded on $\L^p_{t,x,v}$. To save some writing we adapt the strategy of Section \ref{sec:towardsLp}, Section \ref{sec:estimates} and Section \ref{sec:horm}. First, we conjugate with the kinetic shift $\Gamma$ to obtain a singular integral operator, which is convolution in $(x,v)$.  
Let $\mathcal K(\tau,\zeta,\eta)$ be the matrix-valued kernel from \eqref{eq:K_def} in Lemma~\ref{lem:repweak}, extended by $0$ for $\tau\le 0$.
For $\epsilon>0$ define the truncated operators
\begin{align} \label{eq:T_eps_def}
	[T_\epsilon g](t,x,v)  =
	\int_{\R^{1+2d}} [\chi_\epsilon \mathcal K]\!\bigl(t-s,\ x-y-(t-s)w,\ v-w\bigr)  
	g(s,y,w) \dd(s,y,w), 
\end{align}
and define the principal value operator $T$ by $Tg:=\lim_{\epsilon\to 0}T_\epsilon g$ whenever the limit exists pointwise (in particular it exists for $g\in \C_c^\infty$ by Lemma~\ref{lem:repweak}).

\begin{lem}\label{lem:conj_conv}
Let $\widetilde T_\epsilon:=\Gamma T_\epsilon \Gamma^{-1}$ and $\widetilde T:=\Gamma T\Gamma^{-1}$.
For $t,s\in\R$ and $(x,v)\in\R^d\times\R^d$ define
\begin{equation}\label{eq:kernel_ts_def}
	\widetilde{\mathcal K}_\epsilon(t,s,x,v)
	:=
	[\chi_\epsilon\mathcal K](t-s,\ x+t v,\ v),
	\qquad
	\widetilde{\mathcal K}(t,s,x,v)
	:=
	\mathcal K(t-s,\ x+t v,\ v).
\end{equation}
Then for every $g\in \C_c^\infty(\R^{1+2d};\R^d)$ and every $(t,x,v)\in\R^{1+2d}$,
\begin{equation}\label{eq:Ttilde_eps_conv}
	[\widetilde T_\epsilon g](t,x,v)
	=
	\int_{\R}\Bigl[\bigl(\widetilde{\mathcal K}_\epsilon(t,s,\cdot,\cdot)\ast_{x,v} g(s,\cdot,\cdot)\bigr)(x,v)\Bigr]\dd s.
\end{equation}
Consequently,
\begin{equation}\label{eq:Ttilde_conv_form}
	[\widetilde T g](t,x,v)
	=
	\mathrm{p.v.}\!\int_{\R}\Bigl[\bigl(\widetilde{\mathcal K}(t,s,\cdot,\cdot)\ast_{x,v} g(s,\cdot,\cdot)\bigr)(x,v)\Bigr]\dd s,
\end{equation}
where the principal value is taken with the truncation $\widetilde{\mathcal K}_\epsilon$.
\end{lem}

\begin{proof}
Fix $g\in \C_c^\infty(\R^{1+2d};\R^d)$ and $(t,x,v)\in\R^{1+2d}$.
By definition,
\[
	[\widetilde T_\epsilon g](t,x,v)
	=
	[\Gamma T_\epsilon \Gamma^{-1}g](t,x,v)
	=
	[T_\epsilon(\Gamma^{-1}g)](t,x+t v,v).
\]
Using \eqref{eq:T_eps_def} and the identity
\[
	[\Gamma^{-1}g](s,y,w)=g(s,y-s w,w),
\]
we get
\begin{align*}
	[\widetilde T_\epsilon g](t,x,v)
	&=
	\int_{\R^{1+2d}}
	[\chi_\epsilon\mathcal K]\bigl(t-s,\ (x+t v)-y-(t-s)w,\ v-w\bigr)\,
	g(s,y-s w,w)\dd(s,y,w).
\end{align*}
Now perform the change of variables
\[
	\bar y:=y-s w,
	\qquad y=\bar y+s w,
	\qquad \dd y=\dd \bar y.
\]
Then
\[
	(x+t v)-y-(t-s)w
	=
	(x+t v)-(\bar y+s w)-(t-s)w
	=
	x-\bar y+t(v-w).
\]
Hence
\begin{align*}
	[\widetilde T_\epsilon g](t,x,v)
	&=
	\int_{\R^{1+2d}}
	[\chi_\epsilon\mathcal K]\bigl(t-s,\ x-\bar y+t(v-w),\ v-w\bigr)\,
	g(s,\bar y,w)\dd(s,\bar y,w)\\
	&=
	\int_{\R^{1+2d}}
	\widetilde{\mathcal K}_\epsilon(t,s,x-\bar y,v-w)\,
	g(s,\bar y,w)\dd(s,\bar y,w),
\end{align*}
which is exactly \eqref{eq:Ttilde_eps_conv}. Taking the limit $\epsilon\to0$ gives
\eqref{eq:Ttilde_conv_form}.
\end{proof}

Next, we provide the following kernel estimates on $\mathcal K$. 

\begin{lem}
	For every $N\in\N$ there exists $C_{N,d}>0$ such that for all $s<t$ and all $x,v\in\R^d$,
	\begin{equation}\label{eq:Ktilde_pointwise_poly}
		|\widetilde{\mathcal K}(t,s,x,v)|
		\le
		C_{N,d}\,(t-s)^{-(2d+1)}
		\left(1+\frac{|x+s v|}{(t-s)^{3/2}}+\frac{|v|}{(t-s)^{1/2}}\right)^{-N},
	\end{equation}
	where $|\widetilde{\mathcal K}|$ denotes any fixed matrix norm.
\end{lem}

\begin{proof}
	Set $\tau:=t-s>0$. By the homogeneity of $\mathcal K$,
	\[
		\mathcal K_{ij}(\tau,\zeta,\eta)
		=
		\tau^{-(2d+1)}
		\mathfrak K_{ij}\!\left(\frac{\zeta}{\tau^{3/2}},\frac{\eta}{\tau^{1/2}}\right),
		\qquad
		\mathfrak K_{ij}(X,V):=\mathcal K_{ij}(1,X,V).
	\]
	Since $\mathcal K_{ij}(1,\cdot,\cdot)$ is a polynomial times a Gaussian, we have
	$\mathfrak K_{ij}\in\mathcal S(\R^{2d})$. Therefore, for every $N\in\N$, $|\mathfrak K_{ij}(X,V)|
		\le
		C_{N,d}(1+|X|+|V|)^{-N}$.
	Using \eqref{eq:kernel_ts_def}, this yields
	\[
		|\widetilde{\mathcal K}_{ij}(t,s,x,v)|
		\le
		C_{N,d}\,\tau^{-(2d+1)}
		\left(1+\frac{|x+t v|}{\tau^{3/2}}+\frac{|v|}{\tau^{1/2}}\right)^{-N}.
	\]
	Finally,
	\[
		\frac{|x+t v|}{\tau^{3/2}}
		\le
		\frac{|x+s v|}{\tau^{3/2}}+\frac{|v|}{\tau^{1/2}},
		\qquad
		\frac{|x+s v|}{\tau^{3/2}}
		\le
		\frac{|x+t v|}{\tau^{3/2}}+\frac{|v|}{\tau^{1/2}},
	\]
	so the weights with $x+t v$ and $x+s v$ are equivalent. This gives
	\eqref{eq:Ktilde_pointwise_poly}.
\end{proof}

\begin{proof}[Proof of Theorem \ref{thm:weakLplocal}]
	With these estimates at hand it is clear that we obtain integrated bounds as in Proposition~\ref{prop:intkernel} and that we may verify H\"ormander's condition with the same calculations provided already in Section \ref{sec:horm} to prove $\L^p_{t,x,v}$-boundedness of $T_\epsilon$. 
	Details are left to the reader.
\end{proof}

\begin{rem}
	The same strategy works for strong solutions when $\gamma = 2$ and more generally for $\gamma \in \Z$. 
\end{rem}

\begin{rem}
	This approach is elegant when working with local diffusion and integer derivatives. 
	However, as soon as we are concerned with the transfer-of-regularity (which requires to treat fractional derivatives) or non-local diffusion $\beta \in (0,1)$, it is unclear to us how to handle the calculations with this strategy  and the Fourier approach seems to be more adapted. 
\end{rem}

\appendix
\section{Singular integrals on spaces of  homogeneous type}
\label{sec:cw}
We recall the Coifman-Weiss theorem \cite{MR499948}. We give a version with fewer assumptions, which we have not seen explicitly in the literature in this context.  
The proof in \cite{MR499948} was given before the representation of singular integrals by kernels was fully understood. 
But nonetheless the proof applies to our statement.  
Let us consider a space $(X,\rho,\mu)$ of homogeneous type. That is, $X$ is a set equipped with a quasi-distance $\rho$ such that $X$ endowed with the topology induced by balls $B(x,r)=\{y\in X; \rho(x,y)<r\}$  is a Hausdorff space. 
Moreover, $\mu$ is assumed to be a doubling Borel measure.

Consider a linear, bounded operator $T \colon \L^2(X) \to \L^2(X)$ for which  there exists a measurable kernel $K \colon X \times X \setminus \{ (x,x) : x \in X \} \to \R$ (or, even,  the space of bounded linear operators between two Hilbert spaces) such that 
\begin{enumerate}
	\item  $K$ is locally integrable on $X \times X \setminus \{ (x,x) : x \in X \}$, 
	\item  $K$ satisfies H\"ormander's condition, i.e. there exists a constant $C>0$ and $N>0$ such that for all $r_0>0$ and almost any $z,z' \in X$ with $\rho(z,z')\le r_0$ we have
	\begin{equation*}
		\int_{\rho(x,z)\ge Nr_0} \abs{K(x,z)-K(x,z')} \dd \mu(x) \le C, 
			\end{equation*}
	\item for all $f,g\in \L^\infty(X)$ supported in sets $A$, $B$ with finite measure and $\rho(A,B)>0$, then 
	\begin{equation*}
	(g,Tf) = \iint_{X\times X} K(x,y) f(y)g(x)  \dd \mu(x) \dd \mu(y).
	\end{equation*}
\end{enumerate}
The first item implies that the integral in the third item is well-defined. And then, the equality linking $T$ and $K$ in third item implies the representation  
\begin{equation*}
	Tf(x) = \int_{X} K(x,y) f(y)   \dd \mu(y),
	\end{equation*}
	for almost every $x\notin A$. 
\medskip
The Coifman-Weiss theorem, which extends H\"ormander's theorem in this context, states the following. 
\begin{thm}
\label{thm:CW} 
	Any such operator $T$  admits a bounded extension to $\L^p(X)$ for $p \in (1,2)$, with bounds depending only on the structural constants for $X$, the $\L^2$ operator norm for $T$, $C$ and $N$ in $($ii$)$.  
\end{thm}

\begin{proof}
	The  proof  in \cite{MR499948} gives the weak-type $(1,1)$ estimate for all $f$ as above.  Note that their assumption $K \in \L^2(X \times X)$ is not used quantitatively, but  only qualitatively  to write a  representation  in order to apply (ii). Thus assuming (iii) suffices. Then, 
	Marcinkiewicz interpolation implies that for all such $f$, $Tf\in \L^p(X)$ with $\|Tf\|_{p}\lesssim \|f\|_{p}$ for some constant independent of $f$. We conclude by density. 
\end{proof}

\begin{rem} \label{rem:CW}
 Of course, the dual condition to (ii) exchanging the roles of the variables, implies the dual conclusion with $2<p<\infty$. 
\end{rem}

\bibliographystyle{plain}

\end{document}